\documentclass[10pt]{article}

\usepackage{amsmath,amsthm}
\usepackage{amssymb}
\usepackage{amsfonts}
\usepackage{amscd}
\font \manual=manfnt at 7pt
\font \sc=cmcsc10 at 9truept
\font \smallrm=cmr10 at 8truept
\font \smallbf=cmbx10 at 8truept
\font \smallsl=cmsl10 at 8truept

\newtheorem{thm}{Theorem}[section]
\newtheorem{theorem}[thm]{Theorem}

\newtheorem{corollary}[thm]{Corollary}
\newtheorem{lemma}[thm]{Lemma}
\newtheorem{proposition}[thm]{Proposition}

\theoremstyle{definition}
\newtheorem{definition}[thm]{Definition}
\newtheorem{example}[thm]{Example}
\newtheorem{examples}[thm]{Examples}
\newtheorem{remark}[thm]{Remark}
\newtheorem{remarks}[thm]{Remarks}

\newtheorem{free text}[thm]{}

\newcommand{\dbend} {$ {}^{\hbox{\manual \char127}} $}
\newcommand{\N} {\mathbb{N}}
\newcommand{\Z} {\mathbb{Z}}

\newcommand{\KK} {\mathbb{K}}
\newcommand{\bk} {\Bbbk}

\newcommand{\bt} {\mathbf{t}}
\newcommand{\bu} {\mathbf{u}}
\newcommand{\bG} {\mathbf{G}}

\newcommand{\bT} {\mathbf{T}}

\newcommand{\tS} {\widetilde{S}}
\newcommand{\tH} {\widetilde{H}}

\newcommand{\cE} {\mathcal{E}}
\newcommand{\cL} {\mathcal{L}}
\newcommand{\cM} {\mathcal{M}}
\newcommand{\cO}{\mathcal{O}}

\newcommand{\fg} {\mathfrak{g}}
\newcommand{\fh} {\mathfrak{h}}
\newcommand{\fb} {\mathfrak{b}}

\newcommand{\talpha} {\widetilde{\alpha}}
\newcommand{\tbeta} {\widetilde{\beta}}
\newcommand{\tgamma} {\widetilde{\gamma}}
\newcommand{\teta} {\widetilde{\eta}}
\newcommand{\tDelta} {\widetilde{\Delta}}

\newcommand{\zero} {{\bar{0}}}
\newcommand{\uno} {{\bar{1}}}

\newcommand{\bigtimes} {\diagup \hskip-9pt \diagdown}

\newcommand{\kzg} {K_{\Z}(\fg)}

\newcommand{\kzgz} {K_\Z(\fg_0)}

\newcommand{\alg} {\mathrm{(alg)}}
\newcommand{\salg} {\mathrm{(salg)}}

\newcommand{\sets}{{(\mathrm{sets})}}
\newcommand{\grps} {\mathrm{(groups)}}
\newcommand{\lie} {\mathrm{(Lie)}}
\newcommand{\modul} {\mathrm{(mod)}}
\newcommand{\spec}{{\hbox{\sl Spec}\,}}
\newcommand{\uspec}{\underline{\hbox{\sl Spec}\,}}

\newcommand{\Ad}{\hbox{\sl Ad}}
\newcommand{\ad}{\hbox{\sl ad}}
\newcommand{\Hom}{\hbox{\sl Hom}}
\newcommand{\End}{\hbox{\sl End}}
\newcommand{\Der}{\hbox{\sl Der}}
\newcommand{\Lie}{\hbox{\sl Lie}}

\newcommand{\rGL}{\mathrm{GL}}
\newcommand{\rgl}{\mathfrak{gl}}
\newcommand{\rsl}{\mathfrak{sl}}
\newcommand{\rso}{\mathfrak{so}}

\begin{document}

{\ }

\vskip-49pt

   \centerline{ {\smallsl Forum Mathematicum\/}  {\smallbf 26}  (2014), no.\ 5, 1473--1564   \ \ \ --- \ \ \  {\smallbf DOI:}  10.1515/forum-2011-0144 }
 \vskip1pt
   \centerline{\smallrm {\smallsl The original publication is available at\/}
\  www.degruyter.com }

\vskip25pt   {\ }

\centerline{\Large \bf ALGEBRAIC SUPERGROUPS}
 \vskip7pt
\centerline{\Large \bf OF CARTAN TYPE}

\vskip17pt

\centerline{ F. Gavarini }

\vskip7pt

\centerline{\it Dipartimento di Matematica, Universit\`a di Roma ``Tor Vergata'' } \centerline{\it via della ricerca scientifica 1  --- I-00133 Roma, Italy}

\centerline{{\footnotesize e-mail: gavarini@mat.uniroma2.it}}

\vskip29pt

\begin{abstract}
 \vskip2pt
 \hskip-9pt   \footnote{\ 2010 {\it MSC}\;: \, Primary 14M30, 14A22; Secondary 17B20.}
%
   I present a construction of connected affine algebraic supergroups  $ \bG_V $  associated with simple Lie superalgebras  $ \fg $  of Cartan type and with  $ \fg $--modules  $ V $.  Conversely, I prove that every connected affine algebraic supergroup whose tangent Lie superalgebra is of Cartan type is necessarily isomorphic to one of the supergroups  $ \, \bG_V $  that I introduced.  In particular,
the supergroup associated in this way with  $ \, \fg = W(n) \, $  and its standard representation is described.
\end{abstract}

\vskip5pt

   \centerline{\sc Dedicated to Pierre Cartier,}
 \vskip1pt
   \centerline{\sc with great admiration, on the occasion of his 80th birthday.}

 \vskip35pt

\section{Introduction}

\smallskip

   {\ } \quad   A real milestone in classical Lie theory is the celebrated classification theorem for complex finite dimensional simple Lie algebras.  A similar key result is the classification of all complex finite dimensional simple Lie superalgebras  (cf.~\cite{ka});  in particular, this ensures that these objects form two disjoint families: those of  {\sl classical\/}  type, and those of  {\sl Cartan\/}  type.  The ``classical'' ones are strict super-analogue of simple, f.d.~complex Lie algebras; the ``Cartan'' ones instead are a super-analogue of complex Lie algebras of Cartan type, which are simple but  {\sl infinite\/}  dimensional.

\smallskip

   As in the standard Lie context, one can base upon this classification result to tackle the classification problem of existence, construction and uniqueness of simple Lie supergroups, or even simple algebraic supergroups.  A super-analogue of Lie's Third Theorem solves it for Lie supergroups: but the question remains for construction and for the whole algebraic point of view.

\smallskip

   In the standard context, a constructive procedure providing all (f.d., connected) simple algebraic groups was provided by Chevalley, over fields; one starts with a (complex) f.d.~simple Lie algebra  $ \fg \, $,  a faithful  $ \fg $--module  $ V $,  and eventually realizes a group of requested type as a subgroup of  $ \rGL(V) \, $.  In particular, this yields all connected algebraic groups whose tangent Lie algebra is a (f.d.) simple one; this method (and result) also extends to the framework of reductive  $ \Z $--group  schemes.
 By analogy,
 one might try to adapt Chevalley's method to the f.d.~simple Lie superalgebras of classical type, so to provide connected algebraic supergroup-schemes (over  $ \Z $)  which ``integrate'' any such Lie superalgebra.  This is done in  \cite{fg2}   --- see also  \cite{fg1}  and  \cite{ga}.  In this paper instead I implement Chevalley's idea to simple Lie superalgebras of Cartan type, with full success: the main result is an existence result, via a constructive procedure, for connected, algebraic supergroup-schemes (over any ring, e.g.~$ \Z $)  whose tangent Lie superalgebra be simple of Cartan type.
%
   As a second result, I prove also a uniqueness theorem
 for algebraic supergroups of the above mentioned type.

\smallskip

   Hereafter I shortly sketch how the present work is organized.

\smallskip

   The initial datum is a f.d.~simple Lie superalgebra of Cartan type, say  $ \fg \, $.  Basing upon a detailed description of the root spaces (with respect to a fixed Cartan subalgebra), I introduce the key notion of  {\sl Chevalley basis}.  Then I prove two basic results: the existence of Chevalley bases, and a PBW-like theorem for the ``Kostant  $ \Z $--form''  of the universal enveloping superalgebra of  $ \fg \, $.
                                                           \par
   Next I take a faithful  $ \fg $--module $ V $,  and I show that there exists a lattice $ M $  in  $ V $  fixed by the Kostant superalgebra and also by a certain integral form  $ \fg_V $  of  $ \fg \, $.  I define a functor  $ G_V \, $  from the category  $ \salg_\bk $  of commutative  $ \bk $--superalgebras  to the category  $ \grps $  of groups as follows: for  $ \, A \in \salg_\bk \, $,  I let  $ G_V(A) $  be the subgroup of  $ \rGL \big( A \otimes_\Z \! M \big) \, $  generated by ``homogeneous one-parameter subgroups'' associated with the root vectors and with the toral elements in a Chevalley basis.  Then I pick the sheafification  $ \bG_V $  (in the sense of category theory) of the functor  $ G_V \, $.
                                                           \par
   Using commutation relations among generators, I find a factorization of  $ \bG_V \, $  into direct product of representable (algebraic) superschemes: thus  $ \bG_V $  itself is representable, hence it is an ``affine (algebraic) supergroup''.  Some extra work shows how  $ \bG_V $  depends on  $ V $,  that it is independent of the choice of  $ M $ and that its tangent Lie superalgebra is  $ \fg_V \, $.  So the construction of  $ \bG_V $  yields an  {\sl existence theorem\/}  of a supergroup having  $ \fg_V $  as tangent Lie superalgebra.  Right after, I prove the converse, i.e.~a  {\sl uniqueness theorem\/}  showing that any such supergroup is isomorphic to some  $ \bG_V \, $.

\smallskip

   Finally, I illustrate the example of  $ \bG_V $  for  $ \fg $  of type  $ W(n) $  and  $ V $  its defining representation   --- i.e.~the Grassmann algebra in  $ n $  odd indeterminates,  $ W(n) $  being the algebra of its superderivations.

\vskip13pt

   \centerline{\bf Acknowledgements}
 \vskip2pt
   \centerline{The author thanks P.~Cartier, M.~Duflo, R.~Fioresi, V.~Serganova for many useful
conversations.}

\bigskip

\section{Preliminaries}  \label{preliminaries}

\smallskip

   {\ } \quad   We introduce hereafter some preliminaries of supergeometry (main references are \cite{dm},  \cite{ma}, \cite{vsv}).

  \subsection{Superalgebras, superspaces, supergroups}  \label{first_preliminaries}

\smallskip

   {\ } \quad   Let  $ \bk $  be a unital, commutative ring.
%
%
   We call  {\it  $ \bk $--superalgebra\/}  any associative, unital  $ \bk $--algebra  $ A $  which is  $ \Z_2 $--graded,  where  $ \Z_2 $  is the two-element group  $ \, \Z_2 := \big\{ \zero, \uno \big\} \, $:  thus  $ \, A = A_\zero \oplus A_\uno \, $  and  $ \, A_{\overline{a}} \, A_{\overline{b}} \subseteq A_{\overline{a}+\overline{b}} \; $.  The  $ \bk $--submodule  $ A_\zero $  and its elements are called  {\it even},  $ A_\uno $  and its elements  {\it odd}\,.  By  $ \, p(x) $  $(\in \Z_2) \, $  we denote the  {\sl parity\/}  of any homogeneous element  $ \, x \in A_{p(x)} \, $.  All  $ \bk $--super\-algebras  form a category, whose morphisms are those in the category of algebras preserving the unit and the  $ \Z_2 $--grading.  For any  $ \, n \in \N \, $  we call  $ A_\uno^{\,n} $  the  $ A_\zero \, $--span  in  $ A $  of all products  $ \, \vartheta_1 \cdots \vartheta_n \, $  with  $ \, \vartheta_i \in A_\uno \, $  for all  $ i \, $,  and  $ A_\uno^{(n)} $  the unital subalgebra of  $ A $  generated by  $ A_\uno^{\,n} \, $.
 A superalgebra  $ A $  is  {\it commutative\/}  iff  $ \; x y = (-1)^{p(x) p(y)} y x \; $  for all homogeneous  $ \, x $,  $ y \in A \; $  and  $ \; z^2 = 0 \; $  for all odd  $ \, z \in A_\uno \; $.  We denote  by $ \salg_\bk $  the category of commutative  $ \bk $--superalgebras,  dropping the subscript  $ \bk $  if unnecessary.

\vskip9pt

\begin{definition}
  A  {\it superspace}  $ \, S = \big( |S|\,, \cO_S \big) \, $  is a topological space  $ |S| $
%
%
 with a sheaf of commutative superalgebras  $ \cO_S $  such that the stalk  $ \cO_{S,x} $  is a local superalgebra for all  $ \, x \in |S| \, $.  A  {\it morphism}  $ \; \phi: S \longrightarrow T \; $  of superspaces consists of a pair $ \; \phi = \big( |\phi|, \phi^* \big) \, $,  where  $ \; \phi : |S| \longrightarrow |T| \; $  is a morphism of topological spaces and  $ \; \phi^* : \cO_T \longrightarrow \phi_* \cO_S \; $
   \hbox{is a sheaf morphism such that}
 $ \; \phi_x^* \big( {\mathbf{m}}_{|\phi|(x)} \big) \subseteq {\mathbf{m}}_x \; $  where  $ {\mathbf{m}}_{|\phi|(x)} $ and  $ {\mathbf{m}}_{x} $  are the maximal ideals in the stalks $ \cO_{T, \, |\phi|(x)} $  and  $ \cO_{S,x} \, $,  and  $ \phi_x^* $  is the morphism induced by  $ \phi^* $  on the stalk.  Here as usual  $ \phi_*\cO_S $  is the direct image (on  $ |T| $)  of  $ \cO_S(V) \, $.
\end{definition}

\vskip2pt

  Given a superspace  $ \, S = \big( |S| \, , \cO_S \big) \, $,  let  $ \cO_{S,\zero} $  and  $ \cO_{S,\uno} $  be the sheaves on  $ |S| $  defined as follows:  $ \; \cO_{S,\zero}(U) := {\cO_{S}(U)}_\zero \; $,  $ \; \cO_{S,\uno}(U) := {\cO_{S}(U)}_\uno \; $  for each open subset  $ U $  in  $ |S| \, $.  Then  $ \cO_{S,\zero} $  is a sheaf of ordinary commutative algebras, while  $ \cO_{S,\uno} $  is a sheaf of  $ \cO_{S,\zero} \, $--modules.

\vskip10pt

\begin{definition}
   A  {\it superscheme\/}  is a superspace  $ S := \big( |S| \, , \cO_S \big) \, $  such that  $ \big( |S| \, , \cO_{S,\zero} \big) \, $  is an ordinary scheme and  $ \cO_{S,\uno} $  is a quasi-coherent sheaf of  $ \cO_{S,\zero} \, $--modules.  A  {\it morphism\/}  of superschemes is one of the underlying superspaces.
   The  {\it (super)dimension\/}  of  $ S $  is by definition the pair  $ \; \text{\it dim}\,(|S|) \,\big|\, \text{\it rk}\,(\cO_{S,\uno}) \; $  where  $ \text{\it rk}\,(\cO_{S,\uno}) $  is the rank of the quasi-coherent sheaf of  $ \cO_{S,\zero} \, $--modules  $ \cO_{S,\uno} \; $.
\end{definition}

\vskip6pt

\begin{definition}  \label{spec}
  Let  $ \, A \in \salg_\bk \, $  and let  $ \cO_{A_\zero} $  be the structural sheaf of the ordinary scheme  $ \, \uspec(A_\zero) = \big( \spec(A_\zero) \, , \cO_{A_\zero} \big) \, $,  where  $ \spec(A_\zero) $  denotes the prime spectrum of
%
%
 $ A_\zero \; $.  Now  $ A $  is an  $ A_\zero $--module,  so we have a sheaf  $ \cO_A $  of  $ \cO_{A_\zero} $--modules  over  $ \spec(A_\zero) $  with stalk  $ A_p \, $,  the  $ p $--localization of the  $ A_\zero \, $--module  $ A \, $,  at any  $ \, p \in \spec(A_\zero) \, $.
   We set  $ \; \uspec(A) := \big( \spec(A_\zero) \, , \cO_A \big) \; $:  by definition, this is a superscheme.  We call  {\it affine\/}  any superscheme which is isomorphic to  $ \, \uspec(A) \, $  for some  $ \, A \in \salg_\bk \, $;  any affine supercheme is  {\it algebraic\/}  if its representing superalgebra is finitely generated.
\end{definition}

\smallskip

   Clearly any superscheme is locally isomorphic to an affine superscheme.

\begin{example}
   We call  {\sl affine superspace\/}  the superscheme  $ \, \mathbb{A}_\bk^{p|q} \! := \uspec \big( \bk[x_1,\dots,x_p] \otimes_\bk \bk[\xi_1 \dots \xi_q] \big) \, $  ($ \, p , q \in \N \, $),  also denoted  $ \, \bk^{p|q} \, $:  here  $ \bk[\xi_1 \dots \xi_q] $  is the exterior
 algebra generated by  $ q $  anticommuting indeterminates, and  $ \, \bk[x_1,\dots,x_p] \, $  the polynomial algebra in  $ p $  commuting indeterminates.
%
%
 \hfill  $ \diamondsuit $
\end{example}

\smallskip

\begin{definition}
   Let  $ X $  be a superscheme.  Its  {\it functor of points\/}  is the functor  $ \; h_X : \salg_\bk \longrightarrow \sets \; $  defined on objects by  $ \; h_X(A) := \Hom \big(\, \uspec(A) \, , X \big) \; $  and on arrows by  $ \; h_X(f)(\phi) := \phi \circ \uspec (f) \; $.  When  $ h_X $  is actually a functor from  $ \salg_\bk $  to  $ \grps $,  the category of groups, we say that  $ X $  is a  {\it supergroup-scheme}. If  $ X $  is affine, this is equivalent to the fact that $ \, \cO(X) \, $   --- the superalgebra of global sections of the structure sheaf on  $ X $  ---   is a (commutative)  {\sl Hopf superalgebra}. More in general, we shall call  {\it supergroup functor\/}  any functor  $ \; G : \salg_\bk \longrightarrow \grps \; $.
\end{definition}

\vskip4pt

%
%
   Any representable supergroup functor is the same as an affine supergroup: indeed, the former corresponds to the functor of points of the latter.
%
%
 See  \cite{ccf},  Ch.~3--5, for more details.
%
%
%
                                                    \par
%
%
%
   In the present work we consider only affine supergroups, described via their functor of points: we introduce them as supergroup functors, and then show that they are representable and algebraic.
%
%
%
%

\smallskip

\begin{examples}  \label{exs-supvecs}  {\ }
 \vskip3pt
   {\it (a)} \,  Let  $ V $  be a free  $ \bk $--supermodule.  For any commutative  $ \bk $--superalgebra $ \, A \, $  we define  $ \; V(A) \, := \, {(A \otimes_\bk V)}_\zero = \, A_\zero \otimes_\bk V_\zero \oplus A_\uno \otimes_\bk V_\uno \; $.  When  $ V $ is finite dimensional, this is a representable functor (from  $ \salg_\bk $  to  $ \bk $--super-vector spaces).
%
%
 Hence  $ V $  can be seen as an affine superscheme.
 \vskip2pt
   {\it (b)} \,  {\sl  $ \rGL(V) $  as an affine algebraic supergroup}.  Let  $ V $  be a free  $ \bk $--supermodule  of finite (su\-per-)rank  $ p|q \, $.  For any superalgebra  $ A \, $,  let  $ \, \rGL(V)(A) := \rGL\big(V(A)\big) \, $  be the set of isomorphisms  $ \; V(A) \longrightarrow V(A) \; $  preserving the  $ \Z_2 $--grading.  If we fix a homogeneous basis for  $ V $,  we see that  $ \, V \cong \bk^{p|q} \; $;  in other words,  $ \, V_\zero \cong \bk^p \, $  and  $ \, V_\uno \cong \bk^q \, $.  In this case, we also denote $ \, \rGL(V) \, $  with  $ \, \rGL_{p|q} \, $.  Now,  $ \rGL_{p|q}(A) $  is the group of invertible matrices of size  $ (p+q) $  with diagonal block entries in  $ A_\zero $  and off-diagonal block entries in $ A_\uno \, $.  It is known that the functor  $ \rGL(V) $  is representable, so  $ \rGL(V) $  is indeed an affine supergroup, and also algebraic; see (e.g.),  \cite{vsv}, Ch.~3,  for further details.
                                                   \par
   Note that every element  $ \rGL\big(V(A)\big) $  extends to a (degree-preserving,  $ A $--linear) automorphism of  $ \, V_A := A \otimes_\bk \! V \, $;  viceversa, any automorphism of  $ V_A $  restricts to an element of  $ \rGL\big(V(A)\big) \, $.  So  $ \rGL\big(V(A)\big) $  identifies with  $ \rGL\big(V_A\big) \, $,  the group of  ($ A $--linear)  automorphisms of  $ V_A \, $.  We call  $ \rGL(V_\bullet) $  the obvious functor from  $ \salg_\bk $  to  $ \grps $  given on objects by  $ \, A \mapsto \rGL(V_\bullet)(A) := \rGL(V_A) \, $.   \hfill  $ \diamondsuit $
\end{examples}

\medskip

  \subsection{Lie superalgebras}  \label{Lie-superalgebras}

\smallskip

   {\ } \quad   The notion of Lie superalgebra over a field is well known, at least for characteristic neither 2 nor 3.  To take into account all cases, we consider the following modified formulation: it is a ``correct'' notion of Lie superalgebras given by the standard notion enriched with an additional  ``$ 2 $--mapping'', a close analogue to the  $ p $--mapping  in a  $ p $--restricted  Lie algebra over a field of characteristic  $ p > 0 \, $.

\smallskip

\begin{definition}  \label{def-Lie-salg}
 {\it (cf.~\cite{bmpz}, \cite{du})} \,
  Let  $ \, A \in \salg_\bk \, $.  We call  {\sl Lie\/  $ A $--superalgebra\/}  any  $ A $--supermodule  $ \, \fg = \fg_\zero \oplus \fg_\uno \, $  endowed with a  {\it (Lie super)bracket\/}  $ \; [\,\ ,\ ] : \fg \times \fg \longrightarrow \fg \, $,  $ \; (x,y) \mapsto [x,y] \, $,  \, and a  {\it  $ 2 $--operation\/}  $ \; {(\ )}^{\langle 2 \rangle} : \fg_\uno \longrightarrow \fg_\zero \, $,  $ \; z \mapsto z^{\langle 2 \rangle} \, $,  \, such that (for all  $ \, x , y \in \fg_\zero \cup \fg_\uno \, $,  $ \, w \in \fg_\zero \, $,  $ \, z, z_1, z_2 \in \fg_\uno $):
 \vskip5pt
   {\it (a)}  \quad  $ [\,\ ,\ ] \, $  is  $ A $--superbilinear (in the obvious sense)\;,
\qquad  $ [w,w] \; = \; 0  \;\; ,  \qquad  \big[z[z,z]\big] \; = \; 0  \quad $;
 \vskip5pt
   {\it (b)}  $ \qquad  [x,y] \, + \, {(-1)}^{p(x) \, p(y)}[y,x] \; = \; 0  \qquad\; $  {\sl (anti-symmetry)}\,;
 \vskip7pt
   {\it (c)}  $ \quad  {(-\!1)}^{p(x) p(z)} [x,[y,z]] + {(-\!1)}^{p(y) p(x)} [y,[z,x]] \, + \, {(-\!1)}^{p(z) p(y)} [z,[x,y]] \, = \, 0 \;\;\, $  {\sl (Jacobi identity)}\,;
 \vskip6pt
   {\it (d)}  \;\quad  $ {(\ )}^{\langle 2 \rangle} \, $  is  $ A $--quadratic, i.e.~$ \;\;  {(a_\zero \, z)}^{\langle 2 \rangle} = \, a^2 \, z^{\langle 2 \rangle} \;\; $,  $ \;\;  {(a_\uno \, w)}^{\langle 2 \rangle} = \, 0 \;\;\; $  for  $ \; a_\zero \in A_\zero \, $,  $ \; a_\uno \in A_\uno \; $;
 \vskip6pt
%
%
   {\it (e)}  \qquad  $ {(z_1 \! + z_2)}^{\langle 2 \rangle}  \, = \;  z_1^{\langle 2 \rangle} + \, [z_1,z_2] \, + \, z_2^{\langle 2 \rangle}  \quad ,  \quad \qquad  \big[ z^{\langle 2 \rangle}, x \big]  \, = \;  \big[ z \, , [z,x] \big] \quad $.
 \vskip9pt
   All Lie  $ A $--superalgebras  form a category, whose morphisms are the  $ A $--superlinear  (in the obvious sense), graded maps preserving the bracket and the  $ 2 $--operation.
 \eject
 \vskip3pt
   A Lie superalgebra is said to be  {\sl simple\/}  if it has no non-trivial homogenenous ideal.  Simple Lie superalgebras of finite dimension over algebraically closed fields of characteristic zero were classified by V.~Kac  (cf.~\cite{ka}),
 to whom we shall refer for the standard terminology and notions.
\end{definition}

%
%
%

\smallskip

\begin{examples}   \label{def-Lie/Ass_End(V)}
 {\it (a)} \,  Let  $ \, \mathcal{A} = \mathcal{A}_\zero \oplus \mathcal{A}_\uno \, $  be any associative  $ \bk $--superalgebra.  There is a canonical structure of Lie superalgebra on  $ \mathcal{A} $  given by  $ \,\; [x,y] \, := \, x\,y - {(-1)}^{p(x) p(y)} y\,x \;\, $  for all homogeneous  $ \, x, y \in \mathcal{A}_\zero \cup \mathcal{A}_\uno \; $
and  $ 2 $--operation  $ \,\; z^{\langle 2 \rangle} := z^2 = z\,z \;\, $  (the associative square in  $ \mathcal{A} $)  for all odd  $ \, z \in \mathcal{A}_\uno \, $.
 \vskip4pt
   {\it (b)} \,  Let  $ \, V = V_\zero \oplus V_\uno \, $  be a  {\sl free\/}  $ \bk $--supermodule,  and consider  $ \End(V) \, $,  the endomorphisms of  $ V $  as an ordinary  $ \bk $--module.  This is again a free super  $ \bk $--module,  $ \; \End(V) = \End(V)_\zero \oplus \End(V)_\uno \, $,  \; where $ \End(V)_\zero $  are the morphisms which preserve the parity, while  $ \End(V)_\uno $  are the morphisms which reverse the parity.  By the recipe in  {\it (a)},  $ \End(V) $  is a Lie  $ \bk $--superalgebra
with  $ \; [A,B] := A B - {(-1)}^{p(A) p(B)} B A \, $,  $ \; C^{\langle 2 \rangle} := C^2 \, $,  \; for all  $ \, A, B, C \in \End(V) \; $  homogeneous, with  $ C $  odd.
                                                                         \par
   The standard example is for  $ V $  of finite rank, say  $ \, V := \bk^{p|q} = \bk^p \oplus \bk^q \, $,  with  $ \, V_\zero := \bk^p \, $  and  $ \, V_\uno := \bk^q \, $:  in this case we also write  $ \; \End\big(\bk^{p|q}\big) \! := \End(V) \; $  or  $ \; \rgl_{p|q} := \End(V) \; $.  Choosing a basis for  $ V $  of homogeneous elements (writing first the even ones), we identify  $ \End(V)_\zero $  with the set of all diagonal block matrices, and  $ \End(V)_\uno $  with the set of all off-diagonal block matrices.   \hfill  $ \diamondsuit $
\end{examples}

\medskip

\begin{free text}  \label{Lie-salg_funct}
 {\bf Lie superalgebras and Lie algebra valued functors.}  \ Let us fix
%
%
 $ \bk $  and  $ \salg_\bk $  as in section \ref{first_preliminaries},
 and let  $ \modul_\bk $  and  $ \lie_\bk $  be the category of  $ \bk $--modules  and of Lie  $ \bk $--algebras.  Any  $ \bk $--supermodule  $ \mathfrak{m} $  yields a well-defined functor  $ \; \cM_{\mathfrak{m}} : \salg_\bk \longrightarrow \modul_\bk \; $,  \, given on objects by  $ \; \cM_{\mathfrak{m}}(A) \, := \, \big( A \otimes \mathfrak{m} \big)_\zero \, = \, A_\zero \otimes \mathfrak{m}_\zero \, \oplus \, A_\uno \otimes \mathfrak{m}_\uno \; $,  \, for all  $ \, A \in \salg_\bk \, $.  If in addition  $ \, \mathfrak{m} = \fg \, $  is a Lie  $ \bk $--superalgebra,  then
%
%
 $ \, A \otimes \fg \, $  is a Lie  $ A $--superalgebra,  its Lie bracket being defined via sign rules by  $ \; \big[\, a \otimes  X \, , \, a' \otimes X' \,\big] \, := \, {(-1)}^{|X|\,|a'|} \, a\,a' \otimes \big[X,X'\big] \; $,  and similarly for the  $ 2 $--operation:  then  $ \cL_\fg(A) $  is its even part, so it is a Lie algebra.  Thus we have a Lie algebra valued functor  $ \; \cL_\fg : \salg_\bk \longrightarrow \lie_\bk \; $ (see \cite{ccf}, \S 11.2,  for details).  We shall call  {\sl quasi-representable\/}  any functor  $ \; \cL : \salg_\bk \longrightarrow \lie_\bk \; $  for which there exists a Lie  $ \bk $--superalgebra  $ \fg $  such that  $ \, \cL = \cL_\fg \, $:  indeed, any such functor is even representable as soon as the  $ \bk $--module  $ \fg $  is free of finite rank, since  $ \cL_\fg $  is then represented by the commutative  $ \bk $--superalgebra  $ \, S(\fg^*) \in \salg_\bk \; $.  In particular, when  $ V $  is a free super  $ \bk $--module  we have the Lie superalgebra  $ \, \fg := \End(V) \, $  and the functor  $ \cL_{\text{\sl End}(V)} \, $;  then  $ \rGL(V) $  --- cf.~Example \ref{exs-supvecs}{\it (b)}  ---   is a subfunctor of  $ \cL_{\text{\sl End}(V)} $   --- as a set-valued functor.
                                                                   \par
   This ``functorial presentation'' of Lie superalgebras can be adapted to representations too: if  $ \fg $  is a Lie  $ \bk $--superalgebra and  $ V $  a  $ \fg $--module,  the representation map  $ \; \phi : \fg \longrightarrow \End(V) \; $  clearly induces a natural transformation of functors  $ \, \cL_\fg \longrightarrow \cL_{\text{\sl End}(V)} \, $.
\end{free text}

\vskip19pt

  \subsection{Lie superalgebras of Cartan type}  \label{Lie-superalg_Cartan-type}

\smallskip

   {\ } \quad
 In the following,  $ \KK $  is an algebraically closed field of characteristic zero.
                                                           \par
   By definition, a Lie superalgebra  $ \fg $  over  $ \KK $  is of Cartan type if it is finite dimensional, simple, with the odd part  $ \fg_\uno $  which is not semisimple as a module over the even part  $ \fg_\zero \, $.  Actually, Cartan type Lie superalgebras split into four countable families, denoted  $ \; W(n) \, $,  $ \, S(n) \, $,  $ \, \tS(n) \, $  and  $ \, H(n) \, $  with  $ \, n \geq 2 \, $,  $ \, n \geq 3 \, $,  $ \, n \geq 4 \, $  (with  $ n $  being  {\sl even\/})  and  $ \, n \geq 4 \, $  respectively.
%
%
                                                           \par
   We shall now describe in short all these types.  For further details, see  \cite{ka},  \S 3.

\vskip3pt

   Given  $ \, n \in \N_+ \, $,  denote by  $ \, \Lambda(n) = \KK[\xi_1,\dots,\xi_n] \, $  the free commutative superalgebra over  $ \KK $  with  $ n $  odd generators  $ \xi_1 $,  $ \dots $,  $ \xi_n \, $;  this is (isomorphic to) the Grassmann algebra of rank  $ n \, $,  and is naturally  $ \Z $--graded,  with  $ \, \text{\sl deg}(\xi_i) = 1 \, $.
%
%
 A  $ \KK $--basis  of  $ \Lambda(n) $  is the set
%
%
 $ \; B_{\Lambda(n)} \, := \, \big\{\, \underline{\xi}^{\underline{e}}
\;\big|\; \underline{e} \in {\{0,1\}}^n \,\big\} \; $,
where  $ \, \underline{\xi}^{\underline{e}} := \xi_1^{\wedge \underline{e}(1)} \!\! \wedge \xi_2^{\wedge \underline{e}(2)} \!\! \wedge \cdots \wedge \xi_n^{\wedge \underline{e}(n)} \! = \xi_1^{\underline{e}(1)} \! \cdot \xi_2^{\underline{e}(2)} \cdots \xi_n^{\underline{e}(n)} \; $  (hereafter we shall drop the  $ \wedge $'s).
                                                              \par
   For later use, for every  $ \, \underline{e} \in {\{0,1\}}^n \, $  we define  $ \; |\underline{e}| := \sum_{k=1}^n \underline{e}(k) \; $.

\medskip

\begin{free text}  \label{def-W(n)}
 {\bf Definition of  $ W(n) \, $.}  \ For any  $ \, n \in \N_+ \, $  with  $ \, n \geq 2 \, $,  let  $ \, W(n) := \Der_\KK\big(\Lambda(n)\big) \, $  denote the set of  $ \KK $--(super)derivations  of  $ \Lambda(n) \, $.
 This is a Lie subsuperalgebra of  $ \, \End_\KK\big(\Lambda(n)\big) \, $:  explicitly, one has
 \vskip-9pt
  $$  W(n)  \, = \,  \Big\{\, {\textstyle \sum}_{i=1}^n P_i\big(\underline{\xi}\big) \, \partial_i \;\Big|\; P_i\big(\underline{\xi}\big) \in \Lambda(n) \;\;\, \forall \; i = 1, \dots, n \,\Big\}  $$
 \vskip5pt
\noindent
 where each  $ \partial_i $  is the unique superderivation such that  $ \, \partial_i(\xi_j) = \delta_{i,j} \, $.  This Lie superalgebra  $ W(n) $  is naturally  $ \Z $--graded,  with  $ \, \text{\sl deg}(\partial_i) = -1 \, $,  $ \, \text{\sl deg}(\xi_i) = +1 \, $;  in detail,
 \eject
 \vskip-4pt
  $$  {} \hskip-9pt   W(n) = {\textstyle \bigoplus\limits_{z \in \Z}} \, {W(n)}_z \; ,  \qquad  {W(n)}_z  = \,  \Big\{\, {\textstyle \sum_{i=1}^n} P_i\big(\underline{\xi}\big) \, \partial_i \;\Big|\; P_i\big(\underline{\xi}\big) \in {\Lambda(n)}_{z+1} \;\, \forall \; i \,\Big\}   \eqno (2.1)  $$
 \vskip-5pt
\noindent
 so  $ \, {W(n)}_z \not= \{0\} \, $  iff
%
%
 $ \; -1 \leq z \leq n \! - \! 1 \; $.
%
%
 Thus, if
 $ \; {W(n)}_{[z]} := \hskip-9pt {\textstyle \bigoplus\limits_{\zeta \, \equiv \, z \!\! \mod \! (n+1)}} \hskip-17pt {W(n)}_\zeta \phantom{\Big|} \; $  for all  $ \, [z] \in \Z_{n+1} \, $   (the group of integers modulo  $ \, n\!+\!1 \, $)  the  $ \Z $--grading  above yields the  $ \Z_{n+1} $--grading  $ \; W(n) = \hskip-6pt {\textstyle \bigoplus\limits_{[z] \in \Z_{n+1}}} \hskip-8pt {W(n)}_{[z]} \; $.
                                                              \par
   The  $ \Z $--grading  yields a  $ \Z $--filtration  $ \, \big(W(n) \supseteq \!\big) \cdots \supseteq {W(n)}_{\geq z-1} \supseteq {W(n)}_{\geq z} \supseteq {W(n)}_{\geq z+1} \supseteq \cdots \, $  of  $ W(n) $  as a Lie superalgebra, where  $ \, {W(n)}_{\geq z} := \oplus_{k \geq z} {W(n)}_k \phantom{\Big|} $  for all  $ \, z \in \Z \, $;  the associated graded Lie superalgebra then is graded-isomorphic to  $ W(n) $  itself.  Also, this  $ \Z $--grading  is consistent with the  $ \Z_2 $--grading  of  $ W(n) \, $,  i.e.~the  $ \Z_2 $--grading  is given by  $ \; W(n) \, = \, {W(n)}_{\zero} \oplus {W(n)}_{\uno} \; $  with
 \vskip-5pt
  $$  {W(n)}_{\zero}  \; := \,  {\textstyle \bigoplus\limits_{z \in 2 \Z}} \, {W(n)}_z  \quad ,  \qquad  {W(n)}_{\uno}  \; := \,  {\textstyle \bigoplus\limits_{z \in (2 \Z + 1)}} {W(n)}_z  $$
In particular, one also has the following three facts:
 \vskip5pt
   {\it (a)} \  for each  $ \, z \in \Z \, $,  the set
%
%
 $ \; B_{W(n)\,;\,z} \, := \, \big\{\, \underline{\xi}^{\underline{e}} \; \partial_i \;\big|\; \underline{e} \in {\{0,1\}}^n , \; i = 1, \dots, n \, ; \, |\underline{e}| = z+1 \,\big\} \; $
is a  $ \KK $--basis  of  $ {W(n)}_z \, $;  for each  $ \, \overline{z} \in \Z_2 \, $ the set
%
%
 $ \; B_{W(n)\,;\,\overline{z}} \, := \hskip-7pt {\textstyle \bigcup\limits_{(z \; \text{mod} \, 2) \, = \, \overline{z}}} \hskip-11pt B_{W(n)\,;\,z} \; $
is a  $ \KK $--basis  of  $ {W(n)}_{\overline{z}} \; $;  the set
%
%
 $ \; B_{W(n)} \, :=  {\textstyle \bigcup\limits_{z \in \Z}} B_{W(n)\,;\,z} \; $
is a   ---  $ \Z $--homogeneous  and  $ \Z_2 $--homogeneous  ---   $ \KK $--basis  of  $ W(n) \, $.
 \vskip4pt
   {\it (b)} \  $ {W(n)}_0 \, $  is a Lie subalgebra of the even part  $ \, {W(n)}_\zero \, $  of  $ \, W(n) \, $,  isomorphic to  $ \rgl(n) \, $,
%
%
 via
 $ \; \xi_i \, \partial_j \mapsto \text{e}_{i,j} \; $
 (= the elementary  $ (n \times n) $--matrix  bearing 1 in position  $ (i,j) $  and zero elsewhere);
 \vskip4pt
   {\it (c)} \  $ {W(n)}_{-1} \; $,  as a module for  $ \, {W(n)}_0 \cong \rgl(n) \, $,  is
%
%
 the dual of the standard module of  $ \rgl(n) \; $.
\end{free text}

\medskip

\begin{free text}  \label{Lie-struct_W(n)}
 {\bf The Lie structure in  $ W(n) \, $.}  \
%
%
 Our  $ \, W(n) := \Der_\KK\big(W(n)\big) \, $  is a Lie subsuperalgebra of $ \End_\KK\big(\Lambda(n)\big) \, $,  whose Lie bracket in the latter is the ``supercommutator''  (cf.~Example \ref{def-Lie/Ass_End(V)}{\it (b)\/})).  Thus we first consider the composition product of two basis elements in  $ W(n) \, $.  Calculations give
  $$  \underline{\xi}^{\underline{a}} \; \partial_j \circ \, \underline{\xi}^{\underline{b}} \; \partial_\ell  \, = \,  {(-1)}^{|\underline{b}|} \, \underline{\xi}^{\underline{a}} \; \underline{\xi}^{\underline{b}} \; \partial_j \, \partial_\ell  \, + {(-1)}^{\#\{ s \,|\, s<j \, , \, \underline{b}(s) = 1 \}} \, \delta_{\underline{b}(j),1} \, \underline{\xi}^{\underline{a}} \; \underline{\xi}^{\underline{b} - \underline{e}_{\,j}} \; \partial_\ell   \hskip3pt   \eqno (2.2)  $$
where  $ \, \underline{e}_{\,j} \in {\{0,1\}}^n \, $  is given by  $ \, \underline{e}_{\,j}(k) := \delta_{j,k} \, $,  and  $ \, \underline{b} - \underline{e}_{\,j} \, $  is the obvious element in the set  $ {\{0,1\}}^n \, $.  There\-fore the defining formula
  $ \; \Big[\, \underline{\xi}^{\underline{a}} \; \partial_j \; , \, \underline{\xi}^{\underline{b}} \; \partial_\ell \,\Big] \, = \, \underline{\xi}^{\underline{a}} \; \partial_j \circ \, \underline{\xi}^{\underline{b}} \; \partial_\ell  \, - {(-1)}^{\text{deg}(\underline{\xi}^{\underline{a}} \; \partial_j) \, \text{deg}(\underline{\xi}^{\underline{b}} \; \partial_\ell)} \,
\underline{\xi}^{\underline{b}} \; \partial_\ell \circ \, \underline{\xi}^{\underline{a}} \; \partial_j \; $
%
%
 (taken from  Example \ref{def-Lie/Ass_End(V)}{\it (a)\/})
along with (2.2) yields, taking into account that  $ \; \partial_\ell \; \partial_j \, = \, - \partial_j \; \partial_\ell \; $,
  $$  \Big[\, \underline{\xi}^{\underline{a}} \; \partial_j \; , \, \underline{\xi}^{\underline{b}} \; \partial_\ell \,\Big]  \,\; = \;\,  \pm \; \delta_{\underline{b}(j),1} \, \underline{\xi}^{\underline{a}} \; \underline{\xi}^{\underline{b} - \underline{e}_{\,j}} \; \partial_\ell  \; \pm \; \delta_{\underline{a}(\ell\,),1} \, \underline{\xi}^{\underline{b}} \; \underline{\xi}^{\underline{a} - \underline{e}_{\,\ell}} \; \partial_j   \eqno (2.3)  $$
for all  $ \, \underline{a} \, , \underline{b} \in {\{0,1\}}^n \, $,  $ \, j, \ell = 1, \dots n \, $.  In particular   ---
%
 reordering the various factors  $ \xi_k $
%
%
 ---   this shows that  $ \; \Big[\, \underline{\xi}^{\underline{a}} \; \partial_j \; , \, \underline{\xi}^{\underline{b}} \; \partial_\ell \,\Big] \; $  has coefficients in  $ \, \{-1,0,1\} \, $  with respect to the basis  $ B_{W(n)} \, $.

 \vskip5pt

   When  $ \, \delta_{\underline{b}(j),1} = 1 = \delta_{\underline{a}(\ell\,),1}\, $,  i.e.~in the case  $ \, \underline{b}(j) = 1 = \underline{a}(\ell) \, $,  formula (2.3) looks more precise:
  $$  \Big[\, \underline{\xi}^{\underline{a}} \; \partial_j \, , \, \underline{\xi}^{\underline{b}} \; \partial_\ell \,\Big]  \; = \;  {(-1)}^N \, \underline{\xi}^{\underline{a} - \underline{e}_{\,\ell}} \, \underline{\xi}^{\underline{b} - \underline{e}_{\,j}} \, \big(\, \xi_\ell \, \partial_\ell  \, - \, \xi_j \, \partial_j \,\big)
 \;\; ,  \qquad  \forall \; \underline{a} \, , \underline{b} : \underline{a}(\ell)
= 1 = \underline{b}(j)
    \eqno (2.4)  $$
 \vskip5pt
\noindent
 {\sl Note also that (2.4) takes a special form in the following three cases:}
 \vskip7pt
   \centerline{\quad  --- \; if  $ \, j = \ell \, $,  \, then   \hfill
      $ \; \Big[\, \underline{\xi}^{\underline{a}} \; \partial_j \, , \, \underline{\xi}^{\underline{b}} \; \partial_\ell \,\Big]  \; = \;  0  \quad $;   \hskip55pt \hfill  (2.5)}
 \vskip7pt
   \centerline{\quad  --- \; if  $ \, j \not= \ell \, $,  $ \, \underline{a}(j) = 1 \, $,  \, then  \quad  \hfill
      $ \; \Big[\, \underline{\xi}^{\underline{a}} \; \partial_j \, , \, \underline{\xi}^{\underline{b}} \; \partial_\ell \,\Big]  \; = \;  {(-1)}^N \, \underline{\xi}^{\underline{a} - \underline{e}_{\,\ell}} \, \underline{\xi}^{\underline{b} - \underline{e}_{\,j}} \, \xi_\ell \, \partial_\ell  \quad $;   \hskip25pt \hfill  (2.6)}
 \vskip7pt
   \centerline{\quad  --- \; if  $ \, j \not= \ell \, $,  $ \, \underline{b}(\ell) = 1 \, $,  \, then  \quad  \hfill
      $ \; \Big[\, \underline{\xi}^{\underline{a}} \; \partial_j \, , \, \underline{\xi}^{\underline{b}} \; \partial_\ell \,\Big]  \; = \;  {(-1)}^{N+1} \, \underline{\xi}^{\underline{a} - \underline{e}_{\,\ell}} \, \underline{\xi}^{\underline{b} - \underline{e}_{\,j}} \, \xi_j \, \partial_j  \quad $.   \hskip25pt \hfill  (2.7)}

\vskip7pt

   Finally, for the  $ 2 $--operation  the defining formula (from  Example \ref{def-Lie/Ass_End(V)}{\it (b)\/})  along with (2.2) gives
  $$  {\big( \underline{\xi}^{\underline{a}} \; \partial_j \big)}^{\langle 2 \rangle}  \; = \;\,  {\big( \underline{\xi}^{\underline{a}} \; \partial_j \big)}^2  \; = \;\,  0   \eqno \forall \;\; \underline{a} \,:\, \big|\underline{a}\big| \in 2\,\N   \qquad \qquad   (2.8)  $$
\end{free text}

\medskip

\begin{free text}  \label{def-S(n)}
 {\bf Definition of  $ S(n) \, $.}  \ We retain notations of  Definition \ref{def-W(n)}  above, in particular  $ \Lambda(n) $  and  $ \, W(n) := \Der_\KK\big(\Lambda(n)\big) \, $,  for  $ \, n \! \in \! \N_+ \, $,  are defined as therein; in addition, we assume now  $ \, n \geq 3 \;\, $.
%
%
                                                        \par
   Define the  {\sl divergence operator\/}  $ \; \text{\sl div} : W(n) := \Der_\KK\big(\Lambda(n)\big) \longrightarrow \Lambda(n) \; $  by  $ \; \text{\sl div} \big( \sum_{i=1}^n P_i\big(\underline{\xi}\big) \, \partial_i \big) := \sum_{i=1}^n \partial_i\big( P_i\big(\underline{\xi}\big) \big) \; $  for  $ \; \sum_{i=1}^n P_i\big(\underline{\xi}\big) \, \partial_i \in W(n) \; $.  Then set
  $$  S(n)  \,\; := \;\,  \Big\{\, D := {\textstyle \sum}_{i=1}^n P_i\big(\underline{\xi}\big) \, \partial_i \in W(n) \;\Big|\; \text{\sl div}\big(D\big) = 0 \,\Big\}   \qquad \Big( = \, \text{\it Ker}\,\big(\text{\sl div}\big) \Big)  $$
   \indent   This is a  $ \Z $--graded  Lie subsuperalgebra of  $ \, W(n) $   --- with  $ \Z $--grading induced from  $ W(n) \, $:  so
  $$  {} \hskip-9pt   S(n) \, = \, {\textstyle \bigoplus\limits_{z \in \Z}} \, {S(n)}_z \quad ,  \qquad  {S(n)}_z \, = \; {W(n)}_z \,{\textstyle \bigcap}\, S(n)   \eqno (2.9)  $$
--- cf.~(2.1) ---   where  $ \; {S(n)}_z \not= \{0\} \; $  if and only if  $ \; -1 \leq z \leq n \! - \! 2 \, $   --- see below for more details.  Like for  $ W(n) \, $,  this  $ \Z $--grading  yields also a  $ \Z_n $--grading  $ \,\; S(n) = \! {\textstyle \bigoplus\limits_{[z] \in \Z_n}} \hskip-5pt {S(n)}_{[z]} \;\, $  with  $ \; {S(n)}_{[z]} \, := \hskip-7pt {\textstyle \bigoplus\limits_{\zeta \, \equiv \, z \!\! \mod n}} \hskip-11pt {S(n)}_\zeta
%
%
 \; $  for all  $ \, [z] \in \Z_n \, $  (the group of integers modulo  $ \, n \, $).
 Again, the  $ \Z $--grading  yields a  $ \Z $--filtration,  which coincides with the one induced by  $ W(n) \, $,  whose associated graded Lie superalgebra is graded-isomorphic to  $ S(n) $  itself.  Moreover, the  $ \Z $--grading  of  $ S(n) $  is consistent with the  $ \Z_2 $--grading,  in the obvious sense (like for  $ W(n) \, $).
                                                                  \par
   The construction and the results for  $ W(n) $   --- cf.~Definition \ref{def-W(n)}  ---   give:
 \vskip5pt
   {\it (a)} \  a basis of the  $ \KK $--vector  space  $ S(n) $  is given by
 \vskip-15pt
  $$  B_{S(n)}  \; := \;  \Big\{\, \underline{\xi}^{\underline{e}} \; \partial_i \;\Big|\; \underline{e}(i) = 0 \,\Big\}  \;{\textstyle \bigcup}\;  \bigg\{\, \underline{\xi}^{\underline{e}} \; \big( \xi_j \, \partial_j - \xi_{j'} \, \partial_{j'} \big) \;\bigg|\;
 {{1 \! \leq \! j \! < \! j' \! \leq \! n \, , \; \underline{e}(j) = 0 = \underline{e}(j')} \atop {\underline{e}(j'') = 1  \;\quad \forall \;\; j < j'' < j'}}
 \,\bigg\}  $$
 \vskip-4pt
\noindent
 In the following we call  {\sl ``of first type''\/}  the elements of this basis of the form  $ \, \underline{\xi}^{\underline{e}} \; \partial_i \, $,  and  {\sl ``of second type''\/}  those of the form  $ \, \underline{\xi}^{\underline{e}} \; \big( \xi_j \, \partial_j - \xi_{j'} \, \partial_{j'} \big) \, $; more in general, we call  {\sl ``(elements) of second type''\/}  also all those of the form  $ \, \underline{\xi}^{\underline{e}} \; \big( \xi_j \, \partial_j - \xi_k \, \partial_k \big) \, $,  for any  $ \, j \! < \! k \, $  with  $ \, \underline{e}(j) = 0 = \underline{e}(k) \; $.
                                                                   \par
%
%
   Again, this basis is homogeneous (for both the  $ \Z $--grading  and the  $ \Z_2 $--grading),  i.e.~$ \; B_{S(n)} = \bigcup_{z \in \Z} B_{S(n)\,;\,z} \; $  and  $ \; B_{S(n)} = \bigcup_{\overline{z} \in \Z_2} B_{S(n)\,;\,\overline{z}} \; $  where  $ \, B_{S(n)\,;\,z} := B_{S(n)} \cap {S(n)}_z \, $,  respectively  $ \, B_{S(n)\,;\,\overline{z}} := B_{S(n)} \cap {S(n)}_{\overline{z}} \; $,  is a basis of  $ {S(n)}_z \, $,  respectively of  $ {S(n)}_{\overline{z}} \; $,  for every  $ \, z \in \Z \, $,  $ \, \overline{z} \in \Z_2 \, $.
 \vskip4pt
   {\it (b)} \  $ {S(n)}_0 \, $  is a Lie subalgebra of the even part  $ \, {S(n)}_\zero \, $  of  $ \, S(n) \, $,  isomorphic to  $ \rsl(n) \, $,  via  $ \; \xi_i \, \partial_j \mapsto \text{e}_{i,j} \; $  (notation of Definition \ref{def-W(n)})  for  $ \, i \not= j \, $,  and  $ \; \big( \xi_k \, \partial_k - \xi_\ell \, \partial_\ell \big) \mapsto \big( \text{e}_{k,k} - \text{e}_{\ell,\ell} \big) \; $  for  $ \, k \not= \ell \, $;
 \vskip4pt
   {\it (c)} \  $ {S(n)}_{-1} \, $,  as a module for  $ \, {S(n)}_0 \cong \rsl(n) \, $,  is
%
%
 the dual of the standard module of  $ \rsl(n) \, $.
\end{free text}

\medskip

\begin{free text}  \label{Lie-struct_S(n)}
 {\bf The Lie structure in  $ S(n) \, $.}  \
%
%
 We need formulas for the Lie bracket of elements in  $ B_{S(n)} \, $.
                                                          \par
   First we look at pairs of basis elements of the first type  (cf.~\S \ref{def-S(n)},  so  $ \, \underline{a}(j) = 0 \, $,  $ \underline{b}(\ell) = 0 \, $).  For their bracket, formulas (2.2--7) give
(with the right-hand side which in third case might be zero)
  $$  \Big[\, \underline{\xi}^{\underline{a}} \; \partial_j \; , \, \underline{\xi}^{\underline{b}} \; \partial_\ell \,\Big]  \,\; = \;\,
   \begin{cases}
      \; \pm \; \underline{\xi}^{\underline{a}} \; \underline{\xi}^{\underline{b} - \underline{e}_{\,j}} \; \partial_\ell   &  \quad  \text{if \ }  \underline{a}(\ell) = 0 \, , \; \underline{b}(j) = 1  \\
      \; \pm \; \underline{\xi}^{\underline{b}} \; \underline{\xi}^{\underline{a} - \underline{e}_{\,\ell}} \; \partial_j   &  \quad  \text{if \ }  \underline{a}(\ell) = 1 \, , \; \underline{b}(j) = 0  \\
      \; \pm \; \underline{\xi}^{\underline{a} - \underline{e}_\ell} \; \underline{\xi}^{\underline{b} - \underline{e}_{\,j}} \; \big( \xi_\ell \, \partial_\ell - \xi_j \, \partial_j \big)   &  \quad  \text{if \ }  \underline{a}(\ell) = 1 \, , \, \underline{b}(j) = 1  \\
      \;\;\; 0   &  \quad  \text{if \ }  \underline{a}(\ell) = 0 \, , \; \underline{b}(j) = 0
   \end{cases}
  $$

\smallskip

   Now we consider the Lie bracket of a basis element of first type and one of second type (see  \S \ref{def-S(n)},  thus  $ \, \underline{e}(\ell) = 0 \, $).  Then
%
%
 $ \; \Big[\, \underline{\xi}^{\underline{e}} \; \partial_\ell \; , \, \underline{\xi}^{\underline{e}'} \big( \xi_j \, \partial_j - \xi_k \, \partial_k \big) \,\Big] \, = \, \Big(\! {(-1)}^{N'} \! - {(-1)}^{N''} \Big) \; \underline{\xi}^{\underline{e}' - \underline{e}_\ell} \, \underline{\xi}^{\underline{e}} \; \partial_\ell \; $,
 \, by formulas (2.2--7), for some  $ \, N', N'' \in \N \, $;  a detailed (yet elementary) analysis of signs shows that
 \vskip-5pt
  $$  \Big[\, \underline{\xi}^{\underline{e}} \; \partial_\ell \; , \, \underline{\xi}^{\underline{e}'} \big( \xi_j \, \partial_j - \xi_k \, \partial_k \big) \,\Big]  \,\; = \;\,  0  $$

   Third, we look at pairs of elements both of the second type: formulas (2.2--7) eventually give
 \vskip-11pt
  $$  \displaylines{
   \Big[\, \underline{\xi}^{\underline{a}} \, \big( \xi_j \, \partial_j - \xi_k \, \partial_k \big) \; , \, \underline{\xi}^{\underline{b}} \, \big( \xi_\ell \, \partial_\ell - \xi_t \, \partial_t \big) \,\Big]  \,\; =   \hfill  \cr
   \hfill   = \;\,  \big( \underline{b}(j) - \underline{b}(k) \big) \, \underline{\xi}^{\underline{a}} \, \underline{\xi}^{\underline{b}} \, \big( \xi_\ell \, \partial_\ell - \xi_t \, \partial_t \big)  \, - \,  \big( \underline{a}(\ell) - \underline{a}(t) \big) \, \underline{\xi}^{\underline{a}} \, \underline{\xi}^{\underline{b}} \, \big( \xi_j \, \partial_j - \xi_k \, \partial_k \big)  }  $$
We must stress two facts.  First,  $ \, \underline{\xi}^{\underline{a}} \, \underline{\xi}^{\underline{b}} \, \big( \xi_\ell \, \partial_\ell - \xi_t \, \partial_t \big) \, $  and  $ \, \underline{\xi}^{\underline{a}} \, \underline{\xi}^{\underline{b}} \, \big( \xi_j \, \partial_j - \xi_k \, \partial_k \big) \, $  are both either zero or elements of  $ B_{S(n)} $ of the second type; second,
%
%
 $ \; \big( \underline{b}(j) - \underline{b}(k) \big) \, , \, \big( \underline{a}(\ell) - \underline{a}(t) \big) \, \in \, \big\{ \!-\!1 \, , \, 0 \, , \, +1 \big\} \; $.
 \eject

%
%
   Last, the  $ 2 $--operation:  (2.2) and  $ \, {(z_1 \! + z_2)}^{\langle 2 \rangle} = \, z_1^{\langle 2 \rangle} + \, [z_1,z_2] \, + \, z_2^{\langle 2 \rangle} \, $  (cf.~Definition \ref{def-Lie-salg}{\it (e)\/})  give
  $$  {\big( \underline{\xi}^{\underline{a}} \; \partial_j \big)}^{\langle 2 \rangle}  = \,  0  \; ,  \quad  {\big(\, \underline{\xi}^{\underline{b}} \, \big( \xi_\ell \, \partial_\ell - \xi_t \, \partial_t \big) \big)}^{\langle 2 \rangle}  = \,  0  \; ,   \eqno \forall \;\; \underline{a} \, , \, \underline{b} \;:\; \big|\underline{a}\big| \in 2\,\N \; , \; \big|\underline{b}\big| \in (2\,\N +1)   \qquad   (2.10)  $$
\end{free text}

\medskip

\begin{free text}  \label{def-tildeS(n)}
 {\bf Definition of  $ \tS(n) \, $.}  \ We retain notations of  Definitions \ref{def-W(n)}  and  \ref{def-S(n)}  above.
%
%
 In addition, we assume now that  $ n $  is  {\sl even\/}  and  $ \, n \geq 4 \;\, $.  Define
  $$  \tS(n)  \,\; := \;\,  \Big\{\, D \in W(n) \;\Big|\; \big( 1 + \xi_1 \cdots \xi_n \big)
\, \text{\sl div}\big(D\big) + D\big(\xi_1 \cdots \xi_n\big) = 0 \,\Big\}  $$
   \indent   In order to describe  $ \tS(n) \, $,  write  $ \, D \in W(n) \, $  as  $ \, D := {\textstyle \sum}_{z=-1}^{n-1} D_z \, $  with  $ \, D_z \in {W(n)}_z \, $.
%
%
 The defining equation of  $ \tS(n) $   takes place in the  {\sl graded\/}  superalgebra  $ \,  \Lambda(n) = {\textstyle \bigoplus}_{z=0}^n {\Lambda(n)}_z \, $:  when we single out the different homogeneous summands
%
%
the left-hand side of this equation reads
  $$  \displaylines{
   \big( 1 + \xi_1 \cdots \xi_n \big) \, \text{\sl div}\big(D\big)  \, + \,  D\big(\xi_1 \cdots \xi_n\big)  \,\; = \;\,  \text{\sl div}(D_0)  \, + \,  \cdots  \, + \,  \text{\sl div}(D_{n-2})  \; +   \hfill  \cr
   \hfill   + \; \big( \text{\sl div}(D_{n-1}) + D_{-1}(\xi_1 \cdots \xi_n) \big)  \, + \,  \big( D_0(\xi_1 \cdots \xi_n) + \xi_1 \cdots \xi_n \, \text{\sl div}(D_0) \big)  }  $$
(each  $ \, \text{\sl div}(D_z) \, $  is homogeneous of degree  $ z \, $,  in particular  $ \, \text{\sl div}(D_{-1}) = 0 \, $).  Thus the defining equation  $ \; \big( 1 + \xi_1 \cdots \xi_n \big) \, \text{\sl div}\big(D\big) + D\big(\xi_1 \cdots \xi_n\big) \, = \, 0 \; $  of  $ \, \tS(n) \, $  is equivalent to the system
  $$  \circledast : \begin{cases}
   &  \hskip-5pt   \text{\sl div}(D_z)  \, = \,  0   \hskip121pt  \big( \in {\Lambda(n)}_z \,\big)  \hskip35pt \forall \;\;\; 0 \, \leq \, z \, \leq \, n\!-\!2  \\
   &  \hskip-5pt   \text{\sl div}(D_{n-1}) + D_{-1}(\xi_1 \cdots \xi_n)  \, = \,  0  \qquad \quad  \big( \in {\Lambda(n)}_{n-1} \,\big)  \\
   &  \hskip-5pt   D_0(\xi_1 \cdots \xi_n) + \xi_1 \cdots \xi_n \, \text{\sl div}(D_0)  \, = \,  0  \qquad  \big( \in {\Lambda(n)}_n \,\big)  \\
      \end{cases}  $$
The solution of the system  $ \, \circledast \, $  is immediate:  $ \, D_z \in \text{\it Ker}\,\big(\text{\sl div}\big) \,{\textstyle \bigcap}\, {W(n)}_z = {S(n)}_z \, $,  for  $ \, 0 \! \leq \! z \! \leq \! n\!-\!2 \, $,  for the first  $ \, n\!-\!1 \, $  equations,  while the last one has solution  $ \, D_0 \in \text{\it Ker}\,\big(\text{\sl div}\big) \,{\textstyle \bigcap}\, {W(n)}_0 = {S(n)}_0 \, $,  hence it is redundant.  For the last but one, let us write  $ \; D_{n-1} := {\textstyle \sum}_{j=1}^n c_j \, \xi_1 \cdots \xi_n \, \partial_j \, $  and  $ \, D_{-1} := {\textstyle \sum}_{j=1}^n d_j \, \partial_j \; $:  then these yield solutions of that equation if and only if  $ \, c_j + d_ j = 0 \, $  for all  $ \, j = 1, \dots, n \, $,  in other words if and only if  $ \; D_{-1} + D_{n-1} = {\textstyle \sum}_{j=1}^n c_j \, \big( \xi_1 \cdots \xi_n - 1 \big) \, \partial_j \;\, $.  The outcome is that
  $$  {} \hskip-9pt   \tS(n) \, = \, {\textstyle \bigoplus\limits_{[z] \in \Z_n}} \! {\tS(n)}_{[z]} \quad ,  \qquad
   {{{\tS(n)}_{[z]} \, := \, {S(n)}_z  \qquad  \forall \quad 0 \leq z \leq n-2}
     \atop
   {{\tS(n)}_{[n-1]} \, := \; \text{\it Span}_{\,\KK}{\big( (\xi_1 \cdots \xi_n - 1) \, \partial_j \big)}_{j=1,\dots,n}}}   \eqno (2.11)  $$
which is a splitting of  $ \tS(n) $  as a  $ \Z_n $--graded  Lie superalgebra   --- see below for more details.  Moreover, the natural  $ \Z $--filtration  of  $ W(n) $  induces a similar filtration on  $ \tS(n) \, $:  the associated graded Lie superalgebra then is graded-isomorphic to  $ S(n) \, $.  Finally, the  $ \Z_n $--grading  of  $ \tS(n) $  is consistent with the  $ \Z_2 $--grading,
%
 again in the obvious sense (like for  $ W(n) $  and  $ S(n) \, $).

\smallskip

\noindent
   {\sl  $ \underline{\text{Remark}} $:}  to have a uniform notation we shall also write  $ \; {\tS(n)}_z := {\tS(n)}_{[z]} \; $  for all  $ \, -1 \leq z \leq n-2 \; $  and  $ \; {\tS(n)}_z := \{0\} \; $  for all  $ \, z \in \Z \setminus \{ -1, \dots, n-2 \} \; $.  Then  $ \; \tS(n) = {\textstyle \bigoplus}_{z \in \Z} {\tS(n)}_z \; $  as a vector space.

\bigskip

   The results we found for  $ W(n) $   --- cf.~Definition \ref{def-W(n)}  ---   and  $ S(n) $   --- cf.~Definition \ref{def-S(n)}  ---   give:
 \vskip5pt
   {\it (a)} \  a basis of the  $ \KK $--vector  space  $ \tS(n) $  is given by the union of the set  $ \; \bigcup_{z=0}^{n-2} B_{S(n)\,;\,z} \; $
 with the set  $ \; {\big\{\! (\xi_1 \cdots \xi_n \! - \! 1) \, \partial_j \big\}}_{j=1,\dots,n} \; $;  in detail, it is
  $$  B_{\tS(n)}  :=  \Big\{ \underline{\xi}^{\underline{e}} \, \partial_i \,\Big|
 {\textstyle {{\underline{e}(i) = 0} \atop {\phantom{\big|} |\underline{e}| > 0 \phantom{\big|}}}} \!\Big\}  \,{\textstyle \bigcup}\,  \Big\{ \underline{\xi}^{\underline{e}} \, \big( \xi_j \partial_j - \xi_{j'} \partial_{j'} \!\big) \Big|
 {\textstyle {{1 \leq j < j' \leq n-1 \, , \; \underline{e}(j) = 0 = \underline{e}(j')} \atop {\underline{e}(j'') = 1  \;\quad \forall \;\; j < j'' < j'}}}
 \!\Big\}  \,{\textstyle \bigcup}\,   {\Big\{\! \big( \xi_1 \cdots \xi_n \! - \! 1 \big) \, \partial_j \Big\}}_{1 \leq j \leq n}   $$
   \indent   Again, this basis is homogeneous (for both the  $ \Z_n $--grading  and the  $ \Z_2 $--grading),  i.e.~$ \; B_{\tS(n)} = \bigcup_{[z] \in \Z_n} B_{\tS(n)\,;[z]} \; $  and  $ \; B_{\tS(n)} = \bigcup_{\overline{z} \in \Z_2} B_{\tS(n)\,;\,\overline{z}} \; $  where  $ \, B_{\tS(n)\,;[z]} := B_{\tS(n)} \cap {\tS(n)}_{[z]} \, $,  respectively  $ \, B_{\tS(n)\,;\,\overline{z}} := B_{\tS(n)} \cap {\tS(n)}_{\overline{z}} \; $,  is a basis of  $ {\tS(n)}_{[z]} \, $,  respectively of  $ {\tS(n)}_{\overline{z}} \; $,  for every  $ \, [z] \in \Z_n \, $,  $ \, \overline{z} \in \Z_2 \, $.

 \vskip4pt

   {\it (b)} \  $ {\tS(n)}_{[0]} \, $  is a Lie subalgebra of the even part  $ \, {\tS(n)}_\zero \, $  of  $ \, \tS(n) \, $;  as it coincides with  $ \, {S(n)}_0 \, $,  it is (again) isomorphic to  $ \rsl(n) \, $,  see \S \ref{def-S(n)};
 \vskip4pt
   {\it (c)} \  $ {\tS(n)}_{[-1]} \, $,  as a module for  $ \, {\tS(n)}_0 \cong \rsl(n) \, $,  is
 the dual of the standard module of  $ \rsl(n) \, $.
\end{free text}

\medskip

\begin{free text}  \label{Lie-struct_tildeS(n)}
 {\bf The Lie structure in  $ \tS(n) \, $.}  \  To describe the Lie (super)structure of  $ \tS(n) $  we can use explicit formulas for the Lie bracket and the  $ 2 $--operation  of elements in  $ B_{\tS(n)} \, $.  Part of this basis is a subset of the basis  $ B_{S(n)} $  of  $ S(n) $  in  \S \ref{def-S(n)},  thus for
 these elements we refer to formulas therein.
                                                      \par
   We look now at the remaining cases.  The first case
 is
  $$  \displaylines{
   \Big[ \big( \xi_1 \cdots \xi_n - 1 \big) \, \partial_i \; , \, \big( \xi_1 \cdots \xi_n - 1 \big) \, \partial_j \,\Big]  \,\; = \;\,  {(-1)}^j \, \xi_1 \cdots \widehat{\xi_j} \cdots \xi_n \, \partial_i \, + \, {(-1)}^i \, \xi_1 \cdots \widehat{\xi_i} \cdots \xi_n \, \partial_j  \,\; =   \hfill  \cr
   \hfill   = \;\,
   \begin{cases}
      \; {(-1)}^{i+j-1} \, \xi_1 \cdots \widehat{\xi_i} \cdots \widehat{\xi_j} \cdots \xi_n \, \big( \xi_i \, \partial_i - \xi_j \, \partial_j \big)   &  \quad \forall \;\;  i < j  \\
      \; 2 \, {(-1)}^j \, \xi_1 \cdots \widehat{\xi_j} \cdots \xi_n \, \partial_j   &  \quad \forall \;\;  i = j  \\
      \; {(-1)}^{i+j-1} \, \xi_1 \cdots \widehat{\xi_j} \cdots \widehat{\xi_i} \cdots \xi_n \, \big( \xi_j \, \partial_j - \xi_i \, \partial_i \big)   &  \quad \forall \;\;  i > j
   \end{cases}  }  $$
   \indent   The second case
 splits in turn into several subcases.  Namely, the first
two
 subcases are
  $$  \displaylines{
   \hfill   \Big[ \big( \xi_1 \cdots \xi_n - 1 \big) \, \partial_j \; , \, \underline{\xi}^{\underline{e}} \, \partial_i \,\Big]  \,\; = \;\,
   \begin{cases}
      \,\; \pm \, \underline{\xi}^{\underline{e} - \underline{e}_{\,j}} \, \partial_j   &  \quad \text{if} \quad  \underline{e}(j) = 1  \\
      \,\; 0   &  \quad \text{if} \quad  \underline{e}(j) = 0
   \end{cases}   \hfill  \text{for  $ \,\; |\underline{e}| > 1 \;\, $  with  $ \,\; \underline{e}(i) = 0 \; $,}  \cr
   \hfill   \Big[ \big( \xi_1 \cdots \xi_n - 1 \big) \, \partial_j \; , \, \xi_k \, \partial_i \,\Big]  \,\; = \;\,  \delta_{j,k} \, \big( \xi_1 \cdots \xi_n - 1 \big) \, \partial_i   \hfill  \text{for  $ \,\; |\underline{e}| = 1 \;\, $  with  $ \,\; k \not= i \; $.} \;\;  }  $$
The third subcase is
  $$  \displaylines{
   \Big[ \big( \xi_1 \cdots \xi_n - 1 \big) \, \partial_j \; , \, \underline{\xi}^{\underline{e}} \, \big( \xi_h \, \partial_h - \xi_k \, \partial_k \big) \,\Big]  \,\; =   \hfill  \cr
   \hfill   = \;\,
   \begin{cases}
      \,\; \pm \underline{\xi}^{\underline{e} - \underline{e}_{\,j}} \,
\big( \xi_h \, \partial_h - \xi_k \, \partial_k \big)  + \, {(-1)}^{|\underline{e}| + 1} \big( \delta_{j,h} - \delta_{j,k} \big) \, \underline{\xi}^{\underline{e}} \, \partial_j   &  \quad \text{if} \;\;  \underline{e}(j) = 1  \\
      \,\; {(-1)}^{|\underline{e}| + 1} \big( \delta_{j,h} - \delta_{j,k} \big) \, \underline{\xi}^{\underline{e}} \, \partial_j   &  \quad \text{if} \;\;  \underline{e}(j) = 0
   \end{cases}  }  $$
where  $ \, \underline{e}(h) = 0 = \underline{e}(k) \, $,  $ \, |\underline{e}| > 0 \; $  (with  $ \, k = h+1 \, $  if we want the second element to belong to  $ B_{S(n)} $   --- yet the formula above holds in general for any  $ h $  and  $ k \, $).  The fourth, last subcase is
  $$  \Big[ \big( \xi_1 \cdots \xi_n - 1 \big) \, \partial_j \; , \, \big( \xi_h \, \partial_h - \xi_k \, \partial_k \big) \,\Big]  \,\; = \;\,  \big( \delta_{j,h} - \delta_{j,k} \big) \, \big( \xi_1 \cdots \xi_n - 1 \big) \, \partial_j  $$

\vskip5pt

   Finally, for the  $ 2 $--operation  we have that (2.2) and the identity in  Definition \ref{def-Lie-salg}{\it (e)\/}  give
  $$  {\big( \big( \xi_1 \cdots \xi_n - 1 \big) \, \partial_j \big)}^{\langle 2 \rangle}  \, = \;  {(-1)}^j \, \xi_1 \cdots \widehat{\xi_j} \cdots \xi_n \, \partial_j   \eqno \forall \;\; j = 1, \dots , n \, .   \qquad   (2.12)  $$
\end{free text}

\medskip

\begin{free text}  \label{def-H(n)}
 {\bf Definition of  $ H(n) \, $.}  \ We retain again notations of  Definition \ref{def-W(n)}  above, with $ \, n \geq 4 \;\, $.
                                                        \par
   Let  $ \; \underline{\omega} := {\big( \omega_{i,j} \big)}_{i=1,\dots,n;}^{j=1,\dots,n;} \; $  be a symmetric, non-singular square matrix of order  $ n $  with entries in  $ \Lambda(n) \, $:  this defines canonically a symplectic form in  $ \Lambda(n) $  which we still denote by  $ \, \underline{\omega} \, $,  namely  $ \, \underline{\omega} := \sum_{i,j=1}^n \omega_{i,j} \, d\xi_i \circ d\xi_j \, $  --- cf.~\cite{ka}, \S 3.3.  For any such form and any  $ \; D := \sum_{i=1}^n P_i\big(\underline{\xi}\big) \, \partial_i \, \in W(n) \; $  the form  $ \, D \underline{\omega} \, $  is naturally defined, and we set
  $$  \tH(\underline{\omega})  \, := \,  \Big\{\, D \in W(n) \;\Big|\; D \underline{\omega} = 0 \,\Big\}  $$
This is a Lie subsuperalgebra of  $ W(n) \, $.  We define a special Lie subsuperalgebra of  $ \widetilde{H}(\underline{\omega}) \, $,  namely
  $$  H(\underline{\omega})  \; := \;  \big[ \tH(\underline{\omega}) \, , \, \tH(\underline{\omega}) \big]  $$
   \indent   All Lie superalgebras  $ \tH(\underline{\omega}) \, $,  for different forms  $ \underline{\omega} \, $,  are isomorphic with each other; the same holds for the various  $ H(\underline{\omega}) \, $.  Thus we can fix a specific form of the matrix  $ \, \underline{\omega} \, $:  we choose it to be
  $$  \underline{\omega}  \, := \,
   \begin{pmatrix}
      \; 0_r  &  I_r \;  \\
      \; I_r  &  0_r \;
   \end{pmatrix}  \quad  \text{if} \; n = 2 \, r \; ,  \quad \qquad
      \underline{\omega}  \, := \,
   \begin{pmatrix}
      \; 0_{r \times r}  &  I_{r \times r}  &  0_{r \times 1} \;  \\
      \; I_{r \times r}  &  0_{r \times r}  &  0_{r \times 1} \;  \\
       \;  0_{1 \times r}   &   0_{1 \times r}    &  1 \;
   \end{pmatrix}  \quad  \text{if} \; n = 2 \, r + 1   \eqno (2.13)  $$
(where  $ I_{r \times r} $  is the identity matrix of order  $ r \, $,  and so on), so that the corresponding form is
  $$  \underline{\omega} = {\textstyle \sum\limits_{i=1}^r} \big( d\xi_i \! \circ d\xi_{r+i} + d\xi_{r+i} \! \circ d\xi_i \big)  \; \text{\ if \ } n \! = \! 2 r \, ,  \;\;
      \underline{\omega} = {\textstyle \sum\limits_{i=1}^r} \big( d\xi_i \! \circ d\xi_{r+i} + d\xi_{r+i} \! \circ d\xi_i \big) + d\xi_n \! \circ d\xi_n \; \text{\ if \ }\! n \! = \! 2 r \! + 1 \, .  $$
   {\sl For this specific choice of form  $ \underline{\omega} \, $},  we use hereafter the notation  $ \; \tH(n) := \tH(\underline{\omega}) \; $,  $ \; H(n) := H(\underline{\omega}) \; $.
 \eject

   The natural  $ \Z $--filtration  on  $ W(n) $  induces a  $ \Z $--filtration  on  $ \tH(\underline{\omega}) $  and  $ H(\underline{\omega}) \, $,  for any  $ \underline{\omega} \, $.  Even more, on  $ \tH(n) $  and  $ H(n) $  the  $ \Z $--grading  on  $ W(n) $  induces  $ \Z $--gradings  as well.  Then the graded Lie superalgebra associated with  $ \tH(\underline{\omega}) \, $,  for any  $ \underline{\omega} \, $,  is isomorphic to  $ H(n) $  as a graded Lie superalgebra.  Several properties of these graded Lie superalgebras are recorded in  \cite{ka},  \S 3.  Here we just recall
  $$  \tH(n) = \! {\textstyle \bigoplus\limits_{z \in \Z}} {\tH(n)}_z \; ,
\;\;  {\tH(n)}_z = {W(n)}_z \,{\textstyle \bigcap}\, \tH(n) \; ,  \qquad
      H(n) = \! {\textstyle \bigoplus\limits_{z \in \Z}} {H(n)}_z \; ,
\;\;  {H(n)}_z = {W(n)}_z \,{\textstyle \bigcap}\, H(n)  $$
 \vskip-7pt
\noindent
 where  $ \; {\tH(n)}_z \not= \{0\} \; $  iff  $ \, -1 \leq z \leq n\!-\!2 \, $,  $ \; {H(n)}_z \not= \{0\} \; $  iff  $ \, -1 \leq z \leq n\!-\!3 \, $.  Moreover  $ \; {\tH(n)}_z = {H(n)}_z \; $  for  $ \, -1 \leq z \leq n\!-\!3 \, $  and  $ \, \text{\it dim}\big({\tH(n)}_{n-2}\big) = 1 \, $:  in particular,  $ \; \tH(n) = H(n) \oplus {\tH(n)}_{n-2} \; $.
                                                         \par
   Like for  $ W(n) $  and  $ S(n) \, $,  this  $ \Z $--grading  yields also gradings by cyclic groups, namely  $ \,\; \tH(n) = \! {\textstyle \bigoplus\limits_{[z] \in \Z_n}} \hskip-5pt {\tH(n)}_{[z]} \;\, $
and  $ \,\; H(n) = \! {\textstyle \bigoplus\limits_{[z] \in \Z_{n-1}}} \hskip-5pt {H(n)}_{[z]} \;\, $
with  $ \; {\tH(n)}_{[z]} \, := \hskip-7pt {\textstyle \bigoplus\limits_{\zeta \, \equiv \, z \!\! \mod n}} \hskip-11pt {\tH(n)}_\zeta
 \phantom{\Big|}
 \; $  for all  $ \, [z] \in \Z_n \, $
and  $ \; {H(n)}_{[z]} \, := \hskip-7pt {\textstyle \bigoplus\limits_{\zeta \, \equiv \, z \!\! \mod (n-1)}} \hskip-11pt {H(n)}_\zeta
 \phantom{\Big|}
 \; $  for all  $ \, [z] \in \Z_{n-1} \, $  (the integers modulo  $ \, n\!-\!1 \, $).
 Again, in both cases the  $ \Z $--grading  yields a  $ \Z $--filtration,  coinciding with the one induced by  $ W(n) \, $,  whose associated graded Lie superalgebra is isomorphic to  $ H(n) $  and  $ \tH(n) $  respectively.  Finally (as for  $ W $,  $ \tS $  and  $ S \, $)  the  $ \Z $--grading  is consistent with the  $ \Z_2 $--grading  both in  $ \tH(n) $  and in  $ H(n) \, $,
%
 in the obvious sense.

\smallskip

   To describe  $ \tH(n) \, $,  $ H(n) $  and their graded summands, we exploit a different realization of them.

\smallskip

  For any given closed differential form  $ \underline{\omega} $  as above, consider on  $ \, V_n := \text{\it Span}\big(\xi_1, \dots, \xi_n\big) \, $  the bilinear form corresponding to  $ \underline{\omega} \, $,  and take on the vector space  $ \Lambda(n) $  the structure of Clifford algebra associated to  $ V_n $  with such a form.  Then the supercommutator on  $ \Lambda(n) $  reads
  $$  \{f,g\}  \; := \;  {(-1)}^{p(f)} \, {\textstyle \sum_{i,j=1}^n} \, \dot{\omega}_{i,j} \, \partial_i(f) \, \partial_j(g)   \eqno (2.14)  $$
(we use braces instead of square brackets for psychological reasons) where  $ \, \dot{\underline{\omega}} = {\big( \dot{\omega}_{i,j} \big)}_{i=1,\dots,n;}^{j=1,\dots,n;} = {\underline{\omega}}^{-1} \, $  is the inverse of the matrix  $ \, \underline{\omega} \, $.  If we consider on  $ \Lambda(n) $  its natural associative product and the Lie superbracket in (2.14), it is a (supercommutative) Poisson superalgebra, which we denote by  $ P_{\underline{\omega}}(n) \, $.
 By the analog to Poincar{\'e}'s lemma,
there exists a Lie superalgebra epimorphism
  $$  \phi : P_{\underline{\omega}}(n) \longrightarrow \tH(\underline{\omega}) \; ,  \quad  f \mapsto D_f := {\textstyle \sum_{i,j=1}^n} \, \dot{\omega}_{i,j} \, \partial_i(f) \, \partial_j   \eqno (2.15)  $$
which shifts the  $ \Z $--grading  by  $ -2 \, $,  \, i.e.~$ \, \phi\big( {P_{\underline{\omega}}(n)}_z \big) = {\tH(\underline{\omega})}_{z-2} \, $  for all  $ \, z \, $  (so the induced  $ \Z_2 $--gra-ding is preserved) and has kernel the  $ \KK $--span  of  $ 1 \, $;  so  $ \; P_{\underline{\omega}}(n) \! \Big/ \KK \cdot 1_{\scriptscriptstyle P_{\underline{\omega}}(n)} \cong \tH(\underline{\omega}) \; $  via an isomorphism induced by  $ \phi \, $.  Moreover, the restriction of  $ \phi $  to  $ \, \big\{\, f \in P_{\underline{\omega}}(n) = \Lambda(n) \,\big|\, \epsilon(f) = 0 \,\big\} \, $,  where  $ \epsilon(f) $  is the constant term in  $ f $  (thought of as a skew-polynomial in the  $ \xi_i $'s),  is a bijection: thus we have
  $$  \tH(\underline{\omega})  \; = \;  \big\{ D_f \,\big|\, \epsilon(f) = 0 \big\}  \qquad  \text{and}  \quad  \big[ D_f \, , \, D_g \big]  \, = \,  D_{\{f,g\}}   \eqno (2.16)  $$
   The outcome is that we can describe  $ \tH(\underline{\omega}) $  via the isomorphism of it with  $ \, P_{\underline{\omega}}(n) \! \Big/ \KK \cdot 1_{\scriptscriptstyle P_{\underline{\omega}}(n)} \, $,  for which we can compute the Lie superbracket using (2.14).  We do it now for the canonical  $ \underline{\omega} \, $.

\vskip7pt

   Let  $ \underline{\omega} $  be the canonical matrix chosen as in (2.13).  Then we write  $ \, P(n) := P_{\underline{\omega}}(n) \, $  for the corresponding Poisson superalgebra.
 In this case (2.14) and (2.15) take the simpler form
  $$  \displaylines{
   \hfill   \{f,g\}  \; := \;  {(-1)}^{p(f)} \Big( {\textstyle \sum_{s=1}^r} \, \big( \partial_s(f) \, \partial_{r+s}(g) + \partial_{r+s}(f) \, \partial_s(g) \big) + \, \delta_{n \in (2\,\N+1)} \, \partial_{2r+1}(f) \, \partial_{2r+1}(g) \Big)   \hfill (2.17)  \cr
   \hfill   f  \,\; \mapsto \;\,  D_f \, := \, {\textstyle \sum_{s=1}^r} \, \big( \partial_s(f) \, \partial_{r+s} + \partial_{r+s}(f) \, \partial_s \big) \, + \; \delta_{n \in (2\,\N+1)} \, \partial_{2r+1}(f) \, \partial_{2r+1}   \hfill (2.18)  }  $$
where  $ \; \delta_{n \in (2\,\N+1)} := 1 \; $  if  $ n $  is odd  (written as  $ \, n = 2 \, r +1 \, $)  and  $ \; \delta_{n \in (2\,\N+1)} := 0 \; $  otherwise.

\vskip5pt

   For each  $ \, z \in \Z \, $  the set
 $ \; B_{P(n)\,;\,z} \, := \, \big\{\, \underline{\xi}^{\underline{e}} \;\big|\; \underline{e} \! \in \! {\{0,1\}}^n , \; |\underline{e}| = z \,\big\} \; $
is a  $ \KK $--basis  of  $ {P(n)}_z \, $,  and for each  $ \, \overline{z} \in \Z_2 \, $  the set
 $ \; B_{P(n)\,;\,\overline{z}} \, :=  \hskip-7pt  {\textstyle \bigcup\limits_{(z \; \text{mod} \, 2) \, = \, \overline{z}}}  \hskip-9pt  B_{P(n)\,;\,z}
 \; $  is a  $ \KK $--basis  of  $ {P(n)}_{\overline{z}} \;\, $.  It follows that
 $ \; B_{P(n)} := {\textstyle \bigcup\limits_{z \in \Z}} B_{P(n)\,;\,z} \; $
%
is a  $ \KK $--basis  ($ \Z $--homogeneous  and  $ \Z_2 $--homogeneous)  of  $ P(n) \, $. Applying  $ \phi \, $,  we get bases for  $ \tH(n) \, $,  $ H(n) $  and their graded summands.  Focusing on  $ H(n) \, $,  we find:
 \vskip5pt
   {\it (a)} \  a basis of the  $ \KK $--vector  space  $ H(n) $  is given by
  $$  B_{H(n)}  \; := \;  \Big\{\, D_{\underline{\xi}^{\underline{e}}} \;\Big|\; \underline{e} \in {\{0,1\}}^n , \, 0 < |\underline{e}| < n \,\Big\}  $$
 \eject

\noindent
 This basis is homogeneous (for the  $ \Z $--grading  and the  $ \Z_2 $--grading),  i.e.~$ \; B_{H(n)} = \bigcup_{z \in \Z} B_{H(n)\,;\,z} \; $  and  $ \; B_{H(n)} = \bigcup_{\overline{z} \in \Z_2} B_{H(n)\,;\,\overline{z}} \; $  where  $ \, B_{H(n)\,;\,z} \! := B_{H(n)} \cap {H(n)}_z \, $  is a basis of  $ {H(n)}_z \, $  and  $ \, B_{H(n)\,;\,\overline{z}} := B_{H(n)} \cap {H(n)}_{\overline{z}} \; $  is a basis of  $ {H(n)}_{\overline{z}} \; $,  \, for every  $ \, z \in \Z \, $  and  $ \, \overline{z} \in \Z_2 \; $.
 \vskip4pt
   {\it (b)} \  $ {H(n)}_0 \, $  is a Lie subalgebra of the even part  $ \, {H(n)}_\zero \, $  of  $ \, H(n) \, $,  isomorphic to  $ \rso(n) \, $,  the latter being considered with respect to the canonical form  $ \underline{\omega} \, $;  an isomorphism is given by
  $$  \displaylines{
   \hfill   D_{\xi_h \xi_k} \, = \; \xi_k \, \partial_{r+h} - \xi_h \, \partial_{r+k}  \,\; \mapsto \;\,  \text{e}_{k,r+h} - \text{e}_{h,r+k}   \hfill  \forall \;\;\; 1 \leq h < k \leq r  \cr
   \hfill   D_{\xi_h \xi_{r+k}} \, = \; \xi_{r+k} \, \partial_{r+h} - \xi_h \, \partial_k  \,\; \mapsto \;\, \text{e}_{r+k,r+h} - \text{e}_{h,k}   \hfill  \forall \;\;\; 1 \leq h \leq r \, , \; 1 \leq k \leq r  \cr
   \hfill   D_{\xi_{r+h} \xi_{r+k}} \, = \; \xi_{r+k} \, \partial_h - \xi_{r+h} \, \partial_k \,\; \mapsto \;\,  \text{e}_{r+k,h} - \text{e}_{r+h,k}   \hfill  \forall \;\;\; 1 \leq h < k \leq r  \cr
   \hfill   D_{\xi_t \xi_{2r+1}} \, = \; \xi_{2r+1} \, \partial_{r+t} - \xi_t \, \partial_{2r+1}  \,\; \mapsto \;\,  \text{e}_{2r+1,r+t} - \text{e}_{t,2r+1}   \hfill  \forall \;\;\; 1 \leq t \leq r \, ,  \;\; n = 2 \, r + 1  \cr
   \hfill   D_{\xi_{r+t} \xi_{2r+1}} \, = \; \xi_{2r+1} \, \partial_t - \xi_{r+t} \, \partial_{2r+1}  \,\; \mapsto \;\,  \text{e}_{2r+1,r+t} - \text{e}_{t,2r+1}   \hfill  \forall \;\;\; 1 \leq t \leq r \, ,  \;\; n = 2 \, r + 1  }  $$
(with notation as in Definition \ref{def-W(n)}),  the last two formulas being in use only for odd  $ \, n = 2 \, r + 1 \, $.
 \vskip4pt
   {\it (c)} \  $ {H(n)}_{-1} \, $,  as a module for  $ \, {H(n)}_0 \! \cong \rso(n) \, $,  \, is
 the dual of the standard module of  $ \rso(n) \, $.

%
\end{free text}

\medskip

\begin{free text}  \label{Lie-struct_H(n)}
 {\bf The Lie structure in  $ H(n) \, $.}  \  We describe now the Lie (super)structure of  $ H(n) $
in terms of its basis  $ B_{H(n)} \, $.  We make use of the isomorphism  $ \; P(n) \! \Big/ \KK \cdot 1_{\scriptscriptstyle P(n)} \cong \tH(n) \; $  along with formulas (2.16--17) and the fact that  $ \; \tH(n) = H(n) \oplus {\tH(n)}_{n-2} \; $  with  $ \; {\tH(n)}_{n-2} = \KK \cdot D_{\xi_1 \xi_2 \cdots \xi_n} \; $.  In short, we have to compute the brackets  $ \; \big\{\, \underline{\xi}^{\underline{a}} \; , \, \underline{\xi}^{\underline{b}} \,\big\} \; $  in  $ P(n) $  for all  $ \, \underline{a} \, , \underline{b} \in {\{0,1\}}^n \, $  such that  $ \, 0 < |\underline{a}| \, , |\underline{b}| < n \; $.

\medskip

   By (2.17) we have
  $$  \Big\{\, \underline{\xi}^{\underline{a}} \; , \, \underline{\xi}^{\underline{b}} \,\Big\}  \; = \;
  {(-1)}^{|\underline{a}|} \Bigg(\, {\textstyle \sum\limits_{s=1}^r} \, \Big( \partial_s\big( \underline{\xi}^{\underline{a}} \big) \, \partial_{r+s}\big( \underline{\xi}^{\underline{b}} \big) + \partial_{r+s}\big( \underline{\xi}^{\underline{a}} \big) \, \partial_s\big( \underline{\xi}^{\underline{b}} \big) \Big) + \, \delta_{n \in (2\,\N+1)} \, \partial_{2r+1}\big( \underline{\xi}^{\underline{a}} \big) \, \partial_{2r+1}\big( \underline{\xi}^{\underline{b}} \big) \!\Bigg)  $$
 With a detailed (yet elementary) analysis, one finds only two possibilities.  The first one is
  $$  \exists \; s \, : \; \underline{a}(s) = 1 = \underline{b}(r \! + s) \, , \, \underline{a}(r \! + s) = 1 = \underline{b}(s)  \quad \Longrightarrow \quad  \Big\{\, \underline{\xi}^{\underline{a}} \; , \, \underline{\xi}^{\underline{b}} \,\Big\} \; = \; 0   \eqno (2.19)  $$
   \indent   On the other hand, the second possibility is either
  $$  \hskip5pt  \begin{matrix}
   \nexists \; s \, : \; \underline{a}(s) = 1 = \underline{b}(r \! + \! s) \, , \, \underline{a}(r \! + \! s) = 1 = \underline{b}(s)  \\
   \hfill   \big(\, \underline{a}(2\,r\!+\!1) \, , \, \underline{b}(2\,r\!+\!1) \big) \; \not= \; (1,1)  \quad
      \end{matrix}
   \hskip3pt \Bigg\}  \hskip5pt \Longrightarrow \hskip4pt  \Big\{\, \underline{\xi}^{\underline{a}} \; , \, \underline{\xi}^{\underline{b}} \,\Big\}  \, = \,  {\textstyle \sum\limits_{k=1}^r} \, \eta_k \, \underline{\xi}^{\underline{a} + \underline{b} - \underline{e}_{\,k} - \underline{e}_{\,r+k}}   \eqno (2.20)  $$
for some  $ \, \eta_k \in \{ +1, 0, -1 \} \, $,  \, or (only possible if  $ n $  is odd, written as  $ \, n = 2 \,r + 1 \, $)
 \vskip-9pt
  $$  \begin{matrix}
   \nexists \; s \, : \; \underline{a}(s) = 1 = \underline{b}(r \! + \! s) \, , \, \underline{a}(r \! + \! s) = 1 = \underline{b}(s)  \\
   \hfill   \big(\, \underline{a}(2\,r\!+\!1) \, , \, \underline{b}(2\,r\!+\!1) \big) \; = \; (1,1)
      \end{matrix}
   \hskip3pt \Bigg\}  \hskip5pt \Longrightarrow \hskip4pt
     \Big\{\, \underline{\xi}^{\underline{a}} \; , \, \underline{\xi}^{\underline{b}} \,\Big\}  \, = \,  \eta_{2r+1} \, \underline{\xi}^{\underline{a} + \underline{b} - 2 \underline{e}_{\,2r+1}}  \quad   \eqno (2.21)  $$
%
%
for some  $ \, \eta_{2r+1} \in \{ +1, 0, -1 \} \, $.  For later use, we also record the following fact
  $$  \begin{matrix}
   \underline{a}(s) = \underline{a}(r \! + \! s) \, , \; \underline{b}(s) = \underline{b}(r \! + \! s)  \;\;\; \forall \, s  \\
   \hfill   \underline{a}(2\,r\!+\!1) \, = \, 0 \, = \, \underline{b}(2\,r\!+\!1) \big)  \quad \qquad
      \end{matrix}
   \hskip3pt \Bigg\}  \hskip9pt \Longrightarrow \quad  \Big\{\, \underline{\xi}^{\underline{a}} \; , \, \underline{\xi}^{\underline{b}} \,\Big\}  \, = \,  0   \eqno (2.22)  $$
 also proved by straightforward inspection.
   Similar results, still proved by direct analysis, are
  $$  \displaylines{
   \hfill   \hskip21pt  \Big\{\, \underline{\xi}^{\underline{a}} \; , \, \Big\{\, \underline{\xi}^{\underline{a}} \; , \, \underline{\xi}^{\underline{b}} \,\Big\} \!\Big\}  \, = \,  0   \hskip55pt   \forall \;\; \underline{a} \, , \, \underline{b} \, \in \, {\{0,1\}}^n \; : \; |\underline{a}| > 3   \hfill (2.23)  \cr
   \hfill   \hskip21pt  \big\{ \xi_{2r+1} \, , \, \xi_{2r+1} \big\} \, = \, -1  \;\; ,  \qquad  \Big\{\, \underline{\xi}^{\underline{a}} \; , \, \underline{\xi}^{\underline{a}} \,\Big\} \, = \, 0   \hskip35pt  \forall \;\; \underline{a} \in {\big\{0,1\big\}}^n \setminus \big\{ \underline{e}_{\,2r+1} \big\}   \hfill (2.24)  }  $$

\vskip5pt

   All the formulas above yield also the Lie brackets among elements of  $ B_{H(n)} \, $,  via the identity  $ \; \Big[ D_{\underline{\xi}^{\underline{a}}} \; , \, D_{\underline{\xi}^{\underline{b}}} \Big] \, = \, D_{[\underline{\xi}^{\underline{a}} \, , \, \underline{\xi}^{\underline{b}}]} \; $   --- see formula (2.16).  Similarly, from these formulas and from the identity in  Definition \ref{def-Lie-salg}{\it (e)},  taking also (2.24) into account, we get for the  $ 2 $--operation  the formulas
  $$  {D_{\underline{\xi}^{\underline{a}}}}^{\langle 2 \rangle}  \, = \;  0   \eqno \forall \;\; \underline{a} \in {\{0,1\}}^n   \qquad \qquad \qquad   (2.25)  $$

\end{free text}

\vskip5pt

   \centerline{\sl  \dbend \;  From now on,  $ \, \fg $  will be a Lie superalgebra of Cartan type:  $ \, W(n) \, $,  $ S(n) \, $,  $ \tS(n) $  or  $ H(n) \, $.  \,\dbend  }
 \eject

\begin{free text}  \label{def_Cartan-subalg_roots_etc}
 {\bf Cartan subalgebras, roots, root spaces.}  Let  $ \fg $  be a Lie superalgebra of Cartan type.  We call  {\it Cartan subalgebras of\/  $ \fg $}  the Cartan subalgebras of the reduc\-tive Lie algebra  $ \, \fg_0 \, $,  which is  $ \rgl(n) $,  $ \rsl(n) $  or  $ \rso(n) $  respectively if  $ \fg $  is  $ W(n) \, $,  $ S(n) $  or  $ \tS(n) \, $,  or  $ H(n) \, $.  We fix one of them (the ``standard'' one), namely
  $ \, \fh \, := \, \text{\it Span}_{\,\KK}\big(\{\, \xi_k \, \partial_k \,\}_{1 \leq k \leq n}\big) \, $ {\sl for case  $ W(n) \, $},
  $ \, \fh \, := \, \text{\it Span}_{\,\KK}\big(\{ (\xi_k \, \partial_k - \xi_{k+1} \, \partial_{k+1}) \,\}_{1 \leq k \leq n-1}\big) \, $  {\sl for cases  $ S(n) $  and  $ \tS(n) \, $},  and
  $ \, \fh \, := \, \text{\it Span}_{\,\KK}\big(\{ (\xi_k \, \partial_k - \xi_{r+k} \, \partial_{r+k}) \,\}_{1 \leq k \leq r}\big) \, $  {\sl for case  $ H(n) \, $,  with  $ \, r := \big[n/2\big] \, $}.
 In all cases, the spanning set we considered in  $ \fh $  is actually a  $ \KK $--basis.

\smallskip

   Now consider the element  $ \, \cE := \sum_{i=1}^n \xi_i \, \partial_i \in {W(n)}_0 \, $:  we set  $ \, \overline{\fh} := \fh \, $,  $ \, \overline{\fg}_0 := \fg_0 \, $,  $ \, \overline{\fg} := \fg \, $  when  $ \fg $  is of type  $ W $  or  $ \tS $,  and  $ \, \overline{\fh} := \fh + \KK \, \cE \, $,  $ \, \overline{\fg}_0 := \fg_0 + \KK \, \cE \, $,  $ \, \overline{\fg} := \fg + \KK \, \cE \, $  when  $ \fg $  is of type  $ S $  or  $ H \, $.
  In this way  $ \, \cE \in \overline{\fg} \, $  iff  $ \fg $  is  $ \Z $--graded,
 with
 $ \; \big[ \cE , X \big] = z \, X \; $  for  $ \, X \in \fg_z \, $  (cf.~\cite{ka}, \S 1.2.12 and \S 4.1.2).

\smallskip

   The Cartan subalgebra  $ \overline{\fh} $  acts by adjoint action on  $ \overline{\fg} \, $,  and  $ \fg $  itself is an  $ \overline{\fh} $--submodule.  Thus we have a decomposition of  $ \fg $  into weight spaces for the  $ \overline{\fh} $--action,  namely (cf.~\cite{se},  \S 4)
  $$  \fg  \; = \;  {\textstyle \bigoplus_{\alpha \in \overline{\fh}^{\,*}}} \, \fg_\alpha  \;\; ,  \qquad  \fg_\alpha \, := \, \big\{\, w \in \fg \;\big|\, [h,w] = \alpha(h) \, w \, , \,\; \forall \; h \in \overline{\fh} \,\big\}  $$
called  {\sl root (space) decomposition\/}  of  $ \fg \, $,  where  $ \, \fg_{\alpha=0} = \fh \, $.
 The terminology for such a context is standard:
 $ \; \Delta := \big\{\, \alpha \in \overline{\fh}^{\,*} \setminus \{0\} \;\big|\, \fg_\alpha \not= \{0\} \big\} \, $  is the  {\it root system\/}  of  $ \fg \, $,  its elements are called  {\it roots},  we have  {\it root spaces\/}  (each  $ \, \fg_\alpha \, $  for  $ \, \alpha \not= 0 \, $),  {\it root vectors},  etc.  For every root  $ \alpha $  one has either  $ \, \fg_\alpha \subseteq \fg_\zero \, $  or  $ \, \fg_\alpha \subseteq \fg_\uno \, $:  accordingly, we call  $ \alpha $  {\it even\/}  or  {\it odd\/};  we set  $ \; \Delta_{\overline{z}} := \big\{\, \alpha \in \Delta \,\big|\, \fg_\alpha \subseteq \fg_{\overline{z}} \,\big\} \; $  for  $ \, \overline{z} \in \big\{\, \overline{0} \, , \, \overline{1} \,\big\} \, $.

\smallskip

   We call a root  $ \alpha $  {\it essential\/}  if  $ \, -\alpha \in \Delta \, $,  and  {\it nonessential\/}  if  $ \, -\alpha \not\in \Delta \, $.

\smallskip

   Finally, the  {\it multiplicity\/}  of a root  $ \alpha \, $,  by definition, is the non-negative integer  $ \, \mu(\alpha) := \text{\it dim}\,(\fg_\alpha) \, $;  \, then we call the root  $ \alpha $  respectively  {\it thin\/}  or  {\it thick\/}  if  $ \, \mu(\alpha) = 1 \, $  or  $ \, \mu(\alpha) > 1 \, $.

\vskip7pt

  We denote by  $ Q $  the  $ \Z $--lattice  $ \, Q := \Z.\Delta \, $  spanned by  $ \Delta $  inside  $ \overline{\fh}^{\,*} \, $.  The root space decomposition is a  $ Q $--grading  of  $ \fg \, $  (as a Lie superalgebra).  This  $ Q $--grading  is compatible with the  $ \Z $--grading  $ \; \fg = {\textstyle \bigoplus}_{z \in \Z} \fg_z \; $  (only as vector space for  $ \fg $  of type  $ \tS \, $:  cf.~the  $ \underline{\text{\sl Remark}} \/ $  in  \S \ref{def-tildeS(n)}),  in the following sense: one has  $ \, \fh \subseteq \fg_0 \, $  and for each root  $ \alpha \, $  also  $ \, \fg_\alpha \subseteq \fg_{\text{\it ht}(\alpha)} \, $  for a unique integer  $ \, \text{\it ht}(\alpha) \in \Z \; $,  so that
  $$  {\ } \hskip-17pt   \fg_0  \, =  {\textstyle \bigoplus\limits_{\alpha \,:\, \text{\it ht}(\alpha) = 0}} \hskip-5pt \fg_\alpha \,\; {\textstyle \bigoplus} \;\, \fh \;\; ,  \; \qquad  \fg_z  \, =  {\textstyle \bigoplus\limits_{\alpha \,:\, \text{\it ht}(\alpha) = z}} \hskip-5pt \fg_\alpha  \qquad \big(\, \forall \; z \in \Z \setminus \{0\} \,\big)  $$
The unique integer  $ \, \text{\it ht}(\alpha) \in \Z \, $  thus associated with every root  $ \, \alpha \, $  is called the  {\it height\/}  of  $ \alpha \, $.  As another consequence, we can also partition the set of roots according to the height, namely,
 $ \; \Delta \, = \, {\textstyle \coprod}_{z \in \Z} \, \Delta_z \; $  with  $ \; \Delta_z := \big\{\, \alpha \in \Delta \,\big|\, \text{\it ht}\,(\alpha) = z \,\big\} \; $;
 then  $ \; \Delta_{\overline{0}} = {\textstyle \coprod}_{z \in 2\Z} \, \Delta_z \; $  and  $ \; \Delta_{\overline{1}} = {\textstyle \coprod}_{z \in (2\Z+1)} \, \Delta_z \; $.

\vskip3pt

   We set  $ \; \Delta_{\zero^{\,\uparrow}} := \Delta_\zero \setminus \Delta_0 \, $,  $ \; \Delta_{\uno^{\,\uparrow}} := \Delta_\uno \setminus \Delta_{-1} \; $  and  $ \; \tDelta := {\big\{ (\alpha,j) \big\}}^{\alpha \in \Delta}_{1 \leq j \leq \mu(\alpha)} \; $;  we denote  $ \, \pi : \tDelta \relbar\joinrel\twoheadrightarrow \Delta \, $  the map  $ \, \talpha = (\alpha,k) \;{\buildrel \pi \over \mapsto}\; \alpha \; $,  and  $ \, \tDelta_{\hat{z}} := \pi^{-1}\big(\Delta_{\hat{z}}\big) \, $  for  $ \, {\hat{z}} \in \big\{ \zero \, , \uno \, , \zero^\uparrow , \uno^\uparrow \big\} \cup \{-1,0,1,\dots,n\} \, $.
%
%

\vskip5pt

   We conclude introducing (or recalling) the notion of  {\it coroot\/}  adfssociated with a classical root.

\vskip5pt

   When  $ \fg $  is of type  $ W $,  $ S $  or  $ H \, $,  the Lie algebra  $ \overline{\fg}_0 $  is reductive, of the form  $ \; \overline{\fg}_0 \, = \, \overline{\fg}_0^{\,\text{\it ss}} \oplus \KK \, \cE \; $  where  $ \overline{\fg}_0^{\,\text{\it ss}} $  is its semisimple part (actually simple) and  $ \; \KK \, \cE \; $  is its radical (actually the centre); similarly  $ \overline{\fh} $  splits as  $ \; \overline{\fh} = \overline{\fh}_{\text{\it ss}} \oplus \KK \, \cE \, $  with  $ \; \overline{\fh}_{\text{\it ss}} := \overline{\fh} \, \cap \, \overline{\fg}_0^{\,\text{\it ss}} \, $:  explicitly,  $ \overline{\fh}_{\text{\it ss}} $  is the (standard) Cartan subalgebra of  $ \, \overline{\fg}_0^{\,\text{\it ss}} \cong \rsl_n \, $   --- if  $ \fg $  is of type  $ W(n) $  or  $ S(n) $  ---   or of  $ \, \fg_0 \cong \rso_n \, $  --- if $ \fg $  is of type  $ H(n) \, $.
                                                                        \par
   When  $ \fg $  is of type  $ \tS(n) $  the situation is simpler: the Lie algebra  $ \, \overline{\fg}_0 = \fg_0 $  is (semi)simple, isomorphic to  $ \rsl_n \, $.  In order to unify notation, we write then  $ \, \overline{\fg}_0^{\,\text{\it ss}} := \overline{\fg}_0 = \fg_0 \, $  and  $ \; \overline{\fh}_{\text{\it ss}} := \overline{\fh} = \fh \, $:  explicitly,  $ \, \overline{\fh}_{\text{\it ss}} = \fh \, $  is the (standard) Cartan subalgebra of  $ \, \overline{\fg}_0 = \fg_0 \cong \rsl_n \, $.

\vskip2pt

   As in any semisimple Lie algebra, the Killing form induce a  $ \KK $--linear  isomorphism  $ \, {\big( \overline{\fh}_{\text{\it ss}} \big)}^{\!*} \! {\buildrel \cong \over {\lhook\joinrel\relbar\joinrel\twoheadrightarrow}} \, \overline{\fh}_{\text{\it ss}} \, $,  denoted by  $ \, \gamma \mapsto t_\gamma \; $.  When  $ \fg $  is of type  $ \tS(n) $  we use this isomorphism to define (as usual:  cf.~\cite{hu})  the  {\it coroot}  $ \, H_\alpha \in \overline{\fh}_{\text{\it ss}} = \overline{\fh} \, $  associated with any  {\sl classical\/}  root  $ \, \alpha \in \Delta_0 \, \big(\! \subset {\big( \overline{\fh}_{\text{\it ss}} \big)}^* \! = {\overline{\fh}}^{\,*} \,\big) \; $.
                                                                \par
   When  $ \fg $  is of type  $ W $,  $ S $  or  $ H $  instead, note that every linear functional  $ \phi $  on  $ \overline{\fh}_{\text{\it ss}} $  uniquely extends to a linear functional on  $ \overline{\fh} \, $  such that  $ \, \phi(\cE) = 0 \, $:  this yields an embedding of  $ {\big( \overline{\fh}_{\text{\it ss}} \big)}^* $  into  $ {\overline{\fh}}^{\,*} \, $,  so every root (in classical sense) of the simple Lie algebra  $ \overline{\fg}_0^{\,\text{\it ss}} $  identifies with an element of  $ {\overline{\fh}}^{\,*} $.  Conversely, let  $ \, \delta \in {\overline{\fh}}^{\,*} \, $  be the unique  $ \KK $--linear  functional on  $ \, \overline{\fh} = \overline{\fh}_{\text{\it ss}} \oplus \KK \, \cE \, $  such that  $ \, \delta(\cE) = 1 \, $  and  $ \, \delta\big(\overline{\fh}_{\text{\it ss}}\big) = \{0\} \, $:  coupled with the embedding  $ \, {\big( \overline{\fh}_{\text{\it ss}} \big)}^* \lhook\joinrel\longrightarrow {\overline{\fh}}^{\,*} \, $,  this yields  $ \, {\overline{\fh}}^{\,*} = {\big( \overline{\fh}_{\text{\it ss}} \big)}^* \oplus \KK \, \delta \, $.  Now, using this last description of  $ {\overline{\fh}}^{\,*} $  and the isomorphism  $ \, {\big( \overline{\fh}_{\text{\it ss}} \big)}^{\!*} \! {\buildrel \cong \over {\lhook\joinrel\relbar\joinrel\twoheadrightarrow}} \, \overline{\fh}_{\text{\it ss}} \, $  induced by the Killing form, we extend the latter to an isomorphism  $ \, {\overline{\fh}}^{\,*} \! \lhook\joinrel\relbar\joinrel\twoheadrightarrow \overline{\fh} \, $,  $ \, \big(\, \gamma \mapsto t_\gamma \, , \, \cE \mapsto \delta \,\big) \; $:  via this, we define again the  {\it coroot}  $ \, H_\alpha \in \overline{\fh}_{\text{\it ss}} \subsetneqq \overline{\fh} \, $  associated with every  {\sl classical\/}  root  $ \, \alpha \in \Delta_0 \, \big(\! \subset {\big( \overline{\fh}_{\text{\it ss}} \big)}^* \! \subsetneqq {\overline{\fh}}^{\,*} \,\big) \; $.
 \eject

   We describe now in detail all these objects; we distinguish cases  $ W $,  $ S \, $,  $ \tS $  and  $ H \, $  (cf.~\cite{se}).

\vskip11pt

   {\sl  $ \, \underline{\text{Case \ }  \, \fg = W(n)} \; $}:  \; Let  $ \, \fg := W(n) \, $.  Then  $ \, \fg_0 \cong \rgl_n \, $,  and we fix the Cartan subalgebra  $ \, \fh = \overline{\fh} \, $  as above.  Let  $ \phantom{\Big|} \{ \varepsilon_1, \varepsilon_2, \dots, \varepsilon_n \} \, $  be the standard basis of  $ \fh^* \, $:  then the root system of  $ \fg $  is given by
 \vskip-4pt
  $$  \Delta  \; = \;  \big\{\, \varepsilon_{i_1} \! + \cdots + \varepsilon_{i_k} \! - \varepsilon_j \;\big|\; 1 \! \leq \! i_1 \! < \! i_2 \! < \! \cdots \! < \! i_k \! \leq \! n \, , \; j = 1, \dots, n \,\big\} \setminus \{0\}  $$
%
%
\vskip2pt

   From this description, one sees at once that every root  $ \, \alpha \in \Delta \, $  can be written as  $ \pm $  the sum of ``simple roots'' chosen in a ``simple root system'', just like for simple Lie superalgebras of  {\sl basic\/}  type.  For instance, one can take as ``simple root system'' the set  $ \, \Pi := \{\, \varepsilon_1 - \varepsilon_2 \, , \, \dots \, , \, \varepsilon_{n-1} - \varepsilon_n \, , \, \varepsilon_n \,\} \, $,  which is even ``distinguished'' (i.e., it contains only one odd root, in this case  $ \varepsilon_n \, $).

\vskip4pt

  If  $ \, \alpha = \sum_{s=1}^k \, \varepsilon_{i_s} \! - \varepsilon_j \, $  with  $ \, j \not\in \{i_1,\dots,i_k\} \, $,  \, then  $ \, \fg_{\alpha = \sum_{s=1}^k \varepsilon_{i_s} \! - \varepsilon_j} \, $  has  $ \KK $--basis  the singleton  $ \, \big\{ \xi_{i_1} \! \cdots \xi_{i_k} \partial_j \big\} \, $,  so that  $ \; \mu(\alpha) = 1 \; $;  \, if instead  $ \, \alpha = \sum_{r=1}^h \, \varepsilon_{i_r} \, $,  \, then  $ \, \fg_{\alpha = \sum_{r=1}^h \varepsilon_{i_r}} \, $  has  $ \KK $--basis  $ \, \big\{\, \xi_{i_1} \! \cdots \xi_{i_h} \, \xi_\ell \, \partial_\ell \;\big|\; \ell \in \{1,\dots,n\} \setminus \{ i_1, \dots, i_h \} \big\} \, $,  so  $ \; \mu(\alpha) = n-h \; $.  Thus the  {\sl thick\/}  roots are those of  $ \, \big\{ \sum_{s=1}^h \varepsilon_{i_s} \,\big|\; h < n\!-\!1 \big\} \, $  and the  {\sl thin\/}  ones those of  $ \, \big\{ \sum_{s=1}^k \varepsilon_{i_s} \! - \varepsilon_j \,\big|\, j \not\in \{i_1,\dots,i_k\} \big\} \,\bigcup\, \big\{ \sum_{s=1}^{n-1} \varepsilon_{i_s} \big\} \, $.

\vskip7pt

   Now for any  $ \, \alpha = \varepsilon_{i_1} \! + \cdots + \varepsilon_{i_k} \! - \varepsilon_j \in \Delta \, $,  its  {\sl height\/}  is easy to read off: it is given by  $ \, \text{\it ht}(\alpha) = k-1 \, $.

\vskip5pt

   If a root  $ \, \alpha \, $  is  {\sl even essential},  then  $ \; \fg_\alpha \, , \fg_{-\alpha} \subseteq \fg_0 \; $,  the multiplicity of both  $ \alpha $  and  $ -\alpha $  is 1 and  $ \fg_\alpha \, $,  $ \fg_{-\alpha} $  together generate a Lie subalgebra isomorphic to  $ \rsl(2) $  inside  $ \fg_0 \, $.  The set of even essential roots is  $ \, \big\{ \varepsilon_i \! - \! \varepsilon_j \big\}_{i,j=1,\dots,n;}^{i\not=j} \, $,  the corresponding root spaces being  $ \, \fg_{\varepsilon_i - \varepsilon_j} \! = \! \text{\it Span}_\KK\big(\{ \xi_i \, \partial_j \}\big) \! = \! \KK \, \xi_i \, \partial_j \; $.
                                                                 \par
   If instead  $ \, \alpha \, $  is  {\sl odd essential},  then either  $ \; \fg_\alpha \subseteq \fg_{-1} \, $,  $ \; \mu(\alpha) = 1 \, $,  $ \; \mu(-\alpha) = n-1 \, $,  \,
 or  $ \; \fg_{-\alpha} \subseteq \fg_{-1} \, $,  $ \; \mu(-\alpha) = 1 \, $,  $ \; \mu(\alpha) = n-1 \; $.
   In the first case,  $ \, \alpha \in \big\{\! -\!\varepsilon_j \big\}_{1 \leq j \leq n} \, $  is thin while  $ \, -\alpha \in \big\{ \varepsilon_j \big\}_{1 \leq j \leq n} \, $  is thick, and for each  $ j $  the root spaces  $ \, \fg_{\alpha = -\varepsilon_j} \, $  and  $ \, \fg_{-\alpha = \varepsilon_j} \, $  have  $ \KK $--basis  $ \, \big\{ \partial_j \big\} \, $  and  $ \, \big\{\, \xi_j \, \xi_\ell \, \partial_\ell \;\big|\; 1 \leq \ell \leq n \, , \; \ell \not= j \,\big\} \, $  respectively (see above).  In the second case the converse holds.

\vskip3pt

   Finally the  {\sl nonessential\/}  roots  $ \, \alpha \, $  are all those such that  $ \, \text{\it ht}\,(\alpha) > 1 \, $  together with all those of the form  $ \; \alpha = \varepsilon_{i_1} \! + \varepsilon_{i_2} \! - \varepsilon_j \; $  (hence  $ \, \text{\it ht}\,(\alpha) = 1 \, $)  with  $ \, \, j \not= i_1, i_2 \, $.   \hfill  $ \diamondsuit $

\vskip9pt

   {\sl  $ \, \underline{\text{Case \ }  \, \fg = S(n)} \; $}:  \; Let  $ \, \fg := S(n) \, $.  Now  $ \, \fg_0 \cong \rsl_n \, $,  we fix the Cartan subalgebra  $ \fh $  and set  $ \, \overline{\fh} := \fh \oplus \KK \, \cE \, $  as above.  Clearly the root system of  $ \, \fg := S(n) \, $  is a subset of that of  $ W(n) \, $,  namely
 \vskip-4pt
  $$  \Delta  \; = \;  \big\{\, \varepsilon_{i_1} \! + \cdots + \varepsilon_{i_k} \! - \varepsilon_j \;\big|\; 1 \! \leq \! i_1 \! < \! i_2 \! < \! \cdots \! < \! i_k \! \leq \! n \, , \; k \! < \! n \, , \; j \! = \! 1, \dots, n \,\big\} \setminus \{0\}  $$
in short, the roots of  $ S(n) $  are those of  $ W(n) $  whose height is less than  $ \, n\!-\!1 \, $.  In particular, the characterization of the height of a root of  $ S(n) $  is exactly the same as for  $ W(n) \, $.

\vskip4pt

   By construction, root spaces of  $ S(n) $  or  $ W(n) $  enjoy the relation  $ \, {S(n)}_\alpha = {W(n)}_\alpha \bigcap S(n) \, $.  The explicit description given for case  $ W $  implies that  $ \, {S(n)}_\alpha = {W(n)}_\alpha \, $  when  $ \alpha $  is thin for  $ W(n) \, $,  so that it is  {\sl thin\/}  for  $ S(n) $  as well (the multiplicity being 1 in both cases).  Instead, if a root  $ \alpha $  is of the form $ \, \alpha = \sum_{s=1}^h \, \varepsilon_{i_s} \, $,  with  $ \, h < n \! - \! 1 \, $  (so it is thick for  $ W(n) \, $),  then the space  $ {S(n)}_\alpha $  has  $ \KK $--basis
 $ \; \big\{\, \underline{\xi}^{\underline{i}} \; \big( \xi_j \, \partial_j - \xi_{j+1} \, \partial_{j+1} \big) \;\big|\; 1 \leq j \leq n\!-\!1 \, , \; \underline{i}(j) = 0 = \underline{i}(j\!+\!1) \,\big\} \; $
with  $ \, \underline{\xi}^{\underline{i}} := \xi_{i_1} \! \cdots \xi_{i_h} \, $;  so  $ \alpha $  in  $ S(n) $  has multiplicity  $ \; \mu(\alpha) = n\!-\!h\!-\!1 \; $,  hence  $ \alpha $  is  {\sl thick\/}  for  $ S(n) $  if  $ \, h < n\!-\!2 \, $,  and it is  {\sl thin\/}  if  $ \, h = n\!-\!2 \, $.

\smallskip

  Finally, it is clear that roots of  $ S(n) $  have a certain degree (for the  $ \Z_2 $--grading  or the  $ \Z $--grading),  and they are essential or non-essential, exactly as they have or they are for  $ W(n) \, $.  \hfill  $ \diamondsuit $

\vskip11pt

   {\sl  $ \, \underline{\text{Case \ }  \, \fg = \tS(n)} \; $}:  \; Let  $ \, \fg := \tS(n) \, $.  Like for  $ S(n) \, $,  we have  $ \, \fg_0 \cong \rsl_n \, $  and we fix  $ \, \overline{\fh} := \fh \, $  with  $ \fh $  as above.  By the analysis in  \S\S \ref{def-tildeS(n)}, \ref{Lie-struct_tildeS(n)},  we see that the root system of  $ \tS(n) $ is
 \vskip-4pt
  $$  \Delta  \; = \;  \big\{\, \varepsilon_{i_1} \! + \cdots + \varepsilon_{i_k} \! - \varepsilon_j \;\big|\; 1 \! \leq \! i_1 \! < \! i_2 \! < \! \cdots \! < \! i_k \! \leq \! n \, , \; k \! < \! n \, , \; j \! = \! 1, \dots, n \,\big\} \setminus \{0\}  $$
like for  $ S(n) \, $,
 but now  $ \; \varepsilon_1 + \cdots + \varepsilon_n = 0 \; $  in  $ \, {\overline{\fh}}^{\,*} \!\! = \fh^* \, $.
Also, for root spaces we have the following.
                                                      \par
   For every root  $ \alpha $  of  $ S(n) $  having non-negative height we have  $ \; {\tS(n)}_\alpha = {S(n)}_\alpha \; $.
 Instead, for the roots of height $ \, -1 \, $  (which are  $ \, -\varepsilon_1 \, $,  $ \dots \, $,  $ \, -\varepsilon_n \, $)  we have  $ \; {\tS(n)}_{-\varepsilon_j} = \text{\it Span}_{\,\KK}\big(\! (\xi_1 \cdots \xi_n - 1) \, \partial_j \big) \; $  for all  $ \, j=1, \dots, n \, $,  with  $ \, \big\{ (\xi_1 \cdots \xi_n - 1) \, \partial_j \big\} \, $  being a $ \KK $--basis  of any such root space.
                                                      \par
    It is worth stressing that the roots of the form  $ \, -\varepsilon_j \, $  are the only  $ \, \alpha \in \Delta \, $  such that  $ \, 2\,\alpha \in \Delta \, $.  Indeed,
 $ \, 2\,(-\varepsilon_j) = (\,\varepsilon_1 + \cdots + \widehat{\,\varepsilon_j\,} + \cdots + \varepsilon_n) - \varepsilon_j \, \in \, \Delta \, $  (for all  $ \, j = 1, \dots, n \, $),  using the identity  $ \; \varepsilon_1 + \cdots + \varepsilon_n = 0 \; $  (in  $ \, \fh^* $);
instead,
direct analysis shows that  $ \, 2\,\alpha \not\in \Delta \, $  for all  $ \, \alpha \not\in {\{-\varepsilon_j\}}_{j=1,\dots,n;} \, $.

\smallskip

   Finally, we can say that the roots of  $ \tS(n) $  have a certain degree (for the  $ \Z_2 $--grading  or the  $ \Z_n $--grading),  they are essential or non-essential, much like we did for  $ W(n) $  and  $ S(n) \, $.   \hfill  $ \diamondsuit $

\vskip11pt

   {\sl  $ \, \underline{\text{Case \ }  \, \fg = H(n)} \; $}:  \; Let  $ \, \fg := H(n) \, $.  Now  $ \, \fg_0 \cong \rso_n \, $,  we fix the Cartan subalgebra  $ \fh $  as above, and then  $ \phantom{\Big|}  \overline{\fh} := \fh + \KK \, \cE \, \supsetneqq \, \fh \; $.  To describe the root system and the root spaces in this case, we point out the explicit form of the action of  $ \cE $  on the basis vectors of  $ \, \fg := H(n) \, $  considered in  \S \ref{def-H(n)}.  Let  $ \, D_{\underline{\xi}^{\underline{a}}} \, $  be any one of these elements, with  $ \, \underline{a} \in {\{0,1\}}^n \, $,  with  $ \, 0 < |\underline{a}| < n \, $.  The formulas in  \S \ref{Lie-struct_W(n)}  give
 \vskip-5pt
  $$  \Big[\, \cE \, , D_{\underline{\xi}^{\underline{a}}} \,\Big]  \; = \;  \big( |\underline{a}| - 2 \big) \, D_{\underline{\xi}^{\underline{a}}}   \eqno \forall \;\; \underline{a} \in {\{0,1\}}^n \, : \, 0 < |\underline{a}| < n   \qquad (2.26)  $$
   \indent   As before, we write  $ \, n = 2 \, r \, $  or  $ \, n = 2 \, r + 1 \, $,  withe  $ \, r := \big[ n/2 \big] \; $.  Let now  $ \, \{ \varepsilon_1, \dots, \varepsilon_r \} \, $  be the stan\-dard basis in the weight space of  $ \fg_0 = \rso_n \, $:  adding  $ \delta $  we get a basis of  $ \, {\overline{\fh}}^{\,*} \, $.  The root system is
  $$  \displaylines{
   \Delta  \, = \,  \big\{ \pm \varepsilon_{i_1} \! \pm \cdots \pm \varepsilon_{i_k} \! + m \, \delta \;\big|\; 1 \! \leq \! i_1 \! < \! \cdots \! < \! i_k \! \leq \! r , \, k - \! 2 \leq \! m \! \leq n - \! 2 , \, m \! \geq \! -1 , \, m\!-\!k \in 2\,\Z \big\}   \hfill \quad \text{if \ }  n = 2 \, r  \cr
   \qquad   \Delta  \, = \,  \big\{ \pm \varepsilon_{i_1} \! \pm \cdots \pm \varepsilon_{i_k} \! + m \, \delta \;\big|\; 1 \! \leq \! i_1 \! < \! \cdots \! < \! i_k \! \leq \! r \, , \; k - \! 2 \leq \! m \! \leq n - \! 2 \, , \; m \! \geq \! -1 \,\big\}   \hfill \quad \text{if \ } \;  n = 2 \, r + 1  }  $$
   \indent   As to root spaces, consider any root  $ \; \alpha = \pm \varepsilon_{i_1} \! \pm \cdots \pm \varepsilon_{i_k} \! + m \, \delta \; $  written as  $ \; \alpha = \sum_{j=1}^r d_j \, \varepsilon_j + m \, \delta \; $  with  $ \, d_j \in \{+1,0,-1\} \, $  for all $ j \, $.  Then the formulas in  \S \ref{Lie-struct_H(n)}  along with (2.26) lead to find that for every root  $ \; \alpha = \sum_{j=1}^r d_j \, \varepsilon_j + m \, \delta \; $  the root space  $ \, \fg_{\alpha = \sum_{j=1}^r d_j \, \varepsilon_j + m \, \delta} \, $  has  $ \KK $--basis  the set
 \vskip-4pt
  $$  \Big\{\, D_{\underline{\xi}^{\underline{a}}} \,\;\Big|\;\, \underline{a} \in {\{0,1\}}^n \, : \; |\underline{a}| \! - \! 2 = m \, , \;\, \underline{a}(j) - \underline{a}(r\!+\!j) = d_j \;\; \forall \, j \,\Big\}  $$
   \indent   The  {\sl height\/}  of  $ \; \alpha = \pm \varepsilon_{i_1} \! \pm \cdots \pm \varepsilon_{i_k} \! + m \, \delta \; $  is  $ \, \text{\it ht}(\alpha) = m \, $;  in particular, if $ \, D_{\underline{\xi}^{\underline{a}}} \in \fg_\alpha \, $  then  $ \, |\underline{a}| = \text{\it ht}(\alpha) + 2 \; $;  the parity of  $ \alpha $  is the same as  $ m \, $,  and its multiplicity is  $ \; \mu(\alpha) = \Big( {r \atop {[(m-k)/\,2]}} \Big) \, $,  where
 $ \big[ (m-k)/2 \,\big] $  is the integral part of  $ \, (m-k)/2 \; $.

\vskip7pt

   Finally, note that the roots  $ \, \alpha = m \, \delta \, $  are the only ones whose double might be a root too.   \hfill  $ \diamondsuit $
\end{free text}

\vskip9pt

\begin{free text}  \label{fin-root-vects}
 {\bf Finiteness properties of roots and root vector action.}
  As we saw in  \S \ref{def_Cartan-subalg_roots_etc},  the root space decomposition yields a  $ Q $--grading  of  $ \fg \, $.  As a consequence, each root vector acts nilpotently.
 Actually, we can make this result more precise.  We begin with some easy properties of roots:
\end{free text}

\vskip-9pt

\begin{lemma}  \label{root-props}
 Let  $ \, \alpha \, , \beta \in \Delta \, $.  Then
 $ \, (t \, \alpha + \beta) \not\in \! \Delta \, $  for all  $ \, t > 2 \, $.  Moreover, if  $ \, 2 \, \alpha \in \! \Delta \, $  or  $ \, 3 \, \alpha \in \! \Delta \, $,  then  $ \, \fg = H(n) \, $,  $ \, \alpha \in \big\{\, m \, \delta \,\big|\, m \in \Z \,\big\} \, $,  or (only for the first case)  $ \, \fg = \tS(n) \, $,  $ \, \alpha \in {\{-\varepsilon_i\}}_{i=1,\dots,n;} \; $.
\end{lemma}

\vskip5pt

   Here now are the finiteness properties of the root vector action we need:

\vskip9pt

\begin{proposition}  \label{ad_square/cube}
 The following hold:
 \vskip3pt
   {\it (a)} \,  Let  $ \, \alpha \! \in \! \Delta \, $,  $ \, \alpha \! \not= \! -\varepsilon_i \, $  if  $ \; \fg \! = \! \tS(n) \, $  or  $ \, \alpha \! \not= \! (2\,\N+1) \, \delta \, $  if  $ \, \fg \! = \! H(2\,r\!+\!1) \; $.  Then  $ \; \big[ \fg_\alpha \, , \fg_\alpha \big] = \{0\} \; $.
 \vskip3pt
   {\it (b)} \, Let  $ \, \alpha \in \Delta_\zero \, $,  $ \, x_\alpha \in \fg_\alpha \, $.  Then  $ \; \ad(x_\alpha)^{\,3} = 0 \; $;  if  $ \, \alpha \not\in \Delta_0 \, $,  then  $ \; \ad(x_\alpha)^{\,2} = 0 \; $.
\end{proposition}

\begin{proof}
 {\it (a)} \,  If  $ \fg $  is of type  $ W $,  $ S $  or  $ \tS $,  then direct inspection shows  $ \, 2 \, \alpha \not\in \! \Delta \, $  for  $ \, \alpha \in \! \Delta \, $,  with  $ \, \alpha \not= -\varepsilon_i \, $  in type  $ \tS \, $.  If instead  $ \fg $  is of type  $ H(n) $  then  $ \, 2 \, \alpha \not\in \Delta \, $  whenever  $ \, \alpha \not\in \Z\,\delta \, $  (as direct inspection shows).  Then in all these cases one has  $ \; \big[ \fg_\alpha \, , \fg_\alpha \big] \subseteq \fg_{2\alpha} = \{0\} \; $.  Finally, assume  $ \fg $  is of type  $ H(n) $  and  $ \, \alpha \in \Z\,\delta \, $,  say  $ \, \alpha = k \, \delta \, $.  The root space  $ \fg_\alpha $  has  $ \KK $--basis  the set  $ \, \big\{\, D_{\underline{\xi}^{\underline{a}}} \;\big|\; \underline{a}(s) = \underline{a}(r\!+\!s) \;\; \forall \; 1 \leq s \leq r \,\big\} \, $,  and by formulas (2.19--21) we may have  $ \, \Big[\, D_{\underline{\xi}^{\underline{a}}} \; , \, D_{\underline{\xi}^{\underline{b}}} \,\Big] = D_{\{\, \underline{\xi}^{\underline{a}} \; , \; \underline{\xi}^{\underline{b}} \,\}} \not= 0 \, $  only if  $ n $  is odd  $ \, \big(\! = 2\,r + 1 \big) \, $.
 \vskip3pt
   {\it (b)} \,  First of all, assume  $ \, \alpha \in \Delta_\zero \, \big(\! \setminus \Delta_0 \big) \, $,  and consider any other root  $ \, \beta \in \Delta \, $.  By  Lemma \ref{root-props}  one has  $ \, 3\,\alpha \not\in \Delta \, $  and  $ \, 3\,\alpha + \beta \not\in \Delta \, $:
 this easily
implies  $ \; \ad(x_\alpha)^{\,3} = 0 \; $.
 \vskip3pt
%
%
%
   Now let  $ \fg $  be of type  $ W $,  $ S $  or  $ \tS \, $.  Direct inspection shows that, given  $ \, \alpha \in \Delta_\zero \, $  and  $ \, \beta \in \Delta \, $,  one has  $ \; 2 \, \alpha + \beta \in \Delta \; $  {\sl only if}  $ \, \alpha \in \Delta_0 \, $.  Then one can argue as above for the action of  $ \ad(x_\alpha)^{\,2} $  onto  $ \fh $  and onto root spaces  $ \fg_\beta \, $,  and eventually prove the second part of the claim for types  $ W $,  $ S $  and  $ \tS \, $.

\vskip4pt

   Let now  $ \fg $  be of type  $ H $,  say  $ \, \fg = H(n) \; $.  From  \S \ref{def_Cartan-subalg_roots_etc} we know that the root vector  $ \, x_\alpha \, $  is a linear combination of root vectors of the form  $ \, D_{\underline{\xi}^{\underline{a}}} \, $,  say  $ \; x_\alpha = \sum_k c_k \, D_{\underline{\xi}^{\underline{a}_k}} \; $ for some  $ \, c_k \in \KK \, $.
                                                                  \par
   First assume  $ \; \alpha \not\in \Z\,\delta \; $.  Then  $ \, 2\,\alpha \not\in \Delta \, $,  by  Lemma \ref{root-props},  thus  $ \, \fg_{2\alpha} = \{0\} \, $  and for any pair of summands in the expansion of  $ x_\alpha \, $  we have  $ \; \Big\{ D_{\underline{\xi}^{\underline{a}_{k'}}} \, , \, D_{\underline{\xi}^{\underline{a}_{k''}}} \!\Big\} \in \big\{ \fg_\alpha \, , \, \fg_\alpha \big\} \subseteq \fg_{2\alpha} = \{0\} \; $.
 \eject

   Now let  $ \; \alpha \in \Z\,\delta \, $,  say  $ \, \alpha = m \, \delta \, $.  Then  $ m $  is even because $ \, \alpha \in \Delta_\zero \, $;  but  $ \, |\underline{a}_k| = \text{\it ht}(\alpha) + 2 = m + 2 \, $  for each  $ \underline{a}_k $  occurring in the expansion of  $ x_\alpha \, $,  so in particular  $ |\underline{a}_k| $  is (independent of  $ k $  and) even, thus in the end  $ \underline{a}_k $  has the property that  $ \, \underline{a}_k(s) = \underline{a}_k(\ell \! + \! s) \, $  for all  $ \, 1 \leq s \leq \ell \, $.  Now (2.22) applies, yielding  $ \; \Big\{ \underline{\xi}^{\underline{a}_{k'}} \, , \, \underline{\xi}^{\underline{a}_{k''}} \!\Big\} = 0 \; $  and so in the end  $ \; \Big\{ D_{\underline{\xi}^{\underline{a}_{k'}}} \, , \, D_{\underline{\xi}^{\underline{a}_{k''}}} \!\Big\} = D_{\{ \underline{\xi}^{\underline{a}_{k'}} , \, \underline{\xi}^{\underline{a}_{k''}} \,\}} = D_0 = 0 \; $.
                                                                 \par
   In any case, we found that all summands in the expansion of  $ x_\alpha $  do commute with each other.  It follows that
%
%
 $ \; \ad(x_\alpha)^{\,2} =
%
%
\sum_k c_k^{\;2} \, \ad\big( D_{\underline{\xi}^{\underline{a}_k}} \big)^2 \; $;  so it is enough to prove the claim for  $ \, x_\alpha = D_{\underline{\xi}^{\underline{a}}} \; $.
 \vskip3pt
   Again,  $ \, x_\alpha = D_{\underline{\xi}^{\underline{a}}} \, $  implies  $ \, |\underline{a}| = \text{\it ht}(\alpha) + 2 \, $:  as  $ \, \alpha \in \Delta_\zero \setminus \Delta_0 \, $  implies  $ \, \text{\it ht}(\alpha) \geq 2 \, $,  this yields  $ \, |\underline{a}| \geq 4 \, $.  Then (2.23) gives  $ \; \ad\big( \underline{\xi}^{\underline{a}} \big)^2 \big(\underline{\xi}^{\underline{b}}\big) = \Big\{ \underline{\xi}^{\underline{a}} \, , \big\{ \underline{\xi}^{\underline{a}} \, , \underline{\xi}^{\underline{b}} \big\} \Big\} = 0 \, $  for  $ \, \underline{b} \in {\{0,1\}}^n \, $:  thus $ \; \ad\big( D_{\underline{\xi}^{\underline{a}}} \big)^2(\fg_\beta) = \{0\} \; $  for all  $ \, \beta \in \Delta \; $.  Similarly  $ \; \ad\big( D_{\underline{\xi}^{\underline{a}}} \big)^2(h) = \big[ D_{\underline{\xi}^{\underline{a}}} \, , [ D_{\underline{\xi}^{\underline{a}}} \, , h \,] \big] = -\alpha(h) \, \big[ D_{\underline{\xi}^{\underline{a}}} \, , D_{\underline{\xi}^{\underline{a}}} \,\big] = 0 \; $  for any  $ \, h \in \fh \, $,  as  $ \, D_{\underline{\xi}^{\underline{a}}} \in \fg_\zero \; $:  thus  $ \; \ad\big( D_{\underline{\xi}^{\underline{a}}} \big)^2(\fh) = \{0\} \; $  too.  Then arguing as above we get  $ \; \ad\big( D_{\underline{\xi}^{\underline{a}}} \big)^2 = 0 \; $.
\end{proof}

\vskip11pt

\begin{free text}  \label{triang-dec_Borel-subs_Lie-sub-objs}
 {\bf Triangular decompositions, Borel subalgebras and special sub-objects.}  Since  $ \, \text{\sl char}(\KK) = 0 \, $,  we can identify hereafter the fundamental subfield of $ \KK $  with  $ \mathbb{Q} \, $.
 \vskip3pt
   Let  $ \; \fh_{\mathbb{Q}} := \mathbb{Q} \otimes_\Z \big\{ H \in \fh \,\big|\, \alpha(H) \in \Z \, , \; \forall \; \alpha \in \Delta \big\} \; $;  one sees easily that  $ \fh_{\mathbb{Q}} $  is an integral  $ {\mathbb{Q}} $--form  of  $ \fh \, $,  and  $ \, \alpha(h) \in \mathbb{Q} \, $  for all $ \, h \in \fh_{\mathbb{Q}} \, $.
   We call  $ \, h \in \fh_{\mathbb{Q}} \, $  {\it regular\/}  if  $ \, \alpha(h) \not= 0 \, $  for all  $ \, \alpha \in \Delta \, $.  Any regular  $ \, h \in \fh \, $  defines a partition  $ \; \Delta = \Delta^+ \coprod \Delta^- \; $  where  $ \, \Delta^+ := \big\{ \alpha \in \Delta \,\big|\, \alpha(h) > 0 \,\big\} \, $  and  $ \, \Delta^- := \big\{ \alpha \in \Delta \,\big|\, \alpha(h) < 0 \,\big\} \, $:  the roots in  $ \Delta^+ $  are said to be  {\it positive},  those in  $ \Delta^- $  {\it negative}.
All this defines a  {\it triangular decomposition}
%
 $ \; \fg \, = \, \fg^+ \oplus \, \fh \, \oplus \, \fg^- \; $  with  $ \; \fg^\pm \, := \, {\textstyle \bigoplus_{\alpha \in \Delta^\pm}} \, \fg_\alpha \, $,
 as well as  {\it Borel subalgebras}  $ \; \fb^\pm := \fh \oplus \fg^\pm \; $.  From now on, we fix a specific Borel subalgebra of  $ \, \fg_0 \, $,  denoted  $ \fb_0 \, $,  and we restrict ourselves to consider those Borel subalgebras of  $ \fg $  containing  $ \fb_0 \, $.  Among these,
 when  $ \, \fg \not\cong \tS(n) \, $
there is a  {\sl maximal\/}  one,  $ \; \fb_{\max} := \fb_0 \oplus \Big(\, {\scriptstyle \bigoplus\limits_{i>0}} \, \fg_i \Big) \; $,  \, and a  {\sl minimal\/}  one,  $ \; \fb_{\min} := \fb_0 \oplus \fg_{-1} \; $.
 \vskip5pt
   For later use, we introduce notation (consistent with  \S \ref{def_Cartan-subalg_roots_etc})
 $ \; \Delta_{\zero^{\,\uparrow}} \! := \Delta_\zero \setminus \Delta_0 \; $,
$ \; \Delta_{\uno^{\,\uparrow}} \! := \Delta_\uno \setminus \Delta_{-1} \; $,
 $ \; \Delta_\bullet^\pm := \Delta_\bullet \cap \Delta^\pm \, $,  $ \; \tDelta_\bullet^\pm := \tDelta_\bullet \cap \big( \Delta_\bullet^\pm \! \times \N_+ \big) \; $
where  $ \, \Delta_\bullet \in \big\{ \Delta \, , \, \Delta_0 \, , \, \Delta_\zero \, , \, \Delta_\uno \, , \, \Delta_{\zero^{\,\uparrow}} \, , \, \Delta_{\uno^{\,\uparrow}} \big\} \, $.
\end{free text}

\vskip3pt

   Starting from the root decomposition of  \S \ref{def_Cartan-subalg_roots_etc},  we can introduce special ``sub-objects'' of  $ \, \fg \, $:

\vskip11pt

\begin{definition}  \label{def_Lie-sub-objs}
   Basing on the root decomposition in  \S \ref{def_Cartan-subalg_roots_etc},  we have
 $ \; \fg_\zero \, = \, \fh \oplus \Big( {\textstyle \bigoplus_{\alpha \in \Delta_\zero \,}} \fg_\alpha \Big) \, $,  $ \; \fg_\uno \, = \, {\textstyle \bigoplus_{\gamma \in \Delta_\uno \,}} \fg_\gamma \; $  and  $ \; \fg_0 \, = \, \fh \oplus \Big( {\textstyle \bigoplus_{\alpha \in \Delta_0 \,}} \fg_\alpha \Big) \, $.  Then set
  $ \; \fg_{\zero^{\,\uparrow}}  :=  {\textstyle \bigoplus_{\alpha \in \Delta_{\zero^{\,\uparrow}}}} \fg_\alpha  \; $,
  $ \; \fg_{\uno^{\,\uparrow}}  :=  {\textstyle \bigoplus_{\gamma \in \Delta_{\uno^{\,\uparrow}}}} \fg_\gamma \; $,
and  $ \; \fg_{t^\uparrow} = \oplus_{z > t \,} \fg_z \; $  for all  $ \, t \geq -1 \, $.  Note then that  $ \; \fg_{-1^\uparrow} = \fg_\zero \oplus \fg_{\uno^{\,\uparrow}} \; $  and  $ \; \fg_{0^\uparrow} = \fg_{\zero^{\,\uparrow}} \! \oplus \fg_{\uno^{\,\uparrow}} \; $;  note also that  $ \; {\big( \fg_{t^\uparrow} \!\big)}_\zero = \fg_\zero \cap \fg_{t^\uparrow} = \oplus_{\textstyle \!{{z > t} \atop {z \in 2\Z}}\,} \fg_z \; $,  for all  $ \, t \geq -1 \, $.  For  $ \, \fg \not\cong \tS(n) \, $  consider also  $ \; \fg_{-1,0} := \fg_{-1} \oplus \fg_\zero \; $.
\end{definition}

\vskip5pt

\begin{remark}  \label{rem_sub-objs}
  Note that  $ \fg_0 \, $,  $ \, \fg_{\zero^{\,\uparrow}} \, $  and  $ \fg_\zero $  are Lie subalgebras of  $ \fg \, $,  while  $ \, \fg_{\zero^{\,\uparrow}} \, $  is a Lie ideal of  $ \fg_\zero \, $,  and we have a Lie algebra splitting  $ \; \fg_\zero = \fg_0 \ltimes \fg_{\zero^{\,\uparrow}} $  (semidirect product of Lie algebras),  with  $ \fg_0 $  reductive and  $ \fg_{\zero^{\,\uparrow}} $  nilpotent.  Similarly,  $ {\big( \fg_{t^\uparrow} \!\big)}_\zero \, $  is a Lie subalgebra of  $ \, \fg_\zero \, $,  for all  $ \, t \geq -1 \, $;  when  $ \, q > p \, $,  $ \, {\big( \fg_{q^\uparrow} \!\big)}_\zero \, $  is a Lie ideal of  $ \, {\big( \fg_{p^\uparrow} \!\big)}_\zero \, $.  On the other hand,  $ \, \fg_{-1^\uparrow} \, $  and  $ \, \fg_{0^\uparrow} \, $  are Lie supersubalgebras of  $ \fg \, $,  with  $ \, \fg_{-1^\uparrow} \supseteq \fg_{0^\uparrow} \, $.  Moreover,  $ \, \fg_{0^\uparrow} \, $  is a nilpotent Lie superalgebra, and a Lie ideal of  $ \fg_{-1^\uparrow} \, $:  then  $ \; \fg_{-1^\uparrow} = \fg_0 \ltimes \fg_{0^\uparrow} \, $  (semidirect product); similarly, for  $ \, \fg \not\cong \tS(n) \, $  we have  $ \; \fg_{-1,0} = \fg_{-1} \rtimes \fg_0 \; $.
\end{remark}

\vskip19pt

  \subsection{Basics on  $ \fg $--modules}  \label{g-modules}

\smallskip

   {\ } \quad   Later on we shall work with  $ \fg $--modules  and  $ \overline{\fg} $--modules,  so we specify now a few definitions.

\vskip9pt

\begin{definition}  \label{def_weight-modules-etc}
 Let  $ \fh \, $,  resp.~$ \overline{\fh} \, $,  be a fixed Cartan subalgebra of  $ \fg \, $,  resp.~of  $ \overline{\fg} \, $,  as in  \S \ref{def_Cartan-subalg_roots_etc}.
 \vskip2pt
   {\it (a)} \,  Any  $ \fg $--module  $ V $  is said to be a  {\it weight module\/}  if  $ \; V = \bigoplus_{\lambda \in \fh^*} V_\lambda \; $  where we set  $ \; V_\lambda := \big\{\, v \in \! V \;\big|\; h.v = \lambda(h)\,v \; , \,\; \forall \; h \in \fh \,\big\} \; $  for all  $ \, \lambda \in \fh^* \, $.  In this case, every  $ V_\lambda $  is called a  {\it weight space\/}  of  $ V $,  and every  $ \, \lambda \in \fh^* \, $  such that  $ \, V_\lambda \not= \{0\} \, $  is called a  {\it weight\/}  of  $ V $.
 \vskip4pt
   Let now  $ V $  be a weight module (for  $ \, \fg $),  and set  $ \, \text{\it Supp}\,(V) := \big\{ \lambda \in \fh^* \,\big|\, V_\lambda \not= \{0\} \big\} \; $:
 \vskip2pt
   {\it (b)} \,  for every  $ \, \lambda \in \fh^* \, $  we call  {\it multiplicity\/}  of  $ \lambda $  the dimension  $ \, \text{\it mult}_V(\lambda) := \text{\it dim}\big(V_\lambda\big) \, $;
 \vskip2pt
   {\it (c)} \,  we call  $ V $  {\it integrable\/}  if all root vectors of  $ \fg $  act locally nilpotently on it;
 \vskip2pt
   {\it (d)} \,  if a splitting of roots into positive and negative ones has been fixed as in  \S \ref{triang-dec_Borel-subs_Lie-sub-objs},  we call  {\it highest weight\/}  of  $ V $  any  $ \, \lambda \in \text{\it Supp}(V) \, $  such that  $ \, \fg_\alpha.V_\lambda = \{0\} \, $,  i.e.~$ \, \lambda + \alpha \not\in \text{\it Supp}(V) \, $,  for all  $ \, \alpha \in \Delta^+ \, $;  then we call  {\it highest weight vector (of weight  $ \lambda $)\/}  any  $ \, v_\lambda \in V_\lambda \setminus \{0\} \, $.
 \vskip5pt
   {\sl We adopt similar definitions and terminology when  $ \overline{\fg} $  and  $ \overline{\fh} $  replace  $ \fg $  and  $ \fh $  respectively}.
 \vskip3pt
   Also, notice that any  $ \fg $--module  bears a structure of  $ \overline{\fg} $--module  too, with  $ \cE $  acting semisimply: moreover, if the  $ \fg $--module  is in fact a weight module, then its structure of  $ \overline{\fg} $--module  can also be chosen so that it is still a weight module for  $ \overline{\fg} $  as well (all this is standard, see  e.g.~\cite{gr}  or  \cite{se}).
\end{definition}

\vskip9pt

\begin{remarks}  \label{rems_weight-mods}  {\ }
 \vskip2pt
   {\it (a)} \,  by standard theory of reductive Lie algebras, every finite dimensional  $ \fg $--module  on which the element  $ \cE $  acts semisimply is automatically a weight module;
 \vskip2pt
   {\it (b)} \,  if  $ \; V = {\textstyle \bigoplus_{\lambda \in \fh^*}} V_\lambda \; $  is any weight module, then  $ \; x_\alpha.V_\lambda \subseteq V_{\lambda + \alpha} \; $  for every root vector  $ \, x_\alpha \in \fg_\alpha \, $  of  $ \, \fg \, $  ($ \, \alpha \in \Delta \, $),  by elementary calculations; it follows that every finite dimensional  $ \fg $--module   --- which is a weight module by  {\it (a)}  ---   is automatically integrable;
 \vskip2pt
   {\it (c)} \,  $ \fg \, $  itself is an integrable weight  $ \overline{\fg} $--module  for the adjoint representation: the set of weights is  $ \, \Delta \cup \{0\} \, $,  and weight spaces are the root spaces and  $ \fh \, $;  similarly (up to details) for  $ \fg $  as a  $ \fg $--module.
 \vskip2pt
   {\it (d)} \,  if  $ \; V = {\textstyle \bigoplus_{\lambda \in \fh^*}} V_\lambda \; $  is any integrable (weight) module, then for each root vector  $ \, x_\alpha \in \fg_\alpha \, $  ($ \, \alpha \in \Delta \, $)  the formal infinite sum  $ \; \exp(x_\alpha) := \sum\limits_{n \in \N} x_\alpha^n \big/ n! \; $  gives a well-defined operator in  $ \rGL(V) \, $.
\end{remarks}

\bigskip

\section{Integral structures}  \label{integr-struct}

\smallskip

   {\ } \quad   In this section we introduce the first, fundamental results we shall build upon to construct our ``Chevalley type'' supergroups associated with  $ \, \fg \, $.  We keep notation and terminology as before.

\medskip

  \subsection{Chevalley bases and Chevalley superalgebras}  \label{che-bas_alg}

\smallskip

  {\ } \quad   In this subsection we extend a classical result: the notion of ``Chevalley basis'' for (semi)simple Lie agebras.  A similar notion was introduced in  \cite{fg2}  for simple Lie superalgebras of  {\sl classical\/}  type, and used to construct affine algebraic supergroups.  We now do the same for the Cartan type case.

\medskip

\begin{definition}  \label{def_che-bas}
 We call  $ \, r := \text{\it rk}\,(\fg) \, $  the  {\it rank of}  $ \, \fg \, $:  by definition, it is the rank of the reductive Lie algebra  $ \fg_0 \, $,  so  $ \; \text{\it rk}\,\big(W(n)\big) = n \, $,  $ \; \text{\it rk}\,\big(S(n)\big) = \text{\it rk}\,\big(\tS(n)\big) = n-1 \; $  and  $ \; \text{\it rk}\,\big(H(n)\big) = \big[n/2\big] \; $.
                                                       \par
   We call  {\sl Chevalley basis\/}  of  $ \, \fg \, $  any  $ \KK $--basis  $ \,\; B \, = \, {\big\{ H_i \big\}}_{i=1, \dots, r;} \, \coprod \, \Big( \coprod_{\alpha \in \Delta} {\big\{ X_{\alpha,k} \big\}}_{k=1, \dots, \mu(\alpha);} \Big) \, = \, {\big\{ H_i \big\}}_{i=1, \dots, r;} \, \coprod \, {\big\{ X_{\talpha} \big\}}_{\talpha \in \tDelta} \;\, $  of  $ \fg $  which is  {\sl homogeneous\/}  (for the cyclic grading  of  $ \fg \, $,  cf.~subsec.~\ref{Lie-superalg_Cartan-type})  and enjoying the following properties (with notation of  \S \ref{def_Cartan-subalg_roots_etc}  for coroots):
 \vskip7pt
   \text{\it (a)}  \hskip5pt   $ \big\{ H_1 , \dots , H_r \big\} \, $  is a  $ \KK $--basis  of  $ \fh \, $,  \, such that  $ \; \beta(H_1), \dots , \beta(H_r) \in \Z \; $  for all  $ \, \beta \in \Delta \, $  and  $ \; H_\alpha \in \fh_\Z \, := \, \text{\it Span}_{\,\Z} \big( H_1 , \dots , H_r \big) \; $  for all  $ \, \alpha \in \Delta_0 \, $;
in particular,  $ \; \big[ H_i \, , H_j \big] = 0 \; $  for all  $ \, 1 \leq i, j \leq r \, $;
 \vskip9pt
   \text{\it (b)}  \hskip4pt   $ {\big\{ X_{\alpha,k} \big\}}_{k=1, \dots, \mu(\alpha);} \; $  is a  $ \KK $--basis  of  $ \, \fg_\alpha \, $,  for all  $ \, \alpha \! \in \! \Delta \, $;  thus  $ \, \big[ H_i \, , X_{\alpha,k} \big] = \alpha(H_i) \, X_{\alpha,k} \;\, \forall\; i \, , \, k \, $;
 \vskip9pt
   \text{\it (c)}  \hskip7pt   $ \big[ X_{\alpha,k} \, , X_{\alpha,k} \big] = 0   \hskip25pt  \forall \;\, \alpha \in \Delta_\zero \, $,  $ \, k \in \{1, \dots, \mu(\alpha)\} \, $;
 \vskip9pt
   \text{\it (d)}  \hskip7pt   $ \big[ X_{\alpha,1} \, , \, X_{-\alpha,1} \big]  \, = \,  H_\alpha  \hskip29pt  \forall \;\; \alpha \in \Delta_0 \; $,
 \vskip5pt
   \indent   \phantom{\text{\it (d)}}  \hskip7pt   $ \big[ X_{\gamma,1} \, , \, X_{-\gamma,k} \big]  \, = \,  \pm H_{\sigma_\gamma(k)}  \hskip25pt  \forall \;\; \gamma \in \Delta_{-1} \, $,  $ \; k = 1, \dots, \mu(-\gamma) \, $,  \;\;  for some embeddings  $ \,\; \sigma_\gamma \! : \big\{ 1, \dots, \mu(-\gamma) \big\} \! \lhook\joinrel\relbar\joinrel\rightarrow \! \big\{ 1, \dots, r \big\} \;\, $  such that  $ \,\; {\big\{\! \pm H_{\sigma_\gamma(k)} \big\}}_{\gamma \in \Delta_{-1} \, , \, k = 1, \dots, \mu(-\gamma)} = \, \big\{ \pm H_1 , \dots , \pm H_r \big\} \; $;
 \vskip7pt
   \text{\it (e)}  \quad  $ \big[ X_{\alpha,k} \, , \, X_{\beta,h} \big]  \, = \,  0  \hskip27pt   \forall \;\, \alpha , \beta \in \Delta \, : \, \alpha + \beta \not\in \big( \Delta \cup \{0\} \big) \, $;
 \vskip11pt
   \text{\it (f)}  \quad  $ \big[ X_{\alpha,k} \, , \, X_{\beta,h} \big]  \, = \, \sum_{t=1}^{\mu(\alpha+\beta)} c_{\alpha,k}^{\,\beta,h}(t) \, X_{\alpha + \beta, t}  \hskip27pt   \forall \;\; \alpha , \beta \in \Delta \, : \, \alpha + \beta \in \Delta \; $,
 \vskip3pt
\noindent
 for some  $ \; c_{\alpha,k}^{\,\beta,h}(t) \in \big\{\, 0 \, , \pm 1 \, , \pm 2 \,\big\} \; $  such that
 \eject

   \hskip11pt   \text{\it (f.1)}  \ if  $ \; \alpha \, , \beta \in \Delta_0 \, $  then  $ \; c_{\alpha,1}^{\,\beta,1}(1) = \pm (p+\!1) \; $  where  $ \, p \in \N \, $  is such that  $ \, \beta - p \, \alpha \in \Delta \, $  and  $ \, \beta - (p+\!1) \, \alpha \not\in \Delta \; $  (note that in this case  $ \, \mu(\alpha) = \mu(\beta) = \mu(\alpha+\beta) = 1 $),
 \vskip4pt
   \hskip11pt   \text{\it (f.2)}  \ if  $ \; (\alpha \, , \beta) \in \big( \Delta_{-1} \times \Delta \big) \cup \big( \Delta \times \Delta_{-1} \big) \; $,  \, then there is one and only one index  $ t' $  such that  $ \;\; c_{\alpha,k}^{\,\beta,h}\big(t'\big) = \pm 1 \; $,  $ \,\; c_{\alpha,k}^{\,\beta,h}(t) = 0  \quad  \forall \;\; t' \not= t \;\; $,
 \vskip3pt
   \hskip11pt   \text{\it (f.3)}  \ if  $ \; \alpha = \beta \; $  (hence  $ \, 2 \,\alpha \in \Delta \, $;  this occurs only for  $ \, \fg = \tS(n) \, $,  $ \, \alpha = -\epsilon_i \, $,  or  $ \, \fg = H(2\,r\!+\!1) \, $,  $ \, \alpha = m \, \delta \, $  with  $ m $ odd),  there is at most one  $ t' $  such that  $ \; c_{\alpha,k}^{\,\beta,h}\big(t'\big) \not= 0 \; $,  \, and then  $ \,\; c_{\alpha,k}^{\,\beta,h}\big(t'\big) = \pm 2 \;\; $;
 \vskip9pt
   \text{\it (g)}  \quad  if  $ \, \alpha \! \in \! \Delta_0 \, $  and  $ \, 2 \, \alpha + \beta \in \Delta \; $,  \, there exists a unique  $ \, t' \in \{1,\dots,\mu(2\,\alpha+\beta)\} \, $  such that
 \vskip3pt
   \centerline{ $ \,\; \big[\, X_{\alpha,1} \, , \big[\, X_{\alpha,1} \, , X_{\beta,h} \,\big] \big] = \pm \, 2 \, X_{2\alpha+\beta,t'} \;\; $; }
 \vskip9pt
   \text{\it (h)}  \quad  for all  $ \; \gamma \! \in \! \Delta_\uno \, $,  we have (with notation as in  {\it (f.3.)\/}  above)
 \vskip3pt
   \centerline{ $ X_{\gamma,k}^{\langle 2 \rangle} = \, 0 \quad $  if all $ c_{\gamma,k}^{\,\gamma,k}(t) $'s  are zero,  \quad  and  \qquad  $ X_{\gamma,k}^{\langle 2 \rangle} = \, 2^{-1} c_{\gamma,k}^{\,\gamma,k}\big(t'\big) \, X_{2\gamma,t'} \;\, $  otherwise. }
 \vskip11pt
   If  $ B $  is any Chevalley basis of  $ \fg \, $,  we set  $ \; \fg^\Z := \Z\text{\sl --span of}~B \, $,  and we call it  {\it Chevalley superalgebra}.
\end{definition}

\smallskip

\begin{remarks}   \label{rem-Chev_1}  {\ }
 \vskip2pt
   {\it (a)} \,  The above definition extends to the Lie superalgebra  $ \fg $  the notion of Chevalley basis for (semi)simple Lie algebras: in particular, if  $ B $  is a Chevalley basis of  $ \fg $  then  $ \, B \cap \fg_0 \, $  is a Chevalley basis   --- in a standard sense, extended to reductive Lie algebras in case  $ W $  ---   of  $ \, \fg_0 \, $.
                                                             \par
   In the present formulation, the conditions we give are clearly redundant, and may be simplified.
 \vskip2pt
   {\it (b)} \,  By its very definition, the Chevalley superalgebra  $ \fg^\Z $  is a Lie superalgebra over  $ \Z \, $.
 \vskip2pt
   {\it (c)} \,  When  $ \fg $  is not of type  $ W $,  so that  $ \, \fg \subsetneqq \overline{\fg} \, $,  if  $ B $  is any Chevalley basis for  $ \fg $  we can as well consider  $ \, \overline{B} := B \coprod \{H_\delta\} \, $:  this plays the role of a ``Chevalley basis'' for  $ \overline{\fg} \, $,  and we can develop all the theory which follows hereafter with  $ \overline{\fg} $  and  $ \overline{\fh} $  replacing  $ \fg $  and  $ \fh $  respectively.

 \vskip2pt
   {\it (d)} \,  For notational convenience, in the following I shall also use the notation  $ \, X_{\eta,k} := 0 \, $  when  $ \, k \in \N_+ \, $  and  $ \eta $  belongs to  $ Q $  (the  $ \Z $--span  of  $ \Delta $)  but  $ \, \eta \not\in \Delta \; $.
\end{remarks}

\vskip7pt

   We prove the  {\sl existence\/}  of Chevalley bases of  $ \, \fg \, $  by providing explicit examples, as follows:

\vskip11pt

\begin{examples}  \label{examples_che-bas}
  \, {\it Explicit examples of Chevalley bases.}
 \vskip10pt
   {\it (a)} \,  {\it  $ \underline{\text{Case \ }  W(n)} \; $,  first example:} \,  Let  $ \, \fg := W(n) \, $,  and take the subset
 \vskip-4pt
  $$  B'  \,\; \equiv \;\,  B_{W(n)}  \,\; := \;\,  {\big\{ H_i \big\}}_{i=1,\dots,n;} \; {\textstyle \coprod} \; \Big( {\textstyle \coprod_{\alpha \in \Delta}} {\big\{ X_{\alpha,k} \big\}}_{k=1,\dots,\mu(\alpha);} \Big)  $$
 \vskip-2pt
\noindent
 considered in  \S \ref{def-W(n)},  but now written in different notation, namely
 \vskip7pt
   $ \; H_i  \, := \;  \xi_i \, \partial_i \, $   \hfill  for all  $ \; i=1,\dots, r \, (=n) \, $,  \hskip175pt {\ }
 \vskip5pt
\noindent
 $ \quad X_{\alpha,1}  \, := \;  \xi_{i_1} \cdots \xi_{i_s} \, \partial_j \; $   \hfill  for every root of the form  \hskip7pt  $ \alpha = \varepsilon_{i_1} \! + \cdots + \varepsilon_{i_s} \! - \varepsilon_j $   \hskip5pt  ($ \, j \not\in \{i_1,\dots,i_s\} \, $)\,,
 \vskip5pt
\noindent
 $ \quad X_{\alpha,k} \, := \, \xi_{i_1} \! \cdots \xi_{i_s} \, \xi_{j_k} \partial_{j_k} \; $   \hfill  for every root of the form  \hskip7pt  $ \alpha = \varepsilon_{i_1} \! + \cdots + \varepsilon_{i_s} \; $,  \hskip7pt  where  $ j_k $  is the
                                                                         \par
   \hfill   $ k $--th  index in  $ \, \big\{1,\dots,n\big\} \setminus \big\{i_1,\dots,i_s\big\} \, $, \, for  $ \, k = 1, \dots, \mu(\alpha) \, $.
 \vskip7pt
\noindent
 By the results in  \S \ref{Lie-struct_W(n)}  and  \S \ref{def_Cartan-subalg_roots_etc}  one checks by direct analysis that  $ B' $  is a Chevalley basis.
 \vskip4pt
   Indeed, this specific Chevalley basis has even stronger properties than the prescribed ones.  Namely, the  $ \, c_{\alpha,k}^{\beta,h}(t) \, $  occurring in part  {\it (f)\/}  of  Definition \ref{def_che-bas}  satisfy, besides  {\it (f.1)\/}  and  {\it (f.2)},  the following (stronger) properties:
 $ \quad c_{\alpha,k}^{\beta,h}(t) \in \big\{ 0 \, , \pm 1 \big\} \, $,  \, and there exist  {\sl at most two indices}  $ t_1 $  and  $ t_2 $  such that  $ \; c_{\alpha,k}^{\beta,h}(t_1) = \pm 1 \, $,  $ \; c_{\alpha,k}^{\beta,h}(t_2) = \pm 1 \, $,  and  $ \; c_{\alpha,k}^{\beta,h}(t') = 0 \, $  for all  $ \; t' \not\in \big\{ t_1 \, , t_2 \big\} \; $.
 \vskip10pt
   {\it (b)} \,  {\it  $ \underline{\text{Case \ }  S(n)} \; $:} \,  Let  $ \, \fg := S(n) \, $,  and take the subset
 \vskip-5pt
  $$  B  \,\; \equiv \;\,  B_{S(n)}  \,\; := \;\,  {\big\{ H_i \big\}}_{i=1,\dots,n-1;} \; {\textstyle \coprod} \; \Big( {\textstyle \coprod_{\alpha \in \Delta}} {\big\{ X_{\alpha,k} \big\}}_{k=1,\dots,\mu(\alpha);} \Big)  $$
 \vskip-2pt
\noindent
 considered in  \S \ref{def-S(n)},  but now written with another notation, namely
 \vskip6pt
\noindent
 $ \quad H_i  \, := \;  \xi_i \, \partial_i - \xi_{i+1} \, \partial_{i+1} \, $   \hfill  for all  $ \; i=1,\dots, r \; (=n\!-\!1) \, $,  \hskip75pt {\ }
 \vskip5pt
\noindent
 $ \quad X_{\alpha,1}  \, := \;  \xi_{i_1} \cdots \xi_{i_s} \, \partial_j \; $   \hskip17pt  for every root of the form  \hskip5pt  $ \alpha = \varepsilon_{i_1} \! + \cdots + \varepsilon_{i_s} \! - \varepsilon_j \, $,  $ \, j \not\in \{i_1,\dots,i_s\} \; $,
 \vskip5pt
\noindent
 $ \quad X_{\alpha,k} \, := \; \xi_{i_1} \! \cdots \xi_{i_s} \big( \xi_{j_k} \partial_{j_k} \! - \xi_{j_{k+1}} \partial_{j_{k+1}} \!\big) \; $   \hfill  for every root of the form  \hskip5pt  $ \alpha = \varepsilon_{i_1} \! + \cdots + \varepsilon_{i_s} \, $,  \hskip3pt  where
                                                                           \par
   \hfill   $ \; \big\{ j_1 , \dots , j_{\mu(\alpha)} , j_{\mu(\alpha)+1} \big\} := \big\{1,\dots,n\big\} \setminus \big\{ i_1 , \dots , i_s \big\} \; $  with  $ \; j_1 < \cdots < j_{\mu(\alpha)} < j_{\mu(\alpha)+1} \; $.
 \vskip6pt
   Here again, direct analysis based on  \S \ref{Lie-struct_S(n)}  and  \S \ref{def_Cartan-subalg_roots_etc}  shows that  $ B $  is a Chevalley basis of  $ \, S(n) \; $.
 \vskip10pt
   {\it (c)} \,  {\it  $ \underline{\text{Case \ }  \tS(n)} \; $:} \,  Let  $ \, \fg := \tS(n) \, $,  and take the subset
 \vskip-5pt
  $$  B  \,\; \equiv \;\,  B_{\tS(n)}  \,\; := \;\,  {\big\{ H_i \big\}}_{i=1,\dots,n-1;} \; {\textstyle \coprod} \; \Big( {\textstyle \coprod_{\alpha \in \Delta}} {\big\{ X_{\alpha,k} \big\}}_{k=1,\dots,\mu(\alpha);} \Big)  $$
 \vskip-2pt
\noindent
 introduced in  \S \ref{def-tildeS(n)}  and now written with a different notation, namely
 \vskip6pt
\noindent
 $ \quad H_i  \, := \;  \xi_i \, \partial_i - \xi_{i+1} \, \partial_{i+1} \, $   \hfill  for all  $ \; i=1,\dots, r \; (=n\!-\!1) \, $,  \hskip75pt {\ }
 \vskip5pt
\noindent
 $ \quad X_{\alpha,1}  := \,  \xi_{i_1} \cdots \xi_{i_s} \, \partial_j \; $   \hskip5pt  for each root of the form  \hskip5pt  $ \alpha = \varepsilon_{i_1} \! + \cdots + \varepsilon_{i_s} \! - \varepsilon_j \, $,  $ \, j \not\in \{i_1,\dots,i_s\} \, $,  $ \, s > 0 \; $,
 \vskip5pt
\noindent
 $ \quad X_{\alpha,1}  := \,  \big( \xi_1 \cdots \xi_n - 1 \big) \, \partial_j \; $   \hfill  for every root of the form  \hskip5pt  $ \alpha = - \varepsilon_j \;\, (\,j=1, \dots, n) \, $,   \hfill   {\ }
 \vskip5pt
\noindent
 $ \quad X_{\alpha,k} \, := \, \xi_{i_1} \! \cdots \xi_{i_s} \big( \xi_{j_k} \partial_{j_k} \! - \xi_{j_{k+1}} \partial_{j_{k+1}} \big) \; $   \hfill  for every root of the form  \hskip7pt  $ \alpha = \varepsilon_{i_1} \! + \cdots + \varepsilon_{i_s} \, $,  \hskip5pt  where
                                                                           \par
   \hfill   $ \; \big\{ j_1 , \dots , j_{\mu(\alpha)} , j_{\mu(\alpha)+1} \big\} := \big\{1,\dots,n\big\} \setminus \big\{ i_1 , \dots , i_s \big\} \; $  with  $ \; j_1 < \cdots < j_{\mu(\alpha)} < j_{\mu(\alpha)+1} \; $.
 \vskip6pt
   Again, direct analysis (via  \S \ref{Lie-struct_tildeS(n)}  \S \ref{def_Cartan-subalg_roots_etc})  shows that this  $ B $  is indeed a Chevalley basis of  $ \, \tS(n) \; $.
 \vskip10pt
   {\it (d)} \,  {\it  $ \underline{\text{Case \ }  W(n)} \; $,  second example:} \,  Let again  $ \, \fg := W(n) \, $.  For any  $ \, i \in \{1,\dots,n\} \, $  and any  $ \, j_\alpha \in \big\{ 1, \dots, \mu(\alpha) \big\} \, $  for  $ \, \big\{ \alpha \in \Delta : \text{\it ht}(\alpha) \geq 1 \big\} \, $   ---  $ \mu(\alpha) $  being  {\sl the multiplicity in}  $ W(n) $  ---   consider
 \vskip-5pt
  $$  B''  \,\; := \;\,  B_{S(n)} \; {\textstyle \coprod} \; \big\{ \xi_i \, \partial_i \big\} \; {\textstyle \coprod} \; \big\{ X_{\alpha,j_\alpha} := \xi_{a_1} \! \cdots \xi_{a_s} \, \xi_{j_\alpha} \partial_{j_\alpha} \,\big|\, \alpha \in \Delta : \text{\it ht}(\alpha) > 1 \big\}  $$
 \vskip-1pt
\noindent
 where we wrote  $ \, \alpha = \varepsilon_{a_1} \! + \cdots + \varepsilon_{a_s} \, $  for every root  $ \alpha $  with  $ \, \text{\it ht}(\alpha) > 1 \, $  (so that the string  $ \, (a_1, \dots, a_s) \, $  depends on  $ \alpha $  itself).
 Yet another direct check shows that  $ B'' $  is a Chevalley basis of  $ \, W(n) \; $.
%
%
 \vskip11pt
   {\it (e)} \,  {\it  $ \underline{\text{Case \ }  H(n)} \; $:} \,  Let  $ \, \fg := H(n) \, $,  and  $ \; B_{H(n)} \, := \, \Big\{\, D_{\underline{\xi}^{\underline{e}}} \;\Big|\; \underline{e} \in {\{0,1\}}^n , \, 0 < |\underline{e}| < n \,\Big\} \; $.
 Set
 \vskip5pt
   $ \; H_i  \, := \,  -D_{\xi_i \, \xi_{r+i}}  \, = \,  \xi_i \, \partial_i \, - \, \xi_{r+i} \, \partial_{r+i} \, $   \hfill  for all  $ \; i=1,\dots, r \; \big(\! = [n/2] \big) \, $;  \hskip75pt {\ }
 \vskip5pt
\noindent
 then for any root  $ \; \alpha = \varepsilon_{i_1} \! + \cdots + \varepsilon_{i_p} \! - \varepsilon_{j_1} \! - \cdots - \varepsilon_{j_q} \! + m \, \delta \, $ set  $ \; s := \big[ (m-p-q)/2 \big] \; $  and pick the root vector (with  $ \phi $  as in  \S \ref{def-H(n)})
 \vskip4pt
\noindent
 $ \quad  X_{\alpha,k}  \, := \;  \pm \, D_{\xi_{i_1} \cdots \xi_{i_p} \, \xi_{r+j_1} \cdots \xi_{r+j_q} \, \xi_{t_1} \cdots \xi_{t_s} \, \xi_{r+t_1} \cdots \xi_{r+t_s}}
 \; $   \hfill   if  $ \, (m\!-\!p\!-\!q) \, $  is  {\sl even},
 \vskip4pt
\noindent
 $ \quad  X_{\alpha,k}  \, := \;  \pm \, \sqrt{2} \; D_{\xi_{i_1} \cdots \xi_{i_p} \, \xi_{r+j_1} \cdots \xi_{r+j_q} \, \xi_{t_1} \cdots \xi_{t_s} \, \xi_{r+t_1} \cdots \xi_{r+t_s} \, \xi_{2r+1}}
 \; $   \hfill   if  $ \, (m\!-\!p\!-\!q) \, $  is  {\sl odd},
 \vskip5pt
\noindent
%
%
 for every choice of  $ \;\; t_1, \dots, t_s \in \{1,\dots,r\} \setminus \big( \{i_1,\dots,i_p\} \cup \{j_1,\dots,j_q\} \big) \; $  with  $ \; t_1 < \cdots < t_s \, $,  \,
 where  $ \, k \in \{1, \dots, \mu(\alpha)\} \, $  is used to order the possible choices of ordered subset of indices  $ \, \{t_1,\dots,t_s\} \, $.
 \vskip3pt
   {\sl N.B.:\/}  the root vectors of second type have to be considered only when  $ n $  itself is odd.
 \vskip7pt
   Now, using the formulas and results in  \S \ref{Lie-struct_H(n)}  and  \S \ref{def_Cartan-subalg_roots_etc}, one checks
 that the set of all  $ H_i $'s  and all  $ X_{\alpha,k} $'s  defined above (for suitable choice of signs) is indeed a Chevalley basis of  $ \, H(n) \; $.
 \hfill  $ \diamondsuit $
\end{examples}

\medskip

  \subsection{The Kostant superalgebra}  \label{kost-form}

\smallskip

   {\ } \quad   For any  $ \KK $--algebra  $ A \, $,  given  $ \, m \in \N \, $  and  $ \, y \in A \, $  we define the  {\sl  $ m $--th  binomial coefficient}  $ \, {\Big(\! {y \atop m} \!\Big)} \, $  and the  {\sl  $ m $--th  divided power}  $ \, y^{(m)} \, $  by
  $ \; \Big(\! {y \atop m} \!\Big) := {\frac{\, y \, (y-1) \cdots (y-m+1) \,}{m!}} \; $,  $ \; y^{(m)} := y^m \!\big/ m! \; $.
 \vskip5pt
%
   We start with a (standard) classical result, concerning  $ \Z $--valued  polynomials:

\medskip

\begin{lemma}  \label{int-polynomials}
   (cf.~\cite{hu},  \S 26.1)  Let  $ \, \KK\big[\,\underline{y}\,\big] := \KK[\,y_1, \dots, y_t] \, $  be the\/  $ \KK $--algebra  of polynomials in
 $ y_1 \, $, $ \dots \, $,  $ y_t \, $,  and
 $ \; \text{\it Int}_{\,\Z} \big( \KK\big[\,\underline{y}\,\big] \big) \, := \, \big\{\, f \! \in \KK\big[\,\underline{y}\,\big] \;\big|\, f(z_1, \dots, z_t) \in \Z \;\; \forall \, z_1, \dots, z_t \in \! \Z \,\big\} \; $.
Then  $ \, \text{\it Int}_{\,\Z} \big( \KK\big[\,\underline{y}\,\big] \big) \, $  is a  $ \Z $--subalgebra  of  $ \, \KK\big[\,\underline{y}\,\big] \, $,
 free as a  $ \Z $--(sub)module,  with
 basis  $ \; \big\{\, {\textstyle \prod_{i=1}^t} \, \big({y_i \atop n_i}\big) \;\big|\; n_1, \dots, n_t \in \N \,\big\} \; $.
\end{lemma}

\medskip

   Let  $ U(\fg) $  be the universal enveloping superalgebra of  $ \, \fg \, $.  We recall that this can be realized as the quotient of the tensor superalgebra  $ \, T(\fg) \, $  by the two-sided homogeneous ideal generated by  $ \; \Big\{\, x \otimes y - {(-1)}^{p(x) p(y)} y \otimes x - [x,y] \; , \; z \otimes z - z^{\langle 2 \rangle} \;\Big|\; x, y \in \fg_\zero \cup \fg_\uno \, , \, z \in \fg_\uno \,\Big\} \; $.
                                                              \par
   Fix a Chevalley basis  $ \; B = {\big\{ H_i \big\}}_{i=1,\dots,r;} \coprod {\big\{ X_{\alpha,k} \big\}}_{\alpha \in \Delta}^{k=1,\dots,\mu(\alpha);} \, = \, {\big\{ H_i \big\}}_{i=1,\dots,r;} \, \coprod \, {\big\{ X_{\talpha} \big\}}_{\talpha \in \tDelta} \; $  of  $ \fg $  as in  Definition \ref{def_che-bas},  and let  $ \fh_\Z $  be the free  $ \Z $--module  with basis  $ \, \big\{ H_1, \dots, H_r \big\} $.  For  $ \, h \in \fh_\Z \, $,  we denote by  $ \, h(H_1,\dots,H_r) \, $  the expression of  $ h $  as a function of the  $ H_i $'s.  From  Lemma \ref{int-polynomials}  we have:

\medskip

\begin{corollary}  \label{int-polynom_kost} {\ }
 $ \mathbb{H}_\Z := \big\{\, h \! \in \! U(\fh) \;\big|\; h\big(z_1,\dots,z_r\big) \in \Z \, , \,\; \forall \, z_1, \dots, z_r \in \Z \,\big\} \, $  is a  {\sl free}  $ \Z $--sub\-module  of  $ \, U(\fh) \, $,  with basis  $ \, B_{U(\fh)} := \Big\{ {\textstyle \prod_{i=1}^r \! \Big(\! {H_i \atop m_i} \!\Big)} \,\Big|\, m_1, \dots, m_n \! \in \! \N \Big\} \, $.  Moreover, it coincides with the  $ \Z $--subalgebra  of  $ \, U(\fg) $  generated by all the elements  $ \, \Big(\! {{H - z} \atop m} \!\Big) \, $  with  $ \, H \in \fh_\Z \, $,  $ \, z \in \Z \, $,  $ \, m \in \N \, $.
\end{corollary}

\vskip7pt

   We are now ready to define the Kostant superalgebra.  Like in  \cite{fg2},  we mimic the classical construction, but making a suitable distinction between the roles of even and odd root vectors.

\vskip13pt

\begin{definition}  \label{def-kost-superalgebra}
   We call  {\sl Kostant superalgebra\/}  of  $ U(\fg) $  the unital  $ \Z $--subsuperalgebra  $ \kzg $  of  $ U(\fg) $  generated by all elements
 $ \; \Big(\! {H_i \atop m} \!\Big) \; $,  $ \; X_{\talpha}^{\,(m)} \; $,  $ \; X_{\tgamma} \;\, $
for  $ \; m \in \N \, $,  $ \, 1 \leq i \leq r \, $,  $ \; \talpha \in \tDelta_\zero \; $,  $ \; \tgamma \in \tDelta_\uno \; $.
\end{definition}

\vskip7pt

\begin{remarks}
   {\it (a)} \,  The classical notion   --- suitably adapted to the reductive Lie algebra  $ \rgl_n $  when  $ \, \fg = W(n) \, $  ---   defines the Kostant's  $ \Z $--form  of  $ \, U(\fg_0) \, $,  call it  $ \, \kzgz \, $,  as the unital  $ \Z $--subalgebra  of  $ U(\fg_0) $  generated by the elements  $ \; X_{\talpha}^{(m)} \, $,  $ \, \Big(\! {H_i \atop m} \!\Big) \; $  with  $ \, \talpha \in \tDelta_0 \, $  and  $ \, m \in \N \; $.
 Then  $ \; \kzg \supseteq \kzgz \; $.
 \vskip2pt
   {\it (b)} \,  As a matter of notation, we shall always read  $ \; X_{\talpha}^{(m)} := \delta_{m,0} \; $  if  $ \, \talpha \not\in \tDelta \, $,  for any  $ \, m \in \N \, $.
\end{remarks}

\vskip7pt

   We shall use the following result, proved by induction  (cf.~also  \cite{hu},  \S 26.2, for part  {\it (a)\/}):

\vskip15pt

\begin{lemma}  \label{comm_div-pow}
  Let  $ \, \mathfrak{l} \, $  be a Lie  $ \KK $--algebra,  and  $ \, \ell, m \in \N \, $,  $ \, \ell \wedge m := \min(\ell,m) \, $.
 \vskip4pt
   (a) \,  Let  $ \, E \, , F \in \mathfrak{l} \, $,  $ \, H := [E,F\,] \in \mathfrak{l} \, $,  and assume that  $ \, [H,E\,] = +2 \, E \, $,  $ \, [H,F\,] = -2 \, F \, $.  Then
 \vskip-5pt
  $$  E^{(\ell\,)} \, F^{(m)}  \,\; = \;\,  {\textstyle \sum_{s=0}^{\ell \wedge m}} \; F^{(m-s)} \, {\textstyle \Big({{H - m - \ell + 2 \, s} \atop s}\Big)} \, E^{(\ell-s)}   \eqno  \text{inside  $ \, U(\mathfrak{l}) \; $.}  \quad  $$
 \vskip4pt
   (b) \,  Let  $ \, A \, , B \in \mathfrak{l} \, $,  $ \, C := [A,B\,] \in \mathfrak{l} \, $,  and assume also that  $ \, [A,C\,] = 0 \, $,  $ \, [B,C\,] = 0 \, $.  Then
 \vskip-5pt
  $$  A^{(\ell\,)} \, B^{(m)}  \,\; = \;\,  {\textstyle \sum_{q=0}^{\ell \wedge m}} \; B^{(m-q)} \, C^{(q)} \, A^{(\ell-q)}   \eqno  \text{inside  $ \, U(\mathfrak{l}) \; $.}  \quad  $$
 \vskip4pt
   (c) \,  Let  $ \, L \, , M \in \mathfrak{l} \, $,  $ \, N := [L,M\,] \, $,  $ \, 2\,T := [L,N] \in \mathfrak{l} \, $.  Assume also that  $ \, [M,N\,] = [L,T\,] = 0 \, $  (then  $ \, [M,T\,] = [N,T\,] = 0 \, $  as well).  Then
 \vskip-5pt
  $$  L^{(\ell\,)} \, M^{(m)}  \,\; = \;\,  {\textstyle \sum_{s=0}^{\ell \wedge m}} \; M^{(m-s)} \, {\textstyle \sum_{t+q=s}} \; T^{(t)} \, N^{(q)} \, L^{(\ell-2t-q)}   \eqno  \text{inside  $ \, U(\mathfrak{l}) \; $.}  \quad  $$
 \vskip4pt
   (d) \,  Let  $ \, X, Y \in \mathfrak{l} \, $,  and assume that  $ \, [X,Y\,] = 0 \, $.  Then
 \vskip-5pt
  $$  {\big( X + Y \big)}^{(m)}  \; = \;\,  {\textstyle \sum_{u=0}^m} \; X^{(m-u)} \, Y^{(u)}  \; = \;\,  {\textstyle \sum_{v=0}^m} \; Y^{(m-v)} \, X^{(v)}   \eqno  \text{inside  $ \, U(\mathfrak{l}) \; $.}  \quad  $$
\end{lemma}

\medskip

  \subsection{Commutation rules and Kostant's PBW theorem}  \label{comm-rul_kost}

\smallskip

   {\ } \quad   In the classical setup, a description of  $ \kzgz $  comes from a ``PBW-like'' theorem: namely,  $ \kzgz $  is a free  $ \Z $--module  with  $ \Z $--basis  the set of ordered monomials (w.~r.~to any total order) whose factors are divided powers in the  $ X_{\talpha} \, \big(\, \talpha \in \tDelta_0 \big) \, $  or binomial coefficients in the  $ H_i $  ($ \, i = 1, \dots, n \, $).
                                        \par
   We shall prove a similar result for  $ \, \fg \, $,  our Lie superalgebra of Cartan type.  Like for (semi)sim\-ple Lie algebras   --- and also for simple Lie superalgebras of  {\sl classical\/}  type,  cf.~\cite{fg2},  \S 4 ---   this follows from a direct analysis of commutation rules among the generators of  $ \kzg \, $.  To this end, we list hereafter all such rules, and also some slightly more general relations.
 We split the list into two sections:  {\it (1)\/}  relations involving only  {\sl even\/}  generators;  {\it (2)\/}  relations involving also  {\sl odd\/}  generators.
                                        \par
   The relevant feature is that all coefficients in these relations are in  $ \Z \, $.
 \eject
\noindent
 {\bf (1) Even generators only  {\rm  (that is  $ \, \Big(\! {H_i \atop m} \!\Big) $'s  and  $ \, X_{\talpha}^{(m)} $'s  only,  $ \, \talpha \in \tDelta_\zero \, $)}:}
 \vskip-7pt
  $$  \displaylines{
   \hskip85pt \hfill   {\textstyle \Big({H_i \atop \ell}\Big)} \, {\textstyle \Big({H_j \atop m}\Big)}  \; = \;  {\textstyle \Big({H_j \atop m}\Big)} \, {\textstyle \Big({H_i \atop \ell}\Big)}   \hskip25pt
\phantom{{}_{\big|}} \forall \;\; i,j \in \{1,\dots,r\} \, , \;\;\; \forall \;\; \ell, m \in \N   \qquad    \hfill (3.1)  \cr
   \hskip25pt \hfill   X_{\talpha}^{(m)} \, f(H)  \; = \;  f\big(H - m \; \pi\big(\talpha\big)(H)\big) \, X_{\talpha}^{(m)}   \hskip15pt
 \phantom{{}_{\big|}} \forall \;\; \talpha \in \! \tDelta_\zero \, , \; H \! \in \fh \, , \; m \in \N \, , \; f(T) \in \! \KK[T]   \quad   \hfill (3.2)  \cr
   \hfill   X_{\talpha}^{(\ell\,)} \, X_{\talpha}^{(m)}  \; = \;  {\textstyle \Big(\! {{\ell \, + \, m} \atop m} \!\Big)} \, X_{\talpha}^{(\ell+m)}   \quad \qquad  \forall \; \talpha \in \tDelta_\zero \, , \;\; \forall \;\; \ell, m \in \N  \phantom{{}_{\big|}}   \hfill (3.3)  \cr
   \hfill \quad   X_{\talpha}^{(\ell\,)} \, X_{\tbeta}^{(m)}  = \,  X_{\tbeta}^{(m)} \, X_{\talpha}^{(\ell\,)}   \hskip35pt  \forall \;\; \talpha, \tbeta \! \in \! \tDelta_\zero \,:\, \pi\big(\talpha\big) \! + \! \pi\big(\tbeta\,\big) \not\in \! \big( \Delta \! \cup \! \{0\} \big) \, , \,\;\; \forall \;\; \ell, m \in \N  \phantom{{}_{\big|}}   \hfill (3.4)  \cr
   \hfill   X_{\alpha,1}^{(m)} \, X_{-\alpha,1}^{(\ell)}  \,\; = \;\,  {\textstyle \sum\limits_{s=0}^{\ell \wedge m}} \; X_{-\alpha,1}^{(\ell-s)} \, {\textstyle \Big( {{H_\alpha \, - \, m \, - \, \ell \, + \, 2 \, s} \atop s} \Big)} \, X_{\alpha,1}^{(m-s)}   \qquad   \hfill (3.5)  \cr
   \hfill   \forall \;\; \alpha \in \Delta_0 \, , \;\; \forall \;\; \ell, m \in \N \; ,  \quad
\text{with}  \quad  \ell \wedge m := \min(\ell,m)  \phantom{{}_{\big|}}   \quad \qquad  \cr
%
   \hfill   X_{\alpha,k}^{(\ell\,)} \, X_{\beta,h}^{(m)}  \; = \,  {\textstyle \sum\limits_{q=0}^{\ell \wedge m}} \, X_{\beta,h}^{(m-q)} \, \bigg(\, {\textstyle \sum\limits_{\sum q_t = q}} \!\! {\textstyle \prod\limits_{t=1}^{\mu(\alpha+\beta)}} \! {\big( c_{\alpha,k}^{\,\beta,h}(t) \big)}^{q_t} \, X_{\alpha + \beta, \, t}^{\,(q_t)} \,\bigg) \, X_{\alpha,k}^{(\ell-q)}   \hfill (3.6)  \cr
   \hfill   \forall \;\; \alpha \in \Delta_\zero \, , \, \beta \in \Delta_\zero \,:\, \alpha \! + \! \beta \in \Delta \, , \, 2 \, \alpha \! + \! \beta \not\in \Delta \, , \, \alpha \! + \! 2 \, \beta \not\in \Delta \, ,  \quad \forall \;\; \ell, m \in \N  \phantom{{}_{\big|}}  \quad \qquad  \cr
%
%
   \hfill   X_{\alpha,1}^{(\ell\,)} \, X_{\beta,h}^{(m)}  =  \! {\textstyle \sum\limits_{s=0}^{\ell \wedge m}} X_{\beta,h}^{(m-s)} \hskip-10pt {\textstyle \sum\limits_{\hskip11pt p+q=s}} \hskip-8pt \epsilon^p X_{2\alpha+\beta,t'}^{(p)}
 \bigg(\, {\textstyle \sum\limits_{\sum q_t = q}} \!\! {\textstyle \prod\limits_{t=1}^{\mu(\alpha+\beta)}} \!\! {\big( c_{\alpha,1}^{\,\beta,h}(t) \big)}^{q_t} X_{\alpha + \beta, \, t}^{\,(q_t)} \bigg) \,
 X_{\alpha,1}^{(\ell-2p-q)}   \! \phantom{{}_{\big|}}   \hfill (3.7)  \cr
   \hfill   \forall \;\; \alpha \in \Delta_0 \, , \; \beta \in \Delta_\zero \setminus \Delta_0 \,:\, \alpha \! + \! \beta \, , \, 2 \, \alpha \! + \! \beta \in \Delta \, ,  \quad \forall \;\; \ell, m \in \N  \phantom{{}_{\big|}}  \quad \qquad  \cr
   \hfill \qquad \qquad   X_{\alpha,1}^{(\ell\,)} \, X_{\beta,1}^{(m)}  \; = \;  X_{\beta,1}^{(m)} \, X_{\alpha,1}^{(\ell\,)} \, + \, \textit{l.h.t}   \quad \qquad \qquad  \forall \; \alpha, \beta \in \Delta_0 \, , \;\; \forall \;\; \ell, m \in \N   \hfill (3.8)  }  $$
\vskip3pt

\noindent
 where  $ \; c_{\alpha,1}^{\,\beta,h}(t) \; $  and  $ \; X_{\alpha + \beta, \, t} \; $  are as in Definition \ref{def_che-bas}{\it (f)},  while  $ \; \epsilon = \pm 1 \; $  and the index  $ t' $  are such that  $ \; \big[ X_{\alpha,1} , \big[ X_{\alpha,1} , X_{\beta,h} \,\big] \big] = \epsilon \, 2 \, X_{2\alpha+\beta,t'} \; $  as in Definition \ref{def_che-bas}{\it (g)},  and  {\it l.h.t.}~(={\it ``lower height terms''})  stands for a  $ \Z $--linear  combinations of monomials in the  $ X_{\teta}^{(q)} $'s  and in the $ \Big(\! {H_i \atop c} \!\Big) $'s  whose ``height''   --- i.e., the sum of all ``exponents''  $ q $ occurring in such a monomial ---   is less than  $ \, \ell + m \, $.
 \vskip5pt

\begin{proof}
 Relations (3.1), (3.2), (3.3) and (3.5) hold by definitions, along with  Lemma \ref{comm_div-pow}{\it (a)}.
 \vskip4pt
   If  $ \, \talpha, \tbeta \in \tDelta \, $  and  $ \, \pi\big(\talpha\big) \! + \! \pi\big(\tbeta\big) \not\in \big( \Delta \cup \{0\} \big) \, $,  then we get  $ \; \big[ X_{\talpha} \, , \, X_{\tbeta} \big] =  0 \; $  by  Definition \ref{def_che-bas}{\it (e)},  so  $ X_{\talpha} $  and  $ X_{\tbeta} $  commute with each other: this implies (3.4).
 \vskip4pt
   Relations (3.6) follow as an application of  Lemma \ref{comm_div-pow}{\it (b)\/}  to  $ \, \mathfrak{l} := \fg \, $,  $ \, A := X_{\alpha,k} \, $  and  $ \, B := X_{\beta,h} \, $,  taking  Definition \ref{def_che-bas}{\it (f)\/}  into account.  Indeed, in this case  Definition \ref{def_che-bas}{\it (f)\/}  gives
  $$  C  \; := \;  [A,B]  \; = \;  \big[\, X_{\alpha,k} \, , X_{\beta,h} \,\big]  \, = \;  {\textstyle \sum_{t=1}^{\mu(\alpha+\beta)}} \, c_{\alpha,k}^{\,\beta,h}(t) \, X_{\alpha + \beta, \, t}  $$
Moreover, the assumptions $ \; (2 \, \alpha \! + \! \beta) \, , (\alpha \! + \! 2 \, \beta) \not\in \Delta \; $  imply  $ \, [A,C\,] = 0 = [B,C\,] \; $.  Thus we can apply  Lemma \ref{comm_div-pow}{\it (b)\/}  to expand  $ \; A^{(\ell)} \, B^{(m)} = X_{\alpha,k}^{(\ell)} \, X_{\beta,h}^{(m)} \; $;  also, in the expansion we find we can still expand each divided power
  $ \, C^{(q)} = {\Big( {\textstyle \sum_{t=1}^{\mu(\alpha+\beta)}} \, c_{\alpha,k}^{\,\beta,h}(t) \, X_{\alpha + \beta, \, t} \Big)}^{\!(q)} \, $
 via the formula in  Lemma \ref{comm_div-pow}{\it (d)},  which applies as  $ \; \Big[ c_{\alpha,k}^{\,\beta,h}(t') \, X_{\alpha + \beta, \, t'} \, , \, c_{\alpha,k}^{\,\beta,h}(t'') \, X_{\alpha + \beta, \, t''} \Big] = 0 \; $  for all  $ \, t', t''
 \, $,  by  Proposition \ref{ad_square/cube}{\it (a)}.
 \vskip4pt
   Relations (3.7) follow as an application of  Lemma \ref{comm_div-pow}{\it (c)}.  Indeed, in the present case we can apply  Lemma \ref{comm_div-pow}{\it (c)\/}  to  $ \, \mathfrak{l} := \fg_\zero \, $,  $ \, L := X_{\alpha,1} \, $,  $ \, M := X_{\beta,h} \, $,  \, so that
%
%
 \vskip-10pt
  $$  N  \; := \;  [L,M]  \; = \;  \big[ X_{\alpha,1} , X_{\beta,h} \big]  \, = \;  {\textstyle \sum_{t=1}^{\mu(\alpha+\beta)}} \, c_{\alpha,1}^{\,\beta,h}(t) \, X_{\alpha + \beta, \, t}  \;\; ,  \qquad
      T  \; := \;  2^{-1} \, [L,N]  \; = \;  \epsilon \, X_{2\alpha+\beta,k}  $$
 \vskip-2pt
\noindent
 Then in the formula of  Lemma \ref{comm_div-pow}{\it (c)\/}  we still have to expand  $ \; N^{(q)} = {\Big( \sum_{t=1}^{\mu(\alpha+\beta)} \, c_{\alpha,1}^{\,\beta,h}(t) \, X_{\alpha + \beta, \, t} \Big)}^{(q)} \; $  using  Lemma \ref{comm_div-pow}{\it (d)},  which again applies   --- all summands commute with each other ---   by the same arguments as above.  In particular,  $ \, [M,N] = 0 \, $  and  $ \, [L,T] = 0 \, $  because of  Proposition \ref{ad_square/cube}.
 \vskip4pt
   Finally, relations (3.8)   --- concerning roots of  $ \fg_0 $  ---   are well-known by the classical theory.
\end{proof}
 \vskip15pt

\noindent
 {\bf (2) Odd and even generators  {\rm  (also involving the  $ X_{\tgamma} $,  $ \, \tgamma \in \tDelta_\uno \, $)}:}
 \vskip-11pt
  $$  \displaylines{
   \hfill \hskip65pt   X_{\tgamma} \, f(H)  \; = \;  f\big(H - \pi\big(\tgamma\big)(H)\big) \, X_{\tgamma}   \quad \qquad  \forall \;\; \tgamma \in \tDelta_\uno \, , \; h \in \fh \, , \; f(T) \in \KK[T]   \phantom{{}_{\big|}} \quad  \hfill (3.9)  \cr
   \hfill \hskip0pt   X_{\gamma,k}^{\;\,2} \, = \; 0  \quad  \text{if all  $ c_{\gamma,k}^{\,\gamma,k}(t) $  are zero},   \qquad   X_{\gamma,k}^{\;\,2} \, = \; 2^{-1} c_{\gamma,k}^{\,\gamma,k}\big(t'\big) \, X_{2\gamma,t'}  \quad  \text{otherwise}   \hfill (3.10)  \cr
   \phantom{{}_{\big|}} \text{(with notation as in  Definition \ref{def_che-bas}{\it (f.3.)\/})} \phantom{{}_{\big|}}  \cr
   \hfill   X_{\gamma,1} \, X_{-\gamma,k}  \; = \;  - X_{-\gamma,k} \, X_{\gamma,1} \, + \, H_{\sigma_\gamma(k)}   \hskip45pt  \forall \;\, \gamma \in \Delta_{-1}   \qquad  \hfill  (3.11)  \cr
   \quad   \text{with}  \quad  H_{\sigma_\gamma(k)} = \big[ X_{\gamma,1} \, , \, X_{-\gamma,k} \big] \, \in \, \fh_\Z  \quad  \text{as in  Definition \ref{def_che-bas}{\it (d)}}  \phantom{{}_{\big|}}  \cr
   \hfill   X_{\tgamma} \; X_{\teta}  \; = \;  - X_{\teta} \, X_{\tgamma}   \hskip35pt  \forall \;\; \tgamma \, , \, \teta \in \tDelta_\uno \,\; : \;\, \pi\big(\tgamma\big) \! + \pi\big(\,\teta\,\big) \not\in \! \tDelta \, , \,\; \forall \; \ell \in \N   \phantom{{}_{\big|}}  \qquad   (3.12)  \cr
   \hfill   X_{\gamma,k} \; X_{\eta,h}  \, = \,  - X_{\eta,h} \, X_{\gamma,k}  \, +  {\textstyle \sum_{t=1}^{\mu(\alpha+\gamma)}} c_{\gamma,k}^{\,\eta,h}(t) \, X_{\gamma + \eta, \, t}   \hskip21pt   \forall \;\, \gamma \, , \, \eta \in \Delta_\uno \; : \; \alpha + \gamma \in \! \Delta \, , \,\; \forall \; \ell \in \N   \phantom{{}_{\big|}}  \hskip3pt   (3.13)  \cr
   \hfill   X_{\talpha}^{(\ell\,)} \; X_{\tgamma}  \; = \;  X_{\tgamma} \, X_{\talpha}^{(\ell\,)}   \hskip37pt   \forall \;\; \talpha \in \tDelta_\zero \; , \; \tgamma \in \tDelta_\uno \; : \; \pi\big(\talpha\big) \! + \pi\big(\tgamma\big) \not\in \! \tDelta \, , \,\; \forall \; \ell \in \N   \qquad  \phantom{{}_{\big|}}   (3.14)  \cr
   \hfill   X_{\alpha,k}^{(\ell\,)} \; X_{\gamma,h}  \; = \;  X_{\gamma,h} \, X_{\alpha,k}^{(\ell\,)} \, + \, \Big( {\textstyle \sum_{t=1}^{\mu(\alpha+\gamma)}} \, c_{\alpha,k}^{\,\gamma,h}(t) \, X_{\alpha + \gamma, \, t} \Big) \, X_{\alpha,k}^{(\ell-1)}   \hskip115pt  \hfill (3.15)  \cr
   \hfill   \phantom{{}_{\big|}}   \forall \;\; \alpha \in \Delta_\zero \; , \; \gamma \in \Delta_\uno \; : \; \alpha \! + \gamma \in \Delta \, , \, 2 \, \alpha \! + \gamma \not\in \Delta \; ,  \quad \forall \;\; \ell \in \N  \qquad  \cr
   \hfill   X_{\alpha,1}^{(\ell\,)} \, X_{\gamma,h}  \; = \;\,  X_{\gamma,h} \, X_{\alpha,1}^{(\ell\,)} \, + \, \Big( {\textstyle \sum_{t=1}^{\mu(\alpha+\gamma)}} \, c_{\alpha,1}^{\,\gamma,h}(t) \, X_{\alpha + \gamma, \, t} \Big) \, X_{\alpha,1}^{(\ell-1)} + \, \epsilon \, X_{2\alpha+\gamma,t'} \, X_{\alpha,1}^{(\ell-2)}   \qquad   \hfill  (3.16)  \cr
   \hfill   \forall \;\; \alpha \in \Delta_0 \, , \; \gamma \in \Delta_\uno \,:\, \alpha \! + \! \gamma \, , \, 2 \, \alpha \! + \! \gamma \in \Delta \, ,  \quad \forall \;\; \ell \, , m \in \N  \qquad  }  $$
 \vskip3pt
\noindent
 where  $ \; c_{\gamma,k}^{\,\eta,h}(t) \, $,  $ \, c_{\alpha,k}^{\,\gamma,h}(t) \, $,  $ \, X_{\gamma + \eta, \, t} \, $,  $ \, X_{\alpha + \gamma, \, t} \, $,  $ \; \epsilon = \pm 1 \; $  and the index  $ t' $  (namely, the one such that  $ \; \big[ X_{\alpha,1} , \big[ X_{\alpha,1} , X_{\gamma,1} \,\big] \big] = \epsilon \, 2 \, X_{2\alpha+\gamma,t'} \; $)  are given again as in Definition \ref{def_che-bas}{\it (f--g)}.
 \vskip7pt

\begin{proof}
   Almost all of these relations are proved much like those among even generators only.
                                                                 \par
   A first exception is (3.10), which holds by  Definition \ref{def_che-bas}{\it (h)},  taking into account that in the universal enveloping superalgebra one has  $ \, X^2 = X^{\langle 2 \rangle} \, $  for every  $ \, X \in \fg_\uno \; $.  Another exception is (3.11), which is just another way of rewriting what is expressed in  Definition \ref{def_che-bas}{\it (d)}.
 \vskip4pt
   As to the rest, relations (3.9), (3.12), (3.14) directly follow from definitions.  Finally, relations (3.13), (3.15), (3.16) are proved, like relations (3.6) and (3.7), via induction like for  Lemma \ref{comm_div-pow}.
\end{proof}

\vskip5pt

   Here now is our (super-version of) Kostant's theorem for  $ \kzg \, $:

\vskip13pt

\begin{theorem}  \label{PBW-Kost}
 The Kostant superalgebra  $ \, \kzg \, $  is a free\/  $ \Z $--module.
%
%
   More precisely, for any given total order  $ \, \preceq \, $  of the set  $ \; \tDelta \coprod \big\{ 1, \dots, n \big\} \; $
%
%
 a  $ \, \Z $--basis  of  $ \kzg $  is the set  $ {\mathcal B} $  of ordered ``PBW-like monomials'', i.e.~all products (without repeti\-tions) of factors of type  $ X_{\talpha}^{(\ell_{\talpha})} $,  $ \Big(\! {H_i \atop n_i} \!\Big) $  and  $ \, X_{\tgamma} $   --- with  $ \, \talpha \in \tDelta_\zero \, $,  $ \, i \in \big\{ 1, \dots, n \big\} \, $,  $ \, \tgamma \in \tDelta_\uno \, $,  and  $ \, \ell_{\talpha} \, $,  $ n_i  \in \N $  ---   taken in the right order with respect to  $ \, \preceq \; $.
\end{theorem}

\begin{proof}  Let us call ``monomial'' any product
 of several  $ X_{\talpha}^{(\ell_{\talpha})} \, $,  several  $ \Big(\! {{H_i - z_i} \atop s_i} \!\Big) $   --- with  $ z_i \! \in \! \Z $  ---   and several  $ X_{\tgamma}^{m_{\tgamma}} \, $.  For any such monomial, say  $ \cM \, $,  we consider the following three numbers:
 \vskip4pt
%
%
%
 \vskip3pt
   --- its ``height''  $ \, \textit{ht}(\cM) \, $,  \, i.e.~the sum of all  $ \ell_{\talpha} $  and  $ m_{\tgamma} $  in  $ \cM \, $;
 \vskip3pt
   --- its ``factor number''  $ \, \text{\it fac}(\cM) \, $,  \, defined to be the total number of factors (namely  $ X_{\talpha}^{(\ell_{\talpha})} $,  $ \Big(\! {{H_i - z_i} \atop n_i} \!\Big) $  or  $ X_{\tgamma}^{m_{\tgamma}} \, $)  within  $ \cM $  itself;
 \vskip3pt
   --- its ``inversion number''  $ \, \text{\it inv}(\cM) \, $,  \, which is the sum of all inversions of the order  $ \, \preceq \, $  among the indices of factors in  $ \cM $  when read from left to right.
 \vskip5pt
   We can now act upon any such  $ \cM $  with any of the following operations:
 \vskip4pt
   {\it --\,(1)} \,  we move all factors  $ \Big(\! {H_i - z_i \atop s} \!\Big) $  to the leftmost position, by repeated use of (3.2) and (3.9): this produces a new monomial  $ \cM' \, $,  multiplied on the left by several (new) factors  $ \Big(\! {{H_i - \check{z}_i} \atop s_i} \!\Big) \, $;
 \vskip4pt
   {\it --\,(2)} \,  whenever two  {\sl consecutive\/}  factors  $ \, X_{\talpha}^{(\ell_{\talpha})} \, $  and  $ \, X_{\talpha}^{(\ell^{\,\prime}_{\talpha})} \, $  occur in  $ \cM \, $,  we replace their product in  $ \cM $  with an integral coefficient times a single factor, using (3.3); and similarly, we replace any pair of  {\sl consecutive\/}  factors  $ \; X_{\tgamma}^{m_{\tgamma}} \, X_{\tgamma}^{m'_{\tgamma}} \; $  by the single factor  $ \, X_{\tgamma}^{m_{\tgamma} + m'_{\tgamma}} \, $;
 \vskip4pt
   {\it --\,(3)} \,  we replace any power  $ X_{\tgamma}^{m_{\tgamma}} $  of an odd root vector with  $ 0 $  or  $ \, \pm X_{2\gamma,t'} \, $,  for  $ \, \tgamma = (\gamma,k) \, $,  whenever  $ \; m_{\gamma,k} \! > \! 1 \, $,  applying (3.10);
 \vskip4pt
   {\it --\,(4)} \,  whenever two factors within  $ \cM $  occur side by side  {\sl in the wrong order w.~r.~to\/}  $ \preceq \; $,  i.e.~we have any one of the following situations
 \vskip-17pt
  $$  \displaylines{
   \cM  \; = \;  \cdots \, X_{\talpha}^{(\ell_{\talpha})} \, X_{\tbeta}^{(\ell_{\tbeta})} \cdots \quad ,
 \qquad  \cM  \; = \;   \cdots \, X_{\talpha}^{(\ell_{\talpha})} \, X_{\tgamma}^{m_{\tgamma}} \cdots  \cr
   \cM  \; = \;  \cdots \, X_{\teta}^{m_{\teta}} \, X_{\talpha}^{(\ell_{\talpha})} \cdots \quad ,
 \qquad  \cM  \; = \;  \cdots \, X_{\tgamma}^{\,m_{\tgamma}} \, X_{\teta}^{m_{\teta}} \cdots  }  $$
 \vskip-5pt
\noindent
 with  $ \, \talpha \succneqq \tbeta \, $,  $ \, \talpha \succneqq \tgamma \, $,  $ \, \teta \succneqq \talpha \, $  and  $ \, \tgamma \succneqq \teta \, $  respectively, we can use all relations (3.4--8) and (3.11--16) to re-write this product of two distinguished factors, so that all of  $ \cM $  expands into a  $ \Z $--linear  combination of new monomials.  In some cases one has to read these relations the other way round: for instance, one can use (3.7) to re-write  $ \, X_{\beta,h}^{(m)} \, X_{\alpha,1}^{(\ell\,)} \, $  as  $ \; X_{\beta,h}^{(m)} \, X_{\alpha,1}^{(\ell\,)} \, = \, \cdots \;\, $.
 \vskip7pt
   By definition,  $ \kzg $  is  $ \Z $--spanned  by all (unordered) monomials in the  $ X_{\talpha}^{(\ell_{\talpha})} $,  the  $ \Big( {{H_i} \atop n_i} \Big) $  and the  $ X_{\tgamma}^{\,m_{\tgamma}} $.  Let  $ \cM $  be any one of these monomials: it is PBW-like, i.e.~in  $ {\mathcal B} \, $,  if and only if no one of steps  {\it (2)\/}  to  {\it (4)\/}  may be applied; but if not, we now see what is the effect of applying such steps.
 \vskip4pt
   Applying step  {\it (1)\/}  gives  $ \; \cM \, = \, {\mathcal H} \, \cM' \; $  where  $ {\mathcal H} $  is some product of  $ \Big(\! {H_i - \check{z}_i \atop s_i} \Big) \, $,  and  $ \cM' $  is a new monomial such that  $ \; \textit{ht}\big(\cM'\big) = \textit{ht}\big(\cM\big) \; $,  $ \; \textit{fac}\big(\cM'\big) \leq \textit{fac}\big(\cM\big) \; $,  $ \; \textit{inv}\big(\cM'\big) \leq \textit{inv}\big(\cM\big) \, $,  \, and the strict inequality in the middle holds if and only if  $ \; {\mathcal H} \not= 1 \, $,  i.e.~step  {\it (1)\/}  is non-trivial.  Actually, this is clear at once when one realizes that  $ \cM' $  is nothing but  ``$ \cM $  with all factors  $ \Big(\! {{H_i - z_i} \atop s_i} \!\Big) $  removed.''
 \vskip4pt
   Then we apply any one of steps  {\it (2)},  {\it (3)\/}  or  {\it (4)\/}  to  $ \cM' \, $.
 \vskip4pt
   Step  {\it (2)}, if non-trivial, yields  $ \, \cM' = z \, \cM^\vee $,  \, for some  $ \, z \in \Z \, $  and some monomial  $ \cM^\vee \, $  such that  $ \, \text{\it ht}\big(\cM^\vee\big) \! = \text{\it ht}\big(\cM'\big) \, $,  $ \, \text{\it fac}\big(\cM^\vee\big) \lneqq \text{\it fac}\big(\cM'\big) \, $.
%
%
 Instead, step  {\it (3)},  still if non-trivial, gives  $ \; \cM' \, = \, 0 \; $.
 \vskip4pt
   Finally, step  {\it (4)\/}  gives  $ \; \cM' = \, \cM^\curlyvee \! + \sum_k z_k \, \cM_k \; $,  where  $ \, z_k \in \Z \, $  (for all  $ k \, $)  and  $ \cM^\curlyvee $  and the  $ \cM_k $  are monomials such that  $ \; \text{\it ht}\big(\cM_k\big) \lneqq \text{\it ht}\big(\cM_k\big) \;\; \forall \;\, k \; $,  $ \; \text{\it ht}\big(\cM^\curlyvee\big) = \text{\it ht}\big(\cM'\big) \; $,  $ \; \text{\it inv}\big(\cM^\curlyvee\big) \lneqq \text{\it inv}\big(\cM'\big) \; $.
 \vskip5pt
   In short, through either step  {\it (2)},  or  {\it (3)},  or  {\it (4)},  we achieve an expansion
 \vskip-5pt
  $$  \cM'  \; = \;  {\textstyle \sum_h} \; z'_h \, {\mathcal H} \, \cM'_h \;\; ,  \qquad  z'_h \in \Z  \quad  \forall \; h   \eqno (3.17)  $$
 \vskip-1pt
\noindent
 (the sum in right-hand side possibly being void, hence equal to zero) where   --- unless the step is trivial, for then we get all equalities ---   we have
 \vskip-4pt
  $$  \Big( \text{\it ht}\big(\cM'_h\big) \! \lneqq \text{\it ht}\big(\cM'\big) \!\Big) \vee \Big( \text{\it fac}\big(\cM'_h\big) \! \lneqq \text{\it fac}\big(\cM'\big) \!\Big) \vee \Big( \text{\it inv}\big(\cM'_h\big) \! \lneqq \text{\it inv}\big(\cM'\big) \!\Big)   \eqno (3.18)  $$
   \indent   Now we repeat, applying step  {\it (1)\/}  and then step  {\it (2)\/}  or  {\it (3)\/}  or  {\it (4)\/}  to every monomial  $ \cM'_h $  in (3.17).  Then we iterate, until we get only monomials whose inversion number is zero: by (3.18), this is possible indeed, and it is achieved after finitely many iterations.  The outcome reads
 \vskip-5pt
  $$  \cM'  \; = \;  {\textstyle \sum_j} \; \dot{z}_j \, {\mathcal H''_j} \, \cM''_j \;\; ,  \qquad  \dot{z}_j \in \Z  \quad  \forall \; j   \eqno (3.19)  $$
 \vskip-3pt
\noindent
 where  $ \; \text{\it inv}\big(\cM''_j\big) = 0 \; $  for every index  $ j \, $,  i.e.~all monomials  $ \cM''_j $  are ordered and without repetitions, that is they belong to  $ {\mathcal B} \, $.  Now each  $ {\mathcal H''_j} $  belongs to  $ \mathbb{H}_\Z $  (notation of  Corollary \ref{int-polynom_kost}),  just by construction.  Then  Corollary \ref{int-polynom_kost}  ensures that each  $ {\mathcal H''_j} $  expand into a  $ \Z $--linear  combination of ordered monomials in the  $ \Big(\! {H_i \atop n_i} \!\Big) $.  Therefore (3.19) yields
 \vskip-5pt
  $$  \cM  \; = \;  {\textstyle \sum_s} \;\, \hat{z}_s \; {\mathcal H^\wedge_s} \, \cM^\wedge_s \;\; ,  \qquad  \hat{z}_s \in \Z  \quad  \forall \; s   \eqno (3.20)  $$
 \vskip-3pt
\noindent
 where every  $ {\mathcal H^\wedge_s} $  is an ordered monomial, without repetitions, in the  $ \Big(\! {H_i \atop n_i} \!\Big) $,  while for each  index  $ s $  we have  $ \, \cM^\wedge_s = \cM''_j \, $  for some  $ j \, $   --- those in (3.19).
                                            \par
   Using again relations (3.2) and (3.9), we can switch positions among the factors  $ \Big(\! {H_i \atop n_i} \!\Big) \, $  in  $ {\mathcal H^\wedge_s} $  and the factors in  $ \cM^\wedge_s $  (for each  $ s $),  so to get a new monomial  $ \cM^\circ_s $  which is ordered, without repetitions, but might have factors of type  $ \Big(\! {{H_i - z_i} \atop m_i} \!\Big) $  with  $ \, z_i \in \Z \setminus \{0\} \, $   --- so that  $ \, \cM^\circ_s \not\in {\mathcal B} \, $.  But then  $ \, \Big(\! {{H_i - z_i} \atop m_i} \!\Big) \in \mathbb{H}_\Z \, $,  hence again  by  Corollary \ref{int-polynom_kost}  that factor expands into a  $ \Z $--linear  combination of ordered monomials, without repetitions, in the  $ \Big(\! {H_i \atop \ell} \!\Big) $.  Plugging every such expansion inside each monomial  $ \cM^\wedge_s $  instead of each factor  $ \Big(\! {{H_i - z_i} \atop m_i} \!\Big) $   ---  $ \, i = 1, \dots, r \, $  ---   we eventually find
  $$  \cM  \; = \;  {\textstyle \sum_q} \, z_q \, \cM^{\,!}_q \;\; ,  \qquad  z_q \in \Z  \quad  \forall \; q  $$
where now every  $ \cM^!_q $  is a PBW-like monomial,  i.e.~$ \, \cM^{\,!}_q \in {\mathcal B} \, $  for every  $ q \, $.
 \vskip5pt
   As  $ \kzg \, $,  by definition, is spanned over  $ \Z $  by all monomials in the  $ X_{\talpha}^{(\ell_{\talpha})} $,  the  $ \Big( {{H_i} \atop n_i} \Big) $  and the  $ X_{\tgamma}^{\,m_{\tgamma}} $,  our analysis yields  $ \; \kzg \subseteq \text{\it Span}_{\,\Z}({\mathcal B}) \, $.  On the other hand, by definition  $ \text{\it Span}_{\,\Z}({\mathcal B}) $  in turn is contained in  $ \kzg \, $.  Therefore  $ \; \kzg = \text{\it Span}_{\,\Z}({\mathcal B}) \, $,  \, i.e.~$ {\mathcal B} $  spans  $ \kzg $  over  $ \Z \, $.
 \vskip5pt
   At last, the PBW theorem for Lie superalgebras over fields ensures that  $ {\mathcal B} $  is a  $ \KK $--basis  for  $ U(\fg) \, $,  as  $ \; B := \big\{ H_1 , \dots , H_r \big\} \, {\textstyle \coprod} \, \big\{ X_{\talpha} \,\big|\, \talpha \in \tDelta \big\} \; $  is a  $ \KK $--basis  of  $ \fg \, $  (cf.~\cite{vsv}).  So  $ {\mathcal B} $  is linearly indepen\-dent over  $ \KK \, $,  hence over  $ \Z \, $.  Therefore  $ {\mathcal B} $  is a  $ \Z $--basis  for  $ \kzg \, $,  and the latter is a free  $ \Z $--module.
\end{proof}

\medskip

\begin{free text}  \label{kost-form_sub-objs}
 {\bf Kostant superalgebras for special sub-objects.}  \  We can consider  $ \Z $--integral  forms for the sub-objects, as follows.  First fix a Chevalley basis  $ B $  and  $ \fg^\Z $  as in  Definition \ref{def_che-bas}.  Second, for  $ \, \alpha \in \Delta \, $,  $ \, q \geq -1 \, $,  we set  $ \; B_\alpha := B \cap \fg_\alpha \; $,  $ \; \fg_\alpha^\Z := \text{\it Span}_{\,\Z}(B_\alpha) \; $,  $ \; \fg_q^\Z := \text{\it Span}_{\,\Z}\big( \coprod_{\text{\it ht}\,(\alpha) = q} B_\alpha \big) \; $,  \, and
 \vskip-13pt
  $$  \displaylines{
   \fg^\Z_\zero  \; := \, \fh_\Z \oplus \Big(\, {\textstyle \bigoplus_{\alpha \in \Delta_\zero \,}} \fg^\Z_\alpha \Big)  \quad ,  \qquad  \fg^\Z_\uno  \; := \,  {\textstyle \bigoplus_{\gamma \in \Delta_\uno \,}} \fg^\Z_\gamma  \quad ,  \qquad  \fg^\Z_0  \; := \, \fh_\Z \oplus \Big(\, {\textstyle \bigoplus_{\alpha \in \Delta_0 \,}} \fg^\Z_\alpha \Big)  \cr
   \fg^\Z_{\zero^{\,\uparrow}}  :=  {\textstyle \bigoplus_{\alpha \in \Delta_{\zero^{\,\uparrow}}}} \fg^\Z_\alpha  \;\, ,
    \;\quad  \fg^\Z_{\uno^{\,\uparrow}}  :=  {\textstyle \bigoplus_{\gamma \in \Delta_{\uno^{\,\uparrow}}}} \fg^\Z_\gamma  \;\, ,
    \;\quad  \fg^\Z_{t^\uparrow}  := \,  {\textstyle \bigoplus_{q>t\,}} \fg^\Z_q  \;\; \big(\, t \geq -1 \big)  \;\, ,
    \;\quad  \fg^\Z_{-1,0}  :=  \fg^\Z_{-1} {\textstyle \bigoplus} \, \fg^\Z_0  }  $$
 \vskip-5pt
\noindent
 (the last only for  $ \fg \not\cong \tS(n) $;  notation of  \S \ref{triang-dec_Borel-subs_Lie-sub-objs})  with  $ \, \fh_\Z := \text{\it Span}_{\,\Z} \big( H_1 , \dots , H_r \big) \, $,  cf.~Definition \ref{def_che-bas}{\it (a)}.
\end{free text}

\vskip11pt

\begin{definition}  \label{def-kost-superalgebra_sub-objs}
 We define  $ K_\Z\big(\fg_{-1^\uparrow}\!\big) \, $  as the unital  $ \Z $--subsuperalgebra  of  $ \, U(\fg) \, $  generated by  $ \, \Big(\! {H_i \atop m} \!\Big) \, $,  $ \, X_{\talpha}^{\,(m)} $  and  $ \, X_{\tgamma} \, $  for all  $ \; m \in \N \, $,  $ \, 1 \! \leq \! i \! \leq \! r \, $,  $ \; \talpha \! \in \! \tDelta_\zero \, $,  $ \; \tgamma \in \tDelta_{\uno^{\,\uparrow}} \; $.  In a similar way, we define the unital  $ \Z $--subsuperalgebras  $ \, K_\Z\big(\fg_\zero\big) \, $,  $ \, K_\Z\big(\fg_0\big) \, $,  $ \, K_\Z\big(\fg_{\zero^{\,\uparrow}}\!\big) \, $,  $ \, K_\Z\big(\fg_{t^\uparrow}\!\big) \, $   --- $ \, t \geq 0 \, $ ---   and  $ \, K_\Z\big(\fg_{-1,0}\big) $   --- for  $ \, \fg \not\cong \tS(n) \, $  ---   as the ones generated by the binomial coefficients, divided powers of even root vectors, and odd root vectors involved in the very definition of  $ \fg_\zero \, $,  $ \fg_0 \, $,  $ \fg_{\zero^{\,\uparrow}} \, $,  $ \fg_{t^\uparrow} $  and  $ \fg_{-1,0} \; $.
 \vskip3pt
   Also, we denote  $ \, \bigwedge \fg_\uno^\Z \, $,  resp.~by  $ \, \bigwedge
\fg_{\uno^{\,\uparrow}}^\Z \, $,  the (unital) exterior\/  $ \Z $--algebra  over  $ \fg_\uno^\Z \, $,
resp.~over  $ \fg_{\uno^{\,\uparrow}}^\Z \; $.
\end{definition}

\vskip9pt

   All these objects are related by the following consequence of  Theorem \ref{PBW-Kost}  (in particular, the first isomorphism is an integral version of the factorization  $ \; U(\fg) \, \cong \, U(\fg_0) \otimes_{\KK} \bigwedge \fg_\uno \, $,  see  \cite{vsv}),  whose proof follows from the arguments used for  Theorem \ref{PBW-Kost},  or as a direct consequence of it:

\vskip15pt

\begin{corollary}  \label{kost_tens-splitting}
 There exist isomorphisms of  $ \, \Z $--modules
  $$  \kzg  \; \cong \, K_\Z\big(\fg_\zero\big) \otimes_\Z {\textstyle \bigwedge} \, \fg_\uno^\Z \;\; ,
   \quad  K_\Z\big(\fg_{-1^\uparrow}\!\big)  \; \cong \,  K_\Z\big(\fg_\zero\big) \otimes_\Z {\textstyle \bigwedge} \, \fg_{\uno^{\,\uparrow}}^\Z \;\; ,
   \quad  K_\Z\big(\fg_{0^\uparrow}\!\big)  \; \cong \,  K_\Z\big(\fg_{\zero^{\,\uparrow}}\big) \otimes_\Z {\textstyle \bigwedge} \, \fg_{\uno^{\,\uparrow}}^\Z  $$
and of  $ \, \Z $--superalgebras
 $ \,\; K_\Z\big(\fg_\zero\big) \, \cong \, K_\Z\big(\fg_0\big) \otimes_\Z K_\Z\big(\fg_{\zero^{\,\uparrow}}\!\big) \; $,
 $ \,\; K_\Z\big(\fg_{-1,0}\big) \, \cong \, K_\Z\big(\fg_0\big) \otimes_\Z {\textstyle \bigwedge} \, \fg_{-1}^\Z \; $.
\end{corollary}

\vskip7pt

\begin{remark}  \label{super-dist}
 \, Following a classical pattern, one defines the  {\sl superalgebra of distributions\/}  $ {\mathcal D}\text{\it ist}\,(G) \, $  on any supergroup  $ G \, $,  by an obvious extension of the standard notion in the even setting; see  \cite{bk},  \S 4, for details.  If  $ G $  is any one of the algebraic supergroups (over  $ \bk $)  that we are going to construct, then  $ \, \Lie\,(G) = \fg \, $   --- with some more precisions: see  subsection \ref{cartan_LieTT}  later on.  Then one can check   --- like in  \cite{bk}, \S 4  ---   that  $ \; {\mathcal D}\text{\it ist}\,(G) = \, \bk \otimes_\Z \kzg \, =: \, K_\bk(\fg) \; $.  An entirely similar remark occurs when the supergroup  $ G $  is one of the ``Chevalley supergroups'' introduced in  \cite{fg2}.
                                               \par
   Any morphism  $ \, \varphi : G' \longrightarrow G'' \, $  between two supergroups induces (functorially) a morphism  $ \, D_\varphi : {\mathcal D}\text{\it ist}\,\big(G'\big) \longrightarrow {\mathcal D}\text{\it ist}\,\big(G''\big) \, $,  which is injective whenever  $ \varphi $  is injective.  If in addition the supergroups  $ G'$  and  $ G'' $  are of the type mentioned above, then  $ \, {\mathcal D}\text{\it ist}\,\big(G'\big) = K_\bk(\fg'\big) \, $  and  $ \, {\mathcal D}\text{\it ist}\,\big(G''\big) = K_\bk(\fg''\big) \, $,  so that  $ \, D_\varphi \! : K_\bk(\fg'\big) \! \rightarrow K_\bk\big(\fg''\big) \, $,  which is an embedding if  $ G' $  is a subsupergroup of  $ G'' $.
\end{remark}

\medskip

  \subsection{Admissible lattices for  $ \fg $--modules}  \label{adm-lat}

\smallskip

   {\ } \quad   The tools of our construction of algebraic supergroups are the Lie superalgebra  $ \, \fg \, $  together with an integrable  $ \fg $--module.  As we need an integral version of  $ \fg $   --- and even more, of  $ U(\fg) $,  namely the Kostant superalgebra ---   we also need a suitable integral form of any integrable  $ \fg $--module.
 \vskip4pt
   Let a Chevalley basis  $ B $  and Kostant algebra  $ \kzg $  be given, as before.  If  $ V $  a  $ \KK $--vector  space we call a subset  $ \, M \subseteq V \, $  a  $ \Z $--{\it form\/}  of  $ V $  if (as usual)  $ \, M \! = \text{\it Span}_\Z(\mathcal{B}) \, $  for some  $ \KK $--basis  $ \mathcal{B} $  of  $ V \, $.

\vskip11pt

\begin{definition}  \label{def_adm-latt}
 Let  $ V $  be a  $ \fg $--module.  We retain terminology and notation of  Definition \ref{def_weight-modules-etc}.
 \vskip3pt
      {\it (a)} \,  We call  $ V $  {\it rational\/}  if  $ \, \fh_\Z := \text{\it Span}_\Z \big( H_1, \dots, H_r \big) \, $  acts diagonally on  $ V $  with eigenvalues in  $ \Z \, $;  in other words,  $ \, V \! = \! \bigoplus_{\mu \in \fh^*} \! V_\mu \, $  is a weight  $ \fg $--module  {\sl and}   $ \, \mu(H_i) \in \Z \, $  for all  $ i $  and all  $ \, \mu \in \text{\it Supp}(V) \, $.
 \vskip2pt
   {\it (b)} \,  Any  $ \Z $--lattice  $ M $  of  $ V $  is said to be  {\it admissible\/}  if it is a  $ \kzg $--stable  $ \Z $--form  of  $ V $.
 \vskip3pt
   {\sl Note\/}  that   --- by the last remark in  Definition \ref{def_weight-modules-etc}  ---   the  $ \fg $--action on any rational  $ \fg $--module  $ V $  can be extended to a  $ \overline{\fg} $--action  so that  $ \cE $  acts diagonally (=semisimply) on  $ V $  with eigenvalues in  $ \Z \, $;  in short,  $ V $  itself is also a  {\it rational  $ \, \overline{\fg} $--module},  with  $ \overline{\fh} $--weight  space decomposition  $ \, V \! = \bigoplus_{\nu \in {\overline{\fh}}^*} V_\nu \; $.
\end{definition}

\vskip5pt

   The first property of admissible lattices is natural (its proof being classical,  cf.~\cite{st},  \S 2, Cor.~1):

\vskip11pt

\begin{proposition}  \label{split_adm-latt}
  Let  $ \, V $  be a weight  $ \fg $--module.  Then any admissible lattice  $ M $  of  $ \, V $  is the direct sum of its  $ \fh $--weight  components,  i.e.~$ \, M = \bigoplus_{\lambda \in \fh^*} \! \big( M \cap V_\lambda \big) \, $,  and similarly for  $ \overline{\fh} $--weights.
\end{proposition}

\vskip3pt

   Next property instead is an existence result, under mild conditions:

\vskip11pt

\begin{proposition}  \label{exist_adm-latt}
  Let  $ \, V $  be a finite dimensional, completely reducible  $ \fg $--module.  Then  $ \, V $  is a weight module.  If it is also rational, then there exists an admissible lattice  $ M $  of it.
\end{proposition}

\begin{proof}
 First of all, by  Remark \ref{rems_weight-mods}  $ V $  is a weight module.  Now assume it is also rational.  Then by the complete reducibility assumption we can reduce to assume  $ V $  irreducible.  In that case, like for  \cite{se},  Theorem 3.1,  we find that  $ V $  is  {\sl cyclic},  i.e.~it can be generated by a single vector, and the latter can be taken to be a highest weight vector  (cf.~Definition \ref{def_weight-modules-etc}).  Letting  $ v $  be such a highest weight vector, set  $ \; M := \kzg.v \;\, $:  then one can repeat the classical proof   --- like in  \cite{st},  \S 2, Corollary 1 ---   and eventually show that such an  $ M $  is indeed an admissible lattice of  $ V $  as required.
\end{proof}

\smallskip

   We can also describe the stabilizer of an admissible lattice:

\smallskip

\begin{proposition}  \label{stabilizer}
   Let  $ V $  be a faithful, rational, finite dimensional\/  $ \fg $--module,  $ M $  an admissible lattice of  $ \, V $,  and  $ \; \fg_V = \big\{ X \! \in \! \fg \,\big|\, X.M \subseteq M \big\} \; $  its stabilizer.  Then,
%
 letting  $ \; \fh_V \, := \, \big\{ H \in \fh \;\big|\; \mu(H) \in \Z \, , \; \forall \; \mu \in \Lambda \big\} \; $,  \, where  $ \Lambda $  is the set of weights of  $ \, V $,  we have  $ \; \fg_V  \, = \,  \fh_V \, {\textstyle \bigoplus} \, \Big(\! {\textstyle \bigoplus}_{\talpha \in \tDelta} \, \Z \, X_{\talpha} \Big) \; $.
 \vskip3pt
   In particular,  $ \fg_V $  is a lattice in  $ \fg \, $,  independent of the choice of the admissible lattice  $ M \, $.
%
\end{proposition}

\begin{proof} The classical proof in  \cite{st}, \S 2, Corollary 2, applies again, with some additional arguments to manage odd root spaces.  Indeed, the same arguments as in  [{\it loc.~cit.}]  prove that  $ \; \fg_V = \fh_V \bigoplus \Big(\! \bigoplus_{\talpha \in \tDelta} \big(\, \fg_V \cap \, \KK\,X_{\talpha} \big) \Big) \phantom{\bigg|} $;  then one still has to prove that  $ \; \fg_V \! \cap \KK\,X_{\talpha} = \, \Z \, X_{\talpha} \; $  for all  $ \, \talpha \in \tDelta \, $.  The arguments in  [{\it loc.~cit.}]  show that  $ \; \fg_V \! \cap \KK \, X_{\talpha} \; $  is a cyclic  $ \Z $--submodule  of  $ \fg_V \, $ which may be  $ \Z $--spanned  by some  $ \, \frac{\,1}{\,n_{\talpha}} \, X_{\talpha} \, $  with  $ \, n_{\talpha} \in \N_+ \, $  (for  $ \, \talpha \in \tDelta \, $).  What is left to prove is that  $ \, n_{\talpha} = 1 \, $.
 \vskip3pt
   For every  $ \, \talpha \in \tDelta_0 \, $  the same arguments as in  [{\it loc.~cit.}]  still yield  $ \, n_{\talpha} = 1 \, $.
%
%
 \vskip3pt
   For every  $ \, \talpha \in \tDelta_{-1} \, $,  by Definition \ref{def_che-bas}{\it (f.2)\/}  we can find  $ \, \tbeta \in \tDelta \, $  such that  $ \, \alpha + \beta \in \Delta_0 \, $,  with  $ \, \alpha := \pi\big(\talpha\big) \, $,  $ \, \beta := \pi\big(\tbeta\,\big) \, $,  and  $ \; \big[ X_{\talpha} \, , X_{\tbeta} \big] = \pm X_{(\alpha + \beta, \, t')} \; $  for some  $ \, t'\in \big\{1,\dots,\mu(\alpha+\beta)\!\big\} \, $.  This yields
%
%
  $ \; \pm \, {\textstyle \frac{\,1}{\,n_{\talpha}}} \, X_{(\alpha + \beta, \, t')} = \Big[ {\textstyle \frac{\,1}{\,n_{\talpha}}} \, X_{\talpha} \, , X_{\tbeta} \,\Big] \in \Big[ \fg_V , \fg^\Z \Big] \subseteq \big[ \fg_V , \fg_V \big] \subseteq \fg_V \; $
 because  $ \; \fg^\Z \subseteq \fg_V \; $  and  $ \, \fg_V \, $  is a Lie subsuperalgebra of  $ \fg \, $.  Therefore
 $ \; {\textstyle \frac{\,1}{\,n_{\talpha}}} \, X_{(\alpha + \beta, \, t')} \in \fg_V \cap \, \KK \, X_{(\alpha + \beta, \, t')} = \, \Z \, X_{(\alpha + \beta, \, t')} \; $
 because  $ \, X_{(\alpha + \beta, \, t')} \in \tDelta_0 \, $   --- and thanks to the previous step --- which eventually forces  $ \, n_{\talpha} = 1 \, $.
 \vskip3pt
   Finally, consider  $ \, \talpha \in \tDelta_z \, $  with  $ \, z \! \geq 1 \, $,  and let  $ \, \alpha := \pi\big(\talpha\big) \in \Delta_z \; $.  Then, by direct analysis, we see that there exists  $ \, \gamma \in \Delta_{-1} \, $  such that  $ \, (\alpha + \gamma) \in \Delta_{z-1} \, $;  therefore, for  $ \, \tgamma := (\gamma,1) \, $  we have again by  Definition \ref{def_che-bas}{\it (f.2)\/}  that  $ \; \big[ X_{\talpha} \, , X_{\tgamma} \big] \, = \, \pm \, X_{(\alpha+\gamma,\,t')} \; $.  Just like before, this yields
 $ \; \pm {\textstyle \frac{\,1}{\,n_{\talpha}}} \, X_{(\alpha+\gamma,\,t')}  \, = \,  \Big[ {\textstyle \frac{\,1}{\,n_{\talpha}}} \, X_{\talpha} \, , X_{\tgamma} \Big]  \, \in \,  \Big[ \fg_V , \fg^\Z \Big]  \, \subseteq \,  \big[ \fg_V , \fg_V \big]  \, \subseteq \,  \fg_V \; $
 hence  $ \, \frac{\,1}{\,n_{\talpha}} \, X_{(\alpha+\gamma,\,t')} \, \in \, \fg_V \cap \, \KK \, X_{(\alpha+\gamma,\,t')} \, $.
 Since  $ \alpha $  has height  $ \, |\alpha| = z \, $  and  $ |\alpha + \gamma| \! = \! z \! - \! 1 \, $,  by induction on  $ z $  we assume  $ \, \fg_V \cap \, \KK \, X_{(\alpha+\gamma,\,t')} = \Z X_{(\alpha+\gamma,\,t')} \, $:  the basis of induction is  $ \, z = 0 \, $  which corresponds to roots in  $ \Delta_0 \, $,  that we already disposed of.  Therefore  $ \; \frac{\,1}{\,n_{\talpha}} \, X_{(\alpha+\gamma,\,t')} \, \in \, \Z \, X_{(\alpha+\gamma,\,t')} \; $,  \, so that  $ \; n_{\talpha} = 1 \; $.
\end{proof}

\bigskip

\section{Algebraic supergroups  $ \bG_V $  of Cartan type}  \label{car-sgroups}

\smallskip

   {\ } \quad   Classically, Chevalley groups are defined as follows.  Fix a finite dimensio\-nal semisimple Lie algebra  $ \fg $  over an algebraically closed field  $ \KK $  of characteristic zero, a Chevalley basis of  $ \fg $  and the associated Kostant form  $ K_\Z(\fg) $  of  $ U(\fg) \, $.  Then any simple finite dimensional  $ \fg $--module  $ V $  contains a  $ \Z $--lattice  $ M $,  which is  $ K_\Z(\fg) $--stable,  so  $ K_\Z(\fg) $  acts on  $ M \, $.  Using this action and its extensions by scalars to any field  $ \bk \, $,  one defines one-parameter subgroups  $ x_\alpha(t) \, $,  for all roots  $ \alpha $  and  $ \, t \in \bk \, $,  within  $ \rGL(V_\bk) \, $, $ \, V_\bk := M \otimes \bk \, $:  the Chevalley group (associated with  $ \fg $  and  $ V \, $)  is
 the subgroup of  $ \rGL(V_\bk) $  generated by the  $ x_\alpha(t) $.  If one has to extend this construction to that of a genuine  $ \Z $--group  scheme,
 then slight variations are in order, e.g.~one has to ``add by hand''
 a maximal torus.

\smallskip

   This construction has been adapted to simple Lie superalgebras  $ \fg $  of  {\sl classical\/}  type in  \cite{fg2},  \cite{ga}.  We do now the same for all simple Lie superalgebras  $ \fg $  of Cartan type.

\bigskip

  \subsection{One-parameter supersubgroups}  \label{One-param-ssgroups}

\vskip3pt

   {\ } \quad   The supergroups we
 look for
 will be realized as subgroup-functors of some linear supergroup functors  $ \rGL(V) \, $,  generated by suitable subgroup functors: these are super-analogues of one-parameter subgroups in the classical theory, thus we call them ``one-parameter supersubgroups''.  Like in the classical setup, they will be of two types: multiplicative
 and additive:
 the latter ones then will split into two more types, according to the type (even or odd) of the roots involved.

\vskip7pt

   We retain the notation of  sections \ref{preliminaries}  and  \ref{integr-struct}.
 In particular,  $ V $  is a fixed faithful, rational, finite dimensional weight  $ \fg $--module  with an admissible lattice  $ M $  in it (e.g., if  $ V $  is completely reducible).

\vskip5pt

   Fix a commutative unital  $ \Z $--algebra  $ \bk \, $,  and set  $ \; \fg_{V,\bk} := \, \bk \otimes_\Z \fg_V \, $,  $ \, V_\bk := \, \bk \otimes_\Z M \, $,  $ \, U_\bk(\fg) := \, \bk \otimes_\Z \kzg \; $.  Then  $ \fg_{V,\bk} $  acts faithfully on  $ V_\bk \, $,  which yields a Lie superalgebra monomorphism  $ \; \fg_{V,\bk} \! \lhook\joinrel\longrightarrow \End(V_\bk) \; $  and a superalgebra morphism  $ \; U_\bk(\fg) \longrightarrow \End(V_\bk) \; $.  Now, for every  $ \, A \in \salg_\bk \, $  define  $ \; \fg_A := A \otimes_\bk \fg_{V,\bk} \; \big(\!= A \otimes_\Z \fg_V \big) \; $,  $ \; V_A := A \otimes_\bk M_\bk \; \big(\!= A \otimes_\Z M \,\big) \; $  and  $ \; U_{\!A}(\fg) := A \otimes_\bk U_\bk(\fg) \; \big(\!= A \otimes_\Z \kzg \big) \; $.  Then  $ \fg_A $  acts faithfully on  $ V_A \, $,  which yields morphisms  $ \; \fg_A \! \lhook\joinrel\longrightarrow \End(V_A) \; $  and  $ \; U_{\!A}(\fg) \longrightarrow \End(V_A) \; $. Moreover (as sketched in  \S \ref{Lie-salg_funct}) all these constructions are functorial in  $ A \, $.

\vskip7pt

   The splitting  $ \, V = {\oplus_{\mu \in \fh^*}} V_\mu \, $  of the  $ \overline{\fg} $--module  $ V $  into  $ \overline{\fh} $--weight  spaces yields, for any  $ \, A \in \salg_\bk \, $,  a similar split\-ting  $ \, V_\bk(A) = {\oplus_{\mu \in {\overline{\fh}}^*}} V_\mu(A) \, $   ---  using notation as in  Examples \ref{exs-supvecs}{\it (a)}.  Now fix any element  $ \, H \! \in \fh_\Z := \text{\it Span}_{\,\Z} \big( H_1 , \dots , H_r \big) \, $   --- see Definition \ref{def_che-bas}{\it (a)\/}:  then  $ \, \mu(H) \in \Z \, $  for any  $ \, \mu \in \text{\it Supp}(V) \, $,  as  $ V $  is rational.  Let  $ U(A_\zero) $  the group of invertible elements in  $ A_\zero \, $:  we set
  $$  h_H(u).v  \; := \;  u^{\mu(H)} \, v   \eqno \forall \;\; v \in V_\mu(A) \, , \; \mu \in \text{\it Supp}(V) \, , \; u \in U(A_\zero)  \quad  $$
which defines an operator  $ \; u^H := h_H(u) \, \in \, \rGL\big(V_\bk(A)\big) \; $  for all  $ \, u \in U(A_\zero) \, $.
                                                               \par
   Note that the formal identity  $ \; u^H \! = \! \sum\limits_{m=0}^{+\infty} \! {(u \! - \! 1)}^m \Big({H \atop m}\Big) \, $,  whose right-hand side becomes a finite sum if acting on a single weight space  $ V_\mu(A) \, $,  shows that the operator  $ \, u^H \! = h_H(u) \, $  is one of those given by the  $ U_{\!A}(\fg) $--action  on  $ V $.  Note also that  $ \, H = \sum\limits_{i=1}^r \! z_i H_i \, $  ($ z_i \in \Z $)  yields  $ \, h_H(u) := \! \prod\limits_{i=1}^r \! {h_{H_i}\!(u)}^{z_i} \, $.

\vskip7pt

\begin{definition}  \label{mult_one-param-ssgrs}
   For any  $ \, H \in \fh_\Z \, $,  we define the supergroup functor  $ h_H $   --- also referred to as a ``multiplicative one-parameter supersubgroup'' ---   from  $ \salg_\bk $  to  $ \grps $  as given on objects by
%
%
 $ \; h_H(A) \, := \, \big\{\, u^H := h_H(u) \;\big|\; u \in \! U(A_\zero) \big\} \; $
and given on morphisms in the obvious way.
                                                                  \par
   We also write  $ \; h_i := h_{H_i} \, $  for  $ \, i = 1, \dots, r \, $,  and  $ \; h_\alpha := h_{H_\alpha} \; $  for  $ \, \alpha \in \Delta \; $.
\end{definition}

\vskip4pt

   Let  $ \, \talpha \in \tDelta_\zero \, $,  $ \, \tbeta \in \tDelta_\uno \, $,  and let  $ X_{\talpha} \, $  and  $ X_{\tbeta} $  be associated root vectors in  $ B \, $.  Both  $ X_{\talpha} $  and  $ X_{\tbeta} $  act as nilpotent operators on  $ V \, $,  thus on  $ M \, $  and  $ V_\bk \, $;  the same holds for
  $$  t \, X_{\talpha} \; ,  \;\; \vartheta \, X_{\tbeta}  \;\; \in \;
\End\big(V_\bk(A)\big)   \eqno \forall \;\; t \in A_\zero \, ,  \; \vartheta \in A_\uno \;\, .   \qquad (4.1)  $$
   Since  $ \; \big( \vartheta \, X_{\tbeta} \big)^2 = - \vartheta^2 X_{\tbeta}^{\,2} = 0 \; $,  we have  $ \, Y^m \! \big/ m! \in \big(\kzg\big)(A) \, $  for any  $ \, m \in \N \, $,  $ \, Y \in \big\{ t \, X_{\talpha} \, , \, \vartheta \, X_{\tbeta} \big\} \, $  as in (4.1); moreover,  $ \, Y^m \! \big/ m! \, $  acts as zero for  $ \, m \gg 0 \, $,  by nilpotency, so  $ \; \exp(Y) \! := \sum_{m=0}^{+\infty} Y^m \! \big/ m! \in \rGL\big(V_\bk(A)\big) \, $  is well defined.  In particular,  $ \; \exp\big(\vartheta \, X_{\tbeta}\big) := \sum_{m=0}^{+\infty} \, \big( \vartheta \, X_{\beta,k} \big)^m \! \Big/ m! \, = \, 1 + \, \vartheta \, X_{\tbeta} \; $.

\vskip7pt

\begin{definition}  \label{add_one-param-ssgrs}
 Let  $ \, \talpha \in \tDelta_\zero \, $,  $ \, \tbeta \in \tDelta_\uno \, $,  and let  $ X_{\talpha} \, $  and  $ X_{\tbeta} $  be as above; then set  $ \; x_{\talpha}(t) := \exp\!\big( t X_{\talpha} \big) = 1 \! + t X_{\talpha} + t^2 X_{\talpha}^{\,(2)} + \cdots \; $,  \, for all  $ \, t \in A_\zero \, $,  \, and  $ \; x_{\tbeta}(\vartheta) := \exp\!\big( \vartheta \, X_{\tbeta} \big) = 1 +  \vartheta \, X_{\tbeta} \; $,  \, for all  $ \, \vartheta \in A_\uno \, $.  We define the supergroup functors  $ x_{\talpha} $  and  $ x_{\tbeta} $  from  $ \salg_\bk $  to  $ \grps $  setting them on objects as  $ \; x_{\talpha}(A) \, := \, \big\{ x_{\talpha}(t) \,\big|\; t \in \! A_\zero \,\big\} \; $,  $ \; x_{\tbeta}(A) \, := \, \big\{ x_{\tbeta}(\vartheta) \;\big|\; \vartheta \in A_\uno \,\big\} \; $  for all  $ \, A \in \salg_\bk \, $   --- the definition on morphisms then should be clear.
 \vskip2pt
   In order to unify the notation, we shall denote by  $ \, x_{\teta}(\bt) \, $,  for  $ \, \teta \in \tDelta \, $,  any one of the two possibilities above, so that  $ \, \bt \in A_\zero \cup A_\uno \, $.  Finally, for later convenience we shall also formally write  $ \, x_{\widetilde{\zeta}}(\bt) := 1 \, $  when  $ \pi\big(\widetilde{\zeta}\,\big)  $  belongs to the  $ \Z $--span  of  $ \Delta $  but  $ \, \pi\big(\widetilde{\zeta}\,\big) \not\in \Delta \; $.
\end{definition}

\smallskip

   As in the Lie supergroup setting (see subsection 2.3 in  \cite{fg2}),  one can easily prove the following:

\smallskip

\begin{proposition} \label{hopf-alg} {\ }
 The following hold:
 \vskip3pt
   {\it (a)} \,  Every supergroup functor  $ h_H $
%
%
is representable, so it is an  {\sl affine}  supergroup, of (super)\-dimension  $ 1\big|0 \, $.  Indeed,
  $ \, h_H(A) = \Hom\big(\,\bk\big[z,z^{-1}\big],A\big) \, $,  for  $ \, A \in \salg_\bk \, $,  with  $ \, \varDelta\big(z^{\pm 1}\big) = z^{\pm 1} \otimes z^{\pm 1} \, $.
 \vskip5pt
   {\it (b)} \,  The supergroup functors  $ x_{\talpha} $  and $ x_{\tbeta} $
%
%
are representable, so they are  {\sl affine}  supergroups, respectively of (super)dimension  $ \, 1\big|0 \, $  and  $ \, 0\big|1 \, $.  Indeed, for every  $ \, A \in \salg_\bk \, $  one has
  $ \; x_{\talpha}(A) \, = \, \Hom\big(\bk[x]\,,A\big) \; $  with  $ \; \varDelta(x) \, = \, x \otimes 1 + 1 \otimes x \; $  and
  $ \; x_{\tbeta}(A) \, = \, \Hom\big(\bk[\xi]\,,A\big) \; $  with  $ \; \varDelta(\xi) \, = \, \xi \otimes 1 + 1 \otimes \xi \; $.
 \vskip3pt
   ({\sl Remark:}  in both cases,  $ \varDelta \, $  denotes the comultiplication in the Hopf superalgebra under exam)
\end{proposition}

\medskip

  \subsection{Construction of supergroups  $ \bG_V $  of Cartan type}  \label{const-car-sgroups}

\smallskip

   {\ } \quad   We now define our supergroups of Cartan type
 as suitable
subgroup functors   --- from  $ \salg_\bk $  to  $ \grps $  ---   of  $ \rGL\big(V_\bk\big) \, $.
 Further details about the formalism of (Grothendieck) topologies in categories and sheafification of functors can be found in  \cite{fg2},  Appendix, and in references therein.

\medskip

   Once and for all, we let  $ \fg $  and  $ V $  as above, and we fix also a partition  $ \; \Delta = \Delta^+ \coprod \Delta^- \; $  of the roots into positive and negative ones as in  \S \ref{triang-dec_Borel-subs_Lie-sub-objs}.

\medskip

\begin{definition}  \label{def_Cartan-sgroup_funct}
 \, We call  {\it Cartan (type) supergroup functor},  associated with  $ \fg $  and  $ V \, $,  the functor  $ \, G_V : \salg_\bk \! \longrightarrow \! \grps \, $  defined as follows.  Let  $ \, A \, , B \in \text{\it Ob}\big(\salg_\bk\big) \, $,  $ \, \phi \in \text{Hom}_{\salg_\bk}\!\big(A\,,B\big) \, $:  then
 \vskip5pt
   \noindent \quad   --- the object  $ \, G_V(A) \, $  is the subgroup of  $ \rGL\big(V_\bk(A)\big) $  generated by the subgroups  $ \, h_H(A) \, $  and  $ \, x_{\talpha}(A) \, $,  with  $ H \! \in \! \fh_\Z \, $,  $ \talpha \! \in \! \tDelta \, $,  i.e.~$ \; G_V(A) \, := \, {\Big\langle h_H(A) \, , \, x_{\talpha}(A) \Big\rangle}_{\! H \! \in \fh_\Z \, , \, \talpha \in \tDelta} \, = \, \Big\langle h_i(A) \, , \, x_{\talpha}(A) \Big\rangle_{\hskip-4pt {{\talpha \, \in \tDelta} \atop {i = 1, \dots, r ;}}} \; $;
 \vskip-1pt
   \noindent \;   --- the morphism  $ \, G_{V\!}(A) \,{\buildrel {G_{V_{\phantom{|}}}\!\!(\phi)} \over {\longrightarrow}}\, G_{V\!}(B) \, $  is the restriction of  $ \rGL\big(V_\bk(A)\big) \! \rightarrow \! \rGL\big(V_\bk(B)\big) $,  the morphism induced by  $ \phi $  by functoriality of  $ \rGL\big(V_\bk) \, $  (which maps the generators of  $ G_{V\!}(A) $  to those of  $ G_{V\!}(B) $).
\end{definition}

\vskip3pt

   For later use, we need to consider several other supergroup functors:

\vskip9pt

\begin{definition}  \label{def_Cartan-subgroup_funct}
 \, Let  $ G_V $  be as above.  We define the full subfunctors  $ T_V $,  $ G_0 $, $ G_0^\pm $,  $ G_\zero \, $,  $ G_\zero^\pm \, $,  $ G_{\zero^{\,\uparrow}} $,  $ G^\pm_{\zero^{\,\uparrow}} $
  and  $ G^\pm $  of  $ G_V $   --- still from $ \salg_\bk $  to  $ \grps $  ---   as given on objects,  for all  $ \, A \in \salg_\bk \, $,  by
  $$  \displaylines{
   T_V(A)  \; := \;  \big\langle\, h_H(A) \;\big|\; H \! \in \fh_\Z \,\big\rangle  \; = \;  \big\langle\, h_i(A) \;\big|\; i = 1, \dots, r \,\big\rangle_{\phantom{\big|}}  \cr
   G_0(A)  \; := \;  \Big\langle h_i(A) \, , \, x_{\talpha}(A) \Big\rangle_{\hskip-3pt {{\talpha \, \in \, \tDelta_0} \atop {i = 1, \dots, r;}}}  \quad ,  \qquad
  G_0^\pm(A)  \, := \,  \Big\langle h_i(A) \, , \, x_{\talpha}(A) \Big\rangle_{\hskip-3pt {{\talpha \, \in \, \tDelta_0^\pm} \atop {i = 1, \dots, r;}}}  \cr
 }  $$
 \eject
  $$  \displaylines{
   G_\zero(A)  \; := \;  \Big\langle h_i(A) \, , \, x_{\talpha}(A) \Big\rangle_{\hskip-3pt {{\talpha \, \in \, \tDelta_\zero} \atop {i = 1, \dots, r;}}}  \quad ,  \qquad
  G_\zero^\pm(A)  \; := \,  \Big\langle h_i(A) \, , \, x_{\talpha}(A) \Big\rangle_{\hskip-3pt {{\talpha \, \in \, \tDelta^\pm_\zero} \atop {i = 1, \dots, r;}}}  \cr
   \qquad   G_{\zero^{\,\uparrow}}(A)  \;\; := \;\,  \Big\langle \, x_{\talpha}(A) \,\;\Big|\; \talpha \in \tDelta_{\zero^{\,\uparrow}} \Big\rangle  \quad ,  \qquad
  G^\pm_{\zero^{\,\uparrow}}(A)  \;\; := \;\,  \Big\langle \, x_{\talpha}(A) \,\;\Big|\; \talpha \in \tDelta^\pm_{\zero^{\,\uparrow}} \Big\rangle  \qquad  \cr
   \qquad    G^\pm(A)  \; := \;  \Big\langle x_{\talpha}(A) \;\Big|\; \talpha \! \in \! \tDelta^\pm \Big\rangle  \; = \;  \big\langle G_\zero^\pm(A) \, , G_\uno^\pm(A) \big\rangle  \qquad  }  $$
\end{definition}

\smallskip

   By definition  $ T_V $,  $ G_0 $, $ G_0^\pm $,  $ G_\zero \, $,  $ G_\zero^\pm $,  $ G_{\zero^{\,\uparrow}} $,  $ G^\pm_{\zero^{\,\uparrow}} $,  $ G^\pm $,  $ G_V $  are subgroup functors of the functor  $ \rGL(V_\bk) \, $,  with obvious mutual inclusions.  As  $ \rGL(V_\bk) \, $  is a  {\sl sheaf\/}  (in the sense of category theory, cf.~\cite{fg2}, Appendix), these subfunctors are  {\sl presheaves}.  This implies that we can take their sheafification   --- with respect to the Zariski topology in  $ \salg_\bk $  ---   so next definition makes sense:

\vskip9pt

\begin{definition}  \label{def_Cartan-supergroup}
 \, Consider on  $ \, \salg_\bk \, $  the Zariski topology, with respect to which  $ \salg_\bk $  itself is a site.  We call  {\it Cartan (type) supergroup},  associated with  $ \fg $  and  $ V $,  the sheafification  $ \bG_V $  of  $ G_V $ (with respect to the Zariski topology).  In particular,  $ \bG_V \!\! : \! \salg_\bk \!\! \rightarrow \! \grps $  is a sheaf functor such that  $ \, \bG_V(A) \! = \! G_V(A) \, $  when  $ \, A \! \in \! \salg_\bk $  is  {\sl local}   --- i.e., it has a unique maximal homogeneous ideal.
                                                       \par
   Similarly, by  $ \bT_V $,  $ \bG_0 $, $ \bG_0^\pm $,  $ \bG_{\zero^{\,\uparrow}} $,  $ \bG^\pm_{\zero^{\,\uparrow}} $,  $ \bG_\zero \, $,  $ \bG_\zero^\pm $  and  $ \bG^\pm $  we shall denote the sheafification respectively of  $ T_V $,  $ G_0 $, $ G_0^\pm $,  $ G_{\zero^{\,\uparrow}} $,  $ G^\pm_{\zero^{\,\uparrow}} $,  $ G_\zero \, $,  $ G_\zero^\pm $  and  $ G^\pm $.
\end{definition}

\vskip5pt

\begin{remarks}  \label{rems-af-defs_Chev-sgrps}  {\ }
 \vskip4pt
   {\it (a)} \, The functors  $ \bG_V $  will eventually prove to be the ``affine algebraic supergroups of Cartan type'' which are our main object of interest.  Later on, we shall prove that they are indeed representable, so they are  {\sl affine supergroups},  and algebraic, with  $ \fg_V $  as tangent Lie superalgebra.
%
 \vskip3pt
   {\it (b)} \, By definition, the functors  $ T_V $,  $ G_0 $, $ G_0^\pm $,  $ G_\zero \, $,  $ G_\zero^\pm $,  $ G_{\zero^{\,\uparrow}} \, $,  $ G_{\zero^{\,\uparrow}}^\pm $  and their sheafifications are all supergroup functors which factor through  $ \, \alg = \alg_\bk \, $,  the category of commutative, unital  $ \bk $--algebras:  thus they pertain to the domain of ``classical'' (i.e.~``non super'') algebraic geometry.
 \vskip3pt
   {\it (c)} \, We shall see later  (cf.~Remak \ref{alt-def_G_V})  that the functor  $ \bG_V $  can also be defined by saying that  $ \bG_V(A) \, $,  for  $ \, A \in \salg_\bk \, $,  is the subgroup of  $ \rGL(V_\bk)(A) $  generated by  $ \bG_0(A) $  and the one-parameter subgroups  $ x_{\talpha}(A) $  with  $ \, \talpha \in \tDelta \setminus \tDelta_0 \, $.  A similar remark holds true for some of the subgroups of  $ \bG_V \, $.
 \vskip3pt
   {\it (d)} \, By definition  $ G_V $  and  $ \bG_V $   --- and all their supersubgroups considered above ---   are supersubgroups of  $ \rGL(V) \, $.  As the latter identifies with  $ \rGL(V_\bullet) $   --- cf.~Examples \ref{exs-supvecs}{\it (b)}  ---   we can also think of  $ G_V $  and  $ \bG_V $  (and their supersubgroups) as supersubgroups of  $ \rGL(V_\bullet) \, $.
%
\end{remarks}

\medskip

   In all our analysis hereafter, the key tool will be given by the commutation relations among the generators of our supergroups: these are detailed in the next lemma.  As a matter of notation, when  $ \varGamma $  is any group and  $ \, g, h \in \varGamma \, $  we denote by  $ \, (g,h) := g \, h \, g^{-1} \, h^{-1} \, $  their commutator in  $ \varGamma \, $.

\medskip
 \vskip9pt

\begin{lemma}  \label{comm_1-pssg} {\ }
 Let  $ \, A \in \salg_\bk \, $  be fixed.
 \vskip5pt
\noindent
 {\ }  (a) \,  Let  $ \, \alpha, \beta \in \Delta \, $  with  $ \, \alpha + \beta \not= 0 \, $;  set  $ \, \Delta_{\alpha,\beta} := \Delta \cap \big( \N_+ \alpha + \N_+ \beta \big) \, $.  Then, for all  $ \, 0 \leq i \leq \mu(\alpha) \, $,  $ \, 0 \leq j \leq \mu(\beta) \, $,  and  $ \, \gamma \in \Delta_{\alpha,\beta} \, $,  $ \, 0 \leq t \leq \mu(\gamma) \, $,  there exist  $ \; c_{\gamma;t}^{k,h} \in \Z \; $  such that
  $$  \big( x_{\alpha,k}(\mathbf{p}) \, , \, x_{\beta,h}(\mathbf{q}) \big)  \; = \;
{\textstyle \prod\limits_{{\gamma = r \alpha + s \beta \in \Delta_{\alpha,\beta}} \atop {0 \leq t \leq \mu(\gamma)}}}
 x_{\gamma,t}\big(c_{\gamma;t}^{k,h} \, \mathbf{p}^r \mathbf{q}^s\big)   \eqno (4.2)  $$
for any  $ \, \mathbf{p}, \mathbf{q} \in A_\zero \cup A_\uno \, $ (notation as in  Definition \ref{add_one-param-ssgrs}),  where the factors in right-hand side commu-te with one another.  In particular (notation of  Definition \ref{def_che-bas})  we have the following special cases:
%
%
 \vskip9pt
   (a.1) \,  assume  $ \; \alpha + \beta \not\in \big( \Delta \cup \{0\} \big) \, $,  and  $ \, \mathbf{p}, \mathbf{q} \in A_\zero \cup A_\uno \; $  (with suitable parity): then
  $$  \big( x_{\alpha,k}(\mathbf{p}) \, , \, x_{\beta,h}(\mathbf{q}) \big)  \,\; = \;\,  1  $$
 \vskip5pt
   (a.2) \,  assume  $ \; \alpha + \beta \in \Delta \, $,  $ \; \alpha \not\in \Delta_0 \, $,  $ \; \beta \not\in \Delta_0 \, $,  and  $ \, \mathbf{p}, \mathbf{q} \in A_\zero \cup A_\uno \; $  (with suitable parity):  then
  $$  \big( x_{\alpha,k}(\mathbf{p}) \, , \, x_{\beta,h}(\mathbf{q}) \big)  \,\; = \;\,  {\textstyle \prod_{\,t=1}^{\,\mu(\alpha+\beta)\,}} x_{\alpha + \beta, t} \Big(\! {(-1)}^{p(\mathbf{p}) p(\mathbf{q})} \, c_{\alpha,k}^{\,\beta,h}(t) \, \mathbf{p} \, \mathbf{q} \Big)  $$
where all factors in the right-hand product do commute with each other;
 \eject

   (a.3) \,  assume  $ \; \alpha + \beta \in \Delta \, $,  $ \; \alpha \in \Delta_0 \, $,  $ \; \beta \not\in \Delta_0 \, $,  and  $ \, \mathbf{p}, \mathbf{q} \in A_\zero \cup A_\uno \; $  (with suitable parity):  then
  $$  \big( x_{\alpha,k}(\mathbf{p}) \, , \, x_{\beta,h}(\mathbf{q}) \big)  \,\; = \;\,  x_{2 \alpha + \beta, \, t'} \big(\! \pm \mathbf{p}^2 \, \mathbf{q} \,\big) \cdot {\textstyle \prod_{\,t=1}^{\,\mu(\alpha + \beta)}} x_{\alpha + \beta, \, t} \big(\! \pm c_{\alpha,k}^{\,\beta,h}(t) \, \mathbf{p} \, \mathbf{q} \,\big)  $$
(for an index  $ t' $  given as in part  {\it (g)\/}  of  Definition \ref{def_che-bas}  if  $ \, 2\,\alpha \! + \! \beta \in \Delta \, $,  and  $ \, x_{2 \alpha + \beta, \, t'} \big(\! \pm \! \mathbf{p}^2 \mathbf{q} \big) \! := \! 1 \, $  if  $ \, 2\,\alpha \! + \! \beta \not\in \Delta \, $) where all factors on the right-hand side commute with each other.
 \vskip9pt
   (a.4) \,  assume  $ \; \alpha, \beta \in \Delta_0 \, $,  and  $ \, p, q \in A_\zero \; $:  then
  $$  \big( x_{\alpha,1}(p) \, , \, x_{\beta,1}(q) \big)  \,\; = \;\,
   {\textstyle \prod\limits_{{\gamma = r \alpha + s \beta \in \Delta_{\alpha,\beta}} \atop {0 \leq t \leq \mu(\gamma)}}} \!\!
   x_{\alpha + \beta, 1}\big( c_{\alpha,1}^{\,\beta,1}(1) \, p \, q \big)  \;\; \in \;\;  G_0(A)  $$
 \vskip5pt
\noindent
 {\ }  (b) \,  Let  $ \, \gamma \in \Delta_{-1} \, $,  let  $ \; 0 \leq j \leq \mu(-\gamma) = n \! - \! 1 \, $,  and let  $ \, \vartheta, \eta \in A_\uno \, $.  Then
  $$  \big( x_{\gamma,1}(\vartheta) \, , \, x_{-\gamma,j}(\eta) \big)  \,\; = \;\,  \big(\, 1 \! \mp \vartheta \, \eta \, H_{\sigma_\gamma(j)} \,\big)  \,\; = \;\,  h_{H_{\sigma_\gamma(j)}}\!\big( 1 \! \mp \vartheta \, \eta \big)  \;\; \in \;\;  G_0(A)  $$
 \vskip7pt
\noindent
 {\ }  (c)  Let  $ \; \alpha \in \Delta \, $,  $ \, 0 \leq i \leq \mu(\alpha) \, $,  $ \, H \in \fh_\Z \, $,  $ \, u \in U\big(A_\zero\big) \, $,  $ \, \bu \in \! A_\zero \cup A_\uno $  (with suitable parity).  Then
  $$  \hskip7pt   h_H(u) \; x_{\alpha,i}(\bu) \; {h_H(u)}^{-1}  \; = \;\,  x_{\alpha,i}\big( u^{\alpha(H)} \, \bu \big)  \;\; \in \;\;  G_{\!p(\alpha)}\!(A)  $$
where  $ \; p(\alpha) := \overline{s} \in \Z_2 \, $,  by definition, if and only if  $ \, \alpha \in \Delta_{\overline{s}} \; $.
\end{lemma}

\begin{proof}
 {\it (a)} \,  The proof follows from a direct analysis, through formal computations, just like in the classical case of Chevalley groups, which is treated in  \cite{st},  \S 3, Lemma 15.  We shall carry it on by looking at the general case, and later specializing to the special ones.
 \vskip5pt
   First of all, fix notation  $ \, X := \mathbf{p} \, X_{\alpha,k} \, $,  $ \, Y := \mathbf{q} \, X_{\beta,h} \, $.  Recall that any (additive) one-parameter supersubgroup can be expressed by a formal exponential: so  $ \; x_{\alpha,k}(\mathbf{p}) \, := \, {\textstyle \sum_{m=0}^{+\infty}} {(\mathbf{p} X_{\alpha,k})}^m \! \big/ m! \, := \, {\textstyle \sum_{r=0}^{+\infty}} \, X^r \big/ r! \, = \, \exp(X) \; $,  and  $ \; x_{\beta,h}(\mathbf{q}) \, := \, {\textstyle \sum_{s=0}^{+\infty}} \, X^s \! \big/ s! \, = \, \exp(Y) \; $.  Now, formal calculation gives
  $$  \displaylines{
   \big( x_{\alpha,k}(\mathbf{p}) \, , x_{\beta,h}(\mathbf{q}) \big)  =
 \Ad\big(x_{\alpha,k}(\mathbf{p})\!\big) \bigg( {\textstyle \sum\limits_{s=0}^{+\infty}} \, Y^s \!\big/ s! \bigg) \cdot {x_{\beta,h}(\mathbf{q})}^{-1}  \! =  {\textstyle \sum\limits_{s=0}^{+\infty}} {\Big(\! \Ad\big(x_{\alpha,k}(\mathbf{p})\big)(Y) \!\Big)}^s \!\! \Big/ \! s!  \; \cdot \;  {x_{\beta,h}(\mathbf{q})}^{-1}  \! =  \cr
   \hfill   = \;\,  {\textstyle \sum_{s=0}^{+\infty}} {\Big( \Ad\big(\exp(X)\big)(Y) \Big)}^s \! \Big/ s!  \; \cdot \;  {x_{\beta,h}(\mathbf{q})}^{-1}  \,\; = \;\,  {\textstyle \sum_{s=0}^{+\infty}} {\Big( \exp\big( \ad(X) \big)(Y) \Big)}^s \! \Big/ s!  \; \cdot \;  {x_{\beta,h}(\mathbf{q})}^{-1}  \cr
  }  $$
where in the last step we used the (formal) identity  $ \; \Ad \circ \exp = \exp \circ \, \ad \; $.
%
%
   Now, moving on we get
  $$  \exp\big(\ad(X)\big)(Y)  \, = \,  {\textstyle \sum_{r=0}^{+\infty}} \, {\ad(X)}^r(Y) \big/ r!  \, = \,  Y \! + [X,Y] + \! \big[X,[X,Y]\big] \big/ 2   \eqno (4.3)  $$
because  $ \, {\ad(X)}^r = 0 \, $  for all  $ \, r > 2 \, $  by  Proposition \ref{ad_square/cube}{\it (b)}.
 \vskip7pt
   As a consequence, if  $ \, \alpha + \beta \not\in \big( \Delta \cup \{0\} \big) \, $  we have  $ \, [X,Y] \in \fg_{\alpha + \beta}(A) = \{0\} \, $,  hence (4.3) reads  $ \; \exp\big(\ad(X)\big)(Y) = Y \; $.  Then the above analysis proves  {\it (a.1)},  since it yields
  $$  \big( x_{\alpha,k}(\mathbf{p}) \, , \, x_{\beta,h}(\mathbf{q}) \big)  \; = \;  {\textstyle \sum_{s=0}^{+\infty}} \, Y^s \! \Big/ s!  \, \cdot \,  {x_{\beta,h}(\mathbf{q})}^{-1}  \; = \;  {x_{\beta,h}(\mathbf{q})} \cdot {x_{\beta,h}(\mathbf{q})}^{-1}  \; = \;  1  $$
 \vskip3pt
   Now assume  $ \, \alpha + \beta \in \Delta \, $  but  $ \, \alpha \, , \beta \not\in \Delta_0 \; $.  Then  $ \, \big[X,[X,Y]\big] = 0 \, $  by  Proposition \ref{ad_square/cube}{\it (b)\/}  if  $ \, \alpha \in \Delta_\zero \, $,  and by  $ \, \mathbf{p}^2 = 0 \, $ if  $ \, \alpha \in \Delta_\uno \, $;  thus (4.3) reads  $ \; \exp\big(\ad(X)\big)(Y) = Y + [X,Y] \; $.  Similarly  $ \, \big[Y,[X,Y]\big] = 0 \, $,  so the summands  $ \, Y \, $  and  $ \, [X,Y] \, $  commute with each other: thus our analysis gives
  $$  \displaylines{
   \big( x_{\alpha,k}(\mathbf{p}) \, , \, x_{\beta,h}(\mathbf{q}) \big)  \; = \;  {\textstyle \sum_{s=0}^{+\infty}} \, {\big( Y + [X,Y] \big)}^s \! \Big/ s!  \, \cdot \,  {x_{\beta,h}(\mathbf{q})}^{-1}  \; = \;  \exp\big( Y + [X,Y] \big) \cdot {x_{\beta,h}(\mathbf{q})}^{-1}  \; =   \hfill  \cr
   \hfill   = \;  \exp(Y) \cdot \exp\big([X,Y]\big) \cdot {x_{\beta,h}(\mathbf{q})}^{-1}  \; = \;  \exp\big([X,Y]\big) \cdot \exp(Y) \cdot {x_{\beta,h}(\mathbf{q})}^{-1}  \; = \;  \exp\big([X,Y]\big)  }  $$
since  $ \; x_{\beta,h}(\mathbf{q}) = \exp(Y) \; $.  Now  $ \; [X,Y]  =  \big[\, \mathbf{p} \, X_{\alpha,k} \, , \, \mathbf{q} \, X_{\beta,h} \big]  =  {(-1)}^{p(\mathbf{p}) p(\mathbf{q})} \, {\textstyle \sum\limits_{\,t=1}^{\,\mu(\alpha + \beta)}} \!\!\! \mathbf{p} \, \mathbf{q} \, c_{\alpha,k}^{\,\beta,h}(t) \, X_{\alpha + \beta, \, t} \; $  by  Definition \ref{def_che-bas}{\it (f)},  and the summands in the last term all commute with each other.
 Indeed, in all cases except for  $ \, \fg = \tS(n) \, $,  $ \, \alpha + \beta = -\varepsilon_i \, $,  or  $ \, \fg = H(2\,r+1) \, $,  $ \, \alpha + \beta \in (2\,\N + 1)\,\delta \, $,  this holds because  Proposition \ref{ad_square/cube}{\it (a)\/}  give
%
%
 $ \; \big[ X_{\alpha + \beta, \, t'} \, , \, X_{\alpha + \beta, \, t''} \big] \in \big[ \fg_{\alpha + \beta} , \fg_{\alpha + \beta} \big] = \{0\} \; $.
 In the remaining cases instead, the root  $ \alpha + \beta \, $  is  {\sl odd},  hence either  $ \alpha $  or  $ \beta $  is odd as well, thus  $ \, \mathbf{p} \in A_\uno \, $  or  $ \, \mathbf{q} \in A_\uno \; $:  therefore $ \; \big[\, \mathbf{p} \, \mathbf{q} \, X_{\alpha + \beta, \, t'} \, , \, \mathbf{p} \, \mathbf{q} \, X_{\alpha + \beta, \, t''} \big] = 0 \; $  just because  $ \, {(\mathbf{p} \, \mathbf{q})}^2 = \pm \mathbf{p}^2 \, \mathbf{q}^2 = 0 \; $.  The outcome is
 \eject

  $$  \displaylines{
   \big( x_{\alpha,k}(\mathbf{p}) \, , \, x_{\beta,h}(\mathbf{q}) \big)  \; = \;  \exp\big([X,Y]\big)  \; = \;  \exp\Big( {(-1)}^{p(\mathbf{p}) p(\mathbf{q})} \, \mathbf{p} \, \mathbf{q} \cdot {\textstyle \sum_{\,t=1}^{\,\mu(\alpha + \beta)}} \, c_{\alpha,k}^{\beta,h}(t) \, X_{\alpha + \beta, \, t} \Big)  \; =   \hfill  \cr
   \hfill   =  {\textstyle \prod_{t=1}^{\mu(\alpha + \beta)}} \exp \! \Big(\! {(-1)}^{p(\mathbf{p}) p(\mathbf{q})} \, \mathbf{p} \, \mathbf{q} \cdot c_{\alpha,k}^{\beta,h}(t) \, X_{\alpha + \beta, \, t} \Big)  =  {\textstyle \prod_{t=1}^{\mu(\alpha + \beta)}} x_{\alpha + \beta, \, t}\Big(\! {(-1)}^{p(\mathbf{p}) p(\mathbf{q})} \, \mathbf{p} \, \mathbf{q} \cdot c_{\alpha,k}^{\,\beta,h}(t) \Big)   }  $$
with all factors in the last product which commute among themselves.  This proves  {\it (a.2)}.

\vskip7pt

   Finally, assume that  $ \, \alpha + \beta \in \Delta \, $  and  $ \, \alpha \in \Delta_0 \, $,  $ \, \beta \not\in \Delta_0 \, $   ({\sl N.B.:}  the case  $ \, \alpha \not\in \Delta_0 \, $,  $ \, \beta \in \Delta_0 \, $  is symmetric, hence we drop it).  Just like before,
%
%
we find  $ \, \big[ Y , [Y,X] \big] = 0 \, $;  therefore  $ \, \big[ Y , [X,Y] \big] = \pm \big[ Y , [Y,X] \big] = 0 \, $,  then also
 $ \; \big[ Y , \big[X,[X,Y]\big] \big] \, = \, \pm \big[ [\,Y,X] \, , [\,Y,X] \big] \pm \big[ X , \big[\, Y, [\,Y,X] \big] \big] = 0 \; $
--- by the super-Leibnitz' rule, and taking into account the identity  $ \, \big[ [\,Y,X] \, , [\,Y,X] \big] = 0 \, $  inside  $ \, \fg(A) := \fg_\zero \otimes_\bk A_\zero \oplus \fg_\uno \otimes_\bk A_\uno \, $,  which is a  {\sl Lie algebra}  ---   and finally (again by super-Leibnitz' rule)
 $ \; \big[ [X,Y] , \big[X,[X,Y]\big] \big] \, = \, \pm \big[ \big[ X , \big[ X , [X,Y] \big] \big] \, , Y \big] \pm \big[ X , \big[ Y , \big[ X, [X,Y] \big] \big] \big] = \pm [\,0\,,Y] \pm [X,0] = 0 \; $
by  Proposition \ref{ad_square/cube}{\it (a)},  for the first summand,  and by  $ \, \big[ Y , \big[X,[X,Y]\big] \big] = 0 \, $  just proved, for the second.  This means that the three summands in right-hand side of (4.3) do commute with each other; thus
  $$  \displaylines{
   \big( x_{\alpha,k}(\mathbf{p}) \, , \, x_{\beta,h}(\mathbf{q}) \big)  \; = \;  \exp \Big( Y + [X,Y] + \big[ X, [X,Y] \big] \big/ 2 \Big) \cdot {x_{\beta,h}(\mathbf{q})}^{-1}  \; =   \hfill  \cr
   \hfill \qquad   = \;  \exp\Big( \big[X,[X,Y]\big] \big/ 2 \Big) \cdot \exp\big([X,Y]\big) \cdot \exp(Y) \cdot {x_{\beta,h}(\mathbf{q})}^{-1}  \; = \;  \exp\Big( \big[X,[X,Y]\big] \big/ 2 \Big) \cdot \exp\big([X,Y]\big)  }  $$
because  $ \; x_{\beta,h}(\mathbf{q}) = \exp(Y) \; $.  As before,
%
%
 $ \; [X,Y]  \, = \,  {(-1)}^{p(\mathbf{p}) p(\mathbf{q})} \, \mathbf{p} \, \mathbf{q} \cdot {\textstyle \sum_{\,t=1}^{\,\mu(\alpha + \beta)}} \, c_{\alpha,k}^{\,\beta,h}(t) \, X_{\alpha + \beta, \, t} \; $
with the summands in the last sum which commute with each other; similarly, we expand  $ \, \big[X,[X,Y]\big] \, $  as
%
%
 $ \; \big[X,[X,Y]\big]  \, = \,  \big[\, \mathbf{p} \, X_{\alpha,1} , [\, \mathbf{p} \, X_{\alpha,1} , \mathbf{q} \, Y_{\beta,h}] \big]  \, = \,  {(-1)}^{p(\mathbf{p}) (p(\mathbf{p}) + p(\mathbf{q}))} \, \mathbf{p}^2 \, \mathbf{q} \, \big[\, X_{\alpha,k} \, , [\, X_{\alpha,k} \, , \, Y_{\beta,h}] \big] \; $.
Now, if  $ \, \alpha \not\in \Delta_0 \, $  Proposition \ref{ad_square/cube}{\it (b)\/}  gives  $ \, \big[\, X_{\alpha,k} \, , [\, X_{\alpha,k} \, , \, Y_{\beta,h}] \big] = 0 \, $,  hence  $ \; \big[X,[X,Y]\big] = 0 \; $.  If instead  $ \, \alpha \in \Delta_0 \, $,  then  ($ \, k = 1 \, $  and)
%
%
 either  $ \; \big[X,[X,Y]\big]  \, = \,  \mathbf{p}^2 \, \mathbf{q} \, \big[\, X_{\alpha,1} , [\, X_{\alpha,1} , \, Y_{\beta,1}] \big]  \, = \,  \pm \, \mathbf{p}^2 \, \mathbf{q} \, 2 \, X_{2 \alpha + \beta, \, t'} \; $
or  $ \, \big[X,[X,Y]\big] \, = \, 0 \, $,  by  Definition \ref{def_che-bas}{\it (g)}   --- for some  $ t' $  as therein.  Note also that  $ \, \big[X,[X,Y]\big] \, $  commutes with each summand in the expansion
 $ \; [X,Y]  \, = \,  {(-1)}^{p(\mathbf{p}) p(\mathbf{q})} \, {\textstyle \sum_{\,t=1}^{\,\mu(\alpha + \beta)}} \, \mathbf{p} \, \mathbf{q} \cdot c_{\alpha,k}^{\,\beta,h}(t) \, X_{\alpha + \beta, \, t} \; $;
indeed, this occurs because  $ \; \big[ \big[X,[X,Y]\big] \, , X_{\alpha + \beta, \, t} \big] \in \fg_{3 \, \alpha + 2 \, \beta}(A) \; $,  and direct (straightforward) inspection shows that  $ \, 3 \, \alpha + 2 \, \beta \not\in \Delta \, $  (having  $ \, \alpha + \beta \in \Delta \, $  and  $ \, \alpha \in \Delta_0 \, $,  by assumption).  So we find
  $$  \displaylines{
   \big( x_{\alpha,1}(\mathbf{p}) \, , \, x_{\beta,h}(\mathbf{q}) \big)  \; = \;  \exp\Big( \big[X,[X,Y]\big] \big/ 2 \Big) \cdot \exp\big([X,Y]\big)  \; =   \hfill  \cr
   \, \hfill   =  \exp\big( \pm \mathbf{p}^2 \, \mathbf{q} \, X_{2 \alpha + \beta, \, t'} \big) {\textstyle \prod\limits_{\,t=1}^{\,\mu(\alpha + \beta)}} \hskip-9pt \exp\big( \mathbf{p} \, \mathbf{q} \, c_{\alpha,1}^{\,\beta,h}(t) \, X_{\alpha + \beta, \, t} \big)  =  x_{2 \alpha + \beta, \, t'} \! \big(\! \pm \mathbf{p}^2 \, \mathbf{q} \big) {\textstyle \prod\limits_{\,t=1}^{\,\mu(\alpha + \beta)}} \hskip-8pt x_{\alpha + \beta, \, t} \big( \pm c_{\alpha,1}^{\,\beta,h}(t) \, \mathbf{p} \, \mathbf{q} \big)  }  $$
(with  $ \, x_{2 \alpha + \beta, \, t'\!} \big(\! \pm \! \mathbf{p}^2 \mathbf{q} \big) \! := \! 1 \, $  if  $ \, 2\,\alpha \! + \! \beta \not\in \Delta \, $)  with all factors pairwise commuting, so  {\it (a.3)\/}  is proved.
 \vskip7pt
   The very last case to consider is  {\it (a.4)},  which is a classical result: see  \cite{st},  \S 3, Lemma 15.
 \vskip11pt
   {\it (b)} \,  The same arguments used for  {\it (a)\/}  give also
 $ \; \big( x_{\gamma,1}(\vartheta) \, , \, x_{-\gamma,j}(\eta) \big)  \; = \,  1 - \, \vartheta \, \eta \, \big[ X_{\gamma,1} \, , X_{-\gamma,j} \big] \; $.
 Then  Definition \ref{def_che-bas}{\it (d)\/}  gives  $ \; \big[ X_{\gamma,1} \, , X_{-\gamma,j} \big]  \, = \,  \pm H_{\sigma_\gamma(j)} \; $.  Plugging this into the previous formula, and noting that  $ \, {(\vartheta \, \eta)}^n = 0 \, $  for all  $ \, n > 1 \, $,  we get exactly  {\it (b)}.
 \vskip11pt
   {\it (c)} \, Let  $ \, v_\mu \in M_\mu := \big( M \cap V_\mu \big) \, $  be any weight vector in the admissible lattice  $ M $  of  $ V $  used to define  $ G_V \, $.  We show now that  $ \; h_H(u) \, x_{\alpha,i}(\bu) \, {h_H(u)}^{-1} \; $  and  $ \; x_{\alpha,i}\big( u^{\alpha(H)} \, \bu \big) \; $  acts on the same way on  $ v_\mu \, $:  taking  $ \mu $  and  $ v_\mu $  arbitrarily, this is enough to prove claim  {\it (c)}.  Direct computation gives
  $$  \displaylines{
   \big( h_H(u) \, x_{\alpha,i}(\bu) \, {h_H(u)}^{-1} \big)(v_\mu)  \,\; = \;\,  u^{\mu(H)} \cdot h_H(u) \big(\, {\textstyle \sum_{n=0}^{+\infty}}\, {\ad(\bu \, X_{\alpha,i})}^n(v_\mu) \big/ n! \,\big)  \,\; =  \hfill  \cr
   = \;\,  u^{-\mu(H)} \cdot {\textstyle \sum_{n=0}^{+\infty}} \, {\textstyle \frac{\,1}{\,n!}} \; h_H(u) \big( {\ad(\bu \, X_{\alpha,i})}^n(v_\mu) \big)  \; = \,  u^{-\mu(H)} {\textstyle \sum_{n=0}^{+\infty}} \, {\textstyle \frac{\,1\,}{\,n!\,}} \; u^{(\mu + n \alpha)(H)} {\ad(\bu \, X_{\alpha,i})}^n(v_\mu)  \, =   \hfill  \cr
   \hfill   = \,  u^{-\mu(H)} \, u^{\mu(H)} {\textstyle \sum_{n=0}^{+\infty}} \, {\textstyle \frac{\,1}{\,n!}} {\ad \big( u^{\alpha(H)} \bu \, X_{\alpha,i} \big)}^{\!n}(v_\mu)  \, = \,  \exp \! \big( {\ad}\big( u^{\alpha(H)} \bu \, X_{\alpha,i} \big) \big)(v_\mu)  \, = \,  x_{\alpha,i}\big( u^{\alpha(H)} \bu \big)(v_\mu)  }  $$
which is exactly what we needed.
\end{proof}

\medskip

  \subsection{The even part  $ \bG_\zero \, $  of  $ \bG_V $}  \label{even-part}

\smallskip

   {\ } \quad   Our definition of the supergroup  $ \bG_V $  does not imply (at first sight) that  $ \bG_V $  be  {\it representable}.
 In order to prove that, we need to know how the ``even part''  $ \bG_\zero $  of  $ \bG_V $  looks like.  
%
 \eject

\begin{proposition}  \label{bG_0-Ch_V}
 \; The functor  $ \, \bG_0 $  is representable, hence   --- as it factors through  $ \alg_\bk $  ---   it is an affine group-scheme; moreover, it is also algebraic.  More precisely, we have natural isomor\-phisms   $ \, \bG_0 \cong \mathbf{Ch}_V \, $,  where  $ \, \mathbf{Ch}_V : \alg_\bk \longrightarrow \grps \, $  is the standard (affine, algebraic) group functor associated with  $ \fg_0 $  and with the  $ \fg_0 $--module  $ V $  by the classical Chevalley-Demazure construction.
\end{proposition}

\begin{proof}
 This is just a consequence of the very definitions.  Indeed, in terms of category theory (cf.~for instance  \cite{vst}),  the category  $ \, \alg = \alg_\bk \, $ is a site and both the functors  $ \bG_0 $  and  $ \mathbf{Ch}_V $  are sheaves.  Moreover, by construction there exist natural transformations  $ \, G_0 \,{\buildrel \alpha \over \longrightarrow}\; \bG_0 \, $,  $ \, G_0 \,{\buildrel \beta \over \longrightarrow}\; \mathbf{Ch}_V \, $,  and  $ G_0 $  coincides with  $ \mathbf{Ch}_V $  via  $ \beta $  on local algebras, that is  $ \, \beta\big(G_0(R)\big) = \mathbf{Ch}_V(R) \, $  for any  $ \, R \in \text{\it Ob}\alg \, $  which is local (this follows from \S 3.5.3 in  \cite{de2}  and  Corollaire 5.7.6 in  \cite{de1}).  As  $ \bG_0 $  by definition is the sheafification of  $ G_0 \, $,  the universal property characterizing the sheafification yields  $ \, \bG_0 \, \cong \, \mathbf{Ch}_V \, $.
\end{proof}

\vskip7pt

   As a second step, we have the following result for some (classical) subgroup functors of  $ \bG_\zero \, $:

\vskip9pt

\begin{proposition}  \label{descr-G_zero}
 \; Fix any total order  $ \, \preceq \, $  in  $ \, \tDelta_{\zero^{\,\uparrow}} \, $.
%
%
 Then we have:
 \vskip7pt
   {\it (a)}  $ \;\; G_{\zero^{\,\uparrow}}(A) = \prod_{\talpha \in \tDelta_{\zero^{\,\uparrow}}} x_{\talpha}(A) \;\; $  for all  $ \, A \in \salg_\bk \, $,  the product being ordered according to  $ \preceq \, $;
 \vskip5pt
   {\it (b)}  \quad  $ G_{\zero^{\,\uparrow}} \! \trianglelefteq G_\zero \;\; $  and  $ \;\; \bG_{\zero^{\,\uparrow}} \! \trianglelefteq \bG_\zero \; $,  \,
 where  $ \, \trianglelefteq \, $  stands for ``normal subgroup functor'';
 \vskip5pt
   {\it (c)}  \quad  $ G_\zero = G_0 \cdot G_{\zero^{\,\uparrow}} = G_{\zero^{\,\uparrow}} \cdot G_0 \;\; $  and  $ \;\; \bG_\zero = \bG_0 \cdot \bG_{\zero^{\,\uparrow}} = \bG_{\zero^{\,\uparrow}} \cdot \bG_0 \; $.
   In particular,  $ \, \bG_\zero \, $  is a  {\sl closed}  subgroup of  $ \, \rGL(V) \, $,  hence it is in turn (on its own) an affine algebraic group.
 \vskip5pt
   {\it (d)}  \;  the group functors  $ \, G_{\zero^{\,\uparrow}} \, $  and  $ \, \bG_{\zero^{\,\uparrow}} \, $  are both unipotent.
\end{proposition}

\begin{proof}
 {\it (a)} \,  The formulas for commutators in  Lemma \ref{comm_1-pssg}  imply that any (unordered) product of several factors  $ x_{\talpha}(t_{\talpha}) $  with  $ \, \talpha \in \tDelta_{\zero^{\,\uparrow}} \, $  can be reordered.  In fact, whenever we have a couple of consecutive unordered factors, say  $ \; x_{\talpha_1}\!(t_{\talpha_1}) \; x_{\talpha_2}\!(t_{\alpha_2}) \; $,  \, we can re-write their product as
 \vskip-7pt
  $$  x_{\talpha_1}\!(t_{\talpha_1}) \; x_{\talpha_2}\!(t_{\alpha_2})  \; = \;  \big( x_{\talpha_1}\!(t_{\talpha_1}) \, , \, x_{\talpha_2}\!(t_{\alpha_2}) \big) \cdot x_{\talpha_2}\!(t_{\talpha_2}) \; x_{\talpha_1}\!(t_{\alpha_1})  $$
Then formula (4.2) for  $ \, \big( x_{\talpha_1}\!(t_{\talpha_1}) \, , \, x_{\talpha_2}\!(t_{\alpha_2}) \big) \, $  tells us that the commutator is either 1, or a product of several  $ x_{\talpha}\!(t_{\talpha}) $  such that
 $ \, \text{\it ht} \big( \pi \big( \talpha \big) \big)  \, \gneqq \,  \text{\it ht} \big( \pi \big( \talpha_1 \big) \big) \, $,  $ \, \text{\it ht} \big( \pi \big( \talpha \big) \big)  \, \gneqq \,  \text{\it ht} \big( \pi \big( \talpha_2 \big) \big) \, $
(cf.~\S \ref{def_Cartan-subalg_roots_etc}).  Therefore, we can iterate this process in order to commute all unordered pairs of factors, up to (possibly) introducing new factors.  However, the above shows that these new factors, if any, will be attached to roots with greater height: as the height is bounded from above, we shall end up with trivial new factors, i.e.~after finitely many steps all pairs can be reordered without introducing new factors.
                                                                    \par
   As a consequence, the multiplication map  $ \; \bigtimes_{\;\talpha \in \tDelta_{\zero^{\,\uparrow}}} x_{\talpha}(A) \, \relbar\joinrel\longrightarrow \, G_V(A) \; $   --- the product on left-hand side being ordered ---   yields a surjection onto  $ \, G_{\zero^{\,\uparrow}}(A) \, $,  realized as  $ \; G_{\zero^{\,\uparrow}}(A) \, = \, \prod_{\talpha \in \tDelta_{\zero^{\,\uparrow}}} \! x_{\talpha}(A) \; $.

\smallskip

   {\it (b)} \,  Again  Lemma \ref{comm_1-pssg}  gives that  $ \, \prod_{\talpha \in \tDelta_{\zero^{\,\uparrow}}} x_{\talpha}(A) \, $  is normalized by  $ G_0 \, $;  then by  {\it (a)},  we deduce that  $ \; G_{\zero^{\,\uparrow}} \trianglelefteq G_\zero \; $,  \, whence  $ \; \bG_{\zero^{\,\uparrow}} \trianglelefteq \bG_\zero \; $  follows too.

\smallskip

   {\it (c)} \,  This follows easily by construction, namely from  $ \, G_0 \leq G_\zero \, $,  $ \, \bG_0 \leq \bG_\zero \, $  and  $ \; G_\zero = \big\langle \, G_0 \, , \, G_{\zero^{\,\uparrow}} \big\rangle \; $,  \, along with  {\it (b)}\,.  By classical theory of algebraic groups, as  $ \bG_0 $  and  $ \bG_{\zero^{\,\uparrow}} $  are closed subgroups of  $ \rGL(V) $  and  $ \bG_0 $  normalizes  $ \bG_{\zero^{\,\uparrow}} $  one argues that the last part of claim  {\it (c)\/}  holds too.

\smallskip

   {\it (d)} \,  This follows from the classical theory of affine group-schemes, because we have embeddings  $ \; G_{\zero^{\,\uparrow}} \! \leq \bG_{\zero^{\,\uparrow}} \! \leq \rGL(V) \, $,  \, and the tangent Lie algebra of  $ \, \bG_{\zero^{\,\uparrow}} \, $,  i.e.~$ \, \Lie\,\big(\bG_{\zero^{\,\uparrow}}\big) = \fg_{\zero^{\,\uparrow}} \, $,  \, is nilpotent.
\end{proof}

\vskip6pt

   The same (type of) arguments proves also the following:

\vskip9pt

\begin{proposition}  \label{descr-G_zero^pm}
 \; Fix any total order  $ \, \preceq \, $  in  $ \, \tDelta^\pm_{\zero^{\,\uparrow}} \, $.
%
%
 Then we have:
 \vskip7pt
   {\it (a)}  $ \;\; G^\pm_{\zero^{\,\uparrow}}(A) = \prod_{\talpha \in \tDelta^\pm_{\zero^{\,\uparrow}}} x_{\talpha}(A) \;\; $  for all  $ \, A \in \salg_\bk \; $,  the product being ordered according to  $ \preceq \, $;
 \vskip5pt
   {\it (b)}  \quad  $ G^\pm_{\zero^{\,\uparrow}} \! \trianglelefteq G^\pm_\zero \;\; $  and  $ \;\; \bG^\pm_{\zero^{\,\uparrow}} \! \trianglelefteq \bG^\pm_\zero \; $.
%
 \vskip5pt
   {\it (c)}  \quad  $ G^\pm_\zero = G^\pm_0 \cdot G^\pm_{\zero^{\,\uparrow}} \;\; $  and  $ \;\; \bG^\pm_\zero = \bG^\pm_0 \cdot \bG^\pm_{\zero^{\,\uparrow}} \; $.
%
 \vskip5pt
   {\it (d)}  \;  the group functors  $ \, G^\pm_{\zero^{\,\uparrow}} \, $  and  $ \, \bG^\pm_{\zero^{\,\uparrow}} \, $  are both unipotent.
\end{proposition}

\vskip9pt

   Proposition \ref{descr-G_zero}  and  Proposition \ref{descr-G_zero^pm}  can also be improved as follows:
 \eject

\begin{proposition} \label{semi-direct-G_zero}
 The group product yields group-functor isomorphisms
%
%
  $$  G_\zero \, \cong \, G_0 \ltimes G_{\zero^{\,\uparrow}}  \; ,  \quad  G^\pm_\zero \, \cong \, G^\pm_0 \ltimes G^\pm_{\zero^{\,\uparrow}}
   \qquad  \text{and}  \qquad
      \bG_\zero \, \cong \, \bG_0 \ltimes \bG_{\zero^{\,\uparrow}}  \; ,  \quad  \bG^\pm_\zero \, \cong \, \bG^\pm_0 \ltimes \bG^\pm_{\zero^{\,\uparrow}} \; .  $$
\end{proposition}

\begin{proof}
 The right-hand side pair of isomorphisms clearly follows from the left-hand side one.  As for the latter, we have to prove that  $ \; G_\zero(A) \, \cong \, G_0(A) \ltimes G_{\zero^{\,\uparrow}}\!(A) \; $  and  $ \; G^\pm_\zero(A) \, \cong \, G^\pm_0(A) \ltimes G^\pm_{\zero^{\,\uparrow}}\!(A) \; $  for every  $ \, A \in \salg_\bk \, $,  and also to show that these isomorphisms are functorial in  $ A \, $:  this second part will be trivial, so we cope just with the first one.  Actually, we prove  $ \; G_\zero(A) \, \cong \, G_0(A) \ltimes G_{\zero^{\,\uparrow}}\!(A) \; $  only, for the proof of  $ \; G^\pm_\zero(A) \, \cong \, G^\pm_0(A) \ltimes G^\pm_{\zero^{\,\uparrow}}\!(A) \; $  is quite the same.

\smallskip

   For any  $ \, A \in \salg_\bk \, $,  we know by  Proposition \ref{descr-G_zero}  that  $ \, G_0(A) \leq G_\zero(A) \, $,  $ \, G_{\zero^{\,\uparrow}}\!(A) \trianglelefteq G_0(A) \, $  and  $ \; G_\zero(A) \, = \, G_0(A) \cdot G_{\zero^{\,\uparrow}}\!(A) \; $.  Thus we are only left to prove  $ \; G_0(A) \, \bigcap \, G_{\zero^{\,\uparrow}}\!(A) \, = \big\{ e_{{}_{G_V}} \!\big\} \; $.

\medskip

   Let  $ \, A \in \salg_\bk \, $,  and let  $ \, g \in G_0(A) \, \bigcap \, G_{\zero^{\,\uparrow}}\!(A) \, $:  then  $ \, g = g_0 \in G_0(A) \, $  and  $ \, g = g_\uparrow \in G_{\zero^{\,\uparrow}}\!(A) \, $,  in particular  $ \, g_0 = g_\uparrow \, $.  Now let  $ V $  be the  $ \fg $--module  we use to define  $ G_V \, $,  $ \bG_V \, $,  etc., splitting as  $ \, V = \oplus_\mu V_\mu \, $  into direct sum of weight spaces.  All root vectors of  $ \fg $  map weight spaces into weight spaces, namely  $ \, X_{\teta}.V_\mu \subseteq V_{\mu + \eta} \, $  if  $ \, X_{\teta} \in \fg_\eta $  (for each root  $ \eta $  and every weight  $ \mu \, $).  This implies that, for all weights  $ \mu $  and  $ \, v_\mu \in V_\mu(A) \, $,  $ \, H \in \fh_\Z \, $,  $ \, \talpha \in \tDelta \, $  and  $ \, \alpha := \pi\big(\talpha\big) \, $,  one has (notation of \S \ref{def_Cartan-subalg_roots_etc})
  $$  \hskip7pt   h_H(A)\,.\,v_\mu \, \in \, A_\zero \, v_\mu \subseteq V_\mu(A)  \; ,  \qquad  x_{\talpha}(A)\,.\,v_\mu \, \in \, v_\mu + \big( {\textstyle \bigoplus}_{n \in \N_+} V_{\mu + n \alpha}(A) \big)   \eqno (4.4)  $$
%
   \indent
   Now, by definition,  $ G_0(A) $  is generated by all the  $ h_H(A) $  and all the  $ x_{\talpha}(A) $ with $ \, \talpha \in \tDelta_0 \, $;  similarly,  $ G_{\zero^{\,\uparrow}}(A) $  is generated by all the  $ x_{\talpha}(A) $  with $ \, \talpha \in \tDelta_{\zero^{\,\uparrow}} \, $.  This together with (4.4) implies
  $$  \hskip3pt   g_0\,.\,v_\mu \, \in \, {\textstyle \bigoplus_{\beta \in \N \tDelta_0}} V_{\mu + \beta}(A)  \;\; ,  \qquad  g_\uparrow\,.\,v_\mu \, \in \, v_\mu + \big( {\textstyle \bigoplus_{\gamma \in \N \tDelta_{\zero^{\,\uparrow}} \setminus \{0\}}} V_{\mu + \gamma}(A) \big)   \eqno (4.5)  $$
for any weight  $ \mu $  and any  $ \, v_\mu \in V_\mu \, $,  where  $ \, \N \tDelta_0 \, $  and  $ \, \N \tDelta_{\zero^{\,\uparrow}} \, $  are the  $ \N $--span  of  $ \tDelta_0 $  and of  $ \tDelta_{\zero^{\,\uparrow}} $  respectively.  Definitions give also  $ \, \N \tDelta_0 \! \cap \N \tDelta_{\zero^{\,\uparrow}} = \{0\} \, $:  therefore, from (4.5) and  $ \, g_0 = g_\uparrow \, $  we infer that  $ \; g_0\,.\,v_\mu = v_\mu = g_\uparrow\,.\,v_\mu  \; $.  Since  $ \mu $  and  $ \, v_\mu \in V_\mu(A) \, $  were arbitrarily chosen, and since  $ G_V(A) $  acts faithfully on  $ V(A) \, $,  we eventually conclude that  $ \; g_0 = e_{{}_{G_V}} \! = g_\uparrow \; $.
\end{proof}

\bigskip

   Like in the classical case of Chevalley groups, one has also the following auxiliary result:

\medskip

\begin{lemma}  \label{closed_roots}
 Let  $ \, \widetilde{\mathcal{S}} \subseteq \tDelta \, $  and  $ \, \mathcal{S} := \Big\{ \pi\big(\talpha\big) \,\Big|\, \talpha \in \widetilde{\mathcal{S}} \,\Big\}_{{}_{\phantom{|}}} $  (cf.~\S \ref{def_Cartan-subalg_roots_etc}).  Assume that  $ \mathcal{S} $  is  {\sl closed},  i.e.~$ \, \alpha, \beta \in \mathcal{S} \, $  and  $ \, \alpha + \beta \in \Delta \, $  imply that  $ \, \alpha + \beta \in \mathcal{S} \, $;  assume also that  $ \, \alpha \in \mathcal{S} \, $  implies  $ \, -\alpha \not\in \mathcal{S} \, $,  let  $ \, G_{\widetilde{\mathcal{S}}} := \Big\langle\, x_{\talpha} \,\Big|\, \talpha \in \widetilde{\mathcal{S}} \;\Big\rangle \, $  be the full subfunctor of  $ \, G_V $  generated by the one-parameter subgroups indexed by the elements in  $ \widetilde{\mathcal{S}} $,  and let  $ \, \bG_{\widetilde{\mathcal{S}}} \, $  be the sheafification of  $ \, G_{\widetilde{\mathcal{S}}} \, $.
                                                      \par
   For any total order in  $ \widetilde{\mathcal{S}} $,  the group product yields scheme isomorphisms
 $ \; \bigtimes_{\;\talpha \in \widetilde{\mathcal{S}}} \; x_{\talpha} \, \cong \, G_{\widetilde{\mathcal{S}}} \; $,
 $ \; \bigtimes_{\;\talpha \in \widetilde{\mathcal{S}}} \; x_{\talpha} \, \cong \, \bG_{\widetilde{\mathcal{S}}} \; $,  \,
where the direct products on the left-hand side are ordered ones.
                                                      \par
   In particular, one has  $ \; G_{\widetilde{\mathcal{S}}} \, = \bG_{\widetilde{\mathcal{S}}} \, \cong \mathbb{A}^{s_\zero | s_\uno} \, $  as superschemes, where  $ \; s_{\overline{z}} \, := \Big|\; \widetilde{\mathcal{S}} \,\bigcap\, \tDelta_{\overline{z}} \;\Big| \; $.
\end{lemma}

\begin{proof}
 Each one-parameter supersubgroup  $ \, x_{\talpha} \, $  is a representable supergroup, so (as a superscheme) it is a sheaf.  Any direct product of sheaves is itself a sheaf, so the left-hand side isomorphisms in the claim, once proved, implies that  $ \, G_{\widetilde{\mathcal{S}}} \, $  is already a sheaf, so it coincides with  $ \, \bG_{\widetilde{\mathcal{S}}} \, $.  Also, the superscheme  $ \, x_{\teta} \, $  is isomorphic to  $ \, \mathbb{A}^{1|0} \, $  or  $ \, \mathbb{A}^{0|1} \, $  according to whether  $ \, \pi\big(\, \teta \,\big) \, $  is even or odd  (cf.~Proposition \ref{hopf-alg}{\it (b)\/}),  so  $ \; \bigtimes_{\;\talpha \in \widetilde{\mathcal{S}}} \; x_{\talpha} \, \cong \, \mathbb{A}^{s_\zero | s_\uno} \; $  is clear.  So we are left to prove the claim for  $ \, G_{\widetilde{\mathcal{S}}} \, $.
                                                                 \par
   Our task is to show that, for any  $ \, A \in \salg_\bk \, $,  the product map  $ \;\; \bigtimes_{\;\talpha \in \widetilde{\mathcal{S}}} \, x_{\talpha}(A)  \, \longrightarrow \, G_{\widetilde{\mathcal{S}}}(A) \;\; $  in  $ \, G_{\widetilde{\mathcal{S}}} \, $  is a bijection: i.e., every  $ \, g \in G_{\widetilde{\mathcal{S}}} \, $  admits a unique factorization as an ordered product  $ \; g = \prod_{\talpha \in \widetilde{\mathcal{S}}} \, x_{\talpha}(\bt_{\talpha}) \; $ for some  $ \, \bt_{\talpha} \in A_\zero \cup A_\uno \; $.  This result can be found via the classical argument   --- cf.~\cite{st},  \S 3, pp.~24--25 ---   which now works again using  Lemma \ref{comm_1-pssg}  as the basic ingredient.
\end{proof}

\medskip

   A direct application of the previous lemma is the following
($ \, \widetilde{\mathcal{S}} \in \Big\{\, \tDelta_{\zero^{\,\uparrow}} \, , \, \tDelta^\pm_\zero \, , \, \tDelta^\pm_0 \, , \, \tDelta^\pm_{\zero^{\,\uparrow}} \, , \, \tDelta^\pm \Big\} $):

\medskip

\begin{proposition}  \label{factorization}
 Fix any total order in  $ \tDelta_{\zero^{\,\uparrow}} \, $,  in  $ \tDelta_\zero^\pm \, $,  in  $ \tDelta_0^\pm \, $,  in  $ \tDelta_{\zero^{\,\uparrow}}^\pm \, $,  in  $ \tDelta^\pm \, $.  Then the group product yields scheme isomorphisms
  $$  \displaylines{
   \bigtimes_{\;\talpha \in \tDelta_{\zero^{\,\uparrow}}} x_{\talpha}  \; \cong \;  G_{\zero^{\,\uparrow}}   \; \cong \;  \bG_{\zero^{\,\uparrow}}
\;\, ,  \hskip17pt
   \bigtimes_{\;\talpha \in \tDelta^\pm_\zero} \, x_{\talpha}  \; \cong \;  G^\pm_\zero  \; \cong \;  \bG^\pm_\zero
\;\, ,  \hskip17pt
   \bigtimes_{\;\talpha \in \tDelta^\pm_0} x_{\talpha}  \; \cong \;  G^\pm_0  \; \cong \;  \bG^\pm_0  \cr
   \bigtimes_{\;\talpha \in \tDelta^\pm_{\zero^{\,\uparrow}}} x_{\talpha}  \; \cong \;  G^\pm_{\zero^{\,\uparrow}}  \; \cong \;  \bG^\pm_{\zero^{\,\uparrow}}
\;\, ,  \hskip37pt
   \bigtimes_{\;\talpha \in \tDelta^\pm} x_{\talpha}  \; \cong \;  G^\pm  \; \cong \;  \bG^\pm  }  $$
where the direct product on the left is always ordered according to the fixed total order.  In particular
%
%
%
  $$  \displaylines{
   G_{\zero^{\,\uparrow}} \, \cong \, \bG_{\zero^{\,\uparrow}} \, \cong \, \mathbb{A}^{N_{\zero^{\,\uparrow}}|\, 0} \;\; ,  \qquad
   G^\pm_\zero \, \cong \, \bG^\pm_\zero \, \cong \, \mathbb{A}^{N^\pm_\zero|\, 0} \;\; ,  \qquad
   G^\pm_0 \, \cong \, \bG^\pm_0 \, \cong \, \mathbb{A}^{N_0|\, 0}  \cr
   G^\pm_{\zero^{\,\uparrow}} \, \cong \, \bG^\pm_{\zero^{\,\uparrow}} \, \cong \, \mathbb{A}^{N^\pm_{\zero^{\,\uparrow}}|\, 0} \;\; ,  \qquad \qquad
   G^\pm \, \cong \, \bG^\pm \, \cong \, \mathbb{A}^{N^\pm|\, 0}  }  $$
as superschemes, where  $ \, N_{\zero^{\,\uparrow}} \! := \big| \tDelta_{\zero^{\,\uparrow}} \big| \, $,  $ \, N^\pm_\zero \! := \big| \tDelta^\pm_\zero \big| \, $,  $ \, N_0 \! := \big| \tDelta_0 \big| \, $,  $ \, N^\pm_{\zero^{\,\uparrow}} \! := \big| \tDelta^\pm_{\zero^{\,\uparrow}} \big| \, $,  $ \, N^\pm \! := \big| \tDelta^\pm \big| \; $.
\end{proposition}


\vskip19pt

  \subsection{The functors  $ \bG_V $  as affine algebraic supergroups}  \label{cartan-sgroup_aff-alg}

\smallskip

   {\ } \quad   In this subsection we shall show that the supergroup functors  $ \bG_V $  defined in  subsection \ref{const-car-sgroups}  are (the functors of points of) affine supergroups, and also algebraic.  We need some more definitions:

\vskip15pt

\begin{definition}  \label{subgrps}
 For any  $ \, A \in \salg_\bk \, $,  we define the subsets of  $ G(A) $
 \vskip-11pt
%
%
  $$  G_\uno(A)  \, := \,  {\Big\{\, {\textstyle \prod_{\,i=1}^{\,m}} \,
x_{\tgamma_i}(\vartheta_i) \,\Big\}}_{m \in \N \, , \, \tgamma_i \in \tDelta_\uno \, , \, \vartheta_i \in A_\uno} \; ,  \quad
   G_\uno^\pm(A)  \, := \,  {\Big\{\, {\textstyle \prod_{\,i=1}^{\,m}} \,
x_{\tgamma_i}(\vartheta_i) \,\Big\}}_{m \in \N \, , \, \tgamma_i \in \tDelta^\pm_\uno \, , \, \vartheta_i \in A_\uno}  $$
 \vskip-5pt
\noindent
 Let  $ \, N_\pm := \Big| \tDelta_\uno^\pm \Big| \; $  and  $ \, N := \Big| \tDelta_\uno \Big| = N_+ + N_- \; $,  and fix total orders  $ \, \preceq \, $  in  $ \, \tDelta_\uno^\pm \, $  and  $ \, \tDelta_\uno \; $:  \, we set
 \vskip-7pt
  $$  \displaylines{
   G_\uno^{\pm,<}(A)  \,\; := \;  \left\{\, {\textstyle \prod_{\,i=1}^{\,N_\pm}} \, x_{\tgamma_i}(\vartheta_i) \,\;\Big|\;\, \tgamma_1 \precneqq \cdots \precneqq \tgamma_{N_\pm} \in \tDelta^\pm_\uno \, , \; \vartheta_1, \dots, \vartheta_{N_\pm} \in A_\uno \,\right\}  \cr
   G_\uno^<(A)  \,\; := \;  \left\{\, {\textstyle \prod_{\,i=1}^N} \, x_{\tgamma_i}(\vartheta_i) \,\;\Big|\;\, \tgamma_1 \precneqq \cdots \precneqq \tgamma_{\scriptscriptstyle N} \in \tDelta_\uno \, , \; \vartheta_1, \dots, \vartheta_N \in A_\uno \,\right\}  }  $$
   \indent   We use also similar notations to denote the sheafifications  $ \, \bG_\uno \, $,
$ \, \bG^\pm_\uno \, $  and  $ \, \bG_\uno^{\pm,<} \, $.
\end{definition}

%

\vskip11pt

   Using once more  Lemma \ref{comm_1-pssg},  we obtain the following factorization result
for the functor  $ \, G_V \, $:

\vskip17pt

\begin{proposition} \label{g0g1}
 Let  $ \, A \in \salg_\bk \, $.  There exist set-theoretic factorizations
 \vskip-7pt
%
 $$  G_V(A)  \; = \;  G_\zero(A) \; G_\uno(A)  \; = \;  G_\uno(A) \; G_\zero(A)  \quad ,  \qquad
     G^\pm(A)  \; = \;  G^\pm_\zero(A) \; G^\pm_\uno(A)  \; = \;  G^\pm_\uno(A) \; G^\pm_\zero(A)  $$
\end{proposition}

\begin{proof}  The proof for  $ \, G_V(A) $  works for  $ G^\pm(A) $  too, so we stick to the former.
 \vskip5pt
%
   It is enough to prove either one of the equalities, say the first one.  Also, it is enough to show that  $ \; G_\zero(A) \, G_\uno(A) \; $ is closed by multiplication:  thus we must show that  $ \; g_\zero \, g_\uno \, \cdot \, g'_\zero \, g'_\uno \, \in \, G_\zero(A) \, G_\uno(A) \; $  for all  $ \, g_\zero \, , g_\zero' \in G_\zero(A) \, $  and  $ \, g_\uno \, , g'_\uno \in G_\uno(A) \; $.  By the very definitions, we need only to prove that
  $$  \big( 1 + \vartheta_1 X_{\tbeta_1} \big) \cdots \big( 1 + \vartheta_m X_{\tbeta_m} \big) \; x_{\talpha}(t) \; , \,\; \big( 1 + \vartheta_1 X_{\tbeta_1} \big) \cdots \big( 1 + \vartheta_m X_{\tbeta_m} \big) \; h_\eta(u) \,\; \in \; G_\zero(A) \, G_\uno(A)  $$
for all  $ \, m \in \N \, $,  $ \, \tbeta_1 , \dots , \tbeta_m \in \tDelta_\uno \, $,  $ \, \talpha \in \tDelta_\zero \, $,  $ \, \eta \in \Delta \, $,  $ \, \vartheta_1, \dots, \vartheta_m \in A_\uno \, $,  $ \, t \in A_\zero \, $  and  $ \, u \in U\big(A_\zero\big) \, $.  But this follows by an easy induction on  $ m \, $,  via the formulas in  Lemma \ref{comm_1-pssg}.
\end{proof}

\vskip11pt

   Carrying further on our analysis, we shall improve the above result by replacing the factor  $ \, G_\uno \, $  with a factor  $ \, G_\uno^< \; $.  As intermediate step, this requires the following technical result:

\vskip17pt

\begin{lemma} \label{pre-realcrucial}
   Let  $ \, A \in \salg_\bk \, $.  Then   --- with notation of  subsection \ref{first_preliminaries}  ---   we have
  $$  \displaylines{
   G_\uno(A)  \; \subseteq \;  G_\zero\big(A_\uno^{(2)}\big) \, G_\uno^<(A)  \quad ,
\qquad  G_\uno(A) \; \subseteq \;  G_\uno^<(A) \, G_\zero\big(A_\uno^{(2)}\big)  \phantom{\Big|}  \cr
   G_\uno^\pm(A)  \; \subseteq \;  G_\zero^\pm\big(A_\uno^{(2)}\big) \, G_\uno^{\pm,<}(A)  \quad ,  \qquad  G_\uno^\pm(A) \; \subseteq \;  G_\uno^{\pm,<}(A) \, G_\zero^\pm\big(A_\uno^{(2)}\big)  \phantom{\Big|}  }  $$
\end{lemma}

\begin{proof}
 We deal with the first identity, the other being similar.  Indeed, we prove the stronger result
  $$  \Big\langle G_\uno(A) \, , \, G_\zero\big(A_\uno^{(2)}\big) \Big\rangle  \; \subseteq \;  G_\zero\big(A_\uno^{(2)}\big) \, G_\uno^<(A)   \eqno (4.6)  $$
where  $ \, \Big\langle\! G_\uno(A) \, , \, G_\zero\big(A_\uno^{(2)}\big) \!\Big\rangle \, $  is the subgroup generated by  $ G_\uno(A) $  and  $ G_\zero\big(A_\uno^{(2)}\big) $.
                                                            \par
   Any element of  $ \Big\langle G_\uno(A) \, , \, G_\zero\big(A_\uno^{(2)}\big) \Big\rangle $  is a product  $ \; g \, = \, g_1 \, g_2 \cdots g_k \; $  in which each factor  $ g_i $  is either of type  $ \, h_{\eta_i}(u_i) \, $,  or  $ \, x_{\talpha_i}(t_i) \, $,  or  $ \, x_{\tgamma_i}(\vartheta_i) \, $,  with  $ \, \eta_i \in \Delta \, $,  $ \, \talpha_i \in \tDelta_\zero \, $, $ \, \tgamma_i \in \tDelta_\uno \, $  and  $ \, u_i \in U\big(A_\uno^{(2)}\big) \, $,  $ \, t_i \in A_\uno^2 \, $,  $ \, \vartheta_i \in A_\uno \, $.  Such a product belongs to  $ \, G_\zero\big(A_\uno^{(2)}\big) \, G_\uno^<(A) \, $  if and only if all factors indexed by the  $ \, \eta_i \in \Delta \, $  and by the  $ \, \talpha_j \in \tDelta_\zero \, $  are on the left of those indexed by the  $ \, \tgamma_\ell \, \in \Delta_\uno \, $,  and moreover the latter occur in the order prescribed by  $ \, \preceq \, $.  In this case, we say that the factors of  $ g $  are ordered.  We shall now re-write  $ g $  as a product of ordered factors, by repeatedly commuting the original factors, as well as new factors which come in along this process.
 \vskip5pt
   As we have only a finite number of odd coefficients in the expression for  $ g \, $,  we can assume without loss of generality that  $ A_\uno $  is finitely generated as an $ A_\zero \, $--module.  If  $ \, \overline{m} \, $  is the cardinality of any (finite) set of (odd) generators of  $ A_1 \, $,  this implies  $ \, A_\uno^{\,m} = \{0\} \, $  and  $ \, A_\uno^{\,(m)} \! = 0 \, $  when  $ \, m > \overline{m} \, $.
 \vskip5pt
   Let us consider two consecutive factors  $ \; g_i \, g_{i+1} \; $  in  $ g \, $.  If they are already ordered, we are done.  Otherwise, there are four possibilities:
 \vskip5pt
   {\it --- (1)} \,  $ \; g_i = x_{\tgamma_i}(\vartheta_i) \, $,  $ \, g_{i+1} = h_{\eta_i}(u_i) \, $.  \quad  In this case we rewrite
%
%
  $$  g_i \, g_{i+1}  \; = \;  x_{\tgamma_i}(\vartheta_i) \, h_{\eta_i}(u_i)  \;
= \; h_{\eta_i}(u_i) \, x_{\tgamma_i}(\vartheta'_i)  $$
with  $ \; \vartheta'_i \in A_\uno^{\,m_i} \, $  if  $ \, \vartheta_i \in A_\uno^{\,m_i} \, $,  thanks to  Lemma \ref{comm_1-pssg}{\it (c)}.  In particular we replace a pair of unordered factors with a new pair of ordered factors.  Even more, this shows that any factor of type  $ h_{\eta_i}\!(u_i) $  can be flushed to the left of our product so to give a new product of the same nature, but with all factors of type  $ h_{\eta_i}(u_i) $  on the left-hand side.

 \vskip5pt
   {\it --- (2)} \,  $ \, g_i = x_{\tgamma_i}(\vartheta_i) \, $,  $ \, g_{i+1} = x_{\talpha_{i+1}}(t_{i+1}) \, $.
\quad  In this case we rewrite
 \vskip-5pt
  $$  g_i \, g_{i+1}  \, = \,  g_{i+1} \, g_i \, g'_i  \;\quad  \text{with}  \;\quad  g'_i := \big(\, {g_i}^{\!-1} , \, {g_{i+1}}^{\!\!\!-1} \big) = \big(\, x_{\tgamma_i}(-\vartheta_i) \, , \, x_{\talpha_{i+1}}(-t_{i+1}) \big)  $$
so we replace a pair of (consecutive) unordered factors with a pair of ordered factors followed by another, new factor  $ g'_i \, $.  {\sl However},  letting  $ \, m_1, m_2 \in \N_+ \, $  be such that  $ \, \vartheta_i \in A_\uno^{\,m'} \, $,  $ \, t_{i+1} \in A_\uno^{\,m''} \, $,  by  Lemma \ref{comm_1-pssg}  {\sl this  $ \, g'_i \, $  is a product of new factors of type  $ \, x_{\tgamma_j}(\vartheta'_j) \, $  with  $ \, \vartheta'_j \in A_\uno^{\;m_j} \, $,  $ \, m_j \geq m' + m'' \, $}.
 \vskip5pt
   {\it --- (3)} \,  $ \, g_i = x_{\tgamma_i}(\vartheta_i) \, $,  $ \, g_{i+1} = x_{\tgamma_{i+1}}(\vartheta_{i+1}) \, $.  \quad  In this case we rewrite
 \vskip-5pt
  $$  g_i \, g_{i+1}  \, = \,  g_{i+1} \, g_i \, g'_i  \;\quad  \text{with}  \;\quad  g'_i := \big(\, {g_i}^{\!-1} , \, {g_{i+1}}^{\!\!\!-1} \big) = \big(\, x_{\tgamma_i}(-\vartheta_i) \, , \, x_{\tgamma_{i+1}}(-\vartheta_{i+1}) \big)  $$
so we replace a pair of unordered factors with a pair of ordered ones, followed by a new factor  $ g'_i $  which, again by  Lemma \ref{comm_1-pssg},  is a product of new factors of type  $ \, x_{\talpha_j}(t'_j) \, $  or  $ \, h_{\eta_j}(u'_j) \, $  with  $ \, t'_j \in A_\uno^{\;m_j} $,  $ \, u'_j \in U\big(A_\uno^{\,(m_j)}\big) \, $,  \, where  $ \, m_j \geq m' + m'' \, $  for  $ \, m', m'' \in \N_+ \, $  such that  $ \, \vartheta_i \in A_\uno^{\,m'} \, $,  $ \, \vartheta_{i+1} \in A_\uno^{\,m''} \, $.
 \vskip5pt
   {\it --- (4)} \,  $ \, g_i = x_{\tgamma}(\vartheta_i) \, $,  $ \, g_{i+1} = x_{\tgamma}(\vartheta_{i+1}) \; $.
%
%
 \quad  In this case we rewrite
 \vskip-5pt
  $$  g_i \, g_{i+1}  \; = \;  x_{\tgamma_i}(\vartheta_i) \, x_{\tgamma_{i+1}}(\vartheta_{i+1})  \; = \;
x_{\tgamma}(\vartheta_i) \, x_{\tgamma}(\vartheta_{i+1})  \; = \;  x_{\tgamma}(\vartheta_i \! + \! \vartheta_{i+1})  \; =: \;  g'_i  $$
so we replace a pair of unordered factors with a single factor.  In addition, each pair  $ \, g_{i-1} \, g'_i \, $  and  $ \, g'_i \, g_{i+2} \, $  respects or violates the ordering according to what the old pair  $ \, g_{i-1} \, g_i \, $  and  $ \, g_{i+1} \, , g_{i+2} \, $  did.
 \vskip6pt
   Now we iterate this process: whenever we have any unordered pair of
consecutive factors in the product we are working with,
we perform any one of steps  {\it (1)\/}  through  {\it (4)\/}  explained
above.  At each step, we substitute an unordered pair with a single factor
(step {\it (4)\/}), which does not form any more unordered pairs than
the ones we had before,  or with an ordered pair (steps {\it (1)--(3)\/}),
possibly introducing new additional factors.  However, any new factor is either
of type  $ \, x_{\talpha}(t) \, $,  with  $ \, t \in A_\uno^{\;m} \, $,  or
of type  $ \, x_{\tgamma}(\vartheta) \, $,  with  $ \, \vartheta \in A_\uno^{\;m} \, $,
or of type  $ \, h_\eta(u) \, $,  with  $ \, u \in U\big(A_\uno^{\,(m)}\big) \, $:  in all cases, the values of  $ m $  are (overall) strictly increasing after
each iteration of this procedure.  As  $ \, A_\uno^{\,m} = \{0\} \, $  for
$ \, m \gg 0 \, $,  after finitely many steps such new factors are trivial,
i.e.~eventually all unordered (consecutive) factors will commute with each
other and will be re-ordered without introducing any new factors.
Thus the process stops after finitely many steps, proving (4.6).
\end{proof}

\medskip

   A direct consequence of Proposition \ref{g0g1}  and  Lemma \ref{pre-realcrucial}  is the following ``factorization result'':
 \eject

\begin{proposition} \label{realcrucial}
   For every  $ \, A \in \salg_\bk \, $  we have
 \vskip-5pt
  $$  G_V(A)  \; = \;  G_\zero(A) \, G_\uno^<(A) \quad ,  \qquad
      G_V(A)  \; = \;  G_\uno^<(A) \, G_\zero(A)  $$
\end{proposition}

\vskip7pt

   We aim to show that the above decompositions are essentially unique.  We need another lemma:

\vskip14pt

\begin{lemma}  \label{red-to-classical}
  Let  $ A , B \! \in \! \salg_\bk \, $,  $ B $  a subsuperalgebra of  $ A $.  Then  $ \, G_V\!(B) $  is a subgroup of  $ \, G_V\!(A) \, $.
\end{lemma}

\begin{proof}
 By definition  $ G_V $  is a subgroup of  $ \rGL(V) \, $,  so elements in  $ G_V(A) $  are realized as matrices with entries in  $ A \, $,  and similarly for  $ B $  replacing  $ A \, $.  Then it is clear that any matrix in  $ G_V(B) $  is in  $ G_V(A) \, $,  and two such matrices are equal in  $ G_V(B) $  if and only if they are equal in  $ G_V(A) $  too.
\end{proof}

\smallskip

   For the proof of the main result we need the following intermediate step:

\medskip

\begin{lemma} \label{pre-crucial}
 Let  $ \, A \in \salg_\bk \, $,  and  $ \, g_\pm \, , \, f_\pm \in G_\uno^{\pm,<}(A) \, $.
%
%
   If  $ \,\; g_- \, g_+ = f_- \, f_+ \;\, $,  \, then  $ \,\; g_\pm = f_\pm \;\, $.
\end{lemma}

\begin{proof}
 To begin with, we write the element  $ g_- $  as an  {\sl ordered\/}  product
  $ \; g_- = \prod_{d=1}^{N_-} \, \big( 1 + \, \vartheta_d \, X_{\tgamma_d} \big) \; $,
for some  $ \, \vartheta_d \in A_\uno \, $,  where the  $ \, \gamma_d \in \tDelta_\uno^- \, $  are all the negative odd roots, ordered as in  Definition \ref{subgrps};  also, hereafter  $ \, N_\pm \! = \big|\tDelta_\uno^\pm\big| \, $.  Expanding the product on the right-hand side we get
  $$  g_-  \;\; = \;\,
   {\textstyle \, \sum_{\hskip-3pt{{0 \leq k \leq N_-} \atop {1 \, \leq \, d_1 < \cdots < \, d_k \leq \, N_-}}}}
    \hskip-7pt  {(-1)}^{k \choose 2} \, \vartheta_{d_1} \cdots \vartheta_{d_k} \, X_{\tgamma_{d_1}} \cdots X_{\tgamma_{d_k}}  $$
Similarly,
 $ \; f_- \, = \, \prod_{b=1}^{N_-} \, \big( 1 + \, \eta_b \, X_{\tgamma_b} \big) \; $,
 \, for some  $ \, \eta_d \in A_\uno \, $,  and then we have the expansion
 $ \; f_-^{\,-1} = \, \prod_{b = N_-}^1 \!\! {\big( 1 + \, \eta_b \, X_{\tgamma_b} \big)}^{-1} = \, \prod_{b = N_-}^1 \!\! \big( 1 - \, \eta_b \, X_{\tgamma_b} \big) \, =
 \, \sum_{\hskip-3pt {{0 \leq h \leq N_-} \atop {N_- \geq \, b_1 > \cdots > \, b_h \geq \, 1}}}
    \hskip-7pt  {(-1)}^{{h+1} \choose 2} \, \eta_{b_1} \cdots \eta_{b_h} \, X_{\tgamma_{b_1}} \cdots X_{\tgamma_{b_h}} \, $.

\medskip

   Now let  $ \, V = \oplus_\mu V_\mu \, $  be the splitting of  $ V $  into a direct sum of weight spaces.  Then  $ \, X_{\teta}.V_\mu \subseteq V_{\mu + \eta} \; $  if  $ \, X_{\teta} \in \fg_\eta \, $  (for each root  $ \eta $  and every weight  $ \mu \, $):  this and the previous expansions yield
 $ \; \big( f_-^{\,-1} \, g_- \big)\,.\,v_\mu \in \bigoplus_{\gamma^- \in \N \Delta_\uno^-} V_{\mu + \gamma^-} \; $
for all weights  $ \mu $  and  $ \, v_\mu \in V_\mu(A) \, $,  with  $ \, \N \Delta_\uno^- $  being the  $ \N $--span  of  $ \, \Delta_\uno^- \, $.  In a similar way   --- with parallel notation ---   we find also
 $ \;  \big( f_+ \, g_+^{\,-1} \big)\,.\,v_\mu \, \in \, {\textstyle \bigoplus_{\gamma^+ \in \N \Delta_\uno^+}} V_{\mu + \gamma^+} \; $.

\vskip5pt

   Now, the assumption  $ \, g_- \, g_+ = f_- \, f_+ \, $  implies  $ \, f_-^{\,-1} \, g_- = f_+ \, g_+^{\,-1} \, $.  Since  $ \, \N \Delta_\uno^- \! \cap \N \Delta_\uno^+ = \{0\} \, $,  by the previous analysis the only weight space in which both  $ \, \big( f_-^{\,-1} \, g_- \big)\,.\,v_\mu \, $  and  $ \, \big( f_+ \, g_+^{\,-1} \big)\,.\,v_\mu \, $  may have a non-trivial weight component is  $ V_\mu(A) $  itself.  In particular, letting  $ \, {\big( \big( f_-^{\,-1} \, g_- \big)\,.\,v_\mu \big)}_{\mu + \gamma^-} \, $  be the weight component of  $ \, \big( f_-^{\,-1} \, g_- \big)\,.\,v_\mu \, $  inside  $ V_{\mu + \gamma^-}(A) \, $,  we have  $ \; {\big( \big( f_-^{\,-1} \, g_- \big)\,.\,v_\mu \big)}_{\mu + \gamma^-} = 0 \; $  for any  $ \, \gamma^- \in \N \, \Delta_\uno^- \setminus \{0\} \, $.  We shall now describe these components, and deduce that  $ \; g_- = f_- \; $.

\smallskip

   From now on we use short-hand notation  $ \, \underline{\eta}_{\,\underline{b}} := \eta_{b_1} \cdots \eta_{b_h} \, $,  $ \, \underline{X}_{\tgamma_{\underline{b}}} := X_{\tgamma_{b_1}} \cdots X_{\tgamma_{b_h}} \, $,  and  $ \, \underline{\vartheta}_{\,\underline{d}} := \vartheta_{d_1} \cdots \vartheta_{d_k} \, $,  $ \, \underline{X}_{\,\tgamma_{\underline{d}}} := X_{\tgamma_{d_1}} \cdots X_{\tgamma_{d_k}} \, $,  for all ordered strings  $ \, \underline{b} := \big(\, b_1 \! > \! \cdots \! > \! b_h \,\big) \, $  and  $ \, \underline{d} := \big(\, d_1 \! > \! \cdots \! > \! d_k \,\big) \, $.  By  $ \, \big| \underline{b} \big| := h \, $  and  $ \, \big| \underline{d} \big| := k \, $  we denote the  {\sl length\/}  of  $ \underline{b} $  and  $ \underline{d} $  respectively.  Now, we have
  $$  \displaylines{
   f_-^{\,-1} g_-  \;\; = \;\;  {\textstyle \sum_{h, \, k \, = \, 0}^{N_-}}
   \; {\textstyle \sum_{{|\underline{b}| = h} \atop {|\underline{d}| = k}}} \;
   {(-1)}^{{{h+1} \choose 2} + {k \choose 2}} \, \underline{\eta}_{\,\underline{b}} \; \underline{X}_{\,\tgamma_{\underline{b}}} \; \underline{\vartheta}_{\,\underline{d}} \; \underline{X}_{\,\tgamma_{\underline{d}}}  \cr
   \hskip49pt   \;\; = \;\;  {\textstyle \sum_{h, \, k \, = \, 0}^{N_-}}
   \; {\textstyle \sum_{{|\underline{b}| = h} \atop {|\underline{d}| = k}}} \;
   {(-1)}^{{{h+1} \choose 2} + {k \choose 2} + h k} \, \underline{\eta}_{\,\underline{b}} \; \underline{\vartheta}_{\,\underline{d}} \; \underline{X}_{\,\tgamma_{\underline{b}}} \; \underline{X}_{\,\tgamma_{\underline{d}}}  }  $$
by the above expansions of  $ f_-^{\,-1} $  and  $ \, g_- \, $.  For every  $ \, \gamma^- \in \N \Delta_\uno^- \, $,  this last formula yields
  $$  {\big( \big( f_-^{\,-1} \, g_- \big)\,.\,v_\mu \big)}_{\mu + \gamma^-}  \; = \;\;  {\textstyle \sum_{h, \, k \, = \, 0}^{N_-}}
   \; {\textstyle \sum_{\hskip-9pt {{|\underline{b}| = h \, , \; |\underline{d}| = k \; \phantom{{}_|}} \atop {\pi(\tgamma_{\underline{b}}) \, + \, \pi(\tgamma_{\underline{d}}) \, = \; \gamma^- }}}} \hskip-3pt
   {(-1)}^{{{h+1} \choose 2} + {k \choose 2} + h k} \, \underline{\eta}_{\,\underline{b}} \; \underline{\vartheta}_{\,\underline{d}} \; \underline{X}_{\,\tgamma_{\underline{b}}} \; \underline{X}_{\,\tgamma_{\underline{d}}}\,.\,v_\mu  $$
where  $ \; \pi\big(\tgamma_{\underline{b}}\big) := \sum_{i=1}^{|\underline{b}|} \pi\big(\tgamma_{b_i}\big) \; $  and  $ \; \pi\big(\tgamma_{\underline{d}}\big) := \sum_{j=1}^{|\underline{d}|} \pi\big(\tgamma_{d_j}\big) \; $   --- notation of  \S \ref{def_Cartan-subalg_roots_etc}.
   In particular, for a  {\sl root\/}  $ \, \gamma^- := \gamma_q^-  \in \Delta_\uno^- \, $  we can single out the only two summands in the last formula indexed by a pair of strings whose lengths are one and zero: then the whole formula reads
  $$  {\big( \big( f_-^{\,-1} g_- \big)\,.\,v_\mu \big)}_{\mu + \gamma_q^-}  \,\;\; = \;\;\,  {\textstyle \sum^{N_-}_{\hskip-10pt {{\scriptstyle p=1} \atop {\hskip-5pt \pi\left(\tgamma_p\right) \, = \; \gamma_q^-}}}} \hskip-5pt
 \big( \vartheta_p - \eta_p \big) \, X_{\tgamma_p}\,.\,v_\mu  \,\; +   \hfill \hskip91pt \qquad  $$
 \vskip-19pt
  $$  {}   \eqno (4.7)  $$
 \vskip-19pt
  $$  + \;\;\,  {\textstyle \sum^{N_-}_{\hskip-7pt {{\scriptstyle h , \, k \, = \, 0} \atop {{(h,k) \not= (0,1)} \atop {(h,k) \not= (1,0)}}}}} \;
   {\textstyle \sum_{\hskip-5pt {{\hskip-3pt |\underline{b}| = h \, , \; |\underline{d}| = k \; \phantom{{}_|}} \atop {\pi(\tgamma_{\underline{b}}) \, + \, \pi(\tgamma_{\underline{d}}) \, = \; \gamma_q^- }}}} \hskip-9pt
   {(-1)}^{{{h+1} \choose 2} + {k \choose 2} + h k} \, \underline{\eta}_{\,\underline{b}} \; \underline{\vartheta}_{\,\underline{d}} \; \underline{X}_{\,\tgamma_{\underline{b}}} \; \underline{X}_{\,\tgamma_{\underline{d}}}\,.\,v_\mu   \hskip13pt  $$
 \eject

   For every  $ \, \gamma^- \in \Delta_\uno^- \, $,  we call  {\sl n-height\/}  of  $ \gamma $  the highest number  $ |\gamma^-| $  such that  $ \gamma^- $  itself is the sum of exactly  $ |\gamma| $  negative roots.  Looking at (4.7), we see that all roots  $ \pi\big(\tgamma_{b_i}\big) $  or  $ \pi\big(\tgamma_{d_j}\big) $  involved in the strings  $ \, \underline{b} \, $  or  $ \, \underline{d} \, $  occurring in the last, double sum necessarily satisfy
 $ \; \big| \pi\big(\tgamma_{b_i}\big) \big| \, \lneqq \, | \gamma_q^- | \; $,
 $ \; \big| \pi\big(\tgamma_{d_j}\big) \big| \, \lneqq \, | \gamma_q^- | \; $.
 Now fix any  $ \, \gamma_q^- \in \Delta_\uno^- \, $  such that  $ \, |\gamma_q^-| = 1 \, $:  then our last remark implies that (4.7) reduces to
 $ \;\; {\big( \big( f_-^{\,-1} g_- \big)\,.\,v_\mu \big)}_{\mu + \gamma_q^-}  = \,
  \sum^{N_-}_{\hskip-9pt {{\hskip-7pt \scriptstyle p=1} \atop {\pi\left(\tgamma_p\right) \, = \; \gamma_q^-}} } \hskip-15pt  \big( \vartheta_p \! - \eta_p \big) \, X_{\tgamma_p}\,.\,v_\mu \;\; $;  \,
thus from  $ \;\; {\big( \big( f_-^{\,-1} g_- \big)\,.\,v_\mu \big)}_{\mu + \gamma_q^-} = \, 0 \;\; $  we eventually get
 $ \;\; \sum^{N_-}_{\hskip-9pt {{\hskip-7pt \scriptstyle p=1} \atop {\pi(\tgamma_p) \, = \; \gamma_q^-}} } \hskip-15pt \big( \vartheta_p \! - \eta_p \big) \, X_{\tgamma_p}\,.\,v_\mu = \, 0 \;\; $.
 As  $ \fg $  acts faithfully on  $ V \, $,  hence  $ \fg_A $  does the same on  $ V_A \, $,  we get
 $ \;\; \sum^{N_-}_{\hskip-9pt {{\hskip-7pt \scriptstyle p=1} \atop {\pi(\tgamma_p) \, = \; \gamma_q^-}} } \hskip-15pt \big( \vartheta_p \! - \eta_p \big) \, X_{\tgamma_p} = \, 0 \;\; $
inside  $ \fg_A \, $.  Since the  $ X_{\tgamma_p} $  are part of a (Chevalley) basis of  $ \fg \, $,  we conclude that  {\it  $ \;\; \vartheta_p = \, \eta_p \; $  for all  $ \, p \in \{1,\dots,N_-\} \, $ such that  $ \, \pi(\tgamma_p) \, $  has n-height 1\,}.
 \vskip5pt
   We shall now extend this result to all root vectors  $ X_{\tgamma_p} \, $,  by induction on the n-height of  $ \, \pi\left(\tgamma^-_p\right) \, $.
 \vskip5pt
   Take in general any  $ \, \gamma_q^- \in \Delta_\uno^- \, $  with  $ \, |\gamma_q^-| > 1 \, $:  as induction hypothesis, we assume that for all  $ \, \tgamma_{p'} \in \tDelta_\uno^- \, $  such that  $ \, \big| \pi\big(\tgamma_{p'}\big) \big| \lneqq \big| \gamma_q^- \big| \, $  we have  $ \; \vartheta_{p'} = \, \eta_{p'} \; $.  Consider the last, double sum in (4.7): any ``monomial'' in the root vectors occurring there is of the form
  $$  \underline{X}_{\,\tgamma_{\underline{b}}} \, \underline{X}_{\,\tgamma_{\underline{d}}}
\; = \;  X_{\tgamma_{b_1}} \cdots X_{\tgamma^-_{b_h}} \, X_{\tgamma_{d_1}} \cdots
X_{\tgamma_{d_k}}   \eqno (4.8)  $$
with  $ \; b_1 > \cdots > b_h \; $  and  $ \; d_1 < \cdots < d_k \; $.  Moreover, by construction we can assume also that  $ \; b_i \not= d_j \; $  for all  $ i $  and  $ j \, $:  indeed, if  $ \; b_i = d_j \; $ in some  $ \underline{b}' $  and  $ \underline{d}' \, $,   then the inductive assumption gives  $ \; \eta_{b_i} = \vartheta_{d_j} \, $,  \, hence  $ \; \underline{\eta}_{\,\underline{b}'} \, \underline{\vartheta}_{\,\underline{d}'} = 0 \; $  and so  $ \; {(-1)}^{{{h+1} \choose 2} + {k \choose 2} + h k} \, \underline{\eta}_{\,\underline{b}'} \; \underline{\vartheta}_{\,\underline{d}'} \; \underline{X}_{\,\tgamma_{\underline{b}'}} \; \underline{X}_{\,\tgamma_{\underline{d}'}}\,.\,v_\mu \, = \, 0 \; $.
                                          \par
   Now, the monomial in (4.8) will occur a second time in the same sum as follows.  A first case is when  $ \, b_h > d_1 \; $:  then  $ \; \underline{X}_{\,\tgamma_{\underline{b}}} \, \underline{X}_{\,\tgamma_{\underline{d}}} \, = \, \underline{X}_{\,\tgamma_{\underline{b}'}} \, \underline{X}_{\,\tgamma_{\underline{d}'}} \; $  for  $ \; \underline{b}' := \big(\, b_1 \, , \dots \, , b_h \, , d_1 \big) \; $  and  $ \; \underline{d}' := \big( d_2 \, , \dots \, , d_k \big) \; $;  this includes also the case  $ \, h = 0 \, $.  The second case is  $ \, b_h < d_1 \; $:  then  $ \; \underline{X}_{\,\tgamma_{\underline{b}}} \, \underline{X}_{\,\tgamma_{\underline{d}}} \, = \, \underline{X}_{\,\tgamma_{\underline{b}'}} \, \underline{X}_{\,\tgamma_{\underline{d}'}} \; $  for  $ \; \underline{b}' := \big(\, b_1 \, , \dots \, , b_{h-1} \big) \; $  and  $ \; \underline{d}' := \big(\, b_h \, , d_1 \, , \dots \, , d_k \big) \; $;  this makes sense for  $ \, k = 0 \, $  as well.
                                                 \par
   In both cases, the new strings  $ \underline{b}' $  and  $ \underline{d}' $  enjoy the property
 $ \; \big| \underline{b}' \big| = \big| \underline{b} \big| \pm 1 \; $  and  $ \; \big| \underline{d}' \big| = \big| \underline{d} \big| \mp 1 \; $.
 \vskip5pt
   Now, whenever we consider any such pair of monomials
 $ \; \underline{X}_{\,\tgamma_{\underline{d}}} \, \underline{X}_{\,\tgamma_{\underline{d}}} \; $
and
 $ \; \underline{X}_{\,\tgamma_{\underline{b}'}} \, \underline{X}_{\,\tgamma_{\underline{d}'}} \; $
occurring in the last (double) sum in (4.7) and such that
 $ \; \underline{X}_{\,\tgamma_{\underline{d}}} \, \underline{X}_{\,\tgamma_{\underline{d}}} \, = \, \underline{X}_{\,\tgamma_{\underline{b}'}} \, \underline{X}_{\,\tgamma_{\underline{d}'}} \; $,  by induction we have
 $ \; \underline{\eta}_{\,\underline{b}} \, \underline{\vartheta}_{\,\underline{d}} \, =
\, \underline{\eta}_{\,\underline{b}'} \, \underline{\vartheta}_{\,\underline{d}'} \; $.
 \; Even more, a direct check shows that for the signs involved one has  $ \; {(-1)}^{{{h+1} \choose 2} + {k \choose 2} + h k} \, + {(-1)}^{{{h'+1} \choose 2} + {k' \choose 2} + h' k'} \, = \, 0 \; $  where  $ \; h = \big| \underline{b} \big| \, $,  $ \, k = \big| \underline{d} \big| \, $,  $ \, h' = \big| \underline{b}' \big| \, $,  $ \, k' = \big| \underline{d}' \big| \; $.  Thus the two (identical!) monomials
 $ \; \underline{X}_{\,\tgamma_{\underline{d}}} \, \underline{X}_{\,\tgamma_{\underline{d}}} \; $
and
 $ \; \underline{X}_{\,\tgamma_{\underline{b}'}} \, \underline{X}_{\,\tgamma_{\underline{d}'}} \; $
in that sum  {\sl cancel out each other}.

\medskip

  The outcome is that the last, double sum in (4.7) is actually  {\sl zero\/}:  thus (4.7) itself reduces to $ \;\; {\big( \big( f_-^{\,-1} g_- \big)\,.\,v_\mu \big)}_{\mu + \gamma_q^-}  = \,
  {\textstyle \sum^{N_-}_{\hskip-9pt {{\hskip-7pt \scriptstyle p=1} \atop {\pi(\tgamma_p) \, = \; \gamma_q^-}} }} \hskip-9pt
 \big( \vartheta_p \! - \eta_p \big) \, X_{\tgamma_p}\,.\,v_\mu \;\; $:
and, as before, one deduces  $ \; \vartheta_p = \, \eta_p \; $  for all  $ p \, $.

\medskip

   Thus the above induction argument yields  $ \, \vartheta_p = \eta_p \, $  for all  $ \, p = 1, \dots, N_- \; $,  hence  $ \; g_- = f_- \; $.

\medskip

   An entirely similar analysis shows that  $ \; g_+ = f_+ \; $,  \, whence the claim is proved.
\end{proof}

\bigskip

   At last, we are ready for our main result:

\bigskip

\begin{theorem} \label{crucial}
 For any  $ \, A \in \salg_\bk \, $,  the group product yields a bijection
  $$  G_\zero(A) \times G_\uno^{-,<}(A) \times G_\uno^{+,<}(A) \,
\lhook\joinrel\relbar\joinrel\relbar\joinrel\relbar\joinrel\twoheadrightarrow
\, G_V(A)  $$
and all the similar bijections obtained by permuting the factors  $ G_\uno^{\pm,<}(A) $  and/or switching the factor  $ G_\zero(A) $  to the right.
\end{theorem}

\begin{proof}  We shall prove the first mentioned bijection.
 \vskip5pt
   In general,  Proposition \ref{realcrucial}  gives  $ \; G_V(A) = G_\zero(A) \, G_\uno^<(A) \, $,  so the product map from  $ \; G_\zero(A) \times G_\uno^<(A) \; $  to  $ G_V(A) $  is onto.  But in particular, we can choose an ordering on  $ \, \Delta_\uno \, $  such that  $ \, \Delta_\uno^- \preceq \Delta_\uno^+ \, $,  hence  $ \, G_\uno^<(A) = G_\uno^{-,<}(A) \, G_\uno^{+,<}(A) \, $,  so we are done for surjectivity.
 \vskip5pt
   To prove that the product map is injective amounts to showing that for any  $ \, g \in G_V(A) \, $  the factorization  $ \, g = g_\zero \, g_- \, g_+ \, $  with  $ \, g_\zero \in G_\zero(A) \, $,  $ \, g_\pm \in G_\uno^{\pm,<}(A) \, $,  is unique.
 \hbox{In other words, if}
 $ \, g = g_\zero \, g_- \, g_+ \! = \! f_\zero f_- f_+ \, $  with  $ \, g_\zero \, , f_\zero \! \in G_\zero(A) \, $,  $ \, g_\pm \, , f_\pm \! \in G_\uno^{\pm,<}(A) \, $,
we must prove  $ \, g_\zero = \! f_\zero \, $,  $ \, g_\pm = \! f_\pm \; $.
 \eject

   By definition of  $ G_\zero(A) \, $,  both  $ g_\zero $  and  $ f_\zero $  are products of finitely many factors of type  $ x_{\talpha}(t_{\talpha}) $  and  $ h_i(s_i) $  for some  $ \, t_{\talpha} \in A_\zero \, $, $ \, s_i \in U\big(A_\zero\big) \, $   --- with  $ \, \talpha \in \tDelta_\zero \, $,  $ \, i = 1, \dots, r \, $.  Moreover, there exist product expansions
 $ \; g_\pm \, = \, \prod_{d=1}^{N_\pm} \, \big( 1 + \vartheta^\pm_d \, X_{\tgamma^\pm_d} \big) \; $,
 $ \; f_\pm \, = \, \prod_{d=1}^{N_\pm} \, \big( 1 + \eta^\pm_d \, X_{\tgamma^\pm_d} \big) \; $
like in the proof of  Lemma \ref{pre-crucial}.  We call  $ B $  the superalgebra of  $ A $  generated by all the  $ \, t_{\talpha} \, $,  the  $ \, s_i \, $,  the  $ \vartheta^\pm_d $  and  the  $ \eta^\pm_d \, $:  this is finitely generated (as a superalgebra), and  $ B_\uno $  is finitely generated as a  $ B_\zero $--module.
                                                    \par
   By  Lemma \ref{red-to-classical},  $ G_V(B) $  is a subgroup into  $ G_V(A) \, $;  therefore the identity  $ \; g_\zero \, g_- \, g_+ \, = \, f_\zero \, f_- \, f_+ \; $  holds inside  $ G_V(B) $  as well.  Thus we can switch from  $ A $  to  $ B \, $,  i.e.~we can assume from scratch that  $ \, A = B \, $.  In particular  $ A $  is finitely generated, so  $ A_\uno $  is finitely generated as an  $ A_\zero \, $--module.

\medskip

  Consider in  $ A $  the ideal  $ A_\uno \, $,  the submodules  $ A_\uno^{\,m} $  (cf.~section \ref{first_preliminaries})  and the ideal  $ \big( A_\uno^{\,m} \big) $  of  $ A $  generated by  $ A_\uno^{\,m} \, $  ($ \, m \! \in \! \N \, $):  as  $ A_\uno^{\,m} $  is  {\sl homogeneous}, we have  $ \; A \big/ \big( A_\uno^{\,m} \big) \in \salg_\bk \; $.  Moreover, as  $ A_\uno $  is finitely generated (over  $ A_\zero $),  we have  $ \, A_\uno^{\,m} = \{0\} = \big(A_\uno^{\,m}\big) \, $  for  $ m \! \gg \! 0 \, $.  So it is enough to prove
  $$  g_\zero \equiv f_\zero  \mod \big(A_\uno^{\,m}\big) \; ,  \qquad  g_\pm \equiv f_\pm  \mod \big(A_\uno^{\,m}\big)   \eqno \forall \;\;\; m \in \N   \qquad (4.9)  $$
hereafter, for any  $ \, A' \in \salg_\bk \, $,  any  $ I $  ideal of  $ A' $  with  $ \, \pi_I : A' \relbar\joinrel\twoheadrightarrow A' \big/ I \, $  the canonical projection, by  $ \; x \equiv y \mod I \; $  we mean that  $ x $  and  $ y $  in  $ G_V(A') $  have the same image in  $ G_V \big( A' \big/ I \big) \, $  via  $ G_V(\pi_I) \, $.

\smallskip

   We prove (4.9) by induction, the case  $ \, m = 0 \, $  being clear, as there is no odd part.
                                                                   \par
   Let (4.9) be true for even  $ m \, $.  In particular,  $ \, g_\pm \equiv f_\pm  \mod \big(A_\uno^{\,m}\big) \, $: then the proof of  Lemma \ref{pre-crucial}  applied to  $ \, G_V\big( A \big/ \! A_\uno^{\,m} \big) \, $  gives  $ \, \vartheta_d^\pm \equiv \eta_d^\pm \mod \! \big(A_\uno^{\,m}\big) \, $  for all  $ d \, $,  hence  $ \, (\vartheta_d^\pm - \eta_d^\pm) \in \big(A_\uno^{\,m}\big) \cap A_\uno \subseteq \big(A_\uno^{m+1}\big) \, $,  for all  $ d \, $,  by an obvious parity argument.  Thus  $ \, g_\pm \equiv f_\pm  \mod \big(A_\uno^{m+1}\big) \, $  too, hence from  $ \, g_\zero \, g_- \, g_+ = f_\zero \, f_- \, f_+ \, $  we get  $ \, g_\zero \, \equiv f_\zero  \mod \big(A_\uno^{m+1}\big) \, $  as well, that is (4.9) holds for  $ \, m\!+\!1 \, $.
                                                                   \par
   Let now (4.9) hold for odd  $ m \, $.  Then  $ \, g_\zero \, \equiv f_\zero  \mod \big(A_\uno^{\,m}\big) \, $;  but  $ \, g_\zero \, , f_\zero \in G_\zero(A) = G_\zero(A_\zero) \, $  by definition, hence  $ \, g_\zero \, \equiv f_\zero \mod \big(A_\uno^{\,m}\big) \cap A_\zero \, $.  Therefore  $ \, g_\zero \, \equiv f_\zero \mod \big(A_\uno^{m+1}\big) \, $  because, by an obvious parity argument again, one has  $ \, \big(A_\uno^{\,m}\big) \cap A_\zero \subseteq \big(A_\uno^{m+1}\big) \, $.  Thus from  $ \, g_\zero \, g_- \, g_+ = f_\zero \, f_- \, f_+ \, $  we get also  $ \; g_- \, g_+ \equiv f_- \, f_+  \mod \big(A_\uno^{m+1}\big) \, $.  Then  Lemma \ref{pre-crucial}  again   --- now applied to  $ G\big( A \big/ \! A_\uno^{m+1} \big) $  ---   eventually gives  $ \, g_\pm \equiv f_\pm  \mod \big(A_\uno^{m+1}\big) \, $,  so that (4.9) holds for  $ \, m\!+\!1 \, $  too.
\end{proof}

\medskip

   The ``overall consequence'' of the last result is the following, straightforward corollary:

\medskip

\begin{corollary}  \label{nat-transf} {\ }
 \vskip4pt
   (a) \,  The group product yields functor isomorphisms
 \vskip-7pt
  $$  G_\zero \times G_\uno^{-,<} \times G_\uno^{+,<} \;{\buildrel \cong \over {\relbar\joinrel\longrightarrow}}\; G_V \quad ,  \qquad  \bG_\zero \times \bG_\uno^{-,<} \times \bG_\uno^{+,<} \;{\buildrel \cong \over {\relbar\joinrel\longrightarrow}}\; \bG_V  $$
 \vskip-3pt
\noindent
 as well as those obtained by permuting the  $ (-) $-factor  and the  $ (+) $-factor  and/or moving the  $ \big(\,\zero\,\big) $-factor  to the right.  All these induce similar functor isomorphisms with the left-hand side obtained by permuting the factors above, like  $ \, G_\uno^{+,<} \! \times \! G_\zero \times \! G_\uno^{-,<} \,{\buildrel \cong \over \longrightarrow}\; G_V \, $,  $ \, \bG_\uno^{-,<} \! \times \! \bG_\zero \times \! \bG_\uno^{+,<} \,{\buildrel \cong \over \longrightarrow}\; \bG_V \, $,  etc.
 \vskip4pt
   (b) \,  The group product yields functor isomorphisms
 \vskip-13pt
  $$  G_\zero^\pm \times G_\uno^{\pm,<} {\buildrel \cong \over {\relbar\joinrel\longrightarrow}}\; G_V^\pm  \; ,  \hskip7pt  \bG_\zero^\pm \times \bG_\uno^{\pm,<} {\buildrel \cong \over {\relbar\joinrel\longrightarrow}}\; \bG_V^\pm  \; ,  \hskip15pt
  G_\uno^{\pm,<} \! \times G_\zero^\pm {\buildrel \cong \over {\relbar\joinrel\longrightarrow}}\; G_V^\pm  \; ,  \hskip7pt  \bG_\uno^{\pm,<} \! \times \bG_\zero^\pm \,{\buildrel \cong \over {\relbar\joinrel\longrightarrow}}\; \bG_V^\pm  $$
 \vskip-1pt
   (c) \,  Let  $ \, \preceq \, $  be a total order on  $ \tDelta_\uno $  such that  $ \, \tDelta_\uno^- \preceq \tDelta_\uno^+ \, $  or  $ \, \tDelta_\uno^+ \preceq \tDelta_\uno^- \, $.  Then the group product yields
isomorphisms
  $ \;  G_\zero \times G_\uno^< \,{\buildrel \cong \over \longrightarrow}\, G_V \; $,
  $ \; \bG_\zero \times \bG_\uno^< \,{\buildrel \cong \over \longrightarrow}\, \bG_V \; $,
  $ \; G_\uno^< \times G_\zero \,{\buildrel \cong \over \longrightarrow}\, G_V \; $,
  $ \; \bG_\uno^< \times \bG_\zero \,{\buildrel \cong \over \longrightarrow}\ \bG_V \; $.
\end{corollary}

\medskip

   Yet another crucial step we can move on now is the following:

\medskip

\begin{proposition} \label{G_uno^pm-repres}
  The functors  $ \, G_\uno^{\pm,<} : \salg_\bk \longrightarrow \sets \; $  are representable: namely, they are the functor of points of the superscheme  $ \, \mathbb{A}_\bk^{0|N_\pm} $,  where  $ \, N_\pm := \big| \tDelta_\uno^\pm \big| \; $.  In particular they are sheaves, hence  $ \; G_\uno^{\pm,<} = \bG_\uno^{\pm,<} \; $.  Similarly, for any total order in  $ \tDelta_\uno $  such that  $ \, \tDelta_\uno^- \! \preceq \tDelta_\uno^+ \, $  or  $ \, \tDelta_\uno^+ \! \preceq \tDelta_\uno^- \, $,  we have  $ \, G_\uno^< = \bG_\uno^< \cong \mathbb{A}^{0|N} $  as super-schemes, with  $ \, N := \big| \Delta_\uno \,\big| = N_+ + N_- \; $.
\end{proposition}

\begin{proof}  Clearly, there exists a natural transformation  $ \, \Psi^\pm : \mathbb{A}_\bk^{0|N_\pm} \! \longrightarrow G_\uno^{\pm,<} \; $  given on objects by
 \vskip3pt
   \centerline{ $  \Psi^\pm(A) \, : \, \mathbb{A}_\bk^{0|N_\pm}(A) \! \longrightarrow G_\uno^{\pm,<}(A)  \quad ,  \qquad  (\vartheta_1,\dots,\vartheta_{N_\pm}) \, \mapsto \, {\textstyle \prod_{i=1}^{N_\pm}} \, x_{\tgamma_i}(\vartheta_i) $ }
 \vskip5pt
   \indent   Now given  $ \; g_\uno^\pm = \prod_{i=1}^{N_\pm} x_{\tgamma_i}\!(\vartheta'_i) \in G_\uno^{\pm,<}(A) \, $, $ \, h_\uno^\pm = \prod_{i=1}^{N_\pm} x_{\tgamma_i}\!(\vartheta''_i) \in G_\uno^{\pm,<}(A) \, $,  \, assume that  $ \; g_\uno^\pm = h_\uno^\pm \, $,  hence  $ \; h_\uno^\pm \, {(g_\uno^\pm)}^{-1} = 1 \, $.  Then we get  $ \, \big( \vartheta'_1, \dots, \vartheta'_{N_\pm} \big) = \big( \vartheta''_1, \dots, \vartheta''_{N_\pm} \big) \, $  just as showed in the proof of  Lemma \ref{pre-crucial}.  This means that  $ \Psi^\pm $  is an isomorphism of functors, which proves the first part of the claim.
%
%
   The last part of the claim then follows like for  Corollary \ref{nat-transf}{\it (c)}.
\end{proof}

\vskip9pt

   Finally, we prove that the supergroup functors  $ \bG_V $  are affine algebraic:

\vskip11pt

\begin{theorem}  \label{representability}
   Every functor\/  $ \bG_V $  is an affine algebraic supergroup.
\end{theorem}

\begin{proof}
 First,  $ \bG_\uno^{-,<} $  and  $ \bG_\uno^{+,<} $  are affine, and algebraic, by  Proposition \ref{G_uno^pm-repres};  moreover, by  Proposition \ref{descr-G_zero}{\it (c)},  or by  Proposition \ref{semi-direct-G_zero},  $ \bG_\zero $  is affine algebraic as well.  Now  Corollary \ref{nat-transf}  gives  $ \; \bG_V \cong \, \bG_\zero \times \bG_\uno^{-,<} \times \bG_\uno^{+,<} \; $  as superschemes.  As any direct product of affine algebraic superschemes is affine algebraic too (see  \cite{ccf},  Ch.~10), we can eventually conclude the same for  $ \bG_V \, $.
\end{proof}

\vskip9pt

\begin{remark}  \label{alt-def_G_V}
 Theorem \ref{representability}  and  Proposition \ref{G_uno^pm-repres}  together  show that  $ \; \bG_V \, \cong \, \bG_\zero \times \bG_\uno^{-,<} \! \times \bG_\uno^{+,<} \; $  as superschemes.  As  $ \, \bG_\uno^{\pm,<} \cong G_\uno^{\pm,<} \, $  is generated by the one-parameter supersubgroups  $ \, x_{\tgamma} \; \big(\, \tgamma \in \tDelta_\uno^\pm \big) \, $,  we conclude that  $ \bG_V  $  can also be described as  $ \; \bG_V(A) = \Big\langle \bG_\zero(A) \cup {\big\{\, x_{\tgamma}(A) \big\}}_{\tgamma \in \tDelta_\uno} \Big\rangle \; $  for all  $ \, A \in \salg_\bk \, $.  Even more, as  $ \; \bG_\zero \, \cong \, \bG_0 \ltimes \bG_{\zero^{\,\uparrow}} \; $  by  Proposition \ref{semi-direct-G_zero}  and  $ \, \bG_{\zero^{\,\uparrow}} \cong \bigtimes_{\;\talpha \in \tDelta_{\zero^{\,\uparrow}}} x_{\talpha} \, $  by  Proposition \ref{factorization},  we have also  $ \; \bG_V(A) = \Big\langle \bG_0(A) \, \cup {\big\{\, x_{\talpha}(A) \big\}}_{\talpha \in \tDelta \setminus \tDelta_0} \Big\rangle \; $,  \, for all  $ \, A \in \salg_\bk \, $.
\end{remark}

\vskip4pt

   We finish with an additional, non-obvious remark: under mild assumptions, the supergroup  $ \bG_V \, $,  which by construction is a supersubgroup of  $ \rGL(V) \, $,  is indeed a  {\sl closed\/}  one:

\vskip11pt

\begin{proposition}  \label{G_V-closed-ssgrp}
   Assume that  $ \, \fg_\uno $  as a  $ \bk $--submodule  of  $ \, {\rgl(V)}_\uno $  is a direct summand with a  $ \bk $--free  complement.  Then  $ \, \bG_V $  is a closed supersubgroup of\/  $ \rGL(V) \, $.  In particular, this is always true if  $ \, \bk $  is a field.
\end{proposition}

\begin{proof}
 By construction we have that  $ \, \bG \leq \rGL(V) \, $.
 Consider the factorization  $ \; \bG_V = \, \bG_\zero \times \bG_\uno^< \; $  in  Corollary \ref{nat-transf}{\it (c)\/}:  by construction,  $ \, \bG_\zero \, $  is just a classical algebraic group(-scheme), embedded into  $ \rGL(V) $  as a  {\sl closed\/}  subgroup, therefore it is enough to show that  $ \; \bG_\uno^< \; $  is closed too.
 \vskip3pt
%
   Recall that  $ \rGL(V) $  can be realized as an open supersubscheme of  $ \, \End(V) = \text{\sl Mat}_{m|n}(\bk) \, $,  where  $ m|n $  is the (super)rank of  $ V \, $;  so it is enough to prove that  $ \bG_\uno^< $  is closed in  $ \text{\sl Mat}_{m|n}(\bk) \, $;  recall also
 \vskip3pt
   \centerline{ $ \cO\big(\End(V)\big)  \; = \;  \cO\big(\text{\sl Mat}_{m|n}(\bk)\big)  \; = \;  \bk\Big[{\big\{ x'_{i,j} \, , \, x''_{r,s} \, , \, \xi'_{i,s} \, , \, \xi''_{r,j} \big\}}_{i,j=1,\dots,m;}^{r,s=1,\dots,n;} \Big] $ }
 \vskip3pt
   Using  Proposition \ref{G_uno^pm-repres},  we identify  $ \, \mathbb{A}_\bk^{0|N} \cong
\bG_\uno^< \, $
 so that the point  $ 0 $  in
$ \mathbb{A}_\bk^{0|N} $  corresponds to the identity  $ I $  in  $ \bG_\uno^< \, $.
Then the tangent superspace to  $ \bG_\uno^< $  at  $ I $  corresponds to the tangent
superspace to  $ \mathbb{A}_\bk^{0|N} $  at  $ 0 \, $,  naturally identified with
$ \mathbb{A}_\bk^{0|N} $  again.  By the assumption on  $ \fg_\uno \, $,  we can complete the  $ \bk $--basis $ \, \big\{ X_{\tgamma_1} \, , \, \dots \, , \, X_{\tgamma_N} \big\} \, $  of  $ \fg_\uno $  to a  $ \bk $--basis  of  $ {\rgl(V)}_\uno \, $:  this in turn correspond to a ``change of odd variables'' in  $ \cO\big(\End(V)\big) \, $,  from  $ \, {\big\{ \xi'_{i,s} \, , \, \xi''_{r,j} \big\}}_{i,j=1,\dots,m;}^{r,s=1,\dots,n;} \, $  to some new set of odd variables, say  $ \, \big\{ \widehat{\xi}_1 \, , \, \dots \, , \, \widehat{\xi}_{2mn} \big\} \, $,  such that  $ \; \big\langle X_{\tgamma_h} \, , \, \widehat{\xi}_k \big\rangle = \delta_{h,k} \;\, $.  Letting  $ \mathcal{J} $  be the embedding map of  $ \bG_\uno^< $  into  $ \End(V) \, $,  the tangent map  $ \, d_{{}_I}\mathcal{J} \, $  (of  $ \mathcal{J} $  at  $ I \, $)  is expressed by a 2-by-2 block matrix whose only non-trivial block (in the right-bottom corner) is  $ \; {\Big( {{\partial \widehat{\xi}_h} \over {\partial \vartheta_k}} \Big)}_{h,k=1,\dots,N;} \; $.  Now, given  $ \, A \in \salg_\bk \, $,  any  $ \; g = {\textstyle \prod_{i=1}^N} \, x_{\tgamma_i}(\vartheta_i) \in \bG_\uno^<(A) \; $  expands as
 $ \,\; g  \, = \,  {\textstyle \prod_{i=1}^N} \, x_{\tgamma_i}(\vartheta_i)  \, = \,
I + {\textstyle \sum_{i=1}^N} \, \vartheta_i \, X_{\tgamma_i} + \cO(2) \; $,  \;
 where  $ \, \cO(2) \, $  stands for some element in  $ \, \mathfrak{gl}\big(V(A)\big) = A_\zero \otimes_\bk \mathfrak{gl}(V)_\zero + A_\uno \otimes_\bk \mathfrak{gl}(V)_\uno \, $  whose (non-zero) coefficients in  $ A_\zero $  and  $ A_\uno $  actually belong to  $ A_\uno^{\,2} $  (cf.~Subsec.~\ref{first_preliminaries}).
 This implies that  $ \, {{\partial \widehat{\xi}_h} \over {\partial \vartheta_k}} = \delta_{h,k} \, $,  so that the only non-trivial block in the matrix of  $ \, d_{{}_I}\mathcal{J} \, $  is the identity matrix of size  $ N \, $.  Thanks to this last remark, we can adapt the  {\sl Inverse Function Theorem\/}  and its corollaries (see  \cite{ccf},  \S\S 5.1--2)  to the present context: the outcome is that
there exists ``a change of variables''
 \vskip3pt
   \centerline{ $ \; {\big\{ x'_{i,j} \, , \, x''_{r,s} \big\}}_{i,j=1,\dots,m;}^{r,s=1,\dots,n;} \mapsto {\big\{ \widetilde{x}'_{i,j} \, , \, \widetilde{x}''_{r,s} \big\}}_{i,j=1,\dots,m;}^{r,s=1,\dots,n;}  \quad  ,  \qquad  {\Big\{\, \widehat{\xi}_t \,\Big\}}_{t=1,\dots,N;} \mapsto {\Big\{\, \widetilde{\xi}_t \,\Big\}}_{t=1,\dots,N;} $ }
 \vskip3pt
\noindent
%
%
 such that the morphism of superalgebras  $ \, \mathcal{J}^* : \cO\big(\End(V)\big) \longrightarrow \cO\big(\bG_\uno^<\big) \, $  corresponding to  $ \mathcal{J} $  is given by mapping  $ \, \widetilde{x}'_{i,j} \mapsto \delta_{i,j} \, $,  $ \, \widetilde{x}''_{r,s} \mapsto \delta_{r,s} \, $,  $ \, \widetilde{\xi}_t \mapsto \vartheta_t \, $  for  $ \, t \leq N \, $,  $ \; \widetilde{\xi}_t \mapsto 0 \, $  for  $ \, t > N \, $.  In turn,  $ \, \bG_\uno^< = \text{\sl Im}(\mathcal{J}) \, $  is the zero locus  $ \, \text{\it Ker}\big(\mathcal{J}^*\big) \, $,  hence it is a closed supersubscheme of  $ \End(V) \, $.
\end{proof}

\medskip

  \subsection{The dependence on  $ V $}  \label{ind_adm-latt}

\smallskip

   {\ } \quad   The construction of the supergroups  $ \bG_V $  was made via the Lie superalgebra  $ \fg $  and the  $ \fg $--module  $ V \, $.  So we have to clarify how supergroups attached to
%
%
different  $ \fg $--modules  are related among them.  Moreover, the construction involves the choice of an admissible  $ \Z $--lattice  $ M $  in  $ V \, $:  nevertheless, we shall presently show that the outcome, i.e.~$ \bG_V $  itself, is actually independent of that choice.
%

\medskip

\begin{free text}  \label{weight-bG_V}
 {\bf The weight lattice of  $ \bG_V \, $.}  \  Let  $ L_w $  be the lattice of all ``integral weights'' of  $ \fg_0 $  (in short,``the weight lattice of  $ \fg_0 $''),  using standard terminology, cf.~for instance  \cite{hu}:  in particular, these are weights with respect to the Cartan subalgebra  $ \fh $  of  $ \fg \, $.  Also, we let  $ L_r $  be the lattice spanned by all the  $ \fh $--roots  of  $ \, \fg \, $  (in short, ``the root lattice''):  {\sl here by ``$ \, \fh $--roots''  we mean the eigenvalues in  $ \fg $  of the adjoint action of\/  $ \fh $  (again), not of\/  $ \overline{\fh} \, $};  these  $ \fh $--roots  are just the restrictions (as linear functionals) from  $ \overline{\fh} $  to  $ \fh $  of the  roots of  $ \fg $  considered in  \S \ref{def_Cartan-subalg_roots_etc}  (which might be called  ``$ \, \overline{\fh} $--roots'').  Actually, nothing changes in all cases but  $ H(n) \, $:  for the latter, an explicit description of the  $ \fh $--roots  follows from considering the description of the  $ \overline{\fh} $--roots  and reading it modulo  $ \delta \, $.
                                                                            \par
   From  \S \ref{def_Cartan-subalg_roots_etc}  we see that the root lattice  $ L_r $   is spanned by the weights  $ \, \varepsilon_1 \, $,  $ \dots $,  $ \varepsilon_r \, $  of the defining representation of the reductive Lie algebra  $ \fg_0 \, $.  On the other hand, the weight lattice  $ L_w $  is spanned by the so-called fundamental dominant weights  $ \omega_1 \, $,  $ \dots $,  $ \omega_r \, $.  Now, looking at the relationship between the  $ \varepsilon_i $  and the  $ \omega_j $  one sees that the quotient module  $ \, L_w \big/ L_r \, $  is
 \vskip1pt
   \qquad  ---  {\sl trivial},  when  $ \fg $  is of type  $ W $,  $ S $  or  $ \tS \, $,
 \vskip1pt
   \qquad  ---  {\sl isomorphic to}  $ \, \Z_2 \, $,  when  $ \fg $  is of type  $ H(2r+1) \, $,
 \vskip1pt
   \qquad  ---  {\sl isomorphic to}  $ \, \Z_2 \oplus \Z_2 \, $,  when  $ \fg $  is of type  $ H(2r) \, $;

\vskip3pt

\noindent
 therefore, in all cases  $ L_w $  is just ``slightly bigger'' than  $ L_r \, $.

\smallskip

   Now let  $ \bG_V $  be a supergroup constructed as in  section \ref{car-sgroups},  associated with the Lie superalgebra  $ \fg $  of Cartan type and with a faithful, rational, finite dimensional  $ \fg $--module  $ V $  with admissible lattice  $ M \, $.  Now  Corollary \ref{nat-transf},  Proposition \ref{G_uno^pm-repres},  Proposition \ref{semi-direct-G_zero}  and  Proposition \ref{factorization}  altogether give
 \vskip-15pt
  $$  \bG_V  \,\; \cong \;\,  \bG_\zero \times \bG_\uno^<  \,\; \cong \;\,  \bG_0 \times \bG_{\zero^{\uparrow}} \times \bG_\uno^<  \,\; \cong \;\,  \bG_0 \times \mathbb{A}^{N_{\zero^{\uparrow}}\!\big|\,0} \times \mathbb{A}^{0|N}
 \,\; \cong \;\,  \bG_0 \times \mathbb{A}^{N_{\zero^{\uparrow}}\!\big|N}  $$
 \vskip-7pt
\noindent
 (with notations used there),  i.e.~$ \; \bG_V \, \cong \, \bG_0 \times \mathbb{A}^{N_{\zero^{\uparrow}}\!\big|N} \; $.  By  Proposition \ref{bG_0-Ch_V},  $ \, \bG_0 \cong \mathbf{Ch}_V \, $  is a classical, split reductive algebraic group.  By classical theory we know that  $ \, \bG_0 \cong \mathbf{Ch}_V \, $  depends only on the lattice of  $ \fg_0 $--weights  (=$ \, \fg $--weights)  of  $ V \, $:  we denote this weight lattice by  $ L_V \, $.
                                                                            \par
   Now, for the lattice  $ L_V $  associated with the supergroup  $ \bG_V $  we have clearly  $ \, L_r \subseteq L_V \subseteq L_w \; $.  By the remarks above about  $ \, L_w \big/ L_r \, $,  we have that  $ L_V $  is always ``very close'' to  $ L_r $  or  $ L_w \, $:  in particular, we always have equalities  $ \, L_r = L_V = L_w \, $  when  $ \fg $ is of type  $ W $,  $ S $  or  $ \tS \, $  (i.e., not  $ H $).
\end{free text}

\medskip

   Let now  $ \bG_V $  and  $ \bG'_{V'} $  be two Cartan supergroups obtained from  $ \fg \, $  via different  $ \fg $--modules  $ V $  and  $ V' $.  We let  $ x_{\talpha}(\bt) \, $,  $ x'_{\talpha}(\bt) \, $,  and  $ h_H(u) \, $,  $ h'_H(u) \, $,  be the points of the one-parameter supersubgroups in  $ \bG_V $  and  $ \bG'_{V'} $  associated with  $ \, \talpha \in \! \tDelta \, $,  $ \, \bt \in A_\zero \! \cup \! A_\uno \, $,  and  $ \, H \! \in \! \fh_\Z \, $,  $ \, u \! \in \! U\big(A_\zero\big) $   --- cf.~subsection \ref{One-param-ssgroups}.

\vskip13pt

\begin{lemma}  \label{morphism}
 Let  $ \, \phi : \bG_V \longrightarrow \bG'_{V'} \, $  be a morphism of the supergroups mentioned above.  Assume that  $ \; \phi_A\big(\bG_0(A)\big) = \bG'_0(A) \; $  and  $ \; \phi_A \big( x_{\talpha}(\bt) \big) = x'_{\talpha}(\bt) \; $  for all  $ \, A \in \salg_\bk \, $,  $ \, \bt \in A_\zero \cup \! A_\uno \, $,  $ \, \talpha \in \tDelta \setminus \tDelta_0 \, $.
                                                          \par
   Then  $ \; \text{\it Ker}\,\big(\phi\big) \subseteq Z(\bG_0) \; $,  \, where  $ Z(\bG_0) $  is the center of  $ \, \bG_0 \; $.
\end{lemma}

\begin{proof}
 First, note that  $ \; \tDelta \setminus \tDelta_0 = \tDelta_{\,\zero^\uparrow} \cup \tDelta^-_\uno \cup \tDelta^+_\uno \; $,  and fix a total order  $ \, \preceq \, $  on this set such that  $ \, \tDelta_{\,\zero^\uparrow} \preceq \tDelta^-_\uno \preceq \tDelta^+_\uno  \, $.  Then take  $ \, g \in \bG_V(A) \, $  with  $ \, g \in \text{\it Ker}\,\big(\phi_A\big) \, $.  By  Corollary \ref{nat-transf},  Proposition \ref{G_uno^pm-repres},  Proposition \ref{semi-direct-G_zero},  Proposition \ref{factorization}  and  Proposition \ref{bG_0-Ch_V},  there is a unique factorization of  $ g $
  $$  g  \; = \;  g_0 \cdot {\textstyle \prod_{\tbeta \in \tDelta_{\,\zero^\uparrow}}} x_{\tbeta}(t_{\tbeta}) \cdot {\textstyle \prod_{\tgamma^- \in \tDelta^-_\uno}} x_{\tgamma^-}(\vartheta_{\tgamma^-}) \cdot {\textstyle \prod_{\tgamma^+ \in \tDelta^+_\uno}} x_{\tgamma^+}(\vartheta_{\tgamma^+})   \eqno (4.10)  $$
(all products being ordered with respect to  $ \preceq \, $)  for some  $ \, g_0 \in \bG_0(A) \, $,  $ \, t_{\tbeta} \in A_\zero \, $  and  $ \, \vartheta_{\tgamma^-} \in A_\uno \, $.  A similar factorization also holds for  $ \, \phi_A(g) \, $  in  $ \, \bG'_{V'}(A) \, $.  All this along with  $ \, \phi_A(g) =  e_{{\bG'_{V'}(A)}} \, $,  with
  $$  \phi_A(g)  \; = \;  \phi_A(g_0) \cdot {\textstyle \prod_{\tbeta \in \tDelta_{\,\zero^\uparrow}}} \phi_A\big( x_{\tbeta}(t_{\tbeta}) \big) \cdot {\textstyle \prod_{\tgamma^- \in \tDelta^-_\uno}} \, \phi_A\big( x_{\tgamma^-}(\vartheta_{\tgamma^-}) \big) \cdot {\textstyle \prod_{\tgamma^+ \in \tDelta^+_\uno}} \, \phi_A\big( x_{\tgamma^+}(\vartheta_{\tgamma^+}) \big)  $$
and with the unicity of the factorization of  $ \, \phi_A(g) \, $  implies, by the assumption  $ \; \phi_A \big( x_{\talpha}(\bt) \big) = x'_{\talpha}(\bt) \, $,  \, that all factors in the product on right-hand side here above are trivial.  In turn, all factors but  $ \, g_0 \, $  on right-hand side of (4.10) are trivial too; therefore  $ \, g = g_0 \in \bG_0(A) \bigcap \text{\it Ker}\,\big(\phi_A\big) = \text{\it Ker}\,\big(\phi_A\big|_{\bG_0(A)}\big) \; $.

\smallskip

   By assumption  $ \; \phi_A\big|_{\bG_0} : \bG_0 \relbar\joinrel\twoheadrightarrow \bG'_0 \; $  is an epimorphism, with  $ \bG_0 $  and  $ \bG'_0 $  being connected split reductive algebraic groups over  $ \Z $  having tangent Lie algebra  $ \fg $  (by  Proposition \ref{bG_0-Ch_V}),  so $ \; d\phi : \bG_0 \relbar\joinrel\rightarrow \bG'_0 \; $  is an isomorphism. By classical theory this forces  $ \; \text{\it Ker}\,\big(\phi_A\big|_{\bG_0}\big) \subseteq Z(\bG_0) \; $.
\end{proof}

\vskip9pt

   Using this last result, we can now show that the relation between supergroups  $ \bG_V $  associated with different  $ \fg $--modules  $ V $
 is the same as in the classical setting.  The result reads as follows:

\medskip

\begin{proposition} \label{mainthm}
 Let  $ \, \bG_V $  and  $ \, \bG'_{V'} $  be two affine supergroups constructed using two  $ \fg $--modules  $ V $  and  $ V' $  as in  subsection \ref{adm-lat}.  If  $ \, L_V \supseteq L_{V'} \, $,  then there exists a unique morphism  $ \; \phi : \bG_V \longrightarrow \bG'_{V'} \; $  such that  $ \; \text{\it Ker}\,(\phi) \subseteq Z\big(\bG_0\big) \; $  and  $ \; \phi_A\big(x_{\talpha}(\bt)\big) = x'_{\talpha}(\bt) \; $  for every  $ \, A \in \salg_\bk \, $,  $ \; \bt \in A_\zero \cup A_\uno \; $,  $ \; \talpha \in \tDelta \setminus \tDelta_0 \, $.  Moreover,  $ \phi $  is an isomorphism if and only if  $ \, L_V = L_{V'} \, $.
\end{proposition}

\begin{proof}
  By classical theory, if  $ \, L_V \supseteq L_{V'} \, $  there exists a well-defined epimorphism  $ \; \phi_0 : \bG_0 \relbar\joinrel\twoheadrightarrow \bG'_0 \; $:  in particular, we can (and we do) choose it so that  $ \, d\phi_0 = \text{\sl id}_{\fg_{{}_0}} \, $  (recall that  $ \, \Lie\big(\bG_0\big) = \fg_0 = \Lie\big(\bG_0\big) \, $  by  Proposition \ref{bG_0-Ch_V}).  As a consequence, we have that  $ \phi_0 $ acts on one-parameter (additive and multiplicative) subgroups of  $ \bG_0 $  as  $ \, \phi_0\big(x_{\talpha}(t)\big) = x'_{\talpha}(t) \, $  and  $ \, \phi_0\big(h_H(u)\big) = h'_H(u) \, $.

\smallskip

   Now we extend  $ \phi_0 $  to a morphism  $ \; \phi : \bG_V \longrightarrow \bG'_{V'} \; $  as follows.
   Fix  $ \, A \in \salg_\bk \, $  and a total order in  $ \, \tDelta \setminus \tDelta_0 \, $;  use the unique factorization in  $ \bG_V(A) $   --- like in the proof of  Lemma \ref{morphism}  ---   to factor  $ \, g \, $  as in (4.10).  Then  {\sl define}
 $ \; \phi_A(g) \, := \, \phi_0(g_0) \, \cdot {\textstyle \prod_{\tbeta \in \tDelta_{\,\zero^\uparrow}}} x'_{\tbeta}(t_{\tbeta}) \, \cdot {\textstyle \prod_{\tgamma^- \in \tDelta^-_\uno}} \, x'_{\tgamma^-}(\vartheta_{\tgamma^-}) \, \cdot {\textstyle \prod_{\tgamma^+ \in \tDelta^+_\uno}} \, x'_{\tgamma^+}(\vartheta_{\tgamma^+}) \; $.
                                                             \par
   This gives a well-defined map  $ \; \phi_A : \bG_V(A) \longrightarrow \bG'_{V'}(A) \; $,  which by construction is functorial in  $ A \, $:  thus we have a natural transformation  $ \phi $   --- a morphism of superschemes ---   from  $ \bG_V $  to  $ \bG'_{V'} \, $.
                                                       \par
   Moreover, this  $ \phi $  is also a morphism of supergroups.  In fact, if  $ \, A \in \salg_\bk \, $  is  {\sl local\/}  then  $ \phi_A $  is a group morphism: indeed,  $ \, \bG_V(A) = G_V(A) \, $  and  $ \, \bG'_{V'}(A) = G'_{V'}(A) \, $,  and we have  $ \; \phi_A(g\,h) = \phi_A(g) \, \phi_A(h) \; $  because all the relations used to commute elements in  $ G_V(A) $  or in  $ G'_{V'}(A) $  so to write a given element in ``normal form'' as in (4.10) actually  {\sl do not depend on the chosen representation}.  Finally, by  Proposition A.12  in  \cite{fg2},  we have that  $ \phi $  is uniquely determined by its effect on local superalgebras, on which we saw it is a morphism: thus we conclude that  $ \phi $  is globally a morphism.

\smallskip

   By construction  $ \phi $  is also onto.  Thus all assumptions of  Lemma \ref{morphism}  hold, and we can conclude that  $ \, \text{\it Ker}\,(\phi) \subseteq Z\big(\bG_0\big) \, $  and  $ \, \phi_A\big(x_{\talpha}(\bt)\big) = x'_{\talpha}(\bt) \; $.  Finally, again by construction  $ \phi $  is an isomor\-phism if and only if  $ \phi_0 $  is an isomorphism itself: but this in turn holds if and only if  $ \, L_V = L_{V'} \, $.
\end{proof}

\smallskip

   As a direct consequence, we have the following ``independence result'':

\vskip9pt

\begin{corollary}
 Every supergroup  $ \, \bG_V $  constructed so far is independent, up to isomorphism, of the choice (which is needed in the very construction) of an admissible lattice  $ M $  of\/  $ V $.
\end{corollary}

\begin{proof}
 Let  $ M $  and  $ M' $  be two admissible lattices of  $ V \, $,  and set  $ \, V' := V \, $.  Construct  $ \bG_V $  and  $ \bG_{V'} $  using respectively  $ M $  and  $ M' \, $:  then we have  $ \, L_V = L_{V'} \, $,  so  Proposition \ref{mainthm}  gives  $ \, \bG_V \cong \bG_{V'} \, $.
\end{proof}

\medskip

\subsection{Lie's Third Theorem for the supergroups  $ \bG_V $}
\label{cartan_LieTT}

   {\ } \quad   Let  $ \bG_V $  be an (affine) supergroup over the ring  $ \bk \, $,  built out of the Lie superalgebra  $ \fg $  (of Cartan type)  over  $ \KK $  and of the  $ \fg $--module  $ V $  as in  subsection \ref{const-car-sgroups}.  In  subsection \ref{One-param-ssgroups}  we have introduced the Lie superalgebra  $ \, \fg_{V,\bk} := \bk \otimes_\Z \fg_V \, $  over  $ \bk $  starting from the  $ \Z $--lattice  $ \fg_V \, $.  We now show that the algebraic supergroup  $ \bG_V $  has  $ \fg_{V,\bk} $  as its tangent Lie superalgebra.

\smallskip

   We start recalling how to associate a Lie superalgebra with a supergroup scheme  (\cite{ccf}, \S\S 11.2--5).

\smallskip

\begin{free text}  \label{tangent_Lie_superalgebra}
 {\bf The Lie superalgebra of a supergroup scheme.}  \ For a given  $ \, A \in \salg_\bk \, $  let  $ \, A[\epsilon] := A[x]\big/\big(x^2\big) \, $  be the  {\sl super\/}algebra  of dual numbers, in which  $ \, \epsilon := x \! \mod \! \big(x^2\big) \, $  is taken to be  {\it even}.  Then  $ \, A[\epsilon] = A \oplus A \epsilon \, $,  and there are natural morphisms
 $ \; i : A \longrightarrow A[\epsilon] \, , \, a \;{\buildrel i \over \mapsto}\; a \, $,  and  $ \; p : A[\epsilon] \longrightarrow A \, $,
 $ \, \big( a + a'\epsilon \big) \;{\buildrel p \over \mapsto}\; a \; $,  \; such that  $ \; p \circ i= {\mathrm{id}}_A \; $.
   Given a supergroup  $ \bk $--functor  $ \, G : \salg_\bk \! \longrightarrow \grps \, $,  denote the morphism associated with  $ \; p : A[\epsilon] \longrightarrow A \; $  by  $ \; G(p)_A : G (A(\epsilon)) \longrightarrow G(A) \; $.  This gives a unique Lie algebra valued functor  $ \,\; \Lie(G) : \salg_\bk \longrightarrow \lie_\bk \;\, $  given on objects by  $ \,\; \Lie(G)(A) \, := \, \text{\it Ker}\,\big(G(p)_A\big) \; $.  For the Lie structure, one first defines the  {\it adjoint action\/}  $ \; \Ad : G  \longrightarrow  \rGL(\Lie( G )) \; $  of  $ G $  on  $ \Lie(G) $  as  $ \; \Ad(g)(x) \, := \, G (i)(g) \cdot x \cdot {\big(G (i)(g)\big)}^{-1} \; $  for all  $ \, g \in G(A) \, $,  $ \, x \in \Lie(G)(A) \, $.  Then one defines the  {\it adjoint morphism\/}  $ \; \ad \, := \, \Lie(\Ad) : \Lie(G) \longrightarrow \Lie(\rGL(\Lie(G))) := \End(\Lie(G)) \, $,  \, and finally sets  $ \; [x,y] := \ad(x)(y) \; $  for all  $ \, x,y \in \Lie(G)(A) \, $.  Further details are in  \cite{ccf},  \S\S 11.3--5; note that the authors there assume  $ \bk $  to be a field, yet this is  {\sl not\/}  required for the present context.
                                                           \par
   When  $ G $  is (the functor of points of) a supergroup  $ \bk $--scheme  and  $ \bk $  is a field, the functor  $ \Lie(G) $  is quasi-representable  (cf.~\S \ref{Lie-salg_funct}):  indeed, it can be identified with (the functor of points of) the tangent superspace at the identity of  $ G \, $,  denoted  $ T_e(G) \, $.  In turn,  $ T_e(G) $  bears a structure of Lie  $ \bk $--superalgebra, as usual (cf.~\cite{ccf}, \S 11.4.);  moreover, we point out that it bears also a canonical $ 2 $--operation, which can be given using the (standard) identification of  $ T_e(G) $ with the  $ \bk $--superalgebra  of left-invariant superderivations (into itself) of  $ \cO(G) \, $,  the Hopf  $ \bk $--superalgebra representing  $ G \, $.
                                                           \par
   We shall presently see that for Cartan  $ \bk $--supergroups  $ \bG_V $  this is the case also if  $ \bk $  is not a field: we shall then denote by  $ \Lie(G) $  both the above functor and the associated  $ \bk $--supermodule.  Note that this is also the case for the  $ \bk $--supergroup  $ \Lie(\rGL_{m|n}) \, $:  indeed, it is well known that, whatever  $ \bk $  is, the functor  $ \Lie(\rGL_{m|n}) $  is quasi-representable, and identifies with the Lie  $ \bk $--superalgebra  $ \rgl_{m|n} \, $;  as the latter is free (as a  $ \bk $--module)  of finite rank,  $ \, \Lie\big(\rGL_{m|n}\big) \, $  is in fact representable too.
\end{free text}

\medskip

   Eventually, we are now ready for the main result of this subsection.

\medskip

\begin{theorem}  \label{3rd-Lie-Theorem}
 Let  $ \, \bG_V $  be the affine supergroup of Cartan type built upon\/  $ \fg $  and the  $ \fg $--module  $ V $  (cf.~section \ref{const-car-sgroups}).  Then  $ \, \Lie(\bG_V) $  is quasi-representable, and actually representable, namely  $ \, \Lie(\bG_V) = \cL_{\fg_{{}_{V,\bk}}} $  as functors from  $ \salg_\bk $  to  $ \lie_\bk \, $.
%
%
\end{theorem}

\begin{proof}
 The result follows from sheer computations: as everything takes place inside  $ \rGL(V) $,  one can argue like in the standard example of  $ \Lie(\rGL_{m|n}) $   --- which can be found, e.g., in  \cite{ccf}, \S 11.3.
                                                                 \par
   First, from the decomposition  $ \; \bG_V = \bG_\zero \times \bG_\uno^< = \bG_\zero \times \big(\,
\bigtimes_{\;\tgamma \in \tDelta_\uno} x_{\tgamma} \big) \; $   --- see  Corollary \ref{nat-transf}  and  Proposition \ref{G_uno^pm-repres}  ---   we find at once that
  $$  \displaylines{
   \qquad   \Lie(\bG_V)(A)  \; = \;  \Lie(\bG_\zero)(A_\zero) \, \times \, \big(\, \bigtimes_{\;\tgamma \in \tDelta_\uno} (\, 1 + \epsilon \, A_\uno \, X_{\tgamma}) \big)  \; =   \hfill  \cr
   \hfill   = \;  \Lie(\bG_\zero)(A_\zero) \, \times \,
\big(\, 1 + \epsilon \, {\textstyle \sum_{\;\tgamma \in \tDelta_\uno}} A_\uno \, X_{\tgamma} \big)  \; = \;  \Lie(\bG_\zero)(A_\zero) \, \times \,
\big(\, 1 + \epsilon \, A_\uno \otimes_\bk \fg_\uno \big)   \qquad  }  $$
   \indent   Second, by the results in  subsection \ref{even-part}  and by the classical theory of Chevalley groups we know that  $ \, \Lie(\bG_\zero) \, $  is quasi-representable (and actually representable), with  $ \, \Lie(\bG_\zero) = \cL_{{(\fg_\zero)}_{V,\bk}} \, $  where  $ \; {(\fg_\zero)}_{V,\bk} := \bk \otimes_\Z {(\fg_\zero)}_V \; $  and  $ \; {(\fg_\zero)}_V := \big\{ X \! \in \! \fg_\zero \,\big|\, X.M \subseteq M \big\} = \fh_V \, {\textstyle \bigoplus} \, \Big(\! {\textstyle \bigoplus}_{\talpha \in \tDelta_\zero } \, \Z \, X_{\talpha} \Big) \; $  much like in  Proposition \ref{stabilizer}  and in  subsection  \ref{One-param-ssgroups}.  From this and the previous remark, it follows that
  $$  \Lie(\bG_V)(A)  \; = \;  \cL_{{(\fg_\zero)}_{V,\bk}}(A_\zero) \, \times \,
\big(\, 1 + \epsilon \, A_\uno \otimes_\bk \fg_\uno \big)  \; = \;  A_\zero \otimes_\bk {(\fg_\zero)}_{V,\bk} + A_\uno \otimes_\bk \fg_\uno  \; = \;  \cL_{{(\fg_\zero)}_{V,\bk} \oplus \fg_\uno}(A)  $$
so that  $ \; \Lie(\bG_V) = \cL_{\fg_{{}_{V,\bk}}} \; $   --- as claimed ---   because  $ \, {(\fg_\zero)}_{V,\bk} \oplus \fg_\uno = \fg_{V,\bk} \; $.  In particular, as  $ \, \fg_{V,\bk} \, $  is free of finite rank it follows that  $ \Lie(\bG_V) $  is representable too.
\end{proof}

\medskip

 \subsection{Special supersubgroups of  $ \bG_V \, $}  \label{spec_supsubgroup_G_V}

\smallskip

   {\ } \quad   In subsections  \ref{const-car-sgroups}  and  \ref{even-part}  we considered the (super)subgroups  $ \bG_\zero $  and  $ \bG_0 $  of  $ \bG_V \, $.  We introduce now some other remarkable supersubgroups, associated
 with
 special Lie supersubalgebras of  $ \fg \, $.

\smallskip

\begin{definition}  \label{spec_supsubgrps-def}
 Fix a splitting  $ \, \tDelta_0 = \tDelta_0^+ \coprod \tDelta_0^- \, $  of the classical root system  $ \tDelta_0 = \Delta_0 $  (of  $ \fg_0 $)  into positive and negative roots.  For all  $ \, A \! \in \! \salg_\bk \, $  and  $ \, t \! > \! -1 \, $,  define the subgroups of  $ \, G_V(A) \, $
  $$  \displaylines{
   G_{-1^\uparrow}(A)  \; := \,  \Big\langle\, h_H(A) \, , \, x_{\talpha}(A) \;\Big|\; H\! \in \fh_\Z \, , \, \talpha \in {\textstyle \coprod_{z>-1}} \tDelta_z \,\Big\rangle \; ,  \;\qquad
   {\big( G_{-1^\uparrow} \!\big)}_\zero (A)  \; := \,  G_{-1^\uparrow}(A) \,{\textstyle \bigcap}\, G_\zero(A)  \cr
   G_{t^\uparrow}(A)  \; := \,  \Big\langle\, x_{\talpha}(A) \;\Big|\; \talpha \in {\textstyle \coprod_{z>t}\,} \tDelta_z \,\Big\rangle  \;\; ,  \qquad \qquad
     {\big( G_{t^\uparrow} \!\big)}_\zero (A)  \; := \,  G_{t^\uparrow}(A) \,{\textstyle \bigcap}\, G_\zero(A)  \cr
   G_{-1}(A)  \; := \;  \Big\langle\, x_{\tgamma}(A) \;\Big|\; \tgamma \in \tDelta_{-1} \Big\rangle  \;\; ,
 \qquad
     G_{-1,0}(A)  \; := \;  \Big\langle\, h_H(A) \, , \, x_{\talpha}(A) \;\Big|\; H \in \fh_\Z \, , \, \talpha \in \tDelta_{-1} {\textstyle \coprod} \tDelta_0 \,\Big\rangle  \cr
 }  $$
 \eject
  $$  \displaylines{
   G_{\text{min}}^-(A)  \,\; := \;  \Big\langle\; x_{\tgamma}(A) \, , \, h_H(A) \, , \, x_{\talpha}(A) \;\Big|\; \tgamma \in \! \tDelta_{-1} \, , H \in \fh_\Z \, , \, \talpha \in \! \tDelta_0^- \,\Big\rangle  \cr
   G_{\text{max}}^+(A)  \,\; := \;  \Big\langle\; h_H(A) \, , \, x_{\talpha}(A) \, , \, x_{\tbeta}(A) \;\Big|\; H \in \fh_\Z \, , \, \talpha \in \! \tDelta_0^+ \, , \, \tbeta \in \! \tDelta_z \; \forall \, z \! > \! 0 \,\Big\rangle  }  $$
(the last three rows only for  $ \, \fg \not\cong \tS(n) \, $).  Let  $ \; G_{t^\uparrow} \, , {\big( G_{t^\uparrow} \!\big)}_\zero \, , G_{-1} \, , G_{-1,0} \, , G_{\text{min}}^- \, , G_{\text{max}}^+ : \salg_\bk \longrightarrow \grps \; $,  for  $ \, t \! \geq \! -1 \, $,  be the corresponding full subgroup functors of  $ \, G_V \, $,  and  $ \bG_{t^\uparrow} \, $,  $ {\big( \bG_{t^\uparrow} \!\big)}_\zero \, $   --- for  $ \, t \! \geq \! -1 \, $  ---   $ \; \bG_{-1} \, $,  $ \bG_{-1,0} \, $,  $ \bG_{\text{min}}^- \, $  and  $ \bG_{\text{max}}^+ \, $  the sheafification functors of the each of the above.
\end{definition}

\smallskip

   Lemma \ref{closed_roots}  and  Lemma \ref{comm_1-pssg}  yield the following properties of these special supersubgroups:

\medskip

\begin{proposition}  \label{props_supsubgrps}
 The following hold:
 \vskip3pt
   (a) \;  $ G_{t^\uparrow} \cong \bigtimes_{\;\talpha \, \in \coprod_{z>t} \! \tDelta_z\,} x_{\tgamma} \; \cong \, \bG_{t^\uparrow} \; $,
           $ \; {\big( G_{t^\uparrow} \!\big)}_\zero \cong \bigtimes_{\;\talpha \, \in \coprod_{z>t} \! \tDelta_z \cap \tDelta_\zero\,} x_{\tgamma} \; \cong \, {\big( \bG_{t^\uparrow} \!\big)}_\zero \;\; $  for all  $ \; \! t > \! -1 \; $;
 \vskip3pt
   (b) \;\;  $ G_{p^\uparrow} = \bG_{p^\uparrow} \, \unlhd \, \bG_{q^\uparrow} = G_{q^\uparrow} \; $,  $ \;\; {\big( G_{p^\uparrow} \!\big)}_\zero = {\big( \bG_{p^\uparrow} \!\big)}_\zero \, \unlhd \, {\big( \bG_{q^\uparrow} \!\big)}_\zero = {\big( G_{q^\uparrow} \!\big)}_\zero \; $,  \quad  for all  $ \; -1 < q \leq p \; $;
 \vskip3pt
%
   (c) \;\;  $ G_{-1,0} = G_{-1} \! \rtimes G_0 \; $,  \;\;  $ \bG_{-1,0} = \bG_{-1} \! \rtimes \bG_0 \; $;  \;\quad\;
             $ G_{-1^\uparrow} = G_0 \ltimes G_{0^\uparrow} \; $,  \;\;  $ \bG_{-1^\uparrow} = \bG_0 \ltimes \bG_{0^\uparrow} \; $;
 \vskip3pt
   (d) \;\;  $ G_{\text{\rm min}}^\pm = G_{-1} \! \rtimes G_0^\pm \; $,  \;  $ \bG_{\text{\rm min}}^\pm = \bG_{-1} \! \rtimes \bG_0^\pm \; $,  \; and \;
 $ G_{\text{\rm max}}^\pm = G_0^\pm \ltimes G_{0^\uparrow} \; $,  \;  $ \bG_{\text{\rm max}}^\pm = \bG_0^\pm \ltimes \bG_{0^\uparrow} \; $
 \vskip2pt
\noindent
 where  $ \, G_0^\pm \, $  and  $ \, \bG_0^\pm \, $  are given as in  Definition \ref{def_Cartan-subgroup_funct}  and  Definition \ref{def_Cartan-supergroup}  with respect to the splitting  $ \, \tDelta_0 = \tDelta_0^+ \coprod \tDelta_0^- \, $  fixed in  Definition \ref{spec_supsubgrps-def}.
 \vskip3pt
   (e) \;  the supergroup functors  $ \, \bG_{t^\uparrow} \, $,  $ \, {\big( \bG_{t^\uparrow} \!\big)}_\zero \, $   --- for  $ \, t \geq -1 \, $  ---   and   --- for  $ \, \fg \not\cong \tS(n) \, $  ---   $ \, \bG_{-1} \, $,  $ \, \bG_{-1,0} \, $,  $ \, \bG_{\text{min}}^\pm $  and  $ \, \bG_{\text{max}}^\pm \, $  are all representable, hence they are affine (algebraic) supergroups.
\end{proposition}


\smallskip

   Finally, the ``Lie's Third Theorem'' holds for these supersubgroups too, by the same arguments:

\medskip

\begin{theorem}
 For every affine supergroup  $ \, \bG_{t^\uparrow} \, $,  $ \, {\big( \bG_{t^\uparrow} \!\big)}_\zero \; (\, t \! \geq \! -1) \, $,  $ \, \bG_{-1} \, $,  $ \, \bG_{-1,0} \, $,  $ \, \bG_{\text{\rm min}}^\pm \, $  and  $ \, \bG_{\text{\rm max}}^\pm \, $  as above the corresponding tangent Lie algebra functor  $ \, \Lie(-) $  is representable, namely
  $$  \displaylines{
   \Lie \big( \bG_{t^\uparrow} \!\big) \, = \, \cL_{\fg^\Z_{{\scriptstyle t}^\uparrow}}  \quad ,  \qquad  \Lie \big( {\big( \bG_{t^\uparrow} \!\big)}_\zero \big) \, = \, \cL_{{(\fg^\Z_{{\scriptstyle t}^\uparrow})}_\zero}   \quad \qquad  \forall \;\; t \geq 0  \cr
   \Lie \big( \bG_{{-1}^\uparrow} \!\big) \, = \, \cL_{{(\fg_{-1^\uparrow})}_{V,\bk}}  \quad ,  \qquad  \Lie \big( {\big( \bG_{{-1}^\uparrow} \!\big)}_\zero \big) \, = \, \cL_{{(\fg_{-1^\uparrow})}_{V,\,\bk\,;\,\zero}}  \cr
   \Lie \big( \bG_{-1} \!\big) \, = \, \cL_{\fg_{-1}^\Z}  \quad ,  \qquad  \Lie \big( \bG_{-1,0} \big) \, = \, \cL_{{(\fg_{-1,0})}_{V,\bk}}  \cr
   \Lie \big( \bG_{\text{\rm min}}^\pm \big) \, = \, \cL_{{(\fb_{\text{\rm min}}^\pm)}_{V,\bk}}  \quad ,  \qquad  \Lie \big( \bG_{\text{\rm max}}^\pm \big) \, = \, \cL_{{(\fb_{\text{\rm max}}^\pm)}_{V,\bk}}  }  $$
as functors from  $ \salg_\bk $  to  $ \lie_\bk \, $,  where  $ \; {(\fg_{-1^\uparrow})}_{V,\bk} := {(\fg_V)}_0 \ltimes \fg^\Z_{0^\uparrow} \; $,  $ \; {(\fg_{-1^\uparrow})}_{V,\,\bk\,;\,\zero} := {(\fg_V)}_0 \ltimes {\big( \fg^\Z_{0^\uparrow} \big)}_\zero \; $,  $ \; {(\fg_{-1,0})}_{V,\bk} := \fg^\Z_{-1} \rtimes {(\fg_V)}_0 \; $,  $ \; {\big( \fb_{\text{\rm min}}^\pm \big)}_{V,\bk} := \fg^\Z_{-1} \rtimes {(\fg_V\!)}_0^\pm \; $  and  $ \; {\big( \fb_{\text{\rm max}}^\pm \big)}_{V,\bk} := {(\fg_V\!)}_0^\pm \ltimes \fg^\Z_{0^\uparrow} \; $,  with  $ \; {(\fg_V\!)}_0 = \fh_V \oplus \big(\! \bigoplus_{\alpha \in \Delta_0} \fg^\Z_\alpha \big) \; $  and  $ \; {(\fg_V\!)}_0^\pm = \fh_V \oplus \big(\! \bigoplus_{\alpha \in \Delta_0^\pm} \fg^\Z_\alpha \big) \, $   --- notation of  Definition \ref{def_Lie-sub-objs}  and Definition \ref{def-kost-superalgebra_sub-objs}.
\end{theorem}

\medskip

 \subsection{The uniqueness theorem}  \label{unique-theorem}

   {\ } \quad   In this subsection we shall prove that every connected affine algebraic  $ \bk $--supergroup  whose tangent Lie superalgebra is of Cartan type and whose  {\sl classical subgroup\/}  (see below) is  $ \bk $--split,  is necessarily isomorphic to one of the supergroups  $ \bG_V $  we constructed.  So, up to isomorphism the supergroups  $ \bG_V $  are the unique ones of the above mentioned type.  We begin with a definition:

\smallskip

\begin{definition}  \label{ass-class-scheme}
  Let  $ \bG $  is an (affine)  {\sl supergroup},  $ \, H := \cO(\bG) \, $  the Hopf  $ \bk $--superalgebra  represen\-ting it,
%
%
 and  $ \; \overline{H} := H \Big/ \big( {H_\uno}^{\!2} \oplus H_\uno \big) = H_\zero \Big/ {H_\uno}^{\!2} \; $,  which is a (classical) commutative Hopf algebra.
                                                       \par
   The affine group-scheme  $ \, \bG_{\text{\it \!ev}} \, $  represented by  $ \; \overline{H} = \overline{\cO(\bG)} \; $   --- so that  $ \, \cO\big(\bG_{\text{\it \!ev}}\big) = \overline{\cO(\bG)} \; $  ---   is called  {\it the classical supersubgroup(-scheme) associated with  $ \bG \, $}.  By construction,  $ \bG_{\text{\it \!ev}} $  coincides, as a group functor, with the restriction of  $ \bG $  to the category of commutative (unital, associative)  $ \bk $--algebras.
                                                       \par
   The quotient map  $ \, \pi : H \! := \cO(\bG) \! \relbar\joinrel\twoheadrightarrow \cO\big(\bG_{\text{\it \!ev}}\big) = \overline{H} \, $  yields an embedding  $ \, j : \bG_{\text{\it \!ev}} \! \lhook\joinrel\longrightarrow \bG \, $,  \, so that  $ \bG_{\text{\it \!ev}} $  actually identifies with a closed (super)subgroup of  $ \bG \, $.
\end{definition}

\medskip

   Again by construction one has that every (closed) supersubgroup  $ \mathbf{K} $  of\/  $ \bG $  which is  {\sl classical}  is actually a (closed) subgroup of\/  $ \bG_{\text{\it \!ev}} \; $.
%
%
 Moreover,
 the functor  $ \, \Lie\big(\bG_{\text{\it \!ev}}\big) \, $  is the restriction of  $ \, {\Lie(\bG)} \, $  to the category of (classical) commutative algebras:  furthermore, when the latter is quasi-representable, say  $ \, {\Lie(\bG)} = \cL_\fg \, $,  then the former is quasi-representable too, with  $ \, \Lie\big(\bG_{\text{\it \!ev}}\big) = \cL_{\fg_\zero} \, $.

\medskip

\begin{remark}  \label{even=zero}
 Let\/  $ \, \bG := \bG_V \, $  be
%
%
 as in  Definition \ref{def_Cartan-supergroup}\,.
   Then  $ \,\; \big(\bG_V\big)_{\text{\it \!ev}} \cong \bG_\zero \;\, $  and  $ \,\; \big(\bG^\pm_V\big)_{\text{\it \!ev}} \cong \bG^\pm_\zero \;\, $.
\end{remark}

\medskip

\begin{free text}  \label{sgrps_tang-Lie_Cartan}
 {\bf Supergroups with tangent Lie superalgebra of Cartan type.}  Let  $ \bG $  be a connected affine algebraic supergroup, defined over  $ \bk \, $.  We assume that the functor  $ \Lie\,(\bG) $  associated with  $ \bG $  (cf.~\S~\ref{tangent_Lie_superalgebra})  is quasi-representable, with  $ \, \Lie\,(\bG) = \cL_{\fg_{{}_{V'\!,\bk}}} $
 where  $ \; \fg_{V'\!,\bk} := \, \bk \otimes_\Z \fg_{V'} \, $:
 in particular,  $ \fg $  is a simple Lie superalgebra of Cartan type, we fix in  $ \fg $  a Chevalley basis,  $ V' $  is a rational  $ \fg $--module  with an admissible lattice,  etc.  In short, we might say that  ``$ \bG $  has tangent Lie superalgebra which is simple of Cartan type''.  In particular, this means that  $ \, \fg_{V'\!,\bk} \, $  is free as a  $ \bk $--module,  with  $ \, \text{\it rk}_{\,\bk}(\fg_{V'\!,\bk}) = \text{\it dim}_{\,\KK}(\fg) \; $,  and it is a Lie  $ \bk $--superalgebra,  whose  $ \Z_2 $--grading  $ \, \fg_{V'\!,\bk} = {(\fg_{V'\!,\bk})}_\zero \oplus {(\fg_{V'\!,\bk})}_\uno \, $  is given by  $ \, {(\fg_{V'\!,\bk})}_{\overline{a}} := \bk \otimes_\Z {(\fg_{V'}\!)}_{\overline{a}} \, $  (for  $ \, a = 0, 1 \, $)  where  $ \, \fg = \fg_\zero \oplus \fg_\uno \, $  is the  $ \Z_2 $--grading  of  $ \fg \, $.
   To simplify notation we shall drop all superscripts  ``$ \, V', \bk \, $'', writing just  $ \fg \, $,  $ \fg_\zero $  and  $ \fg_\uno \, $,  tacitly assu\-ming that all these objects are  $ \bk $--forms  (specified above) of the initial objects defined over  $ \KK \, $.
                                                                        \par
   According to  Definition \ref{ass-class-scheme},  the supergroup  $ \bG $  has a ``classical'' subgroup  $ \bG_{\text{\it \!ev}} \, $,  such that  $ \, \Lie\,(\bG_{\text{\it \!ev}}) = \fg_\zero \; $.  The assumptions also imply that  $ \bG_{\text{\it \!ev}} $  is a connected affine algebraic (classical) group-scheme, defined over  $ \bk \, $.  Now  $ \; \fg_\zero = \fg_0 \oplus \, \fg_{\zero^{\,\uparrow}} \, $  (cf.~Definition \ref{def_Lie-sub-objs}),  with  $ \fg_0 $  reductive and  $ \fg_{\zero^{\,\uparrow}} $  nilpotent: by the classical theory, from  $ \, \Lie\big(\bG_{\text{\it \!ev}}\big) = \fg_\zero \, $  we deduce that $ \, \bG_{\text{\it \!ev}} \, \cong \, \bG'_0 \ltimes \bG'_{\zero^{\,\uparrow}} \, $  for some connected algebraic groups  $ \bG'_0 $  and  $ \bG'_{\zero^{\,\uparrow}} $  such that  $ \, \Lie\,(\bG'_0) \cong \fg_0 \, $  and  $ \, \Lie\big(\bG'_{\zero^{\,\uparrow}}\big) \cong \fg_{\zero^{\,\uparrow}} \, $.  In particular,  $ \bG'_0 $  is reductive and  $ \bG'_{\zero^{\,\uparrow}} $  is unipotent.  In addition,  {\sl we assume that  $ \bG $  is  $ \bk $--split},  by which we mean   --- by definition ---   that the classical reductive group  $ \bG'_0 $  is  $ \bk $--split.
                                                                      \par
   In this subsection we show that  $ \bG $  is (isomorphic to) a ``Cartan supergroup''  $ \bG_V $  associated with  $ \fg $  and with some  $ \fg $--module  $ V $  as in section \ref{car-sgroups}.

\smallskip

   For our arguments to apply, we need yet another technical requirement, namely  {\sl we assume that\/  $ \bG $  is linearizable},  i.e.~it is embeddable into some  $ \rGL_{n|m} $  as a closed supersubgroup (this is true when  $ \, \bG \cong \bG_V \, $,  hence it is a necessary condition).  Note that this is automatically true when the ground ring  $ \bk $  is a field, or  $ \bk $  is a PID   --- e.g.,  $ \, \bk = \Z \, $  --- and  $ \cO(G) $  is free as a  $ \bk $--module.
\end{free text}

\vskip13pt

\begin{free text}  \label{linear-G}
 {\bf Linearizing  $ \bG \, $.}  By classification theory of split reductive groups,  $ \bG'_0 $  can be realized via the classical Chevalley construction: namely, there is a faithful, rational, finite dimensional  $ \fg_0 \, $--module  $ \widetilde{V} \, $,  with an admissible lattice  $ \widetilde{M} \, $,  such that  $ \, \bG'_0 $  is isomorphic to the affine group-scheme associated with  $ \fg_0 $  and with  $ \widetilde{V} $  by the classical Chevalley's construction.  Similarly, by classification theory of unipotent algebraic groups,  $ \, \bG'_{\zero^{\,\uparrow}} \, $  is isomorphic to the group  $ \bG_{\zero^{\,\uparrow}} $  in  Definition \ref{def_Cartan-supergroup}.  Overall, we get  $ \; \bG_{\text{\it \!ev}} \cong \bG'_0 \ltimes \bG_{\zero^{\,\uparrow}} \; $.  Actually, one has even more:  $ \; \bG_{\text{\it \!ev}} \cong \bG'_0 \ltimes \bG_{\zero^{\,\uparrow}} \, $  can be realized at one strike by means of a (slight extension of the) classical Chevalley's construction, based upon a faithful, rational, finite dimensional  $ \fg_\zero \, $--module  $ \widehat{V} $  with an admissible lattice  $ \widehat{M} \, $.  Then the dual  $ \fg_\zero \, $--module  $ \widehat{V}^* $  is also faithful, rational, finite dimensional with  $ \widehat{M}^* $  as an admissible lattice.

%

\medskip

   By assumption  {\sl  $ \, \bG $  is linearizable},  so it identifies with a closed supersubgroup of some  $ \rGL_{n|m} \; $.  Then  $ \bG_{\text{\it \!ev}} $  identifies with a closed subgroup of  $ \, {\big( \rGL_{n|m} \big)}_{\text{\it \!ev}} \, $,  the classical subgroup of  $ \rGL_{n|m} \, $.

\smallskip

   Pick the  $ {\big( \rGL_{n|m} \big)}_{\text{\it \!ev}} $--module  $ \; \widehat{U} := \text{\it Ind}_{\,\bG_{\text{\it \!ev}}}^{\;{(\rGL_{n|m})}_{\text{\it \!ev}}}\big( \widehat{V}^* \big) \; $   --- thought of as a functor from  $ \salg_\bk $  to  $ \text{($ \bk $--mod)} \, $  ---   induced (by the classical theory of representations of algebraic groups) from the  $ \bG_{\text{\it \!ev}} $--module  $ \widehat{V}^* $.  Let  $ \widehat{U}^* $  be the  $ {(\rGL_{n|m})}_{\text{\it \!ev}} $--module  dual to  $ \widehat{U} \, $;  as  $ \, \text{\it Ind}_{\,\bG_{\text{\it \!ev}}}^{\;{(\rGL_{n|m})}_{\text{\it \!ev}}}\big( \widehat{V}^* \big) \, $  maps onto  $ \widehat{V}^* $,  we have that  $ \, \widehat{V} \cong \widehat{V}^{**} \, $  embeds into  $ \widehat{U}^* $,  i.e.~$ \widehat{U}^* $  contains a  $ \bG_{\text{\it \!ev}} $--submodule  isomorphic to  $ \widehat{V} \, $.
                                                                       \par
   Now, as  $ \, \Lie \big( {\big( \rGL_{n|m} \big)}_{\text{\it \!ev}} \big) = {\big( \rgl_{n|m} \big)}_\zero \; $,  the  $ {\big( \rGL_{n|m} \big)}_{\text{\it \!ev}} $--module  $ \widehat{U}^* $  is also a  $ {\big( \rgl_{n|m} \big)}_\zero $--module.  As  $ {\big( \rgl_{n|m} \big)}_\zero $  is a Lie (super)subalgebra of  $ \rgl_{n|m} \, $,  we can perform on  $ \widehat{U}^* $  the induction from  $ {\big( \rgl_{n|m} \big)}_\zero $  to  $ \rgl_{n|m} \, $:  this yields a  $ \rgl_{n|m} $--module  $ \; W := \text{\it Ind}_{\,{(\rgl_{n|m})}_\zero}^{\;\;\rgl_{n|m}}\!\big( \widehat{U}^* \big) \; $,  described by
 \vskip-5pt
  $$  W  \; := \;  \text{\it Ind}_{\,{(\rgl_{n|m})}_\zero}^{\;\;\rgl_{n|m}}\!\big( \widehat{U}^* \big)
\; = \; U\big(\rgl_{n|m}\big) \otimes_{{}_{\scriptstyle U((\rgl_{n|m})_\zero)}} \widehat{U}^*  $$
 \vskip-2pt
   \indent   Now  $ W $  is a  $ \rGL_{n|m} $--module  too: indeed, one simply has to restrict the action of  $ \cL_{\rgl_{n|m}} \! $  to  $ \rGL_{n|m} $  (thought of as a subfunctor of  $ \cL_{\rgl_{n|m}} $).  Yet we need to describe the  $ \rGL_{n|m} $--action  on  $ W $  explicitly.
                                                          \par
   It is known that  $ \rGL_{n|m} $  ``splits'' into a direct product   --- as a superscheme ---   of the subgroup  $ {\big( \rGL_{n|m} \big)}_{\text{\it \!ev}} $  and the totally odd supersubscheme  $ \; {\big( \rGL_{n|m} \big)}_{\text{\it \!odd}} := I + {\big( \rgl_{n|m} \big)}_\uno \; $,  where  $ \, I := I_{n+m} \, $  is the identity (block) matrix of size  $ \, (n\!+\!m) \times (n\!+\!m) \, $:  the splitting is given by the matrix product,
 namely  $ \; \rGL_{n|m}(A) \, \cong\, {\big( \rGL_{n|m} \big)}_{\text{\it \!ev}}(A) \times {\big( \rGL_{n|m} \big)}_{\text{\it \!odd}}(A) \; $
via the unique factorization
 $ \, \displaystyle \bigg(\! {{{\;\; a \;|\; \beta \;\,} \over {\,\; \gamma \;|\; d \;\,}}} \!\bigg) = \bigg(\! {{{\,\; a \;|\; 0 \;\,} \over {\,\; 0 \;|\; d \;\,}}} \!\bigg) \cdot \bigg(\! {{{\,\; \hskip6pt I_n \hskip6pt \;|\; a^{-1} \beta \;\,} \over {\,\; d^{-1} \gamma \;|\; \hskip6pt I_m \hskip6pt \;\,}}} \!\bigg) \, $
  for each block matrix
 $ \, \displaystyle \bigg(\! {{{\;\; a \;|\; \beta \;\,} \over {\,\; \gamma \;|\; d \;\,}}} \!\bigg) \! \in \rGL_{n|m}(A) \, $,  with  $ \, A \in \salg_\bk \; $.
%

%
%

\smallskip

   Every  $ \, g_{\text{\it \!ev}} \in \! {\big( \rGL_{n|m} \big)}_{\text{\it \!ev}}(A) \, $  acts on any decomposable tensor  $ \; y \otimes \widehat{u} \in U\big(\rgl_{n|m}\big) \otimes_{{}_{\scriptstyle U((\rgl_{n|m})_\zero)}} \! \widehat{U}^* = W \; $  via  $ \; g_{\text{\it \!ev}}\,.\big( y \otimes \widehat{u} \big) = \Ad(g_{\text{\it \!ev}})(y) \otimes g_{\text{\it \!ev}}\,.\widehat{u} \; $,  where on left-hand side we take the natural  $ \rGL_{n|m} $--action  on  $ U\big(\rgl_{n|m}\big) $  induced from the adjoint action on  $ \rgl_{n|m} \, $.  Moreover, every  $ \, g_{\text{\it \!odd}} = I + g' \in {\big( \rGL_{n|m} \big)}_{\text{\it \!odd}}(A) \, $  acts on any  $ \, y \otimes \widehat{u} \, $  as above by  $ \; g_{\text{\it \!odd}}\,.\big( y \otimes \widehat{u}\big) = \big( I + g' \big).\big( y \otimes \widehat{u} \big) = \big( y + g' y \big) \otimes \widehat{u} \; $.

\smallskip

   As  $ \bG $  is (embedded as) a closed supersubgroup of  $ \rGL_{n|m} \, $,  the  $ \rGL_{n|m} $--module  $ W $  is also a  $ \bG $--module.  Moreover, by Remark \ref{super-dist},  both for  $ \bG $  and for  $ \rGL_{n|m} $   --- which is a ``Chevalley supergroup'' in the sense of  \cite{fg2}  --- the Kostant superalgebra (with scalars extended to  $ \bk $)  identifies with the superalgebra of distributions: then
 Remark \ref{super-dist}  tells also that  $ U_\bk(\fg) $  embeds into  $ U_\bk\big(\rgl_{\,m|n}\big) \, $.  Then we can consider inside the  $ \bG $--module  $ W $  the subspace  $ \; V := U_\bk(\fg) \otimes_{{}_{\scriptstyle U_\bk(\fg_\zero)}} \!\! \widehat{V} \; $:  also, it is clear (thanks to the explicit description of the  $ \bG $--action) that this  $ V $  is a  $ \bG $--submodule  of  $ W \, $.
                                                            \par
  Tracking through the whole construction, as  $ \widehat{V} $  is rational and faithful as a  $ \bG_{\text{\it \!ev}} $--module  we see that  $ V $  in turn is rational and faithful as a  $ \bG $--module.  Thus  $ \bG $  embeds as a closed supersubgroup inside  $ \, \rGL(V) \, $,  and  $ \bG_{\text{\it \!ev}} $  as a closed subgroup of  $ \bG \, $.  Also, as  $ \widehat{M} $  is an admissible lattice in the  $ \fg_\zero $--module  $ \widehat{V} $,  we see that  $ \, M := K_\bk(\fg) \otimes_{{}_{\scriptstyle K_\bk(\fg_\zero)}} \! \widehat{M} \, $  is admissible in the  $ \fg $--module  $ V \, $  (hereafter, we write  $ \, K_\bk(\mathfrak{t}) := \bk \otimes_\Z K_\Z(\mathfrak{t}) \, $,  and ``admissible lattice'' has the obvious meaning when passing from  $ \Z $--modules  to  $ \bk $--modules).  Finally, as  $ \widehat{V} $  is finite dimensional, and  $ K_\bk(\fg) $  is (free) of finite rank as a  $ K_\bk(\fg_\zero) $--module  (see  Corollary \ref{kost_tens-splitting})  we argue that  $ \, V \, $  is finite dimensional too.

\smallskip

   By construction   --- including the fact that  $ \; V = U(\fg) \otimes_{U(\fg_\zero)} \! \widehat{V} = \bigwedge \! \fg_\uno \! \otimes \! \widehat{V} \; $  as a  $ \fg_\zero $--module  is just  $ \, \widehat{V}^{\,\oplus r} \, $  for  $ \, r := \text{\it rank}_{\,U(\fg_\zero)}(U(\fg)) = 2^{\text{\it dim}(\fg_\uno)} \, $  ---   the  $ \fg_\zero $--action  on  $ V $  is just a diagonalization  ($ r $ times) of the  $ \fg_0 $--action  on  $ \widehat{V} \, $:  as a consequence, the embedded copy of  $ {\big( \bG_V \big)}_{\text{\it \!ev}} $  inside  $ \rGL(V) $  is a  ($ r $  times) diagonalized copy of the group obtained in  $ \rGL\big( \widehat{V} \big) $  from the  $ \fg_\zero $--action  on  $ \widehat{V} $  via the Chevalley construction.  By assumption this group is  $ \bG_{\text{\it \!ev}} \, $,  thus  $ \, {\big( \bG_V \big)}_{\text{\it \!ev}} = \bG_{\text{\it \!ev}} \, $  inside  $ \rGL(V) \, $.
\end{free text}

\vskip7pt

\begin{free text}  \label{splitting-G}
 {\bf Splitting  $ \bG \, $.}  Recall that  $ \rGL(V) $  splits as  $ \, \rGL(V) = {\rGL(V)}_\zero \times {\rGL(V)}_\uno \, $,  with  $ \, {\rGL(V)}_\uno(A) := I + {\mathfrak{gl}(V)}_\uno(A_\uno) \, $
 (cf.~\S~\ref{linear-G}).  Then denote by  $ \, \pi_\zero \, $  and  $ \, \pi_\uno \, $  the projection maps of  $ \rGL(V) $  onto  $ {\rGL(V)}_\zero $  and  $ {\rGL(V)}_\uno \; $.  Note that  $ \, {\rGL(V)}_\zero = \pi_\zero\big(\rGL(V)\big) \, $  coincides with  $ {\big( \rGL(V) \big)}_{\text{\it \!ev}} := {\rGL(V)}\big|_{\alg_\bk} \; $.
                                                       \par
   Look at  $ \bG $  embedded inside  $ \rGL(V) \, $:  then  $ \, \pi_\zero(\bG) \, $  is the restriction  $ \bG\big|_{\alg_\bk} \, $,  hence  $ \, \pi_\zero(\bG) = \bG_{\text{\it \!ev}} \, $  (by Definition \ref{ass-class-scheme}),  so  $ \, \pi_\zero(\bG) \leq \bG \, $.  Given  $ \, g \in \bG(A) \, $,  for  $ \, A \in \salg_\bk \, $,  it factors as  $ \; g = g_\zero \cdot g_\uno \; $  with  $ \; g_\zero := \pi_\zero(g) \in {\rGL(V)}_\zero \; $,  $ \; g_\uno := \pi_\uno(g) \in {\rGL(V)}_\uno \; $.  Then  $ \, g_\zero \in \bG_{\text{\it \!ev}}(A) \leq \bG(A) \, $  and  $ \; \pi_\uno(g) =: g_\uno = {g_\zero}^{\!-1} g \in \bG(A) \; $,  so  $ \, \pi_\uno(g) \in \bG(A) \, $  too; it follows that  $ \; \pi_\uno(\bG) \subseteq \bG \; $  (as a supersubscheme) too.  The outcome is that the factorization  $ \; \rGL(V) \, = \, {\rGL(V)}_\zero \times {\rGL(V)}_\uno \, = \, \pi_\zero\big(\rGL(V)\big) \times \pi_\uno\big(\rGL(V)\big) \; $  of  $ \rGL(V) $  induces the factorization  $ \; \bG = \pi_\zero(\bG) \times \pi_\uno(\bG) = \bG_{\text{\it \!ev}} \times \pi_\uno(\bG) \; $  of  $ \, \bG \, $  as well.
\end{free text}

\vskip9pt

\begin{free text}  \label{const-G_V-&-comparison}
 {\bf Construction of a supergroup  $ \bG_V $  and comparison with  $ \bG \, $.}  The  $ \bG $--module  $ V $  constructed in  \S \ref{linear-G}  is obviously also a representation of the Lie superalgebra  $ \, \Lie(\bG) = \fg \, $.  More precisely,  $ \, V = U(\fg) \otimes_{U(\fg_\zero)} \! \widehat{V} \; $  implies  $ \, V = \text{\it Ind}_{\,\fg_\zero}^{\;\fg}\big( \widehat{V} \big)
\; $.  In addition, we saw that:  \, {\it (a)}  $ \, V $  has finite dimension,  \, {\it (b)}  $ \, V $  is rational,  \, {\it (c)}  $ \, V $  contains  $ \; M := K_\bk(\fg) \otimes_{K_\bk(\fg_\zero)} \widehat{M} \; $  as an admissible lattice.
 \vskip5pt
   Therefore, using  $ V $  and  $ M $  we can construct the affine algebraic supergroup  $ \bG_V \, $,  as in  section \ref{car-sgroups},  which is embeddded inside  $ \rGL(V) $  as a closed  (cf.~Proposition \ref{G_V-closed-ssgrp})  connected supersubgroup.
 \vskip5pt
   By the analysis above, we can embed both  $ \bG $  and  $ \bG_V $  as closed supersubgroups of  $ \rGL(V) \, $:  thus we identify  $ \, \bG $  and  $ \bG_V $  with their images in  $ \rGL(V) \, $,  and their tangent Lie superalgebras with the corresponding images in  $ \rgl(V) \, $.  We can now prove the main result of this subsection:
\end{free text}

\vskip9pt

\begin{theorem}  \label{bG = bG_V}
 {\it Let\/  $ \bG $  and\/  $ \bG_V $  be as above.  Then\/  $ \; \bG = \, \bG_V \; $.}
\end{theorem}

\begin{proof}
 By the analysis in  \S \ref{splitting-G},  the splitting  $ \, \rGL(V) = {\rGL(V)}_\zero \times {\rGL(V)}_\uno \, $,  with  $ \, {\rGL(V)}_\uno(A) :=
 I \! + {\mathfrak{gl}(V)}_\uno(A_\uno) \, $,  and the embedding of  $ \bG $  in  $ \rGL(V) $  provide a splitting  $ \; \bG = \bG_{\text{\it \!ev}} \times \bG_{\text{\it \!odd}} \; $,  with  $ \, \bG_{\text{\it \!ev}} = \pi_\zero(\bG) \, $  and  $ \, \bG_{\text{\it \!odd}} := \pi_\uno(\bG) \; $.  Similarly, the same fact with  $ \bG_V $  replacing  $ \bG $  yields  $ \; \bG_V = \big(\bG_V\big)_{\text{\it \!ev}} \! \times \! \big(\bG_V\big)_{\text{\it \!odd}} \; $,  with  $ \, \big(\bG_V\big)_{\text{\it \!ev}} \!\! = \! \pi_\zero(\bG_V) \! = \! \bG_\zero \, $  and  $ \, \big(\bG_V\big)_{\text{\it \!odd}} \! := \! \pi_\uno\big( \bG_V \big) \, $   --- see Remark \ref{even=zero}.  All these splittings are superscheme isomorphisms given by the group product map.  So as  $ \, \bG = \bG_{\text{\it \!ev}} \cdot \bG_{\text{\it \!odd}}
 \, $  and  $ \, \bG_V \! = \big(\bG_V\big)_{\text{\it \!ev}} \! \cdot \! \big(\bG_V\big)_{\text{\it \!odd}}
 \, $,  it is enough to prove  $ \, \bG_{\text{\it \!ev}} = \big(\bG_V\big)_{\text{\it \!ev}} \, $  and  $ \, \bG_{\text{\it \!odd}} = \! \big(\bG_V\big)_{\text{\it \!odd}} \; $.

\smallskip

   First, the identity  $ \, \bG_{\text{\it \!ev}} = \big(\bG_V\big)_{\text{\it \!ev}} \, $  follows from  \S \ref{linear-G}.  Indeed, therein we last pointed out that (the copy of)  $ \, \bG_{\text{\it \!ev}} $  inside  $ \rGL(V) $  can be realized through the classical Chevalley's construction via the  $ \fg_\zero $--module  $ V $  and the lattice  $ M \, $;  but this is exactly the same outcome as first performing the construction of the supergroup  $ \bG_V $  and then taking its classical subgroup  $ \big(\bG_V\big)_{\text{\it \!ev}} \, $,  so we are done.

\smallskip

%
%
   Second, definitions yield  $ \, \bG_{\text{\it \!odd}} = I + T_{{}_I}\big(\bG_{\text{\it \!odd}}\big) \, $  as a supersubscheme of  $ {\rGL(V)}_\uno \, $,  where  $ T_{{}_I}\big( \bG_{\text{\it \!odd}} \big) $  is the tangent superspace to  $ \bG_{\text{\it \!odd}} $  at  $ I \, $;  similarly  $ \, \big(\bG_V\big)_{\text{\it \!odd}} = I + T_{{}_I}\big( \big(\bG_V\big)_{\text{\it \!odd}} \big) \, $.  But by construction we have also  $ \; T_{{}_I}\big( \bG_{\text{\it \!odd}} \big) = \fg_\uno = T_{{}_I}\big( \big(\bG_V\big)_{\text{\it \!odd}} \,\big) \, $,  \, hence  $ \; \bG_{\text{\it \!odd}} = \big(\bG_V\big)_{\text{\it \!odd}} \; $.
\end{proof}

\bigskip

\section{The standard case}  \label{stand-case}

\smallskip

   {\ } \quad   In this section we look somewhat in detail the example of the supergroup  $ \bG_{\Lambda(n)} $  associated with  $ \, \fg := W(n) \, $  and with the ``standard''  $ \fg $--module  $ \, V := \Lambda(n) \, $   --- i.e., the defining representation of  $ \, \fg := W(n) \, $.  Our analysis then can be easily adapted to the case  $ \, \fg := S(n) \, $  and  $ \, V := \Lambda(n) \, $  again.  More in general, as each Cartan type Lie superalgebra is naturally embedded in  $ W(n) \, $,  from the present analysis one can also deduce (with some extra work) a similar analysis for the other cases.

\medskip

 \subsection{The affine algebraic supergroup  $ \bG_{\Lambda(n)} $}  \label{supergroup-G_g}

\smallskip

   {\ } \quad   From now on, we retain the notation of  subsection \ref{Lie-superalg_Cartan-type},  and we let  $ \, \fg := W(n) = \Der_{\,\KK\!}\big(\Lambda(n)\big) \, $  and  $ \, V := \Lambda(n) \, $.  Fix the  $ \KK $--bases  $ \, B_{\Lambda(n)} := \big\{\, \underline{\xi}^{\underline{e}} \;\big|\; \underline{e} \in \! {\{0,1\}}^n \,\big\} \, $  in  $ \Lambda(n) $  and  $ \, B_{W(n)} := \big\{\, \underline{\xi}^{\underline{a}} \; \partial_i \;\big|\; \underline{a} \in \! {\{0,1\}}^n , \; i \! = \! 1, \dots, n \,\big\} \, $  in  $ W(n) $   --- see  subsections \ref{Lie-superalg_Cartan-type}, \ref{def-W(n)}  and  \ref{examples_che-bas}.  Recall that  $ B_{W(n)} $  is a  {\sl Chevalley basis\/}  of  $ \, \fg := W(n) \, $:  the root vectors  $ \, X_{\talpha} \; \big(\, \talpha \in \! \tDelta \,\big) \, $  are the  $ \, \underline{\xi}^{\underline{a}} \; \partial_i \, $  with  $ \, \underline{a} \not= \underline{e}_i \, $,  while the ``toral type'' elements  $ \, H_i \, $  are just the remaining elements  $ \, \xi_i \, \partial_i \, $  of  $ B_{W(n)} $  ($ \, i = 1, \dots, n \, $).

\smallskip

   By definition,  $ \fg $  acts on  $ \, V := \Lambda(n) \, $  by (super)derivations.  Explicitly, the action of any basis element in  $ B_{W(n)} $  onto any basis element in  $ B_{\Lambda(n)} $  reads
  $$  \underline{\xi}^{\underline{a}} \; \partial_i \big(\, \underline{\xi}^{\underline{e}} \big)  \; = \;  \pm \, \delta_{1,\underline{e}(i)} \, \underline{\xi}^{\underline{a} + \underline{e} - \underline{e}_i}   \eqno (5.1)  $$
    \indent   This simple formula has deep consequences.  The first is that  {\sl the  $ W(n) $--module  $ \Lambda(n) $  is rational},  as the  $ H_i $  act diagonally with integral eigenvalues.  A second consequence is that
  $$  {} \hskip-7pt   {\big(\, \underline{\xi}^{\underline{a}} \; \partial_i \big)}^2  = \;  0  \hskip13pt  \forall \;\, \underline{a} \not= \underline{e}_i \;\; ,   \qquad
      {\big(\, \xi_i \, \partial_i \big)}^2  = \;  \xi_i \, \partial_i  \qquad  \forall \;\, i = 1, \dots, n  \quad   \eqno (5.2)  $$
(note that  $ \, \underline{\xi}^{\underline{e}_i} = \xi_i \, $).  The left-hand part of (5.2) implies that all divided powers  $ \, X_{\talpha}^{(m)} \, $  of even root vectors  $ \, \big(\, \talpha \in \tDelta \,\big) \, $  with  $ \, m > 1 \, $  act as zero on  $ \Lambda(n) \, $.  From this and from (5.1) it follows at once that  {\sl the  $ \Z $--span  of  $ B_{\Lambda(n)} \, $,  call it  $ M $,  is an admissible lattice of  $ \Lambda \, $}.

\smallskip

   As another consequence, we can describe the one-parameter supersubgroups  $ x_{\talpha} $  and  $ h_i $   associated with root vectors  $ \, X_{\talpha} = \underline{\xi}^{\underline{a}} \, \partial_i \,\; \big( \underline{a} \! \not= \! \underline{e}_i \big) \, $  and ``toral'' elements  $ \, H_i = \xi_i \, \partial_i \, $.  For  $ x_{\talpha} $  one has
 $ \,\; x_{\talpha}(\bu) \, := \, \exp\big( \bu \, X_{\talpha} \big) \, = \, \sum_{m=0}^{+\infty} {\big( \bu \, X_{\talpha} \big)}^m \big/ m! \, = \, 1 + \bu \, X_{\talpha}  \, = \,  1 + \bu_{\,} \, \underline{\xi}^{\underline{a}} \; \partial_i \; $,  \,
 for any  $ \, A \in \salg_\bk \, $  and  $ \, \bu \in A_{p(\talpha)} \, $,  where  $ \, p\big(\talpha\big) \, $  is the parity of  $ \pi\big(\talpha\big) \, $.  Matching this with (5.1), the action of  $ \, x_{\talpha}(\bu) \, $  on basis elements of  $ \, \Lambda(n)(A) := A_\zero \otimes_\Z \! M_\zero + A_\uno \otimes_\Z \! M_\uno \, $  reads (for  $ \, \bt \in A_{p(\underline{e})} $)
  $$  \bt \, \underline{\xi}^{\underline{e}}
   \;\; {\buildrel {x_{\talpha}(\bu)} \over {\succ\joinrel\relbar\joinrel\longrightarrow}} \;\;
      x_{\talpha}(\bu) \big(\, \bt \, \underline{\xi}^{\underline{e}} \big)  \; = \;\,  \bt \, \underline{\xi}^{\underline{e}} + \bu \, \underline{\xi}^{\underline{a}} \; \partial_i \big(\, \bt \, \underline{\xi}^{\underline{e}} \big)  \; = \;\,  \bt \, \underline{\xi}^{\underline{e}} \pm \delta_{1,\underline{e}(i)} \, \bu \, \bt \, \underline{\xi}^{\underline{a} + \underline{e} - \underline{e}_i}  $$
Similarly,
 for  $ h_i $
%
%
 we find  $ \;\; \displaystyle{ \bt \, \underline{\xi}^{\underline{e}} \; {\buildrel {h_i(u)} \over {\succ\joinrel\relbar\joinrel\longrightarrow}} \; h_i(u) \big(\, \bt \, \underline{\xi}^{\underline{e}} \big) \, = \;  u^{\delta_{1,\underline{e}(i)}} \, \bt \, \underline{\xi}^{\underline{e}} } \;\; $  for all  $ \, \bt \in A_{p(\underline{e})} \, $.

\bigskip

   In particular, this yields the following ``point-set description'' of  $ \bG_{\Lambda(n)} \, $:

\bigskip

\begin{proposition}
 Let  $ \, \bG_{\Lambda(n)} $  be the supergroup associated with  $ \, \fg := W(n) \, $  and the  $ W(n) $--module  $ \Lambda(n) $  as in  Section \ref{car-sgroups},  and let us fix total orders in  $ \, \tDelta_{\zero^{\,\uparrow}} $  and in  $ \, \tDelta_\uno \, $  such that  $ \, \tDelta_\uno^- \! \preceq \tDelta_\uno^+ \, $  or  $ \, \tDelta_\uno^+ \! \preceq \tDelta_\uno^- \, $.  Then for any  $ \, A \in \salg_\bk \, $  the group  $ \, \bG_{\Lambda(n)}(A) $  is given
%
%
by
 \vskip4pt
   \centerline{ $ \displaystyle{ \bG_{\Lambda(n)}(A)  \,\; = \;\,  \rGL_n(A_\zero) \times \big(\, \bigtimes_{\, \talpha \in \tDelta_{\zero^{\,\uparrow}}\!} (1 + A_\zero \, X_{\talpha}) \big) \times \big(\, \bigtimes_{\, \tgamma \in \tDelta_\uno} (1 + A_\uno \, X_{\tgamma}) \big) } $ }
 \vskip1pt
\noindent
 (the products indexed by  $ \tDelta_{\zero^{\,\uparrow}} \! $  or by  $ \tDelta_\uno $  being ordered according to the fixed orders), as well as by all set-theoretic factorizations that one gets by permuting the three factors above with one another.
\end{proposition}

\begin{proof} As we noticed in  \S \ref{weight-bG_V},  $ \bG_{\Lambda(n)} $  factors into  $ \; \bG_{\Lambda(n)} \, \cong \, \bG_0 \times \bG_{\zero^{\,\uparrow}} \times \bG_\uno^< \; $,  and in addition  $ \; \bG_\uno^< \cong \bigtimes_{\, \tgamma \in \tDelta_\uno} x_{\tgamma}(A) \, $,  $ \; \bG_{\zero^{\,\uparrow}} \cong \bigtimes_{\, \talpha \in \tDelta_{\zero^{\,\uparrow}}\!} x_{\talpha}(A) \; $  and  $ \; \bG_0 \cong \mathbf{Ch}_{\Lambda(n)} \, $. The latter is the standard (affine, algebraic) group functor associated by the classical Chevalley construction with  $ \, \fg_0 \cong \rgl_n \, $  and with the  $ \rgl_n $--module  $ \Lambda(n) \, $:  but the very construction clearly gives  $ \, \bG_0 \cong \mathbf{Ch}_{\Lambda(n)} \cong \rGL_n \; $.  Finally, taking into account the additional remark that  $ \; x_{\talpha}(A) = (1 + A_\zero \, X_{\talpha}) \; $  for  $ \, \talpha \in \tDelta_{\zero^{\,\uparrow}} \, $  and  $ \; x_{\tgamma}(A) = (1 + A_\uno \, X_{\tgamma}) \; $  for  $ \, \tgamma \in \tDelta_\uno \, $   --- by the above analysis ---   we end up with the claim.
\end{proof}

\medskip

\begin{remark}
 The factorization of  $ \bG_{\Lambda(n)}(A) $  in the above Proposition is a special instance of the general result in  \S \ref{weight-bG_V}.  But the present case is much easier to handle, as commutation relations among one-parameter supersubgroups (as in  Lemma \ref{comm_1-pssg})  look simpler: e.g., for  $ \, \talpha, \tbeta \in \tDelta_{\zero^{\,\uparrow}} \coprod \tDelta_\uno \, $  one has
 $ \; \big( x_{\talpha}(\mathbf{p}) \, , \, x_{\tbeta}(\mathbf{q}) \big) = \big( 1 + \mathbf{p} \, X_{\talpha} \, , 1 + \mathbf{q} \, X_{\tbeta} \big) = 1 + \mathbf{p} \, \mathbf{q} \, \big[ X_{\talpha} \, , X_{\tbeta\,} \big] \; $
where (cf.~Examples \ref{examples_che-bas}{\it (a)\/})  the bracket  $ \, \big[ X_{\talpha} \, , X_{\tbeta\,} \big] \, $  is either zero, or a root vector, or a sum (with signs) of two such vectors.
\end{remark}

\medskip

 \subsection{$ \bG_{\Lambda(n)} $  as a supergroup of automorphisms}  \label{G_g-automorf}

\smallskip

   {\ } \quad   In the present subsection we prove that the supergroup functor  $ \bG_{\Lambda(n)} $  actually is a group functor of automorphisms, namely the group functor of superalgebra automorphisms canonically associated with the  $ \bk $--superalgebra  $ \Lambda(n) \, $.  We begin with a (general) definition.

\medskip

\begin{definition}
 Let  $ \, \mathfrak{A} \in \salg_\bk \, $  be any  $ \bk $--superalgebra  which, as a  $ \bk $--module,  is free of finite rank.  We define the supergroup functor  $ \; \mathbf{Aut}\big(\mathfrak{A}\big) : \salg_\bk \longrightarrow \grps \; $  as the full subfunctor of the group functor  $ \rGL(\mathfrak{A}_\bullet) $   --- cf.~\ref{exs-supvecs}{\it (b)}  ---   whose value on objects is  $ \; \mathbf{Aut}\big(\mathfrak{A}\big)(A) := \text{\sl Aut}_{\salg_{\!A}} \! \big( \mathfrak{A}_A \big) \, $ --- the group of all  $ A $--linear  superalgebra automorphisms of  $ \; \mathfrak{A}_A := A \otimes_\bk \mathfrak{A} \; $   --- for all  $ \, A \in \salg_\bk \, $.
\end{definition}

\smallskip

\begin{free text}  \label{Aut(Lambda(n))}
 {\bf The group functor  $ \, \mathbf{Aut}\big(\Lambda(n)\big) \, $.}  Given the  $ \bk $--superalgebra  $ \, \mathfrak{A} = \Lambda(n) \, $,  we are interested into  $ \, \mathbf{Aut}\big(\Lambda(n)\big) \; $:  our ultimate goal is to show that  $ \, \bG_{\Lambda(n)} = \mathbf{Aut}\big(\Lambda(n)\big) \; $.  Note that
   $$  \mathbf{Aut}\big(\Lambda(n)\big)(A)  \, = \,  \text{\sl Aut}_{\salg_{\!A}} \! \big( {\Lambda(n)}_A \big)  \, = \,  \text{\sl Aut}_{\salg_{\!A}} \! \big( A[\xi_1,\dots,\xi_n] \big)   \eqno \forall  \;\; A \in \salg_\bk  \quad  (5.3)  $$
because  $ \; {\Lambda(n)}_A := A \otimes_\bk \Lambda(n) = A \otimes_\bk \bk[\xi_1,\dots,\xi_n] = A[\xi_1,\dots,\xi_n] \; $.  Now, given  $ \,  A \in \salg_\bk \, $,  any  $ \, \phi \in \text{\sl Aut}_{\salg_{\!A}} \! \big( A[\xi_1,\dots,\xi_n] \big) \, $  is uniquely determined by the images of the  $ \xi_j \, $:  these are of the form
  $$  \phi(\xi_j)  \,\; = \;\,  \kappa_j \, + \, {\textstyle \sum_{i=1}^n} \, c_{i,j} \, \xi_i \; + {\textstyle \sum_{\hskip-1pt {{|\underline{e}| > 1} \atop {|\underline{e}| \text{\ is even}}}}} \hskip-2pt \kappa_{\underline{e},j} \, \underline{\xi}^{\underline{e}} \; + {\textstyle \sum_{\hskip-1pt {{|\underline{e}| > 1} \atop {|\underline{e}| \text{\ is odd}}}}} \hskip-2pt c_{\underline{e},j} \, \underline{\xi}^{\underline{e}}   \eqno \forall \;\; j=1, \dots, n  \qquad  (5.4)  $$
with the only constraints that each  $ \, \phi(\xi_j) \, $  be again  {\sl odd},  which means  $ \; \kappa_j , \kappa_{\underline{e},j} \in A_\uno \, $,  $ \; c_{i,j} , c_{\underline{e},j} \in A_\zero \; $, {\sl and\/}  that  $ \phi $  itself be invertible.  By the nilpotency of the  $ \xi_t \, $,  it is clear that  $ \phi $  is invertible if and only if the square matrix of the  $ c_{i,j} $'s  is invertible, i.e.~$ \, C_\phi := {\big( c_{i,j} \big)}_{i=1,\dots,n;}^{j=1,\dots,n;} \in \rGL_n(A_\zero) \; $.  Note also that (5.4) means that  $ \, \phi \in \text{\sl Aut}_{\salg_{\!A}} \! \big( {\Lambda(n)}_A \big) =
 \text{\sl Aut}_{\salg_{\!A}} \! \big( A[\xi_1,\dots,\xi_n] \big) \, $  can be written as
  $$  \phi  \;\; = \;  {\textstyle \sum\limits_{j=1}^n} \kappa_j \, \partial_j \, + \, {\textstyle \sum\limits_{i=1}^n} {\textstyle \sum\limits_{j=1}^n} \, c_{i,j} \, \xi_i \, \partial_j \, + \, {\textstyle \sum\limits_{j=1}^n} {\textstyle \sum_{\hskip-2pt {{|\underline{e}| > 1} \atop {|\underline{e}| \text{\ is even}}}}} \hskip-2pt \kappa_{\underline{e},j} \, \underline{\xi}^{\underline{e}} \; \partial_j \, + \, {\textstyle \sum\limits_{j=1}^n} {\textstyle \sum_{\hskip-2pt {{|\underline{e}| > 1} \atop {|\underline{e}| \text{\ is odd}}}}} \hskip-2pt c_{\underline{e},j} \, \underline{\xi}^{\underline{e}} \; \partial_j   \eqno (5.5)  $$
Thus every  $ \, \phi \in \text{\sl Aut}_{\salg_{\!A}} \! \big( {\Lambda(n)}_A \big) \, $  is uniquely associated with a string of coefficients: the  $ \kappa_j \, $,  the  $ \kappa_{\underline{e},j} \, $,  the  $ c_{i,j} $  and the $ c_{\underline{e},j} $  as above.  Therefore, the overall conclusion is the following:
\end{free text}

\medskip

\begin{proposition}
 The group functor  $ \, \mathbf{Aut}\big(\Lambda(n)\big) \, $  is  {\sl representable}   --- hence it is an (affine) supergroup scheme ---   and isomorphic, as a superscheme, to  $ \; \mathbb{A}^{0|n} \times \rGL_n \times \mathbb{A}^{|\,\tDelta_{{\zero^{\uparrow\!}}}|\big|\,0} \times \mathbb{A}^{0\big||\,\tDelta_{{\uno^{\uparrow\!}}}|} \;\, $.
\end{proposition}

\medskip

   We are now ready for the main result of this subsection:

\medskip

\begin{theorem}  \label{G_Ln-Aut}
  $ \, \bG_{\Lambda(n)} = \mathbf{Aut}\big(\Lambda(n)\big) \; $,  that is  $ \, \bG_{\Lambda(n)} $  coincides with the group functor  $ \, \mathbf{Aut}\big(\Lambda(n)\big) \, $.
\end{theorem}

\begin{proof}
 By construction, we must prove that  $ \; \bG_{\Lambda(n)}(A) \, = \, \mathbf{Aut}\big(\Lambda(n)\big)(A) \, := \, \text{\sl Aut}_{\salg_{\!A}} \! \big( {\Lambda(n)}_A \big) \; $  with  $ \, {\Lambda(n)}_A \! := A \otimes_\bk \Lambda(n) = A[\xi_1,\dots,\xi_n] \, $,  $ \, A \! \in \! \salg_\bk \, $.  We begin by  $ \, \bG_{\Lambda(n)}\!(A) \subseteq \text{\sl Aut}_{\salg_{\!A}} \! \big( {\Lambda(n)}_{\!A} \big) \, $.

\vskip4pt

   By  Remarks \ref{rems-af-defs_Chev-sgrps}{\it (d)},
 $ \bG_{\Lambda(n)}(A) $  is a subgroup of  $ \rGL(V_A) \, $:  we must only prove that its elements are superalgebra automorphisms. As  $ \bG_{\Lambda(n)} $  is the sheafification of  $ G_{\Lambda(n)} \, $,  and  $ \mathbf{Aut} $  is a sheaf, it is enough to prove that  $ \; G_{\Lambda(n)}(A) \subseteq \text{\sl Aut}_{\salg_{\!A}} \! \big( {\Lambda(n)}_A \big) \; $.  Now, the group  $ \, G_{\Lambda(n)}(A) \, $  is generated by such elements as
 $ \; x_{\talpha}(t) := \exp\big(t\,X_{\talpha}\big) \, $,  $ \, x_{\tbeta}(\vartheta) := \exp\big(\vartheta\,X_{\tbeta}\big) \, $,  $ \, h_i(u) := u^{H_i} \, $;
 both  $ X_{\talpha} $  and  $ X_{\tbeta} $  are superderivations of  $ \, \Lambda(n) \, $,  hence they also define (uniquely)  $ A $--linear  superderivations of  $ \, {\Lambda(n)}_A := A \otimes_\bk \Lambda(n) \; $.  But then both  $ \, t \, X_{\talpha} \, $  and  $ \, \vartheta \, X_{\tbeta} \, $  are  $ A $--linear  {\sl derivations\/}  of  $ {\Lambda(n)}_A $  into itself: taking their exponentials we get ($ A $--linear)  {\sl automorphisms\/}  of  $ {\Lambda(n)}_A \, $,  so that  $ \; x_{\talpha}(t) , \, x_{\tbeta}(\vartheta) \in \text{\sl Aut}_{\salg_{\!A}} \! \big( {\Lambda(n)}_A \big) \; $.  A similar argument proves  $ \; h_i(u) \in \text{\sl Aut}_{\salg_{\!A}} \! \big( {\Lambda(n)}_A \big) \; $,  hence we are done.

\vskip5pt

   Now we prove that the above inclusion is an identity.  We begin with an aside observation: by the explicit description of automorphisms in (5.5), one sees that for each  $ \, A \in \salg_\bk \, $  the subsets
  $$  \displaylines{
   {\mathbf{Aut}\big(\Lambda(n)\big)}_{\leq 0}(A)  \; := \;  \big\{\, \phi \in \mathbf{Aut}\big(\Lambda(n)\big)(A) \,\big|\, \kappa_{\underline{e},j} = 0 = c_{\underline{e},j} \; \forall \, \underline{e}, \, j \,\big\}  \cr
   {\mathbf{Aut}\big(\Lambda(n)\big)}_{0^\uparrow}(A)  \; := \;  \big\{\, \phi \in
\mathbf{Aut}\big(\Lambda(n)\big)(A) \,\big|\, \kappa_j = 0 = c_{i,j} \; \forall \, i, j \,\big\}  }  $$
are subgroups of  $ \, \mathbf{Aut}\big(\Lambda(n)\big)(A) \, $,  which altogether generate  $ \, \mathbf{Aut}\big(\Lambda(n)\big)(A) \; $.  This defines two supersubgroups  $ \, {\mathbf{Aut}\big(\Lambda(n)\big)}_{\leq 0} \, $  and  $ \, {\mathbf{Aut}\big(\Lambda(n)\big)}_{0^\uparrow} \, $  which jointly generate  $ \, \mathbf{Aut}\big(\Lambda(n)\big) \; $.
                                                                  \par
   The first supersubgroup  $ \, {\mathbf{Aut}\big(\Lambda(n)\big)}_{\leq 0} \, $  is isomorphic to the algebraic group  $ \, \mathbf{Aff}_n = \mathbb{G}_a^{\times n} \rtimes \rGL_n \, $  of all affine-linear transformations of the (totally odd) affine superspace  $ \mathbb{A}^{0|n} \, $.  It contains the subgroup  $ G'_{\leq 0}(A) $  generated by all the elements  $ \; (1 + \vartheta \, \partial_j) = x_{-\varepsilon_j}\!(\vartheta) \, $,  $ \, (1 + t \, \xi_i \, \partial_j) = x_{\varepsilon_i - \varepsilon_j}\!(t) \, $  and  $ \, h_i(u) \, $  of  $ \bG_{\Lambda(n)}(A) $   --- for  $ \, \vartheta \in A_\uno \, $,  $ \, t \in A_\zero \, $,  $ \, u \in U(A_\zero) \, $,  $ \, i, j = 1, \dots, n \, $.  All these  $ G'_{\leq 0}(A) $  define a supergroup functor, whose sheafification  $ \bG'_{\leq 0} $  clearly  {\sl coincides\/}  with  $ \, {\mathbf{Aut}\big(\Lambda(n)\big)}_{\leq 0} \; $.
                                                                  \par
   The second supersubgroup  $ \, {\mathbf{Aut}\big(\Lambda(n)\big)}_{0^\uparrow} \, $  contains the subgroup  $ G'_{0^\uparrow}\!(A) $  generated by the elements
 $ \, (1 + \bt \, \underline{\xi}^{\underline{e}} \, \partial_j) = x_{\alpha_{\underline{e},j}}\!(\bt) \, $   ---  for  $ \, \bt \in A_\zero \cup A_\uno \, $,  $ \, i, j = 1, \dots, n \, $,  where  $ \, \alpha_{\underline{e},j} \, $  is the unique element of  $ \tDelta $ associated with  $ \underline{e} $  and  $ j \, $.  We shall now show that  $ G'_{0^\uparrow}\!(A) $  {\sl coincides\/}  with  $ {\mathbf{Aut}\big(\Lambda(n)\big)}_{0^\uparrow} \, $:  by the previous analysis, this will be enough to prove that  $ \, \mathbf{Aut}\big(\Lambda(n)\big) = \bG_V \; $.  Consider the subsets
  $$  {\mathbf{Aut}\big(\Lambda(n)\big)}_{>t}(A)  \; := \;  \big\{\, \phi \in
\mathbf{Aut}\big(\Lambda(n)\big)(A) \,\big|\, \kappa_j = 0 = c_{i,j} \, , \, \kappa_{\underline{e},j} = 0 = c_{\underline{e},j} \; \forall \, i, j \, , \, \forall \, |\underline{e}| \leq t \,\big\}  $$
for all  $ \, t = 2, \dots, n \, $;  then easy computations, basing upon (5.5) and upon the formula
  $$  \big( 1 + \bt' \, \underline{\xi}^{\underline{a}} \, \partial_j \big) \, \big( 1 + \bt'' \, \underline{\xi}^{\underline{b}} \, \partial_k \big)  \; = \;  1 \, + \,  \bt' \, \underline{\xi}^{\underline{a}} \, \partial_j \, + \, \bt'' \, \underline{\xi}^{\underline{b}} \, \partial_k \, \pm \, \delta_{\underline{b}(j),1} \, \bt' \, \bt'' \, \underline{\xi}^{\underline{a} + \underline{b} - \underline{e}_j} \, \partial_k  $$
show that these subsets form a strictly decreasing sequence of normal subgroups  $ \, {\mathbf{Aut}\big(\Lambda(n)\big)}_{0^\uparrow} \; $,  which ends with the trivial subgroup.  It is immediate to see that  $ \; {\mathbf{Aut}\big(\Lambda(n)\big)}_{0^\uparrow} \Big/ {\mathbf{Aut}\big(\Lambda(n)\big)}_{>2}(A) \; $  is generated by the cosets  $ \; (1 + \bt \, \underline{\xi}^{\underline{e}} \, \partial_j) \mod {\mathbf{Aut}\big(\Lambda(n)\big)}_{>2}(A) \; $  with $ \, |\underline{e}| = 2 \; $.  Similarly, one sees easily for all  $ t $  (by iteration) that  $ \; {\mathbf{Aut}\big(\Lambda(n)\big)}_{0^\uparrow} \Big/ {\mathbf{Aut}\big(\Lambda(n)\big)}_{>t}(A) \; $  is generated by the cosets  $ \; (1 + \bt \, \underline{\xi}^{\underline{e}} \, \partial_j) \mod {\mathbf{Aut}\big(\Lambda(n)\big)}_{>1}\!(A) \; $  with $ \, 2 \! \leq|\underline{e}| \! \leq t \, $.  For  $ \, t \! = \! n \, $  this yields the expected result.
\end{proof}

\medskip

  \subsection{Special supersubgroups of  $ \bG_{\Lambda(n)} $}  \label{spec_supsubgroups_G_Lbdn}

\smallskip

   {\ } \quad   We finish this section with an explicit description of the special supersubgroups of  $ \bG_{\Lambda(n)} $  that we considered along the way   --- cf.~\S \ref{even-part}  and  \S \ref{spec_supsubgroup_G_V}  ---   namely  (for all  $ \, t \geq -1 \, $)
  $$  \bG_{-1} \;\; ,  \;\quad  \bG_0 \;\; ,  \;\quad  \bG_{-1,0} \;\; ,  \;\quad  \bG_\zero \;\; ,  \;\quad \bG_{\zero^\uparrow}  \;\; ,  \;\quad  \bG_{t^\uparrow}  \;\; ,  \;\quad  {\big( \bG_{t^\uparrow} \big)}_\zero \;\; ,  \;\quad  \bG_{\text{\rm min}}^\pm  \;\; ,  \;\quad  \bG_{\text{\rm max}}^\pm  $$
Such a description follows from  Propositions \ref{descr-G_zero}, \ref{descr-G_zero^pm}, \ref{semi-direct-G_zero}, \ref{factorization}  and  \ref{props_supsubgrps};  using the identification  $ \; \bG_{\Lambda(n)} = \mathbf{Aut}\big(\Lambda(n)\big) \; $  in  Theorem \ref{G_Ln-Aut}  and by (5.5), all those results yield easily the following:

\medskip

\begin{proposition}
 For every  $ \, A \in \salg_\bk \, $,  we have:
 \vskip4pt
   (a)  $ \; \bG_{-1}(A) = \big\{\, \phi \in \mathbf{Aut}\big(\Lambda(n)\big)(A) \,\big|\, c_{i,j} = \kappa_{\underline{e},j} = c_{\underline{e},j} = 0 \,\; \forall \; \underline{e} \, , \, i, j \,\big\} \, \cong \, \mathbb{G}_{a,\text{\it odd}}^{\,\times n}(A) \; $,  \, so that  $ \; \bG_{-1} \, \cong \, \mathbb{G}_{a,\text{\it odd}}^{\,\times n} \; $,  where  $ \, \mathbb{G}_{a,\text{\it odd}} \, $  is defined on objects by  $ \, A \mapsto \mathbb{G}_{a,\text{\it odd}}(A) := A_\uno \; $  (as additive group);
 \vskip4pt
   (b)  $ \; \bG_0(A) = \big\{\, \phi \in \mathbf{Aut}\big(\Lambda(n)\big)(A) \,\big|\, \kappa_j = \kappa_{\underline{e},j} = c_{\underline{e},j} = 0 \,\; \forall \; \underline{e} \, , \, j \,\big\} \, \cong \, \rGL_n(A) \; $,  \, so that  $ \; \bG_0 \, \cong \, \rGL_n \; $,  \, where  $ \, \rGL_n \, $  is the classical general linear affine group extended to superalgebras  via  $ \, \salg_\bk \ni A \mapsto \rGL_n(A) := \rGL_n(A_\zero) \; $;
 \vskip4pt
   (c)  $ \; \bG_{-1,0}(A) = \big\{\, \phi \in \mathbf{Aut}\big(\Lambda(n)\big)(A) \,\big|\, \kappa_{\underline{e},j} = c_{\underline{e},j} = 0 \,\; \forall \; \underline{e} \, , \, j \,\big\} =: {\mathbf{Aut}\big(\Lambda(n)\big)}_{\leq 0}(A) \; $,  \, so that  $ \; \bG_{-1,0} = \bG_{-1} \rtimes \bG_0 \, \cong \, \mathbb{G}_{a,\text{\it odd}}^{\,\times n} \rtimes \rGL_n =: \mathbf{Aff}_{0|n} \; $,  \, the latter being the (classical) algebraic group of all affine-linear transformations of the totally odd affine superspace  $ \, \mathbb{A}^{0|n} \, $;
 \vskip3pt
   (d)  $ \; \bG_{\zero^\uparrow}(A) = \big\{\, \phi \in \mathbf{Aut}\big(\Lambda(n)\big)(A) \,\big|\, \kappa_j = c_{i,j} = \kappa_{\underline{e},j} = 0 \,\; \forall \; \underline{e} \, , \, i, j \,\big\} \; $,  \, so that  $ \; \bG_{\zero^\uparrow} \, \cong \, \mathbb{A}_\bk^{N_{\zero^\uparrow} \!|\, 0} \; $  as affine superschemes, where  $ \; N_{\zero^\uparrow} := \big| \tDelta_{\zero^{\,\uparrow}} \big| = \sum_{z>0} \big| \tDelta_z \cap \tDelta_\zero \big| \; $;
 \vskip4pt
   (e)  $ \; \bG_\zero(A) = \big\{\, \phi \in \mathbf{Aut}\big(\Lambda(n)\big)(A) \,\big|\, \kappa_j = \kappa_{\underline{e},j} = 0 \,\; \forall \; \underline{e} \, , \, j \,\big\} \; $,  \, so that  $ \; \bG_\zero = \bG_0 \ltimes \bG_{\zero^\uparrow} \, \cong \, \rGL_n \times \mathbb{A}_\bk^{N_{\zero^\uparrow} \!|\, 0} \; $  as affine superschemes;
 \vskip4pt
   (f)  $ \; \bG_{t^\uparrow\!}(A) = \big\{\, \phi \in \mathbf{Aut}\big(\Lambda(n)\big)(A) \,\big|\, \kappa_j = c_{i,j} = \kappa_{\underline{e},j} = c_{\underline{e},j} = 0 \,\; \forall \; i, j, \; \forall \, \underline{e} : |\underline{e}| \leq t+1 \,\big\} \; $  for all  $ \, t > -1 \, $,  therefore  $ \; \bG_{t^\uparrow\!} \, \cong \, \mathbb{A}_\bk^{N_{t^\uparrow}^\zero |\, N_{t^\uparrow}^\uno} \, $  as affine superschemes,  where  $ \, N_{t^\uparrow}^{\overline{s}} := \! \sum_{z>t} \big| \tDelta_z \cap \tDelta_{\overline{s}} \big| \; $;  in particular,  $ \; \bG_{0^\uparrow\!}(A) = \big\{\, \phi \in \mathbf{Aut}\big(\Lambda(n)\big)(A) \,\big|\, \kappa_j = c_{i,j} = 0 \,\; \forall \; i, j \,\big\} \, $,  \, hence  $ \; \bG_{0^\uparrow\!} \, \cong \, \mathbb{A}_\bk^{N_{0^\uparrow}^+ |\, N_{0^\uparrow}^-} \; $;
 \vskip1pt
   (g)  $ \; \bG_{-1^\uparrow\!}(A) = \big\{\, \phi \in \mathbf{Aut}\big(\Lambda(n)\big)(A) \,\big|\, \kappa_j = 0 \,\; \forall \, j \,\big\} \; $,  so  $ \; \bG_{-1^\uparrow} = \bG_0 \ltimes \bG_{0^\uparrow} \, \cong \, \rGL_n \ltimes \mathbb{A}_\bk^{N_{0^\uparrow}^+ |\, N_{0^\uparrow}^-} \; $  as affine superschemes;
 \vskip7pt
   (h)  $ \; {\big( \bG_{t^\uparrow\!} \big)}_\zero(A) = \big\{\, \phi \in \mathbf{Aut}\big(\Lambda(n)\big)(A) \,\big|\, \kappa_j = c_{i,j} = \kappa_{\underline{e},j} = c_{\underline{e}',j} = 0 \,\; \forall \; i, j, \underline{e} \, , \; \forall \, \underline{e}' : |\underline{e}'| \leq t+1 \,\big\} \; $  for all  $ \, t \! > \! -1 \, $,  therefore  $ \; {\big( \bG_{t^\uparrow\!} \big)}_\zero \, \cong \, \mathbb{A}_\bk^{N_{t^\uparrow}^+ |\, 0} \; $  as (totally even) affine superschemes;  in particular,  $ \; {\big( \bG_{0^\uparrow\!} \big)}_\zero(A) = \big\{\, \phi \in \mathbf{Aut}\big(\Lambda(n)\big)(A) \,\big|\, \kappa_j = c_{i,j} = \kappa_{\underline{e},j} = 0 \,\; \forall \; i, j, \underline{e} \,\big\} \; $,  \, hence  $ \; \bG_{0^\uparrow\!} \, \cong \, \mathbb{A}_\bk^{N_{0^\uparrow}^+ |\, 0} \; $;
 \vskip7pt
   (i)  $ \; {\big( \bG_{-1^\uparrow\!} \big)}_\zero(A) = \big\{\, \phi \in \mathbf{Aut}\big(\Lambda(n)\big)(A) \,\big|\, \kappa_j = \kappa_{\underline{e},j} = 0 \,\; \forall \; i, j, \underline{e} \,\big\} \; $,  \, thus  $ \; \bG_{-1^\uparrow} = \bG_0 \ltimes \bG_{0^\uparrow} \, \cong \, \rGL_n \ltimes \mathbb{A}_\bk^{N_{0^\uparrow}^+ |\, 0} \; $  as (totally even) affine superschemes;
 \vskip7pt
   (j)  \, let  $ \, \tDelta_0 = \tDelta_0^+ \coprod \tDelta_0^- \, $  be the splitting of the root system  $ \, \tDelta_0 = \Delta_0 \, $  of\/  $ \, \fg_0 = \rgl_n \, $  given by  $ \, \tDelta_0^+ := \big\{\, \varepsilon_i - \varepsilon_j \,\big|\, 1 \leq i < j \leq n \,\big\} \, $  and  $ \, \tDelta_0^- := \big\{\, \varepsilon_i - \varepsilon_j \,\big|\, 1 \leq j < i \leq n  \,\big\} \, $,  and define\/  $ \bG_{\text{\rm min}}^- $  and\/  $ \bG_{\text{\rm max}}^+ $  accordingly as in  Definition \ref{spec_supsubgrps-def}.  Then we have
  $$  \displaylines{
   \bG_{\text{\rm min}}^-(A)  \; = \;  \big\{\, \phi \in \mathbf{Aut}\big(\Lambda(n)\big)(A) \;\big|\; c_{i,j} = \kappa_{\underline{e},j} = c_{\underline{e},j} = 0 \;\; \forall \;\, i, j, \underline{e} : i < j \,\big\}  \,\; \cong \;\,  \bG_{-1\!} \rtimes \mathbf{B}^-  \cr
   \bG_{\text{\rm max}}^+(A)  \,\; = \;\,  \big\{\, \phi \in \mathbf{Aut}\big(\Lambda(n)\big)(A) \;\big|\; \kappa_j = c_{i,j} = 0 \;\; \forall \;\, i, j : i > j \,\big\}  \,\; \cong \;\,  \mathbf{B}^+ \! \ltimes \bG_{0^\uparrow}  }  $$
where  $ \, \mathbf{B}^\pm $  is the Borel subgroup of  $ \, \bG_0 = \rGL_n \, $  of all invertible upper/lower triangular matrices.
\end{proposition}


\bigskip
\bigskip


\begin{thebibliography}{vsv}

\vskip7pt

\bibitem{bmpz} Y.~A.~Bahturin, A.~A.~Mikhalev, V.~M.~Petrogradsky, M.~V.~Zaicev,  {\it Infinite-dimensional Lie superalgebras},  de Gruyter Expositions in Mathematics  {\bf 7},  Walter de Gruyter \& Co., Berlin, 1992.

\bibitem{bk} J.~Brundan, A.~Kleshchev, {\it Modular representations of the supergroup  $ Q(n) $,  I}, J.~Algebra  {\bf 260}  (2003), 64--98.

\bibitem{ccf} C.~Carmeli, L.~Caston, R.~Fioresi, {\it  Mathematical Foundations of Supersymmetry},  EMS Series of Lectures in Mathematics  {\bf 15},  European Mathematical Society, Z{\"u}rich, 2011.

\bibitem{dm} P.~Deligne, J.~Morgan,  {\it Notes on supersymmetry (following J.~Bernstein)},  in:  P.~Deligne et al.~(eds.),  {\sl ``Quantum fields and strings. A course for mathematicians''},  Vol.~1, 2 (Princeton, NJ, 1996/1997), 41-–97, American Mathematical Society, Providence, RI, 1999.

\bibitem{de1} M.~Demazure,  {\it Groupes r{\'e}ductifs: d{\'e}ploiements, sous-groupes, groupes-quotients},  in: M.~Demazure, A.~Grothendieck (eds.),  {\sl ``Structure des Sch{\'e}mas en Groupes R{\'e}ductifs''},  Sch{\'e}mas en Groupes (SGA 3), Lecture Notes in Math.~{\bf 153}, Springer-Verlag, Berlin-New York, 1962/64, Expos{\'e} XXII, 156--262.

\bibitem{de2} M.~Demazure,  {\it Groupes r{\'e}ductifs: unicit{\'e} des groupes {\'e}pingl{\'e}s},  in: M.~Demazure, A.~Grothen\-dieck (eds.),  {\sl ``Structure des Sch{\'e}mas en Groupes R{\'e}ductifs''},  Sch{\'e}mas en Groupes (SGA 3), Lecture Notes in Math.~{\bf 153}, Springer-Verlag, Berlin-New York, 1962/64, Expos{\'e} XXIII, 263--409.

\bibitem{du} M.~Duflo,  {\it Private communication},  2011.

\bibitem{fg1} R.~Fioresi, F.~Gavarini,  {\it On the construction of Chevalley supergroups},  in: S.~Ferrara, R.~Fio\-resi, V.~S.~Varadarajan (eds.),  {\sl ``Supersymmetry in Mathematics \& Physics''},  UCLA Los Angeles, U.S.A.~2010, Lecture Notes in Math.~{\bf 2027}, Springer-Verlag, Berlin-Heidelberg, 2011, pp.~101--123.

\bibitem{fg2} R.~Fioresi, F.~Gavarini,  {\it Chevalley Supergroups},  Memoirs Amer.~Math.~Soc.~{\bf 215},  no.~1014 (2012).

\bibitem{fg3} R.~Fioresi, F.~Gavarini,  {\it Algebraic supergroups with Lie superalgebras of classical type},  Journal of Lie Theory (to appear)   --- preprint  {\tt arXiv:1106.4168v4}  [math.RA], 17 pages (2011).

\bibitem{ga} F.~Gavarini,  {\it Chevalley Supergroups of type  $ D(2,1;a) $},  preprint  {\tt arXiv:1006.0464v1}  [math.RA] (2010).

\bibitem{gr} D.~Grantcharov,  {\it On the structure and characters of weight modules},  Forum Mathematicum  {\bf 18}  (2006), 933--950.

\bibitem{hu} J.~E.~Humphreys,  {\it Introduction to Lie Algebras and Representation Theory},  Graduate Texts in Mathematics  {\bf 9},  Springer-Verlag, New York, 1972.

\bibitem{ka} V.~G.~Kac,  {\it Lie superalgebras},  Advances in Mathematics  {\bf 26}  (1977), 8--26.

\bibitem{ma} Y.~Manin,  {\it Gauge field theory and complex geometry},  Grundlehren der Mathematischen Wissenschaften  {\bf 289},  Springer-Verlag, Berlin, 1988.

\bibitem{se} V.~Serganova, {\it On Representations of Cartan Type Lie Superalgebras},  in: E.~Vinberg (ed.), {\sl ``Lie Groups and Invariant Theory''},  223--239, American Mathematical Society Translations (2)  {\bf 213},  American Mathematical Society, Providence, RI, 2005.

\bibitem{st} R.~Steinberg,  {\it Lectures on Chevalley groups},  Yale University, New Haven, Conn., 1968.

\bibitem{vsv} V.~S.~Varadarajan,  {\it Supersymmetry for mathematicians: an introduction},  Courant Lecture Notes  {\bf 1},  American Mathematical Society, Providence, RI, 2004.

\bibitem{vst} A.~Vistoli,  {\it Grothendieck topologies, fibered categories and descent theory}, in: Fundamental algebraic geometry,  1--104, Math.~Surveys Monogr.  {\bf 123},  American Mathematical Society, Providence, RI, 2005.

\end{thebibliography}
\end{document}


{\ }

\vskip-47pt

   \centerline{\smallrm To appear in  {\smallsl Forum Mathematicum}.}
\vskip11pt   {\ }

\centerline{\Large \bf ERRATUM TO}
 \vskip13pt
\centerline{\Large \bf ``ALGEBRAIC SUPERGROUPS OF CARTAN TYPE''}
%
%

\vskip23pt

\centerline{ Fabio GAVARINI }

\vskip11pt

\centerline{\it Dipartimento di Matematica, Universit\`a di Roma ``Tor Vergata'' } \centerline{\it via della ricerca scientifica 1  --- I-00133 Roma, Italy}

\centerline{{\footnotesize e-mail: gavarini@mat.uniroma2.it}}

\vskip31pt

\begin{abstract}
 \vskip2pt
 \hskip-9pt   \footnote{\ 2010 {\it MSC}\;: \, Primary 14M30, 14A22; Secondary 17B20.}
%
   In this note I fix a mistake in my previous paper  \cite{ga2}:
%
%
 namely, the result concerning the uniqueness (up to isomorphisms) of such supergroups needs a new formulation and proof.  By the same occasion, I explain more in detail the existence result which comes out of the construction of Chevalley supergroups.
\end{abstract}

\vskip15pt

   \centerline{\smallbf ACKNOWLEDGEMENTS}
 \vskip5pt
   \centerline{\smalrm The author is greatly indebted to prof.\ Akira Masuoka, whom he sincerely thanks for his many valuable}
   \centerline{\smalrm comments.  He also thanks the referee, whose precious remarks helped improving the present work.}

 \vskip41pt

\section{Algebraic supergroups of Cartan type: existence}  \label{sect-existence}

\smallskip

   {\ } \quad   All notation and terminology throughout this note will be as in  \cite{ga2}.  In particular, every supergroup  $ \bG $  we shall consider will be  {\sl fine},  which means that its tangent Lie algebra functor  $ \text{\sl Lie}\,(\bG) $  is of the form  $ \, A \mapsto \text{\sl Lie}\,(\bG)(A) = A \otimes_\bk \fg \, $  (where  $ A $  ranges among all commutative superalgebras) for some Lie superalgebra  $ \fg $  over  $ \bk \, $,  with the additional requirement that  $ \fg_\uno $  (as  $ \bk $--module)  be free of finite rank.
 \vskip3pt
   The main result in  \cite{ga2}  was the construction of the ``Chevalley supergroups'' of Cartan type, denoted  $ \bG_V $  as their construction depends on some suitable  $ \fg $--module  $ V \, $:  these are connected algebraic  $ \bk $--supergroups,  defined over  $ \Z \, $,  such that the complexification of their tangent Lie superalgebra be (finite dimensional) simple of Cartan type.  In particular, this proves that supergroups with such properties do exist.  However, the presentation in  \cite{ga2}  might be obscure on this point, since the construction of  $ \bG_V $  is based upon the choice of  $ V $  and of a suitable lattice inside it, and the existence of such data might be unclear.  This point deserves to be made clear, which is what I am doing in this section.
 \vskip3pt
   Let  $ \fg $  be a complex Lie superalgebra which is simple of Cartan type; let then  $ L_r $  and  $ L_w $  be its root lattice and (integral) weight lattice, respectively.   As explained in  \cite{ga2},  Subsec.\ 4.5, for any Chevalley supergroup  $ \bG_V $  associated with  $ \fg $  and a suitable  $ \fg $--module  $ V $  its group of characters is a lattice $ \, \Lambda := \Lambda_V \, $  that lies between  $ L_r $  and  $ L_w \, $;  indeed, it is the lattice of weights spanned by the weights (for the action of a Cartan subalgebra of  $ \fg \, $)  of  $ V $  itself.  In addition, this  $ V $   has to be finite dimensional, faithful and rational, and also has to contain an admissible lattice, say  $ M \, $.  Thus what we need to show is the following:

\smallskip

\begin{lemma}  \label{lemma-exist-V}
 For any choice of an intermediate lattice  $ \Lambda $  lying between  $ L_r $  and  $ L_w \, $,  there exists a finite dimensional, rational, faithful  $ \fg $--module  which contains an admissible lattice  $ M $  and whose set of weights spans  $ \Lambda \, $.
\end{lemma}

\begin{proof}
 We start recalling that the even part of  $ \fg $  is of the form  $ \, \fg_\zero = \fg_0 \oplus \fg_{\zero^\uparrow} \,  $  where  $ \fg_0 $  is a reductive Lie subalgebra and  $ \fg_{\zero^\uparrow} $  is a nilpotent ideal (cf.\ \cite{ga2},  Subsec. 2.3).
                                                                           \par
   First, by classical theory of reductive Lie algebras, for any  $ \Lambda $  as before there is a faithful, finite dimensional, rational  $ \fg_0 $--module  $ W $  whose weights span  $ \Lambda \, $;  moreover, such a  $ W $  contains a lattice  $ N $  that is ``admissible'', which means stable for the Kostant form  $ K_\Z(\fg_0) $  of  $ U(\fg_0) \, $   ---  cf.\ \cite{bou},  Ch.\ VIII, \S 12.7 (taking into account that  $ \fg_0 $  is always  {\sl simple\/}  but when  $ \fg $  is of type  $ W(n) $,  for then it has a one-dimensional center, and Theorem 2 in  [{\it loc.\ cit.}]  applies again).  Since  $ \fg_{\zero^\uparrow} $  is an ideal in  $ \fg_\zero $  and  $ \, \fg_\zero \big/ \fg_{\zero^\uparrow} \cong \fg_0 \, $,  the same  $ W $  is also a  $ \fg_\zero $--module  (by scalar extension), the lattice  $ N $  being again ``admissible'', i.e.\ stable for  $ K_\Z(\fg_{\zero^\uparrow\!}) \, $;  of course  $ W $  is still rational and finite dimensional.  Since  $ \fg_{\zero^\uparrow} $  is an ideal in  $ \fg_\zero $  and  $ \, \fg_\zero \big/ \fg_{\zero^\uparrow} \cong \fg_0 \, $,  the same  $ W $  is also a  $ \fg_\zero $--module  (by scalar extension), the lattice  $ N $  being again ``admissible'', i.e.\ stable for  $ K_\Z(\fg_{\zero^\uparrow\!}) \, $;  of course  $ W $  is still rational and finite dimensional.
                                                                           \par
   Second, consider  $ \, V := \text{\sl Ind}_{\fg_\zero}^{\,\fg}(W) = U(\fg) \otimes_{U(\fg_\zero)} W \, $:  this is a  $ \fg $--module  which is still finite dimensional and rational; moreover, it contains the admissible lattice  $ \, M := K_\Z(\fg) \otimes_{K_\Z(\fg_\zero)} N \, $,  and by construction the lattice  $ \Lambda_V $  spanned by the weights of  $ V $  is exactly  $ \Lambda $  (since we initially assumed that  $ \, \Lambda \supseteq L_r \, $,  i.e.\  $ \Lambda $  contains all roots of  $ \fg \, $).  Finally, as  $ \fg $  is simple and its action  on  $ V $  is non-trivial,  $ V $  itself is also faithful, as required.
\end{proof}

\vskip7pt

\begin{remarks}
   {\it (a)}\,  Every Cartan type simple Lie superalgebra  $ \fg $  can be realized (or defined, if you wish) as a Lie subsuperalgebra of some  $ \mathfrak{gl}(V) $  for a suitable  $ V $   --- more precisely, as a suitable Lie superalgebra of superderivations of a Grassmann algebra, say  $ \bk\big[\xi_1,\dots,\xi_n\big] \, $,  which stands for  $ V \, $  (the ``standard representation'').  It so happens that such a  $ \fg $--module  $ \, V := \bk\big[\xi_1,\dots,\xi_n\big] \, $  is finite dimensional, faithful and rational, and in addition it contains an admissible lattice, namely  $ \, M := \Z\big[\xi_1,\dots,\xi_n\big] \, $:  so everything is in place to construct the corresponding Chevalley supergroup  $ \bG_V \, $.  In Sec.\ 5 of  \cite{ga2}  this construction is explicitly carried out for type  $ W(n) \, $;  in this case, the associated lattice of weights is  $ L_w \, $,  the  {\sl full\/}  lattice of weights of  $ \fg \, $.  One can clearly do the same, along the same lines, for types  $ S \, $,  $ \widetilde{S} $  and  $ H $  still using the standard representation  $ \, V := \bk\big[\xi_1,\dots,\xi_n\big] \, $.
                                                                       \par
   More generally, the weight lattices  $ \Lambda $  between  $ L_r $  and  $ L_w $  are in bijections with the sublattices of the quotient  $ \, L_w \big/ L_r \; $:  but then   --- see \S 4.27 in  \cite{ga2}  ---   there are very few possibilities for  $ \Lambda \, $,  namely four cases for type  $ H(2r) \, $,  two cases for type  $ H(2r+1) \, $,  and just one case for types  $ W \, $,  $ S $  and  $ \widetilde{S} \, $.  In particular, for the last three cases the construction of Chevalley supergroups  $ \bG_V $  with  $ V $  the standard representation exhausts all possibilities.
 \vskip3pt
   {\it (b)}\,  The arguments used to prove  Lemma \ref{lemma-exist-V}  above also apply to give a similar result for the case when  $ \fg $  is simple of  {\sl classical\/}  type: one only needs minimal adaptations, actually simplifications, because  $ \fg_\zero $  is reductive (there is no ``extra nilpotent part'' such as  $ \fg_{\zero^\uparrow} \, $,  say).  As a consequence, one has a proof of the fact that ``Chevalley supergroups of classical type'' as considered in  \cite{fg1},  \cite{fg2},  \cite{fg3}  and  \cite{ga1}  actually do exist.
 \vskip4pt
   {\it (c)}\,  It is proved in  \cite{ga2}, Proposition 4.26,  that, under mild assumptions, every Chevalley supergroup of Cartan type  $ \bG_V $  is a  {\sl closed\/}  supersubgroup of  $ \text{GL}(V) \, $.  Actually, these conditions are slightly ill settled in the statement of that Proposition: indeed, instead of
 \vskip4pt
   \centerline{ \it ``Assume that  $ \fg_\uno $  as a  $ \bk $--submodule  of  $ {\mathfrak{gl}(V)}_\uno $  is a direct summand'' }
 \vskip2pt
\noindent
 one should read
 \vskip-1pt
   \centerline{ \it ``Assume that  $ \fg_\uno $  as a  $ \bk $--submodule  of  $ {\mathfrak{gl}(V)}_\uno $ }
 \vskip1pt
   \centerline{ \it is a direct summand with a  $ \bk $--free  complement'' }
%
%
 \eject

\noindent
 or (what amounts to be the same)
 \vskip2pt
   \centerline{ \it ``Assume that the  $ \bk $--module  $ \, {\mathfrak{gl}(V)}_\uno \Big/ \fg_\uno \, $  is free''. }
 \vskip3pt
\noindent
 In fact, the extra condition of  ``having a  $ \bk $--{\sl free\/}  complement''  was actually used in the proof of the Proposition, but it was not mentioned in the statement itself.  By the way, when  $ \bk $  is  {\sl local\/}  this extra condition automatically holds true, by Kaplansky's theorem.
\end{remarks}

\bigskip

\section{Splittings for supergroups and Hopf superalgebras}

\smallskip

   {\ } \quad   In the following we need the notion of ``splitting'' for both supergroups and Hopf superalgebras.  We take it from  \cite{ga3},  where further details may be found.  Hereafter, we shall think of  $ \bk $  as being a totally even superalgebra.

\vskip9pt

%
%
%

\begin{free text}  \label{strong-split_salg}
 {\bf Strongly split Hopf superalgebras.}  \ Let
 $ H $  be any commutative Hopf  $ \bk $--superalgebra.  Then  $ \, J_H := H_\uno^{\,2} \oplus H_\uno \, $  is in fact a  {\it Hopf ideal\/}  of  $ H \, $,  hence  $ \, \overline{H} := H \big/ J_H \, $  is a classical (i.e.~super but with trivial odd component) commutative Hopf algebra.  Moreover, the coproduct of  $ H $  induces a structure of super left  $ \overline{H} $--comodule  on  $ H \, $  (via the projection  $ \, H \relbar\joinrel\twoheadrightarrow \overline{H} \; $),  such that  $ H $  is a counital super left  $ \overline{H} $--comodule  $ \bk $--algebra.
                                                                         \par
   Now let  $ \, \epsilon : H \longrightarrow \bk \, $  be the counit map, let  $ \, H^+ := \hbox{\sl Ker}\,(\epsilon) \, $,  $ \, H_\zero^+ := H_\zero \cap H^+ \, $,  $ \, W^H := H_\uno \big/ H_\zero^+ H_\uno \, $  and consider  $ \, \bigwedge W^H \, $.  Then  $ \; \overline{H} \otimes \bigwedge W^H \; $  has a natural structure of a  commu-\break{}tative  superalgebra, endowed with a natural ``augmentation'' map (i.e.\ a  $ \bk $--valued  morphism of  $ \bk $--superalgebras);  moreover, the coproduct of  $ \overline{H} $  induces on  $ \; \overline{H} \otimes \bigwedge W^H \; $  a super left  $ \overline{H} $--comodule  structure, so  {\it  $ \; \overline{H} \otimes \bigwedge W^H \, $  is a super counital left  $ \overline{H} $--comodule  $ \bk $--algebra}.
 \vskip4pt
   The notion of ``strongly split'' (commutative) Hopf superalgebra, essentially due to Masuoka   --- as the core idea was already in his papers  \cite{ma1}  and  \cite{ma2}  ---   but the present terminology is borrowed from  \cite{ga3},  reads as follows: a commutative Hopf superalgebra  $ H $  as above is said to be  {\it strongly split\/}  if  $ \, W^H $  is  $ \bk $--free  and there is an isomorphism  $ \, \zeta : H \,{\buildrel \cong \over {\lhook\joinrel\relbar\joinrel\relbar\joinrel\twoheadrightarrow}}\, \overline{H} \otimes_\bk \bigwedge \! W^H \; $  of super counital left  $ \overline{H} $--comodule  $ \bk $--algebras. In particular, Masuoka proved that any commutative Hopf superalgebra over  $ \bk $  is automatically strongly split when  $ \bk $  is a field whose characteristic is not 2\,:  cf.\  \cite{ma1},  Theorem 4.5.
%
%
\end{free text}

\medskip

%
%
%

\begin{free text}  \label{glob-split_sgrp}
 {\bf Global splittings for supergroups.}  \ Let
 $ \bG $  be an (affine)  {\sl supergroup\/}  over  $ \bk \, $,  $ \, H := \cO(\bG) \, $  the Hopf  $ \bk $--superalgebra  representing it, and  $ \; \overline{H} := H \Big/ J_H
 = H_\zero \Big/ {H_\uno}^{\!2} \; $,  which is a (classical) commutative Hopf  $ \bk $--algebra.  The affine group-scheme  $ \, \bG_{\text{\it \!ev}} \, $  represented by  $ \; \overline{H} = \overline{\cO(\bG)} \; $   --- so that  $ \, \cO\big(\bG_{\text{\it ev}}\big) = \overline{\cO(\bG)} \; $  ---   is called  {\it the classical supersubgroup(-scheme) associated with  $ \bG \, $}.
%
%
  The projection  $ \, \pi : H \! := \cO(\bG) \! \relbar\joinrel\twoheadrightarrow \cO\big(\bG_{\text{\it ev}}\big) = \overline{H} \, $  yields an embedding  $ \, j : \bG_{\text{\it ev}} \! \lhook\joinrel\longrightarrow \bG \, $,  \, so  $ \bG_{\text{\it ev}} $  identifies with a closed (super)subgroup of  $ \bG \, $.  Moreover, every closed supersubgroup of\/  $ \bG $  which is  {\sl classical}  is a closed subgroup of  $ \, \bG_{\text{\it ev}} \; $.

\medskip

   Here is now the definition of ``globally split supergroups''.  Let  $ \bG $  be an affine  $ \bk $--supergroup for which there exists a closed subsupercheme  $ \bG_- $  of  $ \bG \, $,  stable by the adjoint  $ \bG_{\text{\it ev}} $--action,  such that
 \vskip3pt
   {\it (a)}\,  $ \; 1_{{}_\bG} \! \in \bG_- \, $,  hence we look at  $ \bG_- $  as a  {\sl pointed superscheme},
 \vskip-1pt
   {\it (b)}\,  the product in  $ \bG $  restricts to an isomorphism  $ \; \bG_{\text{\it ev}} \times \bG_- \;{\buildrel \cong \over {\lhook\joinrel\relbar\joinrel\relbar\joinrel\twoheadrightarrow}}\; \bG \; $  of pointed left  $ \bG_{\text{\it ev}} $--superschemes  (which will be called a  {\sl (global) splitting\/}  of  $ \bG \, $),
 \vskip1pt
   {\it (c)}\,  $ \bG_- $  is isomorphic to a totally odd affine superscheme  $ \mathbb{A}_\bk^{0|d_-} $,  as a pointed superscheme.
 \vskip5pt
\noindent
 When all this holds, we say that  $ \bG $  is  {\it globally strongly split},  or in short that it is  {\it gs-split}.  As the referee kindly suggested, a very inspiring (and suggestive) alternative terminology might be that of  {\sl ``equivariantly split''\/}  supergroup, which stresses the fact the splitting of such a supergroup  $ G $  is  $ \bG_{\text{\it ev}} $--equivariant; nevertheless, we have adopted here the terminology of  \cite{ga3}  as we quote results from there.
%
%
%

\vskip8pt

   The link with ``splittability'' of (commutative) Hopf superalgebras is the one we could expect (cf.\  Theorem 4.5 in  \cite{ma1},  as well as Theorem 3.2.8 and Corollary 3.2.9 in  \cite{ga3}):
\end{free text}

\vskip9pt

\begin{theorem}  \label{gl-str-split_sgroup-Th}
 Let  $ \bG $  be an affine supergroup, defined over a ring\/  $ \bk \, $,  and let  $ \, H := \cO(\bG) \, $  be its representing (commutative Hopf)  $ \bk $--superalgebra.  Then  $ \bG $  is globally strongly split if and only if the Hopf superalgebra  $ \cO(\bG) $  is strongly split.  In particular, if\/  $ \bk $  is a  {\sl field}  whose characteristic is not 2, then  $ \bG $  is automatically globally strongly split.
\end{theorem}

\vskip6pt

   Finally, by  Corollary 4.22{\it (c)\/}  and Proposition 4.23 in  \cite{ga2}  one gets the following:

\vskip9pt

\begin{theorem}  \label{Chev-sgroups-gss}
 All Chevalley supergroups of Cartan type  $ \bG_V $  as in  \cite{ga2}  are gs-split.
\end{theorem}

\bigskip

\section{Algebraic supergroups of Cartan type: uniqueness}

\smallskip

   {\ } \quad   As recalled in  Sec.\ \ref{sect-existence},  the main result in  \cite{ga2}  was the construction of the ``Chevalley supergroups'' of Cartan type: this proved the  {\sl existence\/}  of any possible type of connected algebraic  $ \Z $--supergroup  whose complexified tangent Lie superalgebra is (finite dimensional) simple of Cartan type.  On the other hand, the  {\sl uniqueness\/}  question was addressed in Subsection 4.8 of  \cite{ga2},  devoted to proving that any algebraic supergroup with the above mentioned properties (in particular for its tangent Lie superalgebra) is necessarily isomorphic to some Chevalley supergroup of Cartan type.  However, the result and proof presented there were wrong: hereafter I provide a correct (modified) statement and proof, with changes that affect everything from \S 4.38 through Theorem 4.42 in Sec.\ 4.8 of  \cite{ga2}.

\vskip9pt

\begin{free text}  \label{consist-splitt.s}
 {\bf Gs-split,  $ \bk $--split  supergroups of Cartan type and the Uniqueness Theorem.}  \ Let  $ \bG $  be a connected, gs-split  $ \bk $--supergroup; we assume for it that its tangent Lie superalgebra  $ \, \fg := \Lie(\bG) \, $  be a  $ \bk $--form  of a complex Lie superalgebra $ \fg_\C $   --- i.e., there exists a Lie superalgebra  $ \fg_\Z $  over  $ \Z $  such that  $ \, \fg = \bk \otimes_\Z \fg_\Z \, $  and  $ \, \fg_\C = \C \otimes_\Z \fg_\Z \, $  ---   and this  $ \fg_\C $  is simple of Cartan type.  Moreover, we assume that the classical subgroup  $ \bG_{\text{\it ev}} $  of  $ \bG $  has a  $ \bk $--split  maximal torus.  In short, we say that  {\sl  $ \bG $  is a gs-split,  $ \bk $--split  supergroup of Cartan type}.  The group of characters of any  $ \bk $--split  maximal torus in  $ \bG_{\text{\it ev}} \, $,  call it  $ \Lambda \, $, contains the root lattice, since  $ \bG_{\text{\it ev}} $  acts on  $ \fg $  by the adjoint action: hence  $ \Lambda $  is an intermediate lattice lying between  $ L_r $  and  $ L_w $  (notation of Subsec.\ 1).
 \vskip5pt
   Now, for any pair  $ \big( \fg_\C \, , \Lambda \,\big) $  as above there exists a Chevalley  $ \bk $--supergroup  $ \bG_V $  of Cartan type whose associated pair is exactly  $ \big( \fg_\C \, , \Lambda \,\big) \, $   --- cf.\ Sec.\ 1 above.  This yields an ``Existence Theorem'' for gs-split,  $ \bk $--split  supergroups of Cartan type.
                                                              \par
   A related ``Uniqueness Theorem'' was presented in Subsec.\ 4.8 of  \cite{ga2}.  However, it was based on a wrong analysis, so it requires crucial amendments.  In fact, I shall present a double version of such a result, one holding true for any ring   --- but requiring a stronger assumption than that concerning  $ \Lie(\bG) $  ---   and one that applies only for fields of zero characteristic   -- for which the above requirement on  $ \Lie(\bG) $  is enough.
%
\end{free text}

   Here comes the first result:

\vskip3pt

\begin{theorem}  \label{theor-uniqueness-1}
 Let  $ \bG $  be a connected gs-split,  $ \bk $--split  supergroup of Cartan type, let  $ \big(\, \fg_\C \, , \Lambda \,\big) $  be its associated pair and let  $ \bG_V \! $  be the Chevalley supergroup of Cartan type whose group is associated with the same pair  $ \big(\, \fg_\C \, , \Lambda \,\big) \, $.
 Assume in addition that  $ \bG_{\text{\it ev}} $  is isomorphic to  $ {\big( \bG_V \big)}_{\text{\it ev}} \, $.  Then  $ \bG $  is isomorphic to  $ \bG_V \, $.
\end{theorem}

\begin{proof}
 To begin with, recall that the classical subgroup  $ {(\bG_V)}_{\text{\it ev}} $  of  $ \bG_V $  splits into a semidirect product  $ \, \bG_V = \bG^V_0 \ltimes \bG^V_{\zero^\uparrow} \, $,  where  $ \bG^V_0 $  is a connected  $ \bk $--split  subgroup with  $ \, \Lie\,\big(\bG^V_0\big) = \fg_0 \, $  and group of characters  $ \Lambda \, $,  and  $ \bG^V_{\zero^\uparrow} $  is a connected, unipotent normal subgroup with  $ \, \Lie\,\big(\bG^V_{\zero^\uparrow}\big) = \fg_{\zero^\uparrow} \, $.  In addition, by construction  $ \bG^V_0 $  is obtained via a classical procedure ``\`a la Chevalley''   --- cf.\ Proposition 4.9 in  \cite{ga2}  ---  based on  $ V $  thought of as a  $ \fg_0 $--module,  the Kostant form of  $ U(\fg_0) \, $,  etc.  Similarly,  $ \bG^V_{\zero^\uparrow} $  too is realized via a construction ``\`a la Chevalley'', which roughly speaking ``integrates'' the nilpotent Lie algebra  $ \fg_{\zero^\uparrow} $  (linearized through  $ V \, $).  These Chevalley constructions for  $ \bG^V_0 $  and  $ \bG^V_{\zero^\uparrow} $  are realized simul\-taneously as parts of the overarching procedure which constructs  {\sl all\/}  of  $ \bG_V \, $.  In addition (as a consequence), the adjoint action of  $ {(\bG_V)}_{\text{\it ev}} $  onto  $ \, \Lie\big(\bG_V\big) = \fg \, $  is uniquely determined by the adjoint action of  $ \, \Lie\big({(\bG_V)}_{\text{\it ev}\,}\big) = \fg_\zero \, $  onto  $ \, \Lie\big(\bG_V\big) = \fg \, $
  \hbox{as well as by
  $ \Lambda \, $.}
                                                                                    \par
   From the isomorphism  $ \, \phi : \bG_{\text{\it ev}} \cong {\big( \bG_V \big)}_{\text{\it ev}} \, $  we get a corresponding decomposition of  $ \bG_{\text{\it ev}} $  as a semidirect product, and also that the adjoint action of  $ \bG_{\text{\it ev}} $  onto  $ \, \Lie\big(\bG\big) = \fg \, $  is uniquely determined by the action of  $ \, \Lie\big(\bG_{\text{\it ev}\,}\big) = \fg_\zero \, $  onto  $ \, \Lie\big(\bG\big) = \fg \, $  and  $ \Lambda \, $.
                                                                                    \par
   Putting all this together we get, in the language of  \cite{ga3}  one can express this by saying that the ``super Harish-Chandra pairs'' of  $ \bG_V $  and  $ \bG \, $,  namely  $ \, \big(\! {\big( \bG_V \big)}_{\!\text{\it ev}} \, , \fg \big) \, $  and  $ \, \big( \bG_{\text{\it ev}} \, , \fg \big) \, $,  are isomorphic.  Now, the main result in  \cite{ga3}  is exactly   --- cf.\ Theorem 4.3.14 therein ---   that the (suitably defined) category of ``super Harish-Chandra pairs'' is equivalent to the category of (fine) gs-split supergroups.  But the latter category contains both our supergroups  $ \bG $ and  $ \bG_V $  (by assumption for the former, and by the remark at the beginning of the proof for the latter): therefore, we can conclude that  $ \bG $  is isomorphic to  $ \bG_V \, $.
\end{proof}

\vskip3pt

   The second result is a direct consequence:

\vskip3pt

\begin{theorem}  \label{theor-uniqueness-2}
 Let  $ \bG $  be a connected gs-split,  $ \bk $--split  supergroup of Cartan type, let  $ \big(\, \fg_\C \, , \Lambda \,\big) $  be its associated pair and let  $ \bG_V \! $  be the Chevalley supergroup of Cartan type associated with the  $ \big(\, \fg_\C \, , \Lambda \,\big) \, $.  Moreover, assume that  $ \, \bk $  is a  {\sl field}  of characteristic zero.
                                                           \par
   Then  $ \bG $  is isomorphic to  $ \bG_V \, $.
\end{theorem}

\begin{proof}
 We recall from the proof of  Theorem \ref{theor-uniqueness-1}  above we know that  $ {(\bG_V)}_{\text{\it ev}} $  of  $ \bG_V $  splits as  $ \, \bG_V = \bG^V_0 \ltimes \bG^V_{\zero^\uparrow} \, $,  where  $ \bG^V_0 $  is connected,  $ \bk $--split  reductive with  $ \, \Lie\,\big(\bG^V_0\big) = \fg_0 \, $  and group of characters  $ \Lambda \, $,  and  $ \bG^V_{\zero^\uparrow} $  is a connected, unipotent, normal with  $ \, \Lie\,\big(\bG^V_{\zero^\uparrow}\big) = \fg_{\zero^\uparrow} \, $.  Moreover, the (conjugacy) action of  $ \bG^V_0 $  onto  $ \bG^V_{\zero^\uparrow} $  is entirely encoded by the adjoint action of  $ \fg_0 $  onto  $ \fg_{\zero^\uparrow} $  and by the lattice  $ \Lambda \, $.
                                                                                    \par
   On the other hand, the classical subgroup  $ \bG_{\text{\it ev}} $  of  $ \bG $  has tangent Lie algebra  $ \, \Lie\,\big(\bG_{\text{\it ev}}\big) = \fg_\zero = \fg_0 \oplus \fg_{\zero^\uparrow} \, $,  where  $ \fg_0 $  is a reductive Lie subalgebra and  $ \fg_{\zero^\uparrow} $  is a nilpotent ideal of  $ \fg_\zero $  (see the proof of  Lemma \ref{lemma-exist-V}).  Since  $ \bk $  is a field of characteristic zero, we have a Chevalley decomposition of  $ \bG_{\text{\it ev}} $  into a semidirect product  $ \, \bG_{\text{\it ev}} = \bG_0 \ltimes R_u\big(\bG_\zero\big) \, $,  where  $ R_u\big(\bG_\zero\big) $  is the unipotent radical of  $ \bG_\zero \, $;  then this  $ R_u\big(\bG_\zero\big) $  is a connected, unipotent normal subgroup with  $ \, \Lie\,\big(R_u\big(\bG_\zero\big)\big) = \fg_{\zero^\uparrow} \, $,  while  $ \, \bG_0 \, \big(\! \cong \bG \big/ R_u\big(\bG_\zero\big) \,\big) \, $  is a connected,  $ \bk $--split  reductive subgroup with  $ \, \Lie\,\big(\bG_0\big) = \fg_0 \, $  and group of characters the lattice  $ \Lambda \, $.
                                                                                    \par
   By classification theory of  $ \bk $--split  reductive groups,  $ \bG_0 $  corresponds to the pair  $ \big( \fg_0 \, , \Lambda \big) \, $;  but  $ \bG^V_0 \, $,  which also is  $ \bk $--split  reductive, corresponds to the same pair as well, whence there exists an isomorphism  $ \, \bG_0 \cong \bG^V_0 \, $.  Similarly, in the classification of connected unipotent (algebraic) group-schemes over fields of characteristic zero both  $ \bG^V_{\zero^\uparrow} $  and  $ R_u\big(\bG_\zero\big) $  correspond to the same nilpotent Lie algebra (namely  $ \fg_{\zero^\uparrow} \, $),  hence there exists an isomorphism between them.  It follows then, in addition, that the (conjugacy) action of  $ \bG_0 $  onto  $ R_u\big(\bG_\zero\big) $  is entirely encoded by the adjoint action of  $ \fg_0 $  onto  $ \fg_{\zero^\uparrow} $  and by the lattice  $ \Lambda \, $.
                                                                                    \par
   Comparing now  $ {(\bG_V)}_{\text{\it ev}} $  and  $ \bG_{\text{\it ev}} $  we conclude that both are semidirect products, with pairwise isomorphic factors, and the action of the reductive factor onto the unipotent (normal) one is ruled in the same way; so these semidirect products are isomorphic, i.e.\  $ \, \bG_{\text{\it ev}} \cong {(\bG_V)}_{\text{\it ev}}\; $.  Then  Theorem \ref{theor-uniqueness-1}  applies,  and we find that  $ \bG $  is isomorphic to  $ \bG_V \, $.
\end{proof}

\medskip

\begin{remark}
 The same statements as in  Theorems \ref{theor-uniqueness-1}  and  \ref{theor-uniqueness-2}  above also hold true in the case of a simple  $ \fg $  of  {\sl classical\/}  type (and  $ G_V $  is a ``Chevalley supergroup'' as in  \cite{fg2};  indeed, one proves them via the same arguments, and in the second case the proof is even simpler, as  $ \fg_\zero $  is reductive.  This yields another, more general proof of the fact that ``Chevalley supergroups of  {\sl classical\/}  type'' are unique up to isomorphism (cf.\  \cite{fg3}).
\end{remark}